\documentclass[11pt]{amsart}
\usepackage[utf8]{inputenc}
\numberwithin{equation}{section}
\usepackage[total={6.5in,9in},top=1in, left=1in, right=1in, bottom=1in]{geometry}
\usepackage{fancyhdr}
\usepackage[usenames,dvipsnames]{color}
\usepackage{graphics, setspace}
\newcommand{\mathsym}[1]{{}}
\newcommand{\unicode}[1]{{}}
\usepackage{algorithm}
\usepackage[noend]{algpseudocode}
\usepackage{color,amsmath,amssymb,mathtools,graphicx}
\usepackage{libertine}
\usepackage[libertine]{newtxmath}
\usepackage{bbold}
\usepackage{epsfig,fancyhdr}
\usepackage{dsfont} 
\usepackage{graphicx,float}
\usepackage{graphicx}
\usepackage{epsfig,fancyhdr,setspace}
\usepackage{epstopdf}
\usepackage{wrapfig}
\usepackage{soul}
\usepackage{import}
\usepackage{lineno}
\usepackage{stackrel}
\usepackage[allcolors = blue,colorlinks]{hyperref}
\usepackage[abs]{overpic}
\usepackage[center]{caption}
\usepackage{subcaption}
\usepackage{tikz}
\usepackage{enumitem}
\counterwithin{figure}{section}
\usepackage{tikz-cd}
\usepackage{wasysym}
\usepackage{marvosym}
\usepackage{extarrows}
\usepackage{amsmath}
\usepackage{graphicx,amsmath,mathtools,float, algpseudocode}
\usepackage{setspace}
\graphicspath{{./images/}}
\usepackage{stackrel}
\usepackage{yfonts}
\usetikzlibrary{decorations.pathmorphing}
\usepackage{enumitem, hyperref}\makeatletter
\def\namedlabel#1#2{\begingroup
    #2%
    \def\@currentlabel{#2}%
    \phantomsection\label{#1}\endgroup
}
\makeatother
\usepackage{type1cm} 
\usepackage{blindtext}
\usepackage{scalerel,stackengine}
\usepackage{nicematrix}
\NiceMatrixOptions{
code-for-first-row = \color{blue} ,
code-for-last-row =  ,
code-for-first-col = \color{blue} ,
code-for-last-col = \color{blue}
}

\usepackage{tocbasic}

\newcommand{\zero}{\vcenter{\hbox{\includegraphics[scale=.25]{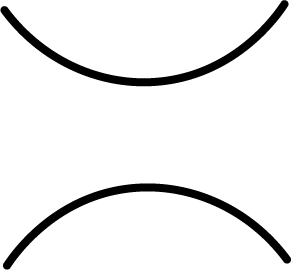}}}}
\newcommand{\one}{\vcenter{\hbox{\includegraphics[scale=.25]{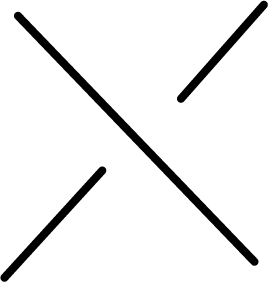}}}}
\newcommand{\none}{\vcenter{\hbox{\includegraphics[scale=.25]{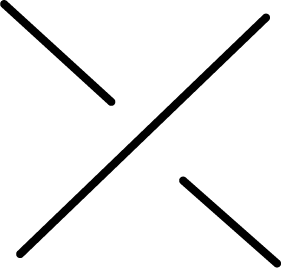}}}}
\newcommand{\infinity}{\vcenter{\hbox{\includegraphics[scale=.25]{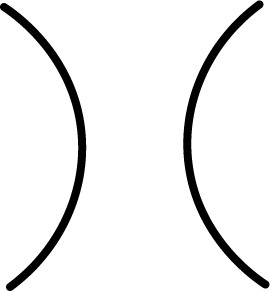}}}}

\newcommand{\numnone}{\vcenter{\hbox{\includegraphics[scale=.25]{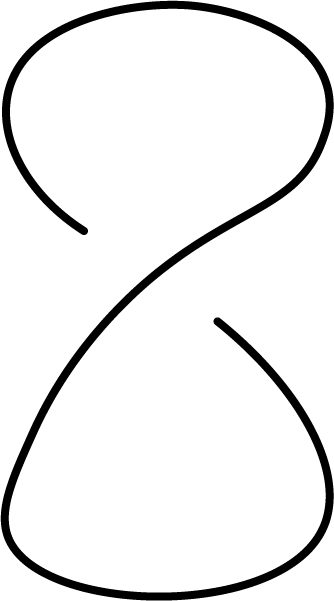}}}}
\newcommand{\numzero}{\vcenter{\hbox{\includegraphics[scale=.25]{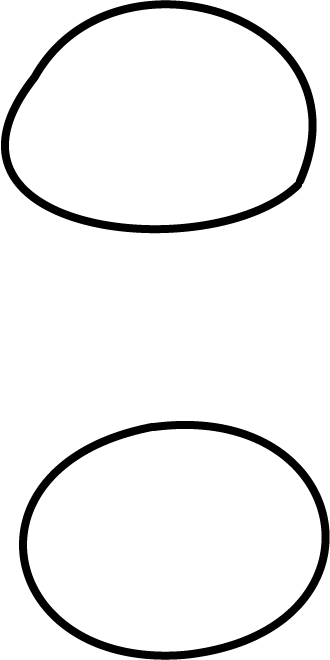}}}}
\newcommand{\numone}{\vcenter{\hbox{\includegraphics[scale=.25]{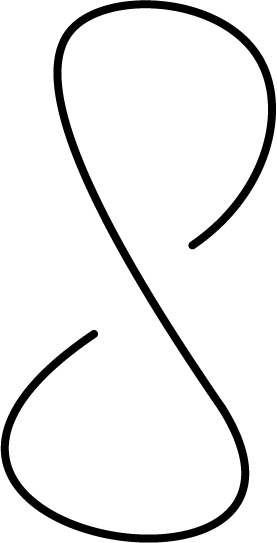}}}}
\newcommand{\numinf}{\vcenter{\hbox{\includegraphics[scale=.25]{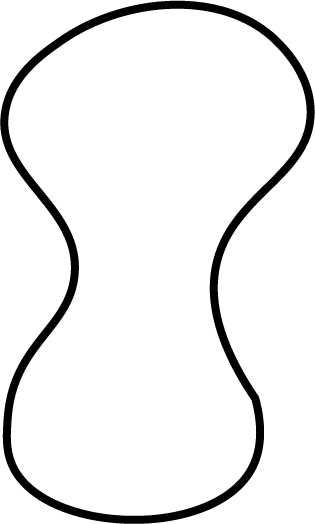}}}}
\newcommand{\dennone}{\vcenter{\hbox{\includegraphics[scale=.25]{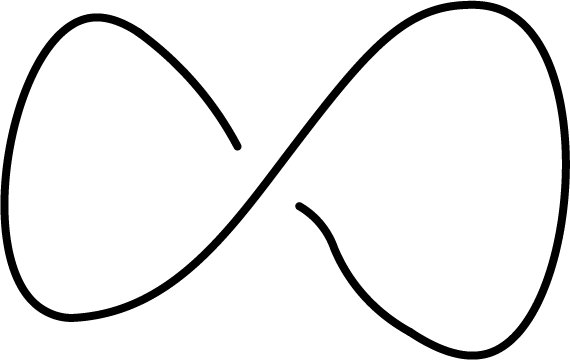}}}}
\newcommand{\denzero}{\vcenter{\hbox{\includegraphics[scale=.25]{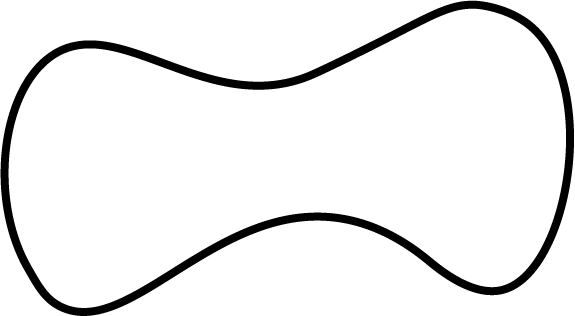}}}}
\newcommand{\denone}{\vcenter{\hbox{\includegraphics[scale=.25]{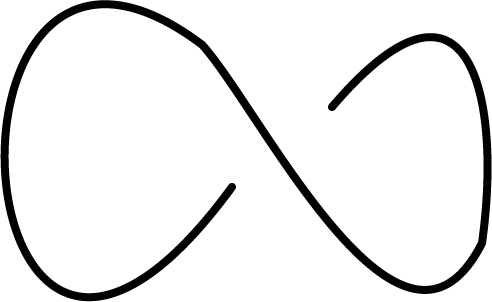}}}}
\newcommand{\deninf}{\vcenter{\hbox{\includegraphics[scale=.25]{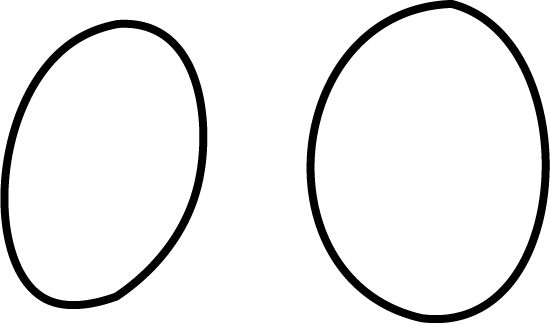}}}}

\newcommand{\Dzero}{\vcenter{\hbox{\includegraphics[scale=.08]{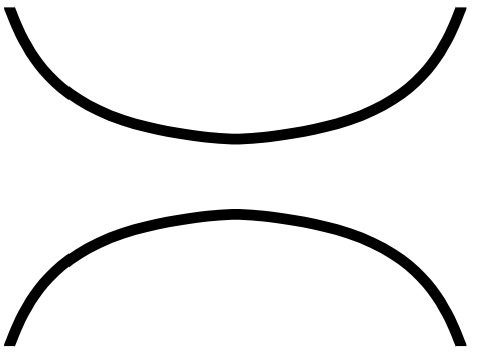}}}}
\newcommand{\Dminus}{\vcenter{\hbox{\includegraphics[scale=.08]{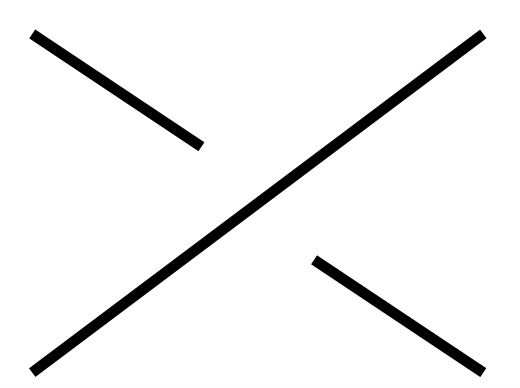}}}}
\newcommand{\Done}{\vcenter{\hbox{\includegraphics[scale=.08]{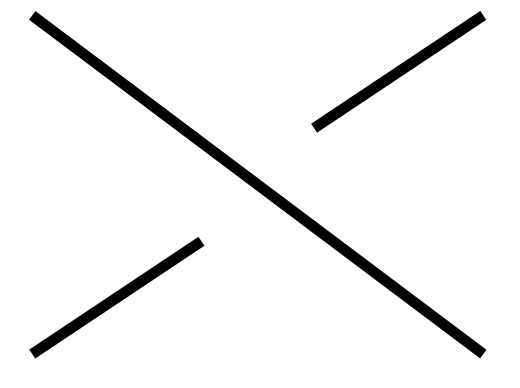}}}}
\newcommand{\Dtwo}{\vcenter{\hbox{\includegraphics[scale=.08]{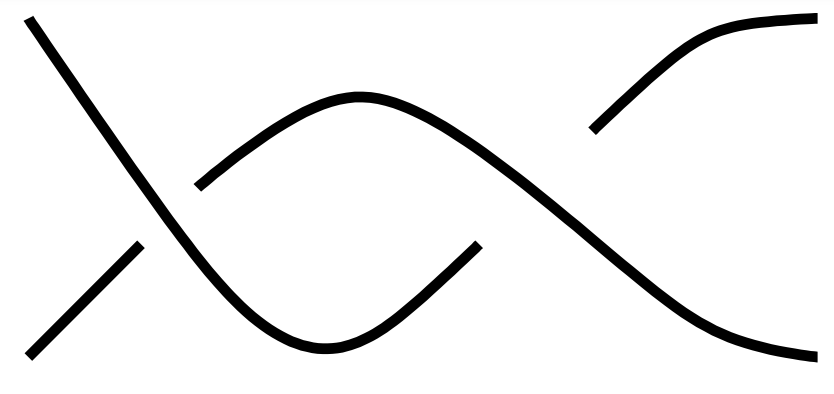}}}}

\newcommand{\Dinfty}{\vcenter{\hbox{\includegraphics[scale=.08]{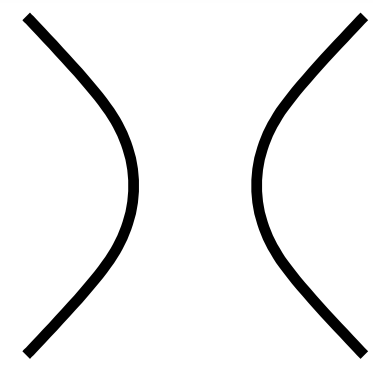}}}}
\newcommand{\Dntwo}{\vcenter{\hbox{\includegraphics[scale=.1]{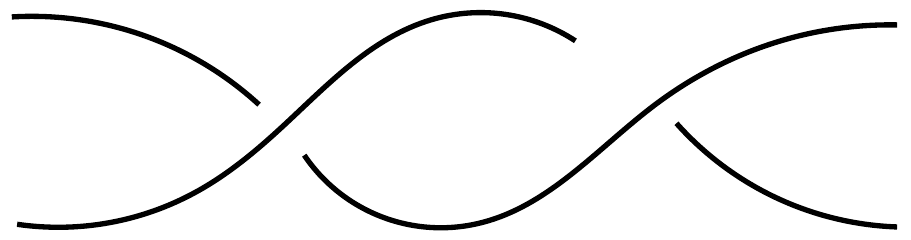}}}}
\newcommand{\Dnzerotwist}{\vcenter{\hbox{\includegraphics[scale=.1]{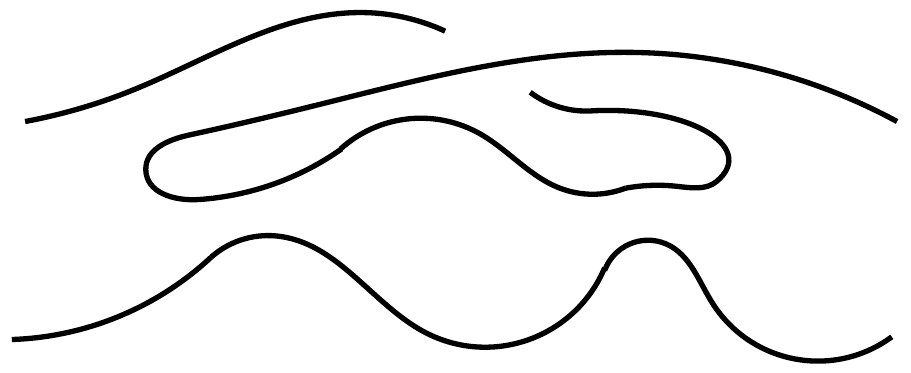}}}}
\newcommand{\Dnonetwist}{\vcenter{\hbox{\includegraphics[scale=.1]{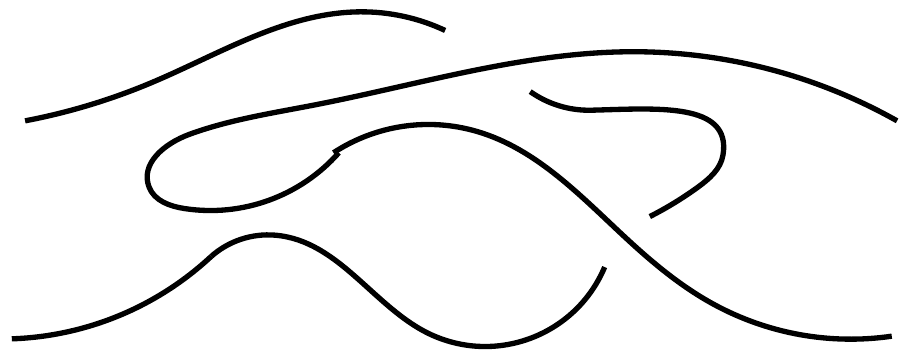}}}}
\newcommand{\Dnnonetwist}{\vcenter{\hbox{\includegraphics[scale=.1]{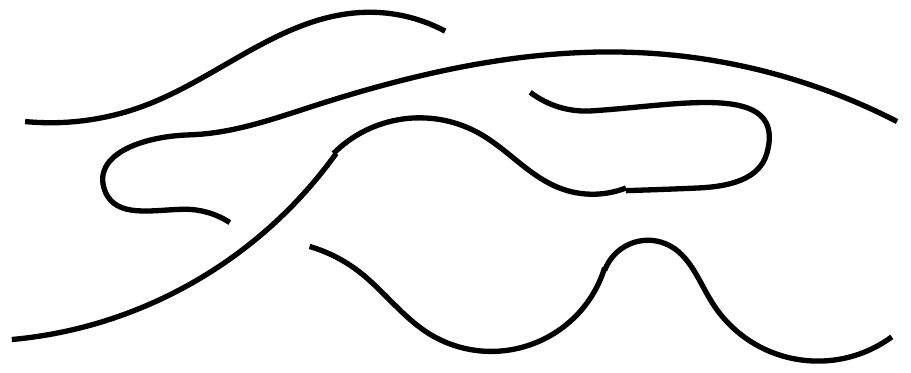}}}}
\newcommand{\Dninftwist}{\vcenter{\hbox{\includegraphics[scale=.1]{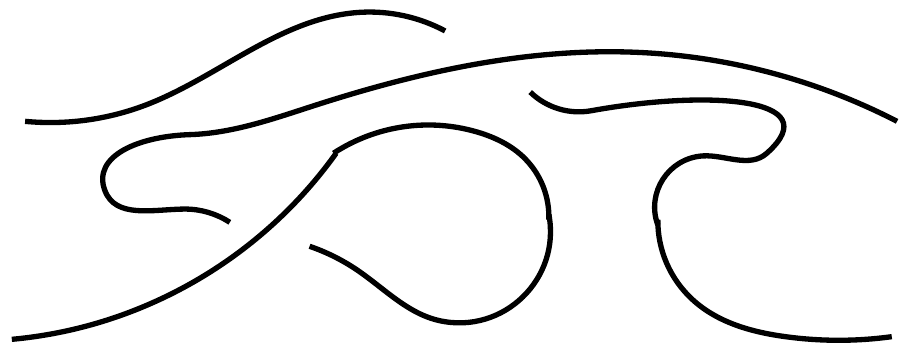}}}}
\newcommand{\Dntwotwist}{\vcenter{\hbox{\includegraphics[scale=.1]{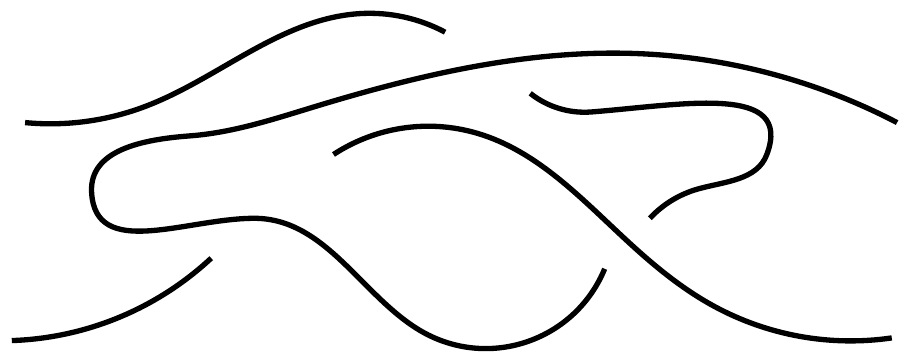}}}}
\newcommand{\poskink}{\vcenter{\hbox{\includegraphics[scale=.1]{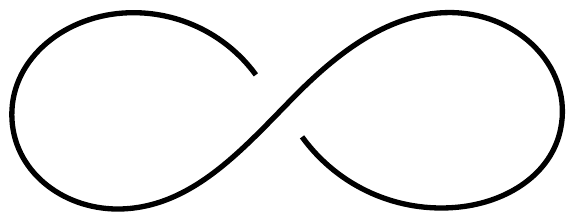}}}}
\newcommand{\negkink}{\vcenter{\hbox{\includegraphics[scale=.1]{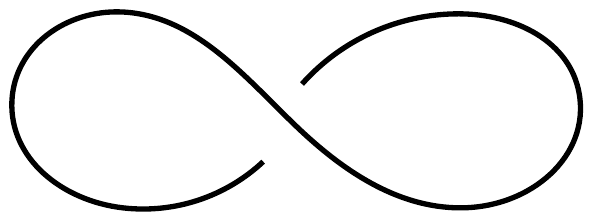}}}}
\newcommand{\poskinkinf}{\vcenter{\hbox{\includegraphics[scale=.1]{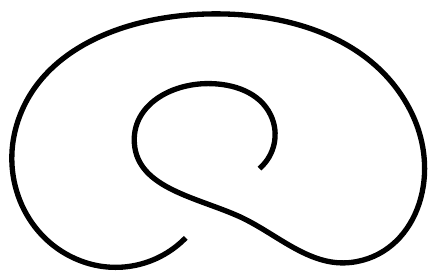}}}}

\newcommand{\postwokinkinf}{\vcenter{\hbox{\includegraphics[scale=.1]{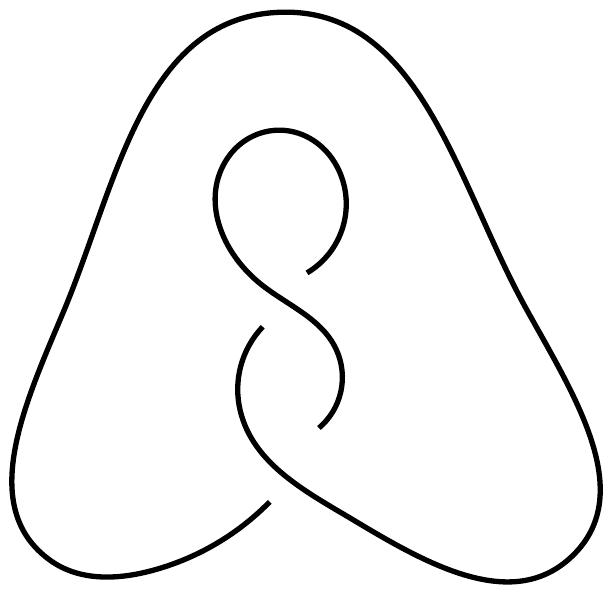}}}}
\newcommand{\trefoilpos}{\vcenter{\hbox{\includegraphics[scale=.1]{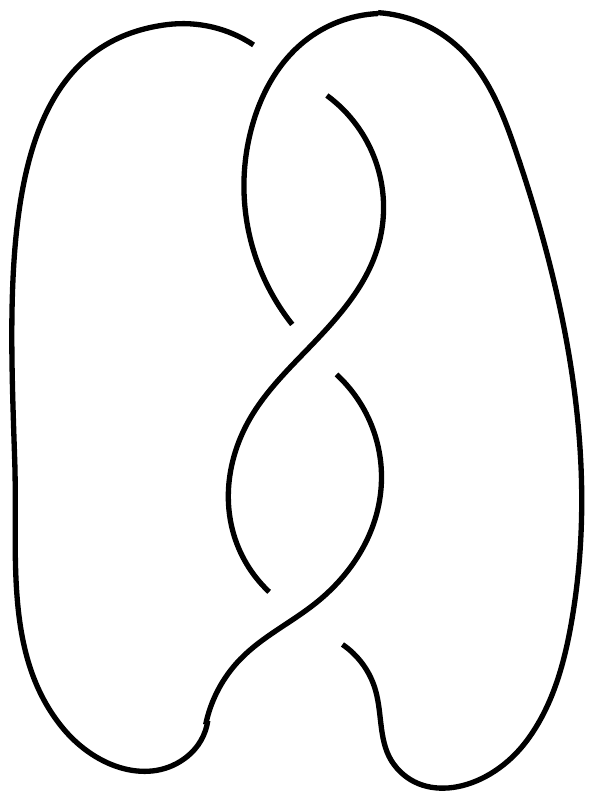}}}}
\newcommand{\trefoilneg}{\vcenter{\hbox{\includegraphics[scale=.1]{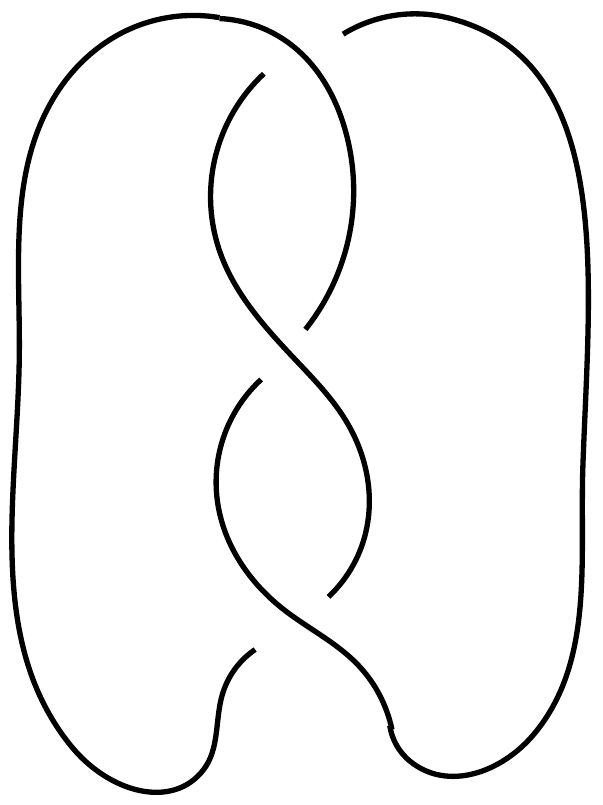}}}}
\newcommand{\figureeightneg}{\vcenter{\hbox{\includegraphics[scale=1]{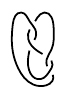}}}}
\newcommand{\figureeightpos}{\vcenter{\hbox{\includegraphics[scale=1]{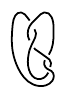}}}}

\newtheorem{theorem}{Theorem}[section]
\newtheorem{lemma}[theorem]{Lemma}
\newtheorem{definition}[theorem]{Definition}
\newtheorem{example}[theorem]{Example}
\newtheorem{proposition}[theorem]{Proposition}
\newtheorem{remark}[theorem]{Remark}

\newtheorem{corollary}[theorem]{Corollary}
\newtheorem{conjecture}[theorem]{Conjecture}

\newtheorem{exercise}[theorem]{Exercise}

\newtheorem{question}[theorem]{Question}

\newtheorem*{namedtheorem}{\theoremname}
\newcommand{\theoremname}{testing}
\newenvironment{named}[1]{\renewcommand{\theoremname}{#1}\begin{namedtheorem}}{\end{namedtheorem}}

\author{Rhea Palak Bakshi}
\address{University of California, Santa Barbara, USA}
\email{{\rm rheapalak@math.ucsb.edu $|$ rheapalakbakshi@gmail.com}}

\author{Anthony Christiana}
\address{Department of Mathematics, The George Washington University, Washington DC, USA}
\email{{\rm ajchristiana@gwmail.gwu.edu}}

\author{Huizheng Guo}
\address{Department of Mathematics, The George Washington University, Washington DC, USA}
\email{{\rm hguo30@gwu.edu}}

\author{Dionne Ibarra}
\address{School of Mathematics, Monash University, Australia}
\email{{\rm dionne.ibarra@monash.edu}}

\author{Louis H. Kauffman}
\address{Department of Mathematics, University of Illinois at Chicago, USA}
\email{{\rm loukau@gmail.com}}

\author{Gabriel Montoya-Vega}
\address{Department of Mathematics, University of Puerto Rico at R\'io Piedras, San Juan, PR, USA }
\email{{\rm gabrielmontoyavega@gmail.com $|$ gabriel.montoya@upr.edu}}

\author{Sujoy Mukherjee}
\address{Department of Mathematics, University of Denver, CO, USA}
\email{{\rm sujoymukherjee.math@gmail.com}}

\author{J\'{o}zef H. Przytycki}
\address{Department of Mathematics, The George Washington University, Washington DC, USA and \newline \indent Department of Mathematics, University of Gda\'{n}sk, Gda\'{n}sk, Poland}
\email{{\rm przytyck@gwu.edu}}

\author{Xiao Wang}

\address{Department of Mathematics, Jilin University, Changchun, China}
\email{{\rm wangxiaotop@jlu.edu.cn}}



\begin{document}

\title{Fundamentals of cubic skein modules}

\maketitle
\begin{abstract}
Over the past thirty-seven years, the study of linear and quadratic skein modules has produced a rich and far-reaching skein theory, intricately connected to diverse areas of mathematics and physics, including algebraic geometry, hyperbolic geometry, topological quantum field theories, and statistical mechanics. However, despite these advances, skein modules of higher degree-those depending on more parameters than the linear and quadratic cases-have received comparatively little attention, with only a few isolated explorations appearing in the literature.  In this article, we undertake a systematic study of the cubic skein module, the first representative of this broader class. We begin by investigating its structure and properties in the $3$-sphere, and then extend the analysis to arbitrary $3$-manifolds. The results presented here aim to establish a foundational framework for the study of higher skein modules, thereby extending the scope of skein theory beyond its classical domains. Furthermore, studying the structure of cubic skein modules may lead to new polynomial invariants of knots. 

\end{abstract}

\tableofcontents

\section{Introduction}
Skein modules are invariants of manifolds that were introduced by Przytycki \cite{Prz2} in 1987 and independently by Turaev \cite{turaevsolidtorus} in 1988. Przytycki's motivation originated from his desire to create an algebraic topology based on knots, the main algebraic object of which is known as a {\it skein module} that is associated to a  
manifold. A skein module is usually constructed as a formal linear combination of embedded
(or immersed) submanifolds of an $n$-dimensional manifold, modulo locally defined relations. In this paper we restrict ourselves to the case of links embedded in oriented $3$-manifolds.\footnote{One may consider relations involving surfaces embedded in $3$-manifolds, as in the Bar-Natan skein module \cite{asaedafrohman}. One may also consider skein modules for manifolds of higher dimensions, such as the skein lasagna module for $4$-manifolds \cite{MWW}.} A skein module of an oriented $3$-manifold is a module composed of linear combinations of links in that manifold, modulo some (usually local) skein relations. By local we imply that the skein relations occur in some $3$-ball in the $3$-manifold.\\

Links in a $3$-manifold may be framed, unframed, oriented, unoriented, or considered up to homotopy, homology, or isotopy. Moreover, we can choose from an arbitrarily large collection of skein relations to construct a skein module. Thus, there are several skein modules that may be associated to a $3$-manifold. 
The most extensively studied of these are linear and quadratic skein modules. The skein relations of linear skein modules involve two horizontal diagrams, and may involve a third vertical diagram. Some examples of linear skein modules are the $q$-homology skein module and Kauffman bracket skein module. On the other hand, the skein relations of quadratic skein modules involve three horizontal diagrams, and may involve a fourth vertical diagram. The Alexander, Jones, HOMFLYPT, Kauffman, and Dubrovnik skein modules are all quadratic, see \cite{Kau, PBIMW}. There is extensive literature on all of these skein modules and we refer the reader to \cite{PBIMW} for a survey. See Figures \ref{skein-linear} and \ref{skein-quadratic1} for illustrations of the diagrams used in linear and quadratic skein relations. The skein relations of these skein modules are special cases of the skein relation given in Equation \ref{ninftyskeinrelation}. \\

\begin{figure}[ht]
\centering
\includegraphics[scale=0.35]{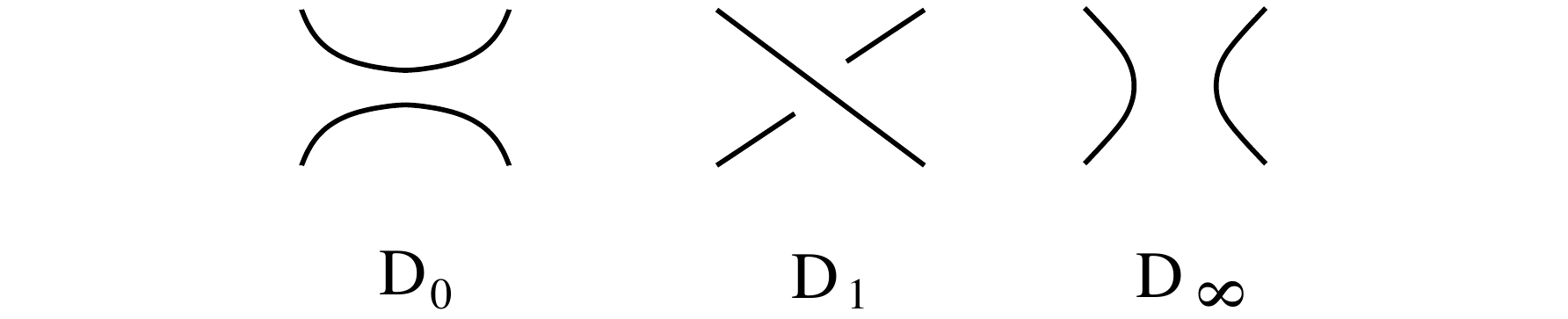} 
\caption{Skein diagrams $D_0,D_1,$ and $D_{\infty}$ for linear skein modules.}
\label{skein-linear}
\end{figure}

\begin{figure}[ht]
\centering
\includegraphics[scale=0.35]{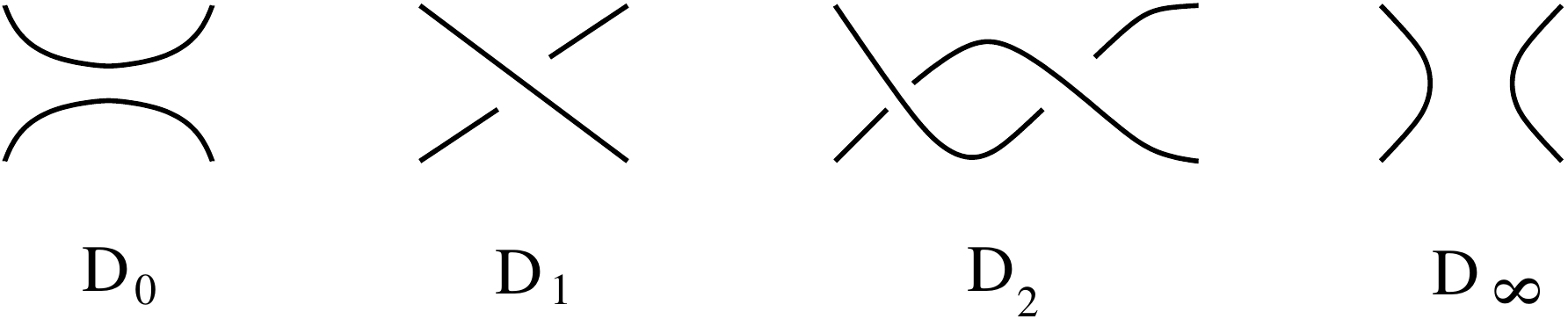} 
\caption{Skein diagrams $D_0,D_1,D_2,$ and $D_{\infty}$ for quadratic skein modules.}
\label{skein-quadratic1}
\end{figure}

The study of linear and quadratic skein modules over the last thirty-eight years has led to a very rich skein theory that is connected to many disciplines of mathematics and physics, such as algebraic geometry, hyperbolic geometry, Topological Quantum Field Theories (TQFTs), and statistical mechanics. There is, however, another class of skein modules with more parameters than the linear and quadratic cases which, save for a few exceptions (see \cite{Cox,Fun,BF,Prz2,CM,Ore}), has been largely neglected until now. 
The cubic skein module is the first object in this class that awaits exploration, first in $S^3$, and then in arbitrary $3$-manifolds. This skein module is the main object of our study in this paper.

\subsection{Skein Modules as Linearizations of Topological Objects}\label{section:skeinlinear}
As envisioned in \cite{Prz2}, skein modules of manifolds are obtained using linear combinations of pairs of manifolds $(M,N)$ where $N$ is properly embedded in $M$. We start this section by discussing a family of relatively general skein modules known as $(n,\infty)$ (or degree $n-1$) unoriented skein modules. When the relations are nonlinear, Przytycki and Traczyk introduced the concept of a Conway algebra (related to the older concept of entropic right quasigroup \cite{PrTr} (see also \cite{Prz6} and Section \ref{Inord})). In this paper, we only consider skein modules where the $3$-dimensional manifold is fixed\footnote{We will  consider $S^3$ for most of the paper.} and the free module generated by links (up to ambient isotopy) is taken modulo linear combinations of properly chosen tangles. The details are in the following subsection.

\subsection{$\boldsymbol{(n,\infty)}$-skein Modules}\label{ninftysm}  
The $(n,\infty)$-skein module of an oriented $3$-manifold $M$ is one whose skein relation involves $n$ horizontal diagrams and one vertical diagram. They are discussed in \cite{Prz2}, \cite{Prz4}, \cite{Prz5}, \cite{PBIMW}, and \cite{PTs}. We denote this skein module by $\mathcal S_{n, \infty}(M)$. The skein relation of $\mathcal S_{n, \infty}(M)$ is: 
 
\begin{equation}\label{ninftyskeinrelation}
    b_0D_0+b_1D_1+...+b_{n-1}D_{n-1} +b_\infty D_\infty =0, 
\end{equation} 
where $D_0,D_1,...,D_{n-1}$, and $D_\infty$ are the framed diagrams illustrated in Figure \ref{skein-n-inf}, subject to the framing relation $D^{(1)}=aD$. Here $D^{(1)}$ denotes the link obtained by twisting the framing of the link $D$ by one positive full twist.  

\begin{figure}[ht]
\centering
\includegraphics[scale=0.30]{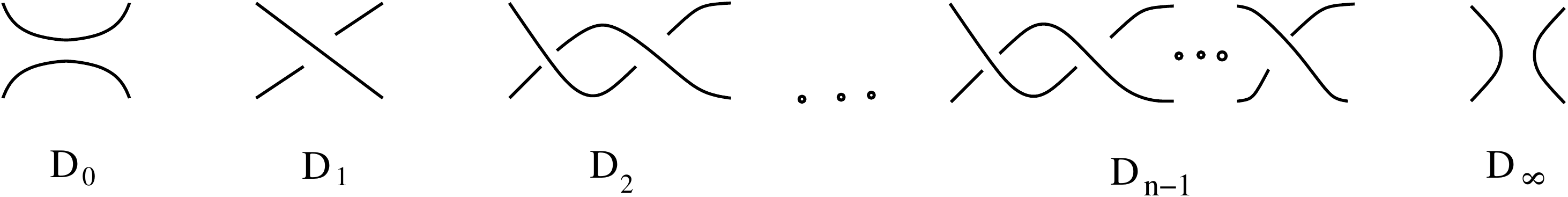} 
\caption{The ($n,\infty$)-skein diagrams $D_0,D_1,$..., $D_{n-1}$, and $D_{\infty}$, assuming blackboard framing.}
\label{skein-n-inf}
\end{figure}

We interpret the skein relation as follows. For any $3$-ball $B^3=B^2\times I$ in an oriented 3-manifold $M$,  we consider $B^3$ to be oriented in agreement with the orientation of $M$. Then for any $(n+1)$ framed $2$-tangles, we consider links, $D_0,D_1,D_2,\hdots, D_{n-1}$, and $D_{\infty}$ which are in agreement outside of $B^3$ and differ only in $B^3 $, where they look like the tangles shown in Figure \ref{skein-n-inf}. 
At this stage we have two main choices: either $b_{\infty}=0$ or  $b_\infty$ is invertible.
In the case of $b_\infty=0$, the $D_i$'s can be diagrams of either oriented or unoriented links; however, if $b_\infty \neq 0$, the diagrams must be unoriented. 

\ 

We may treat the trivial framed knot as a variable $t$, that is, $D\sqcup \bigcirc = tD$. Then the variable $t$ is involved in Equation \ref{eqn:binft} and is obtained by considering the ``half-denominator" 
of the defining relation (see Figure \ref{fig:skein-n-inf-den}). Therefore, in the skein module we get:

\begin{equation}\label{eqn:binft}   
 -tb_\infty D= (b_0+ a^{-1}b_1 +a^{-2}b_2+...+a^{1-n}b_{n-1})D.
\end{equation}

\begin{figure}[ht]
\centering
\includegraphics[scale=0.30]{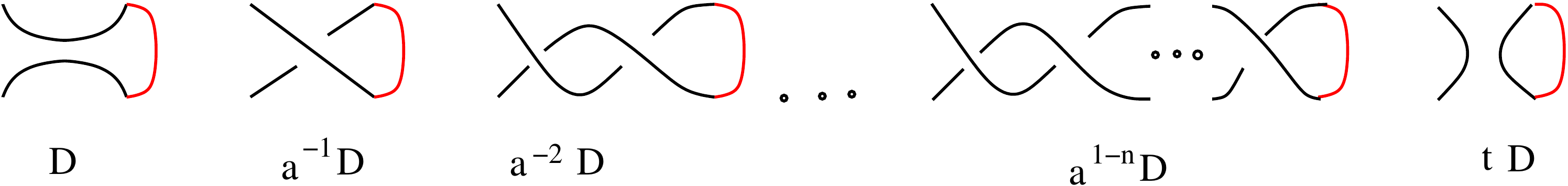} 
\caption{Half-denominator of skein relation diagrams. We assume that diagrams have blackboard framing locally.}
\label{fig:skein-n-inf-den}
\end{figure}

Thus, when $b_\infty$ is invertible we can solve for $t$ to get:
\begin{equation}\label{tgenformula}
t= \frac{b_0+ a^{-1}b_1 +a^{-2}b_2+...+a^{1-n}b_{n-1}}{-b_\infty}.
\end{equation}

Linear, quadratic, and cubic skein modules are special cases of $(n,\infty)$-skein modules. We note that linear skein modules are $(2,\infty)$-skein modules. If $b_\infty$ is invertible, we get the Kauffman bracket skein module denoted by $\mathcal S_{2,\infty}(M; R)$, where 
$R$ is a commutative ring with identity and with chosen elements $b_0=A, b_1=-1, b_\infty=A^{-1}$, and $a=-A^{3}$ where $A$ is invertible.

Quadratic skein modules are $(3,\infty)$-skein modules. If $b_\infty = 0$ and the links are oriented, we get the HOMFLYPT skein module. On the other hand, if $b_{\infty}$ is invertible, we get the Kauffman or Dubrovnik skein modules (see Subsection \ref{secnkaffdubrov}). The cubic skein module, which is the main object of our paper,  is a $(4,\infty)$-skein module, which we denote by $\mathcal S_{4,\infty}(M)$. Thus, its skein relation is:\footnote{The use of ``linear", ``quadratic", and ``cubic" may seem misleading as all relations are linear in many variables. Rather, the names are meant to suggest degree $n$ polynomials, which have $n+1$ coefficients. Here, $b_\infty$ is an additional term which is needed for unoriented links.}

\begin{equation}\label{eqn:CubicSkeinRelation}
b_0D_0+ b_1D_1 + b_2D_2+ b_3D_3+ b_{\infty}D_{\infty}=0, 
\end{equation}
which involves one vertical and three horizontal unoriented framed link diagrams (see Figure \ref{skein-cubicinff}), subject to the framing relation $D^{(1)}=aD$.

\begin{figure}[ht]
\centering
\includegraphics[scale=0.34]{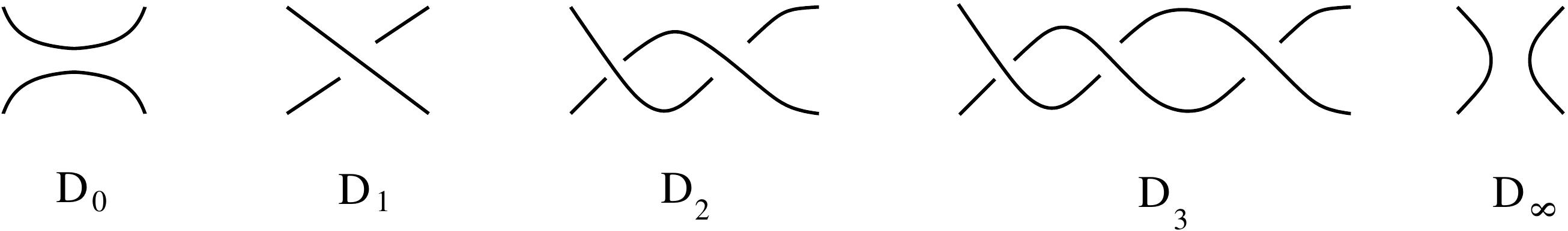} 
\caption{Cubic skein diagrams $D_0,D_1,D_2,D_3$ and $D_{\infty}$.}
\label{skein-cubicinff}
\end{figure}

The cubic skein relation may be thought of as a deformation of a $3$-move (the move from $D_0$ to $D_3$). 
It was earlier thought that every link in $S^3$ could be changed to a trivial link by a finite number of $3$-moves. This was known as the Montesinos-Nakanishi conjecture; \cite{Prz3}. This conjecture was shown in \cite{DP1} to not hold in general. However, it does hold for many families of links, in particular, for $3$-algebraic links; see Section \ref{GS3AT}. Therefore, it is useful to consider the submodule of $ \mathcal{S}_{4,\infty}(S^3)$ generated by trivial links. We denote this submodule by $  S^{(0)}_{4,\infty}(M)$. \\

The paper is organized as follows. 

In Section \ref{section: basics} we define the cubic skein module and discuss some of its properties. We also discuss the choice of the ring for the cubic skein module. In Section \ref{sec:relations}, we discuss relations in the cubic skein module, and compute some initial examples, including the Hopf Relation. In Section \ref{Section 15}, we explore $n$-moves and find closed formulas for some families of links. In Section \ref{Mut}, we analyze when the cubic skein module is preserved under mutation, and present concrete examples. In Section \ref{Inord}, we describe skein relations as $n$-ary operations satisfying an entropic condition, and we interpret this as a locality principle. In Section \ref{quadtocube}, we analyze cubic skein relations coming from quadratic skein relations. In Section \ref{RaTaAl}, we present two algorithms for computing diagrams and relations in the cubic skein module. In Section \ref{MRel}, we present a set of relations which fall outside of the ideal generated by the Hopf Relation. We also present a nonstandard cubic skein relation based on the trefoil. Under the substitutions provided in this section, we pose a question about the structure of a cubic skein module which will not reduce the cubic skein relation into a quadratic one. Studying this skein module may lead to a new polynomial invariant of knots. In Section \ref{GS3AT}, we discuss a generating set for 3-algebraic tangles under the cubic skein relation. In Section \ref{Sec:longerrelations}, we relate our work to Cubic Hecke Algebras, the Links-Gould Invariants, and the Yang-Baxter Equation. In Section \ref{sec:future}, we discuss future research directions and collect conjectures from throughout the paper. In particular, we first make a conjecture about the existence of cubic polynomial invariants. We also conjecture that Fox 7-colorings  are determined by the cubic skein module of $S^3$. Furthermore, we speculate that the wrapping number of a link in the solid torus is determined by the value of that link in the cubic skein module of the solid torus.

\section{Basics of Cubic Skein modules}\label{section: basics}

We precede our discussion of cubic skein modules with some remarks on quadratic also known as $(3, \infty)$-skein modules; in particular the Kauffman and Dubrovnik skein modules. These examples of $(3,\infty)$-skein modules are well-studied, and their definitions here will serve as natural points of comparison for $(4, \infty)$-skein modules, also known as cubic skein modules. They will be defined in Subsection \ref{csmintro}.

\subsection{Kauffman and Dubrovnik Skein Modules}\label{secnkaffdubrov}
 In this subsection, we will work with an oriented $3$-manifold, $M$ and the set of ambient isotopy classes of unoriented framed links, $\mathcal L^{\mathit {fr}}$, in $M$. 
\begin{definition}[Kauffman skein modules]

Let $R = \mathbb Z[a^{\pm 1}, z^{\pm^1}]$ and $R\mathcal L^{\mathit {fr}}$ denote the free $R$-module with basis $\mathcal L^{\mathit {fr}}$. Consider the submodule $\mathcal{S}_{3, \infty}$ of $\mathcal L^{\mathit {fr}}$ generated by the skein relation $D_+ + D_- - zD_0 - z D_\infty $ and the framing relation $D^{(1)} - aD$. The Kauffman skein module of $M$ is defined as the quotient $\mathcal{S}_{3,\infty}(M,R;a,z) = R\mathcal{L}^{\mathit{fr}}/\mathcal{S}_{3, \infty}$.

\end{definition}

Once again, $D^{(1)}$ denotes the link diagram obtained from $D$ by twisting its framing by one positive full twist. The Kauffman skein module generalizes the Kauffman $2$-variable polynomial link invariant to arbitrary $3$-manifolds. We reiterate that the Kauffman skein module is an $(n,\infty)$-skein module where $n =3$. This can be easily seen after comparing its skein relation $D_+ + D_- - zD_0 - z D_\infty $ with the skein relation in Equation \ref{ninftyskeinrelation}. Changing the signs slightly in the skein relation of the Kauffman skein modules gives us the Dubrovnik skein modules, which is a generalization of the Dubrovnik polynomial link invariant. 

\begin{definition}[Dubrovnik skein modules]

Let $R = \mathbb{Z}[a^{\pm 1},z^{\pm 1}]$ and $S^{\mathit{Dub}}_{3, \infty}$ be the submodule of the free $R$-module $R\mathcal{L}^{\mathit{fr}}$ generated by the skein expressions $D_+ -D_- + zD_{\infty} - zD_{0}$ and $D^{(1)} - aD$. The Dubrovnik skein module is defined as the quotient $ S^{\mathit{Dub}}_{3,\infty}(M,R;a,z) = R\mathcal{L}^{\mathit{fr}}/S^{\mathit{Dub}}_{3, \infty}$.
\end{definition}

The Kauffman and Dubrovnik skein modules for the solid torus were computed in \cite{turaevsolidtorus} and are equivalent. In fact, the Kauffman and Dubrovnik skein modules for the product of an oriented surface with the unit interval are infinitely generated and free \cite{lieberumhomfly,mpskauffman}. Since the Kauffman and Dubrovnik polynomials \cite{Kau} are equivalent in $S^3$, it is natural to expect them to be equivalent in any arbitrary $3$-manifold. However, they are inequivalent for lens spaces. In \cite{Mro-3}, Mroczkowski showed that the Kauffman and Dubrovnik skein modules for $\mathbb RP^3$ are not isomorphic. We now define the main object of our paper, the cubic skein module, and discuss some of its basic properties.

\subsection{Cubic Skein Modules} \label{csmintro}

As before, let $M$ be an oriented $3$- manifold and let $\mathcal{L}^{\mathit{fr}}$ denote the set of unoriented framed links in $M$ up to ambient isotopy.

\begin{definition}[Cubic skein modules]

Consider any commutative ring $R$ with elements $a, b_0,b_1,b_2,b_3$, and $b_\infty$ such that the framing variable ``a" is invertible.
Denote by $ R\mathcal{L}^{\mathit{fr}}$ the free $R$-module with basis $\mathcal{L}^{\mathit{fr}}$. Then the cubic skein module of $M$, denoted by
$\mathcal \mathcal{S}_{4,\infty}(M; R,b_0,b_1,b_2,b_3,b_\infty, a^{\pm 1})$, 
is defined to be the quotient of $R\mathcal{L}^{\mathit{fr}}$ by the $R$-submodule generated by cubic skein expressions 
$b_0D_0+b_1D_1+b_2D_2+b_3D_3 +b_\infty D_\infty $ and
framing expression $D^{(1)}-aD$.
\end{definition}

The link diagrams $D_i$ are illustrated in Figure \ref{skein-cubicinff}.  In some cases we denote the submodule generated by the skein and framing expressions to be $\mathcal{S}_{4,\infty}^{sub}$.  We reiterate that the cubic skein module is an $(n,\infty)$-skein module where $n =4$. We can also consider $3$-manifolds with boundary with marked points on the boundary. This gives rise to the relative version of cubic skein modules. 

\begin{definition}[Relative cubic skein modules]

Let $(M, \partial M)$ be a $3$-manifold with boundary with marked points $\{x_i\}_{i=0}^{\infty}$ on $\partial M$. Let $\mathcal{L}^{\mathit{fr}}(2n)$ be the set of all relative framed links in $(M, \partial M)$ considered up to ambient isotopy keeping $\partial M$ fixed, such that $L \cap \partial M = \partial L = \{x_i\}$. Let $R$ be a commutative ring with unity,  with elements $a, b_0,b_1,b_2,b_3$, and $b_\infty$ such that the framing variable $``a"$ is invertible., and $\mathcal{S}_{4,\infty}(2n)$, the submodule of $R\mathcal{L}^{\mathit{fr}}(2n)$, generated by all the cubic skein relations. Then, the relative cubic skein module of $M$ is the quotient: $$\mathcal{S}_{4,\infty}(M, \{x_i\}_1^{2n}; R, ,b_0,b_1,b_2,b_3,b_\infty, a^{\pm 1}) = \frac{R\mathcal{L}^{\mathit{fr}}(2n)}{ \mathcal{S}_{4,\infty}(2n)}.$$

\end{definition}

\subsection{Properties of Cubic Skein Modules} 

In this subsection we discuss some elementary properties of cubic skein modules. Compare \cite{Prz4} and \cite{PBIMW}.

\begin{theorem}\label{theorem: functoriality}

If $i : M \hookrightarrow N$ is an orientation preserving embedding of $3$-manifolds, then $f$ induces a homomorphism $i_*:  \mathcal{S}_{4,\infty} (M)\longrightarrow \mathcal{S}_{4,\infty} (M)$ of the corresponding cubic skein skein modules. This correspondence leads to a functor from the
category of $3$-manifolds and orientation preserving embeddings (up to ambient isotopy) to the category of $R$-modules with specified elements $b_0,b_1,b_2,b_3,b_\infty$, $a, a^{-1} \in R$.
    
\end{theorem}

\begin{theorem}\hfill

\begin{enumerate}

 \item If $N$ is obtained from $M$ by adding a $3$-handle to $M$ and $i : M \hookrightarrow N$ is the associated embedding, then
$i_* : \mathcal{S}_{4,\infty}(M) \longrightarrow \mathcal{S}_{4,\infty}(N)$ is an isomorphism of the corresponding cubic skein modules.

\item Let $M$ be a $3$-manifold with boundary $\partial M$ and let $\gamma$ be a
simple closed curve on $\partial M$. Let $N = M_{\gamma}$ be the $3$-manifold
obtained from $M$ by adding a $2$-handle along $\gamma$ and $i : M \hookrightarrow N$ be the associated embedding. Then
$i_* : \mathcal{S}_{4,\infty}(M) \longrightarrow \mathcal{S}_{4,\infty}(N)$ is an epimorphism of the cubic skein modules.

\end{enumerate}
    
\end{theorem}

\begin{example}

We note that $S^3$ may be obtained from the $3$-ball $D^3 \simeq D^2 \times I$ by gluing a $3$-handle to its boundary. Hence, $\mathcal{S}_{4,\infty}(S^3) \cong \mathcal{S}_{4,\infty}(D^3) \cong \mathcal{S}_{4,\infty}(D^2 \times I)$. 
    
\end{example}

\begin{theorem}

 If $i, j : M \hookrightarrow N$ are isotopic embeddings and $i_*$ and $j_*$ are the induced homomorphisms between the corresponding cubic skein modules, then $i_* = j_*$. 
    
\end{theorem}

\begin{theorem}\label{theorem: disjoint union}

 If $M_1 \sqcup M_2$ is the disjoint sum of $3$-manifolds $M_1$ and $M_2$, then
$\mathcal{S}_{4,\infty}(M_1 \sqcup M_2) = \mathcal{S}_{4,\infty}(M_1) \otimes_R  \mathcal{S}_{4,\infty}(M_2)$.
    
\end{theorem}

The proofs of Theorems \ref{theorem: functoriality} - \ref{theorem: disjoint union} may be found in \cite{PBIMW}. 

\begin{theorem}\label{UCP}(Universal Coefficient Property)

Let $R$ and $R'$ be commutative rings with unity, $b_0, b_1, b_2, b_3, b_\infty$, and $a$ be elements in $R$, and $\varphi : R \longrightarrow R'$ be a homomorphism. Then the identity map on $\mathcal L^{\mathit{fr}}$
induces the following isomorphism of $R'$ (and $R$) modules: 

$$ \mathcal S_{4,\infty} (M; R)(b_0, b_1, b_2, b_3, b_\infty,a)\otimes_R R' \cong \mathcal S_{4,\infty} (M; R')(\varphi(b_0),\varphi(b_1), \varphi(b_2), \varphi(b_3), \varphi(b_\infty), \varphi(a) ).$$

In particular, if $R = \mathbb{Z}[x_0,x_1,x_2,x_3,x_\infty,x_a]$, $R'$ is a ring with chosen elements  $b'_0, b'_1, b'_2, b'_3, b'_\infty, a'$, $\varphi (x_i)=b'_i$,  $\varphi (x_\infty)=b'_\infty$, and $\varphi (x_a)=a'$, then 
$$\mathcal S_{4,\infty} (M;R )(x_0,x_1, x_2, x_3, x_\infty,x_a)\otimes_{R} R' \cong \mathcal S_{4,\infty} (M; R')(b'_0,b'_1,b'_2,b'_3,b'_\infty,a') . $$

In particular, if $x_\infty$ and $x_a$ are invertible in $R$, then 
$$\mathcal S_{4,\infty} (M;R )(x_0,x_1,x_2,x_3,x_\infty^{\pm 1}, x_a^{\pm 1}) \otimes_R R' \cong \mathcal S_{4,\infty} (M; R')(b'_0,b'_1, b'_2, b'_3, b'_\infty,a), $$ where $R = {\mathbb{Z}[x_0,x_1,x_2,x_3,x_\infty^{\pm 1}, x_a^{\pm 1}]}$ and $b'_\infty$ and $a$ are invertible. 
    
\end{theorem}

\begin{proof}

The exact sequence of $R$-modules,
$\mathcal{S}_{4,\infty}^{\mathit{sub}} \longrightarrow R\mathcal{L}^{\mathit{fr}} \longrightarrow \mathcal S_{4,\infty} (M; R)(b_0, b_1, b_2, b_3, b_\infty,a) \longrightarrow 0$,
leads to the following exact sequence of $R'$-modules (this follows from the right exactness of tensor products; see  Chapter 2, Proposition 4.5 in \cite{homalg}):

$$\mathcal{S}_{4,\infty}^{\mathit{sub}} \otimes_R R' \longrightarrow R\mathcal{L}^{\mathit{fr}} \otimes_R R' \longrightarrow \mathcal S_{4,\infty} (M; R)(b_0, b_1, b_2, b_3, b_\infty,a) \otimes_R R' \longrightarrow 0.$$

Now, by applying the {\it five lemma} to the following the commutative diagram with
exact rows (see, for example, Proposition 1.1 in \cite{homalg})

{\small \begin{displaymath} 
\begin{array}{ccccccccccccccc} 
 \mathcal{S}_{4,\infty}^{\mathit{sub}}(R) \otimes_R R' & \to & R{\mathcal L}^{\mathit{fr}}\otimes_R R' 
& \to & \mathcal S_{4,\infty} (M; R)(b_0, b_1, b_2, b_3, b_\infty,a) \otimes_R R' & \to & 0 & \to & 0 \\ 
\big\downarrow {\mathit{epi}} & & \big\downarrow \mathit{iso} & & \big\downarrow \overline r & & \big\downarrow & & \big\downarrow \\ 
\mathcal{S}_{4,\infty}^{\mathit{sub}}(R') & \to & R'{\mathcal L}^{\mathit{fr}} & \to & 
\mathcal S_{4,\infty} (M; R')(\varphi(b_0),\varphi(b_1), \varphi(b_2), \varphi(b_3), \varphi(b_\infty), \varphi(a) ) & \to & 0 & \to & 0
\end{array} 
\end{displaymath} }

we conclude that $\overline{r}$ is an isomorphism of $R'$- (and $R$-) modules.

\end{proof}

\begin{remark}\label{remark:alg}
    The cubic skein module can be enriched with an algebra structure when the $3$-manifold is the product of an oriented surface $F$ and the unit interval $I$. The product of two links is given by stacking, that is, $L_1 \cdot L_2$ is defined by placing $L_1$ above $L_2$ in $F \times I$, with respect to the height given by $I$. In particular, $L_1 \subset (1/2,1)$ and $L_2 \subset (0,1/2)$. The empty link $\varnothing$ is the identity element of this operation. This multiplication is associative but not commutative in general, and is compatible with the module structure, which makes $\mathcal S_{4,\infty} (F \times I;R)$ an algebra over $R$.  
\end{remark}

\subsection{Choosing the Ring for Cubic Skein Modules}\label{CRCSM}

In Section \ref{ninftysm}, we defined $(4, \infty)$-skein modules over an arbitrary commutative, unital ring with chosen elements $b_0,b_1,b_2,b_3,b_{\infty}$ and an invertible element $a$. However, for practicality, we may place further demands on the elements in $R$, such as on their invertibility or non-zero-ness. Each demand placed upon $R$ should be regarded as a choice informed by computational needs.

\smallskip

For instance, for practical and computational reasons we assume, unless otherwise stated,  that $b_0$ and $b_3$ are invertible. Due to the Universal Coefficient Property stated in Theorem \ref{UCP}, we can work with the ``universal" ring of Laurent polynomials 
$\mathbb Z[b_0^{\pm 1},b_1,b_2,b_3^{\pm 1},b_{\infty},a^{\pm 1}]$.

If we consider 
the half-denominator of our skein relation we get (see Figure  \ref{fig:skein-n-inf-den}).
$$0=b_\infty D\sqcup \bigcirc +(b_0+a^{-1}b_1 +a^{-2}b_2+a^{-3}b_3 )D.$$

We simplify this equation by writing $D\sqcup \bigcirc$ as $tD$. (in $\mathbb{R}^3$ or generally any Cartesian product of a surface and the interval, $F\times [0,1]$ our skein module form an algebra; in particular the product of arbitrary $D$ with a trivial knot is well defined and may be denoted by $tD$).
Thus we have $$-tb_{\infty}D= (b_0+a^{-1}b_1 +a^{-2}b_2+a^{-3}b_3 )D.$$

This equation suggests two natural choices:
$b_\infty=0$ or $b_\infty$ invertible.\\

The choice of $b_{\infty}=0$ leads to the skein relation
$$b_0D_0+b_1D_1+b_2D_2+b_3D_3=0.$$ 
This skein module, which can be denoted by $\mathcal S_{4,\infty}(M;R,b_0^{\pm 1},b_1,b_2,b_3^{\pm 1},0,a^{\pm 1})$, is analyzed in some detail in \cite{PTs}. In particular the coefficients satisfy:
$$b_0+a^{-1}b_1 +a^{-2}b_2+a^{-3}b_3 =0.$$

In this paper we impose the invertibility of $b_{\infty} \in R$. This choice (and the fact that $\mathcal{S}_{n,\infty}(S^3)$ defines an algebra) enables us to compute $t$ as an element (polynomial) of $R$. We denote by  $t$ the  trivial component $\bigcirc$ and compute 
\begin{equation}\label{eqn:t1}
t=\frac{(b_0+b_1a^{-1}+b_2a^{-2}+b_3a^{-3})}{-b_{\infty}}.
\end{equation}

The utility of computing $t$ as an element of $R$ is clearly demonstrated first in Section {\ref{sec:Hopfrelation}, and used throughout the remainder of the paper. In Section \ref{Subsec:trivialknot}, we further discuss other calculations related to $t$.
Thus we consider the choice of ring $R = \mathbb Z[ b_0^{\pm 1}, b_1,b_2,b_3^{\pm 1},b_{\infty}^{\pm 1},a^{\pm 1}]$ or its quotients. Unless otherwise stated, we denote $R = \mathbb Z[ b_0^{\pm 1}, b_1,b_2,b_3^{\pm 1},b_{\infty}^{\pm 1},a^{\pm 1}]$. 

\smallskip

Initially, we thought that a good choice of ring would be $R$ divided by the Hopf Relation (see Section \ref{sec:Hopfrelation}) as speculated in \cite{Prz5}, but we have since found other relations in $\mathcal{S}_{4,\infty}(S^3)$ (see Section \ref{sec:trefoilrelation}) not divisible by the Hopf Relation.

We can also approach the choice of a ring from another angle.
For example, we can take $R$ to be the cubic skein module of $ S^3$. This would be a good choice if the embedding of $D^3$ into $M^3$ always induces a monomorphism of skein modules (see Conjecture \ref{mono}).
The disadvantage of such a choice is that since the Montesinos-Nakanishi conjecture does not hold, then $R$ would not be a quotient of $\mathbb Z[b_0^{\pm 1}, b_1,b_2,b_3^{\pm 1},b_{\infty}^{\pm 1},a^{\pm 1}]$. 
To remedy this, we can take as our ring the subring of $R$ generated by trivial links, $S^{(0)}_{4,\infty}(S^3)$.
Thus we have a choice in selecting the ring over which we consider skein modules; and we should be flexible. One of the rules would be that our ring should have no zero divisors and no element of the ring should annihilate all links.

\begin{conjecture}\label{mono}
Let $f:B^3 \to M^3$ be an embedding of a disk into a manifold then the induced map on skein modules $f_*:\mathcal S_{4,\infty}(B^3,R) \to \mathcal{S}_{4,\infty}(M^3,R)$ is a monomorphism. 
\end{conjecture}

\subsection{The Case of Mirror Image}\label{MI} \ 
In classical knot theory with a $3$-manifold $M=\mathbb{R}^3$ (or $S^3=\mathbb{R}^3\cup \infty$), the notion of the mirror image is well defined and does not depend on the position of the mirror (up to ambient isotopy). If we consider, more generally  $M= F\times [0,1]$ then, assuming this product structure, the mirror image is well defined as the image under the reflection with respect to the mirror $F\times \frac{1}{2}$.
To be concrete, let us compare the basic cubic skein relation 
\begin{equation}\label{cubic3}
b_3D_3+ b_2D_2 + b_1D_1+ b_0D_0+ b_{\infty}D_{\infty}=0, 
\end{equation}
with the same relation shifted by two negative twists
\begin{equation}\label{cubic4}
b_0D_{-3}+b_1D_{-2}+b_2D_{-1}+b_3D_0+a^3b_{\infty} D_\infty=0,
\end{equation}
we can conclude that for $M^3=F\times [0,1]$ and 
$D=\sum_{i=1}^Np_iB_i$ the mirror image, $\bar D$ of $D$ with respect to the surface $F\times \{\frac{1}{2}\}$ 
has the description 
$\bar D=\sum_{i=1}^N\phi(p_i) \bar B_i$, where $B_i$ and $\bar B_i$ are elements of the 
KBSM of $M^3$ and the ring element $ \phi(p_i)$ is the element of the ring $\mathcal R$ obtained from 
$p_i\in \mathcal R$ by the involution 
$b_0\leftrightarrow b_3$,
$b_1\leftrightarrow b_2$, $a\leftrightarrow a^{-1}$, and 
$b_v=a^{3/2}b_{\infty}$ is invariant under the involution; in particular, $a^{k}b_\infty \leftrightarrow a^{3-k}b_{\infty}$. We refer the reader to Remark \ref{Pn-sub} where this involution is discussed further.  The result follows directly from the definition of the KBSM but depends on the product structure of $M^3$ or more generally, as we explain below, requires an 
orientation-reversing involution of the 3-manifold $M$.

Thus to describe general mirror image relation in an oriented 3-manifold $M^{(o)}$ we should go above the product structure and consider the manifold $M^{(ro)}$ obtained from $M^{(o)}$ by changing its orientation. Let $\phi : M^{(o)} \to M^{(ro)}$ be the changing orientation identity map. Then we compare a frame link $L$ in  $M^{(o)}$ with the same link $\phi(L)$ in $M^{(ro)}$.
\begin{proposition}\label{mirror}\ \\
If $L=\sum_{i=1}^Np_i  B_i$ in the KBSM of 
$M^{(o)}$ then $\phi(L)=\sum_{i=1}^N\phi(p_i) \phi (B_i)$ in the KBSM of $M^{(ro)}$.
\end{proposition}
\begin{proof} The proof follows directly by comparing the cubic skein relation given in Equation \ref{cubic3} and Equation \ref{cubic4} 
and noting the positive twist of the framing in $M^{(o)}$ 
become the negative twist on framing in $M^{(ro)}$.
\end{proof}
\begin{example} Consider, in $\mathbb{R}^3$ the cubic skein relation 
$-b_3D_2= b_2D_1+b_1D_0+b_0D_{-1}+ ab_{\infty} D_\infty$. Then the mirror image of this relation has the form:
$-b_0D_{-2}=b_1D_{-1} + b_2D_0+ b_3D_1+ a^2b_{\infty} D_\infty $.
We can observe that the numerator of $D_2$ is the positive Hopf link and the numerator of $D_{-2}$ is the negative Hopf link, but in fact they are ambient isotopic leading to a relation in the cubic skein module. We investigate this Hopf relation later.
\end{example}

\section{Relations in the Cubic Skein Module}\label{sec:relations}

Standard in the theory of skein modules are equations like \ref{eqn:t1} in which a link diagram (a representative of an ambient isotopy class) is computed as an element of the base ring. The unknot is computed as $t = -(A^{2} + A^{-2})$ in the Kauffman Bracket Skein Module.

In the cubic skein module, we compute $t$ as Equation \ref{eqn:t1}. If, however, $t$ is computed differently, as in Equation \ref{othertrivial}, then their difference is a ``relation  in $\mathcal{S}_{(4,\infty)}(S^3)$."

More generally, given two ambient isotopic link diagrams $D,D' \subset S^3$, we compute two polynomials $d, d'   \in$  $\mathbb{Z}[b_0^{\pm 1}, b_1 ,b_2 , b_3^{\pm 1}, b_\infty^{\pm 1}, a^{\pm 1}]$ by reducing subtangles of the diagrams until we yield framed, trivial links. Since $[D]=[D']$ in $\mathcal{L}^{\mathit{fr}}$, the relation $d-d'=0$ is in $\mathcal{S}_{(4,\infty)}(S^3)$. 

In Section \ref{RaTaAl}, we present two algorithms for quickly computing polynomials from diagrams in the same ambient isotopy classes. In Proposition \ref{reversecode}, we ensure that we that we can always produce at least one relation for rational links.

In the following subsections, we compute various relations by hand to get a sense of the general method.

\smallskip

\begin{remark}
There are three basic prerequisites to computing polynomials representing link diagrams in the cubic skein module.
 \begin{enumerate}
     \item The cubic skein module admits an algebra structure when the 3-manifold is a thickened surface, see Remark \ref{remark:alg}.

     \item The trivial component can be computed as a word in $R,$ namely
     \[ t = - \frac{b_0 + b_1 a^{-1} + b_2 a^{-2} + b_3 a^{-3}}{b_\infty}. \]
     This specific expression for the trivial component, $t$, is computed from the denominator closure of the cubic skein relation and exploits the algebra structure of $\mathcal{S}_{(4,\infty)}(S^3)$.
     \item The given link diagram that we wish to compute is $3$-move reducible to the trivial link. In particular, a slight modification of the Rational Tangle Algorithm will yield the fact that all rational links are 3-move reducible to unlinks. We note that this condition is nontrivial because of the  Montesinos-Nakanishi Conjecture, see \cite{DP1}.
 \end{enumerate}
\end{remark}

\subsection{The Hopf Relation}\label{sec:Hopfrelation}

In this section, we give an explicit computation of the basic Hopf relation with the assumption that $a, b_{\infty}, b_0,$ and $b_3$ are invertible. 

Consider the two projections of the Hopf link $H_{+}$ and $H_{-}$ as illustrated in Figure \ref{fig:HplusHneg}. In the same figure we see that the two diagrams are related by Reidemeister moves; specifically one Reidemeister 3 move and two Reidemeister 1 moves. We calculate $H_+$ from Equation \ref{eqn:CubicSkeinRelation} by setting the diagram in the $D_3$ position.
\begin{figure}[ht]
    \centering
$$H_+ = \vcenter{\hbox{\includegraphics[scale=.25]{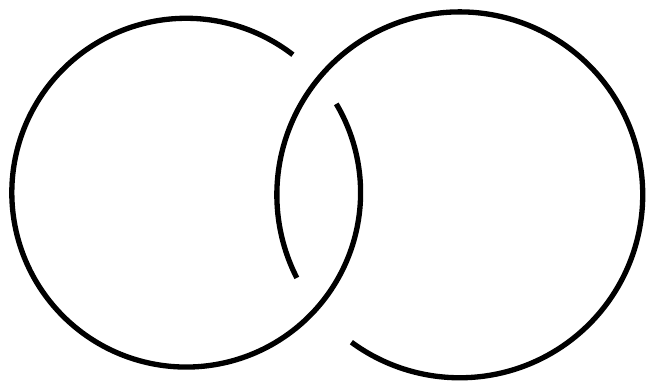}}} \xrightarrow[]{R_1} \vcenter{\hbox{\includegraphics[scale=.25]{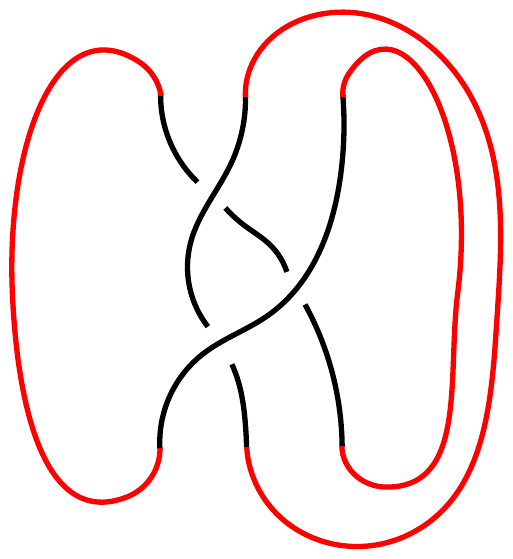}}} \xrightarrow[]{R_3} \vcenter{\hbox{\includegraphics[scale=.25]{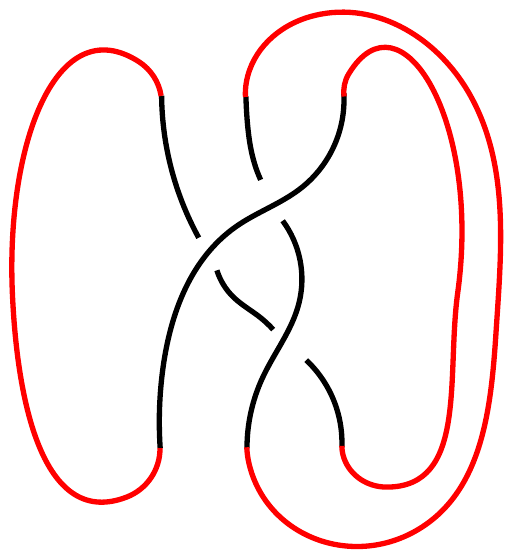}}} \xrightarrow[]{R_1} \vcenter{\hbox{\includegraphics[scale=.25]{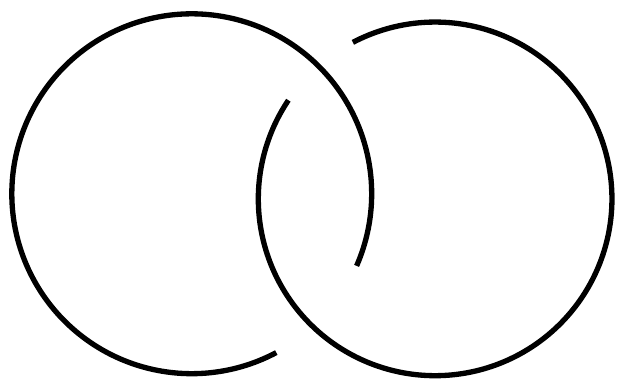}}} = H_-$$
    \caption{Sequence of Reidemeister moves from $H_+$ to $H_-$.}
    \label{fig:HplusHneg}
\end{figure}

\begin{eqnarray}
    0 &=& b_3 \vcenter{\hbox{\includegraphics[scale=.15]{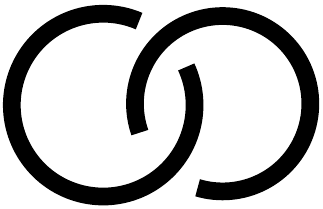}}} + b_2 \poskink+ b_1 t^2+ b_0 \negkink + b_{\infty} \poskinkinf \nonumber \\
    \implies H_+ &=&   -b_3^{-1} t(b_2a+b_0 a^{-1}+b_{\infty} a +b_1 t). \label{eqn:hopfpos}
\end{eqnarray}

We calculate $H_-$ from Equation \ref{eqn:CubicSkeinRelation} by setting the diagram in the $D_0$ position.

\begin{eqnarray}
   0 &=&  b_0\vcenter{\hbox{\includegraphics[scale=.15]{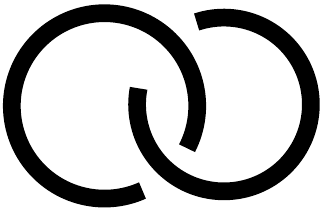}}} + b_1 \negkink+ b_2t^2 + b_3 \poskink+b_{\infty} \postwokinkinf \nonumber \\
   \implies H_- &=&  =-b_0^{-1}t(b_1a^{-1} +b_2t+b_3a+b_{\infty}a^2).\label{eqn:hopfneg}
\end{eqnarray}

\begin{definition}[Hopf relation]
    Define the Hopf relation, denoted by $R_{\mathit{Hopf}}$, as the difference between $H_+$ and $H_-$.

    \begin{eqnarray}
&& R_{\mathit{Hopf}} := H_+ - H_- \label{eqn:Hopfrelation}\\
&=& \frac{t(-b_0^2 - a^2 b_0 b_2 + b_1 b_3 + a^2 b_3^2 - a^2 b_0 b_{\infty} + a^3 b_3 b_{\infty} - 
  a b_0 b_1 t + a b_2 b_3 t)}{a b_0 b_3}. \nonumber
\end{eqnarray}
\end{definition}
Recall that we computed the trivial component $\bigcirc= t$ as

\[  t=\frac{(b_0+b_1a^{-1}+b_2a^{-2}+b_3a^{-3})}{-b_{\infty}}. \]

After substituting $t$ we have our first relation;

\begin{eqnarray*}
 R_{\mathit{Hopf}} &=& \frac{t( (b_1b_0-b_2b_3)(a^3 b_0 + a^2 b_1 + a b_2 + b_3)  -a^2 b_{\infty}( b_0^2+
  a^2 b_0 b_2 - b_1 b_3 - a^2 b_3^2 + a^2 b_0b_{\infty} - 
  a^3 b_3 b_{\infty}))}{a^3 b_0 b_3 b_{\infty}}.
\end{eqnarray*}
In Section \ref{RaTaAl}, we use the language of rational links to denote this relation by $[-1,-1]-[1,1]$.

\begin{remark}
    Notice that computing the Hopf Relation as an element of $\mathbb{Z}[b_0^{\pm 1}, b_1 ,b_2 , b_3^{\pm 1}, b_\infty^{\pm 1}, a^{\pm 1}]$ relied on our ability to reduce the Hopf Link diagrams to an algebraic combination of framed trivial links. Whereas in the Kauffman Bracket Skein Module, every link diagram may be reduced to trivial links,\footnote{We consider the Kauffman Bracket Skein Relation as deformation of a 2-move, and all links are 2-move reducible to trivial links.} the cubic skein relation is a deformation of a 3-move, and not every link is 3-move reducible to trivial links; this is the Montesinos-Nakanishi Conjecture, which was shown to be false in \cite{DP1}. The upshot is that we are only ensured the ability to compute diagrams as polynomial in $R$ for those links which are 3-move generated by trivial links; that is, diagrams which represent classes in $\mathcal{S}^{(0)}_{4,\infty }(S^3)$. 
\end{remark}

\subsection{Trefoil Knots}\label{subsec:trefoil}
The following computation of the left (Conway code [-3])  and right (Conway code [3]) handed diagram of the trefoil will be used in later sections, particularly Section \ref{MRel}.

\begin{equation}\label{eqn:lefttrefoil1}
    \trefoilneg = -b_0^{-1} (b_1 ~H_- + b_2 a^{-1} t + b_3 t^2 + a^3 b_{\infty} t).
\end{equation}

\begin{eqnarray}
    \overline{3_1} :=\trefoilpos &=& -b_3^{-1} \left( b_2 ~H_+ + a b_1 t + b_0 t^2 +b_{\infty} t \right). \label{eqn:trefoilpos1}
\end{eqnarray}

\subsection{A Figure-eight Relation}

Denote by $(4_1)_+$ the diagram of the figure-eight knot given by Conway code $[2,2]$ then use the cubic skein relation by placing $(4_1)_+$ in the $D_3$ position of the cubic skein relation.

\begin{eqnarray}
    (4_1)_+ &:=& \figureeightpos \nonumber\\
    &=& -b_3^{-1} \left( b_2 ~3_1 + b_1 a^2 t + b_0 a^{-1} t +b_{\infty} a ~H_-\right).
    \label{eqn:fig8posone}
\end{eqnarray}

Next, denote by $(4_1)_-$ the diagram of the figure-eight knot given by Conway code $[-2,-2]$, then use the cubic skein relation by placing $(4_1)_-$ in the $D_0$ position.
\begin{eqnarray}
    (4_1)_- &:=& \figureeightneg \nonumber\\
    &=& -b_0^{-1} \left( b_1 ~\overline{3_1} + b_2 a^{-2} t + b_3 a t +b_{\infty} a ~H_+\right).\label{eqn:fig8negone}
\end{eqnarray}

We take the difference of Equations  \ref{eqn:fig8posone} and \ref{eqn:fig8negone}. We call this relation $R_{(4_1)}$: 
\begin{eqnarray*}
    R_{(4_1)} &=& [2,2]-[-2,-2] \\
    &=& -\frac{1}{a^6 b_0^2 b_3^2 b_{\infty }^2} \left(b_0 b_3-b_1 b_2\right) \left(a^3 b_0+a^2 b_1+a b_2+b_3\right)  \\ 
    && (a^5 b_3 b_{\infty }^2-a^4 b_0
   b_{\infty }^2+a^4 b_3^2 b_{\infty }-a^4 b_0 b_2 b_{\infty }+a^3 b_0^2 b_1-a^3 b_0 b_2 b_3-a^2 b_0^2 b_{\infty
   }+a^2 b_1 b_3 b_{\infty } \\
   &&+a^2 b_0 b_1^2 
    -a^2 b_1 b_2 b_3+a b_0 b_1 b_2 -a b_2^2 b_3-b_2 b_3^2+b_0 b_1
   b_3).
\end{eqnarray*}

Notice that this figure eight relation is divisible by the Hopf Relation, consistent with an earlier conjecture.

\section{Cubic Skein Modules and $n$-moves}\label{Section 15}
An $n$-move is a local deformation (replacement) of a link where two parallel pieces of a link are replaced by $n$ right handed half twists (see Figure \ref{n-move2}.

In this section  we develop a formula for an $n$ move which can be use to compute immediately formulas for torus links $T(2,n)$ in cubic skein modules and further for twist knots and more generally pretzel knots. 

\begin{figure}[ht]
\centering
\includegraphics[scale=0.49]{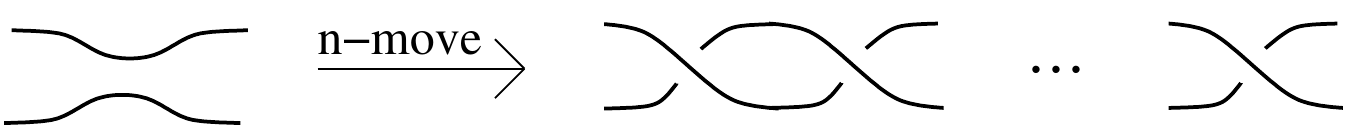} 
\caption{The $n$-move from $D_0$ to $D_n$ ($n$ right handed half twists).}
\label{n-move2}
\end{figure}

In particular, the formula for the 2-bridge link $[m,n]$ would be useful (this 2-bridge link is the same as the pretzel link $P(m,1,1,...,1)$ ($n$-times $1$)).

For simplicity we assume first $b_3=-1$. We do not lose any information, as $b_3$ was assumed to be invertible; we do however lose symmetry when considering mirror images. However, we can easily recover the general case by using the fact that the polynomials involved are homogeneous; see Remark \ref{Pn-sub}. 
The case of $n$-moves and quadratic skein relation was considered in \cite{Prz1} (see also \cite{DIP}).

\subsection{$D_n$ calculation for $b_3=-1$}\label{Dn calculation}
For simplicity we assume that $b_3=-1$. As long as $b_3$ is invertible we do not loose any information. That is the case because 
the cubic skein relation is homogeneous so we can easily move from a homogeneous formulas to a non-homogeneous formula ($b_3=-1$) and back (see Remark \ref{Pn-sub} where we explain how to bring back variable $b_3$ by considering new variables $b_i'=\frac{b_i}{-b_3}$, for $i=0,1,2,\infty$).

Our defining relations are $$b_3D_3+b_2D_2+b_1D_1+b_0D_0+ b_\infty D_\infty =0,$$
see Figure \ref{3-move-skein-cubic}, and $D^{(1)}=aD$ (framing relation).

\begin{figure}[ht]
\centering
\includegraphics[scale=0.31]{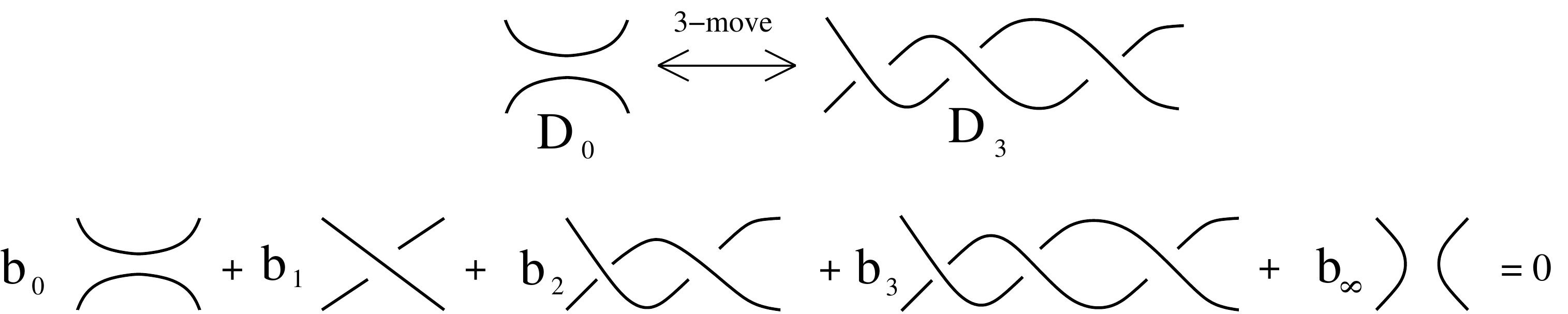} 
\caption{3-move and its cubic deformation.}
\label{3-move-skein-cubic}
\end{figure}

\begin{figure}[ht]
\centering
\includegraphics[scale=0.31]{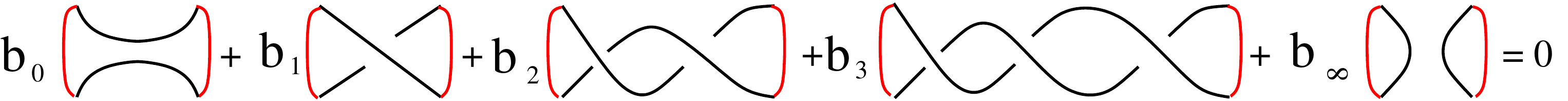} 
\caption{Denominator of cubic skein relation.}
\label{cubic-relation-den2}
\end{figure}

Recall that from the denominator of this relation, compare Figure \ref{cubic-relation-den2}, we get 
\begin{equation}\label{tbinfty}
a^{-3}b_3+ a^{-2}b_2 + a^{-1}b_1 +b_0 +b_\infty t =0
\end{equation} 
For simplicity, without loosing any information, as explained before (see also Remark \ref{Pn-sub}), we work with $b_3=-1$ and from the defining relation we 
get a  recursive relation 
$$D_n= b_2D_{n-1}+b_1D_{n-2}+ b_0D_{n-3}+ a^{3-n}b_\infty D_\infty.$$
Now, using the recursive relation $3$ times we get ($\stackrel{(k)}{=}$ denotes the result after $k$ iterations).
\begin{eqnarray}\label{Dn-kiterations}
D_n &\stackrel{(1)}{=}& b_2D_{n-1}+b_1D_{n-2}+ b_0D_{n-3}+ a^{3-n}b_\infty D_\infty \\
&\stackrel{(2)}{=}&  
b_2(b_2D_{n-2}+b_1D_{n-3}+ b_0D_{n-4}+ a^{4-n}b_\infty D_\infty)+ b_1D_{n-2}+ b_0D_{n-3}+ a^{3-n}b_\infty D_\infty \nonumber \\
&=&
(b_2^2+b_1)D_{n-2} + (b_1b_2+b_0)D_{n-3} + b_0b_2D_{n-4} +\bigg(a^{4-n}b_2 + a^{3-n}\bigg)b_\infty D_\infty  \nonumber \\ &\stackrel{(3)}{=}&
(b_2^2+b_1)(b_2D_{n-3}+b_1D_{n-4}+ b_0D_{n-5}+ a^{5-n}b_\infty D_\infty)+(b_1b_2+b_0)D_{n-3} + b_0b_2D_{n-4} \nonumber \\
&&+\bigg(a^{4-n}b_2 + a^{3-n}\bigg)b_\infty D_\infty \nonumber \\
&=&
(b_2(b_2^2+b_1)+ b_1b_2+b_0)D_{n-3}+ (b_1(b_2^2+b_1)+ b_0b_2)D_{n-4}+ b_0(b_2^2+b_1)D_{n-5}+ \nonumber \\
&&\bigg(a^{5-n}(b_2^2+b_1)+a^{4-n}b_2 + a^{3-n}\bigg)b_\infty D_\infty \nonumber \\
&=&
P_3D_{n-3} + (b_1P_2+b_0P_1)D_{n-4}+ b_0P_2D_{n-5}+ (\sum_{i=0}^{2}a^{3-n+i}P_i)b_\infty D_\infty \nonumber
\end{eqnarray} 

One can already guess the general formula as we see that 
we work with polynomials $P_0=1, P_1=b_2$, $P_2=b_2^2+b_1$, and $P_3=b_2P_2+b_1P_1+b_0P_0= b_2^3+2b_1b_2+b_0$ and 
$$D_n \stackrel{(3)}{=} P_3D_{n-3} + (b_1P_2+b_0P_1)D_{n-4}+ b_0P_2D_{n-5}+ (\sum_{i=0}^{2}a^{3-n+i}P_i)b_\infty D_\infty. $$
However it is useful to go through one more iteration:

\begin{eqnarray}
D_n &\stackrel{(4)}{=}& P_3(b_2D_{n-4}+b_1D_{n-5}+b_0D_{n-6}+ a^{6-n}b_{\infty}D_\infty)+(b_1P_3+b_0P_2)D_{n-5}+b_0P_3D_{n-6} \nonumber \\
&&+ (a^{5-n}P_2+ a^{4-n}P_1+a^{3-n}P_0)b_{\infty}D_\infty) \nonumber\\
&=&
P_4D_{n-4}+ (b_1P_3+b_0P_2)D_{n-5}+b_0P_3D_{n-6}+ (a^{5-n}P_2+ a^{4-n}P_1+a^{3-n}P_0)b_{\infty}D_\infty) \nonumber\\
\stackrel{(5)}{=}
... 
&\stackrel{(k)}{=}&
P_kD_{n-k} + (b_1P_{k-1}+b_0P_{k-2})D_{n-k-1} +b_0P_{k-1}D_{n-k-2}+(\sum_{i=0}^{k-1}a^{3-n+i}P_i)b_\infty D_\infty. \nonumber
 \end{eqnarray}

Thus we obtained our main formula for the $n$-move (and arbitrary $k$):
\begin{lemma}\label{U-formula}
$$D_n\stackrel{(k)}{=}\\
P_kD_{n-k} + (b_1P_{k-1}+b_0P_{k-2})D_{n-k-1} +b_0P_{k-1}D_{n-k-2}+(\sum_{i=0}^{k-1}a^{3-n+i}P_i)b_\infty D_\infty.$$
\end{lemma}
Now let $U^{(3)}_{n,k}= \sum_{i=0}^{k-1}a^{3-n+i}P_i$, where upper index (3) means that we work with the cubic skein relation\footnote{Notice that $a^{n-3}U^{(3)}_{n,k}$ is the truncated generating function of $P_i$ with respect to the variable $a$.}.\\
By using the denominator of Formula \ref{Dn-kiterations} we can get a closed formula for $U^{(3)}_{n,k}$ as follows:
$$a^{-n}= a^{k-n}P_k + a^{k-n+1}(b_1P_{k-1}+b_0P_{n-k-2})+ a^{k-n+2}b_0P_{k-1}+ U^{(3)}_{n,k} b_\infty t.$$
Therefore
$$U^{(3)}_{n,k} b_\infty t = a^{-n} - \bigg(a^{k-n}P_k + a^{k-n+1}(b_1P_{k-1}+b_0P_{k-2})+ a^{k-n+2}b_0P_{k-1}\bigg)$$
and assuming that $b_\infty t$ is invertible and using Formula \ref{tbinfty} we get the closed formula
\begin{equation}\label{ClosedFormulaU}
U^{(3)}_{n,k}= \sum_{i=0}^{k-1}a^{3-n+i}P_i = \frac{-a^{-n}+ a^{k-n}P_k + a^{k-n+1}(b_1P_{k-1}+b_0P_{k-2})+ a^{k-n+2}b_0P_{k-1}}{-a^{-3}+a^{-2}b_2 + a^{-1}b_1 +b_0}.
\end{equation}
Notice that $U^{(3)}_{n+1,k}= a^{-1}U^{(3)}_{n,k}$ and $U^{(3)}_{n,k+1}= U^{(3)}_{n,k}+ a^{3-n+k}P_k$.
\\
We can give a short inductive proof of Formula \ref{ClosedFormulaU} by not requiring a topological trick (but we can do it easily because we know the result!)
\begin{proof} We use induction on $k$. For $k=1$ both sides of the equation are equal to $a^{3-n}$. For an inductive step we assume that formula holds for numbers $\leq k$ and consider the case of $k+1$. 
We have 
{\small
\begin{eqnarray*}
    U^{(3)}_{n,k+1} &=& U^{(3)}_{n,k}+ a^{3-n+k}P_k \\
    &\stackrel{induction}{=}&
 \frac{-a^{-n}+ a^{k-n}P_k + a^{k-n+1}(b_1P_{k-1}+b_0P_{k-2})+ a^{k-n+2}b_0P_{k-1}}{-a^{-3}+a^{-2}b_2 + a^{-1}b_1 +b_0} + a^{3-n+k}P_k \\
 &=&
\frac{-a^{-n}+ a^{k-n}P_k + a^{k-n+1}(b_1P_{k-1}+b_0P_{k-2})+ a^{k-n+2}b_0P_{k-1} + a^{3-n+k}P_k(-a^{-3}+a^{-2}b_2 + a^{-1}b_1 +b_0)}{-a^{-3}+a^{-2}b_2 + a^{-1}b_1 +b_0} \\
&=&
\frac{-a^{-n} + a^{k-n+1}(b_1P_{k-1}+b_0P_{k-2}+b_2P_k)+ a^{k-n+2}(b_1P_k+ b_0P_{k-1}) + a^{3-n+k}P_kb_0}{-a^{-3}+a^{-2}b_2 + a^{-1}b_1 +b_0}\\
&=&
\frac{-a^{-n} + a^{k-n+1}P_{k+1} +a^{k-n+2}(b_1P_k+ b_0P_{k-1}) + a^{k-n+3}b_0P_k}{-a^{-3}+a^{-2}b_2 + a^{-1}b_1 +b_0},
\end{eqnarray*}}

as needed.
\end{proof}
The importance of Formula \ref{ClosedFormulaU} is that for invertible $b_\infty t$ it holds in the ring $Z[a^{\pm 1},b_0^{\pm 1},b_1,b_2,b_3^{\pm 1},b_{\infty}^{\pm 1},t^{\pm 1}]$ and this fact is crucial in the Section \ref{RaTaAl} as it shows that $[n,0]-[0,-n]=0$ if we follow the algorithm.

\begin{remark}\label{Pn-sub} 
As we mention at the beginning of the subsection, we can recover the variable $b_3$ to translate between $P_n^h$, a  polynomial (Laurent polynomial in variable $b_3$) of degree 0, and $P_n$, a homogeneous polynomial of degree $n$. Concretely, $P_n^h = P(b_i')$ where $b_i' = \frac{b_i}{-b_3}$ for $i \in \{0,1,2\}$.

Notice that $P_n'=(-b_3)^n P^h_n $ is a homogeneous polynomial of degree $n$.  
\end{remark}

For our convenience, we record the values of $P_n'$ for 
$n=0,1,2$ and $3$:
$$P'_0=P_0=1, { \ \ } P_1'=P_1=b_2,$$
$$P'_2=(b_2^2-b_1b_3),$$
$$P'_3=(b_2^3-2b_1b_2b_3+b_0b_3^2),$$
$$P'_4=(b_2^4-3b_2^2b_1b_3+b_1^2b_3^2-2b_0b_2b_3^2).$$

It is also worth seeing that division of $-b_3$ on both sides of defining equation 

We also note that dividing both sides of our defining relation
$-b_3D_3 = b_2D_2 +b_1D_1+ b_0D_0 + b_{\infty}D_\infty $ by $-b_3$, we get the form 

\begin{align*}
    D_3 & = \frac{b_2}{-b_3}D_2 + \frac{b_1}{-b_3}D_1 +\frac{b_0}{-b_3}D_0 +\frac{b_\infty}{-b_3}D_{\infty} \\
    & =  b'_2D_2 +b'_1D_1+ b'_0D_0 + b'_{\infty}D_\infty.
\end{align*}

This form is recalled for general $D_n$ in Proposition \ref{Pnformula}.

\begin{remark}
    Though computation with $b_i$ and $b_i'$ are equivalent, the more general form which includes $b_i' = \frac{b_i}{-b_3}$ is useful for analyzing mirror image diagrams; see Section \ref{MI}.
\end{remark}

\subsection{$\boldsymbol{(2,n)}$ Torus Links, Twist Knots, Pretzel Links, and More} 
The torus link of $(2,n)$ (denoted by $T(2,n)$) can be obtained as a numerator of $D_n$. In particular, we use the closed formula for $D_n$ for $k=n-1$: 
$$D_n=
P_{n-1}D_1 + (b_1P_{n-2}+b_0P_{n-3})D_{0} +b_0P_{n-2}D_{-1}+U_{n,n-1}b_\infty D_\infty$$ 
and evaluate this formula by taking its denominator.

\begin{corollary}\label{TorusLink} Let $T(2,n)$ denote the $(2,n)$ torus link, that is the link obtained by taking the numerator
of the $D_n$ tangle considered above.
Then, in the cubic skein module (using previous notations) we have:
\begin{equation}
T(2,n)= t\bigg(aP_{n-1}+ (b_1P_{n-2}+b_0P_{n-3})t+a^{-1}b_0P_{n-2}+ U_{n,n-1} b_\infty \bigg).
\end{equation}
\end{corollary}
The formula can be completed by substituting 
$t= b_\infty^{-1}(a^{-3}-a^{-2}b_2+ a^{-1}b_1 +b_0)$ and 
\begin{equation}
U_{n,n-1}= \sum_{i=0}^{n-2}a^{3-n+i}P_i = \frac{a^{-1}P_{n-1} +(b_1P_{n-2}+b_0P_{n-3}) +ab_0P_{n-2}-a^{-n}}{b_0+b_1a^{-1}+b_2a^{-2}-a^{-3}}.
\end{equation}
We computed the formula assuming that $b_3=-1$ but we can easily get the general formula by replacing $b_i$ by $\frac{b_i}{-b_3}$. Compare Remark \ref{Pn-sub}.

\subsubsection{Torus link in the annulus, $T^{ann}(2,n)$}\label{TLA}

Let us denote by  $T(2,n)^{ann}$ the torus link in the thickened annulus, that is the 2-braid $(\sigma_1\sigma_2)^n$ closed in the annulus, see
Figure \ref{n-moveAnn}. Our formula for $D_n$ allows us to quickly find the closed formula for $T(2,n)^{ann}$ as a linear combination of basic 2-tangles
being annular closures of $D_1,D_0,D_{-1}$, and $D_\infty$. $D_\infty^{ann} = t$ while $D_i^{ann}$ are illustrated in Figure \ref{n-moveAnn}.
We can conjecture that $D_1^{ann},D_0^{ann}$, and $D_{-1}^{ann}$ are linearly independent in the cubic skein module of the solid torus ($Ann\times [0,1]$).
\begin{corollary} In $\mathcal{S}_{4,\infty}(Ann\times [0,1];R)$ we have:
$$T(2,n)^{ann}= P_{n-1}D_1^{ann} +(b_1P_{n-2}+b_0P_{n-3})D_0^{ann}+ b_0P_{n-2}D_{-1}^{ann}+ U_{n,n-1}b_\infty  t .$$
\end{corollary}

\subsubsection{Twist knots in $\mathbb R^3$ and the solid torus, $Ann\times [0,1]$}
\begin{figure}[ht]
    \centering
\includegraphics[width=0.75\linewidth]{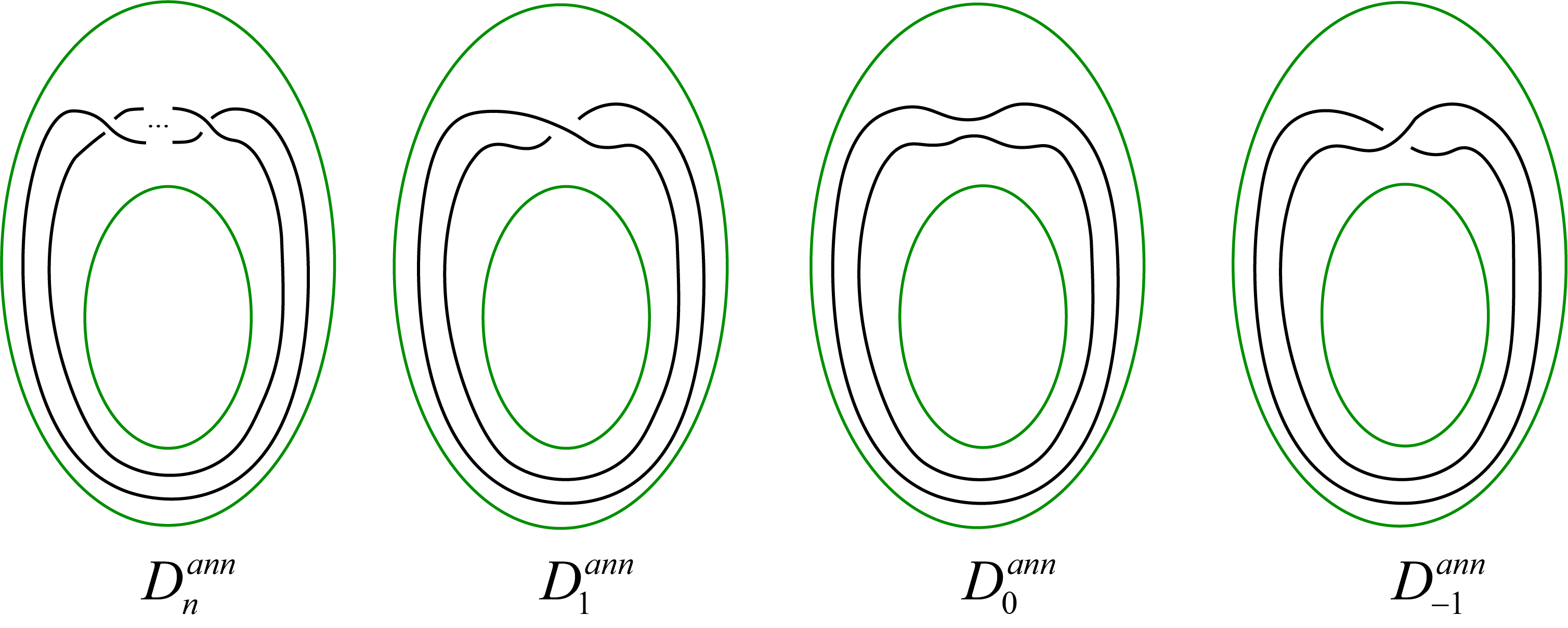}
   \caption{$T(2,n)$ links in $Ann \times [0,1]$: $D_n^{ann}$, $D_1^{ann}$, $D_0^{ann}$, and $D_{-1}^{ann}.$}
\label{n-moveAnn}
\end{figure}

The torus knots of type $(2,2k+1)$ are the simplest knots, starting from the trefoil knot for $k=1$. The next family, as listed in Tait and Rolfsen tables, are twist knots, starting from the figure eight knot. 
We will denote such knots by $[-2,-m]$. Rolfsen, calls  knots with $m=2q$, $q$-twist knots, and the notation $[-2,-m]$ is a notation for rational (2-bridge) knots and links which we review in Section \ref{RaTaAl}. In Figure \ref{TwistKnot1}, we use the notation 
 $[\pm 2,n)$ for a rational tangle, $[2,n]$ for its closure in $\mathbb{R}^3$ and $[2,n]^{ann}$ for its closure in the solid torus. 
\begin{exercise} Compute the cubic polynomial for a twist knot in $S^3=\mathbb{R}^3\cup \infty$ and in the solid torus $D^2 \times S^1= Ann \times [0,1]$.
\end{exercise}
{\bf Hint.} In both cases the beginning is the same: we resolve the 2-tangle of Figure \ref{TwistKnot1} but once projected on the disc $D^2$ and second time projected onto the annulus (see Figure \ref{TwistKnot1}).

Similarly as $T(2,n)$, using polynomials $P_n $ (or $P_n^f$) we can write simple formulas for twist knots and  more generally pretzel links.  We can also write a good, polynomial complexity, algorithm for rational (e.g. 2-bridge) links (see Section \ref{RaTaAl}), or more generally 2-algebraic links or even $3$-algebraic links and tangles ($40 \times 40$ would be handy then.)
For pretzel links we just reduce every column to $1,0,$ or $-1$, so we deal with a torus link of type $(2,n)$ considered before; see Subsection \ref{sec:pretzels}.

\begin{example} We will compute here the formula for the twist knot but because we would like to work with $b_3=-1$ we will work with the twist knot in the Conway form $[-2,-n]$. The reason for this is that the tangle $[-2,-n]$ is 
equivalent to $[2,-n-1]$ with only 
right handed (positive) half-twists.
We will also work with rational tangles (for more generality); compare Section \ref{RaTaAl}. Our conventions are explained in Figure \ref{TwistKnot}.
We get:
$$[-2,-n]=[2,-n-1] \stackrel{b_3=-1}{=} b_2[n]+ a^{-n-1}b_1[\infty]+b_0[n+2]+ab_\infty [n+1].$$
As computed before ($n$-move calculation)
$$[m]\stackrel{b_3=-1}{=} P_{m-1}[1]+ (b_1P_{n-2}+b_0P_{n-3})[0]+b_0P_{m-2}[-1]+U_{m,m-1}b_\infty [\infty].$$
We can now non-trivially close our tangles to get the twist knot with initial conditions
 $[1]=at$, $[0]=t^2$,
$[-1]=a^{-1}t$, and $[\infty]=t$; see Remark \ref{num-dom-base}.

\begin{figure}[ht]
\centering
\includegraphics[width=0.75\linewidth]{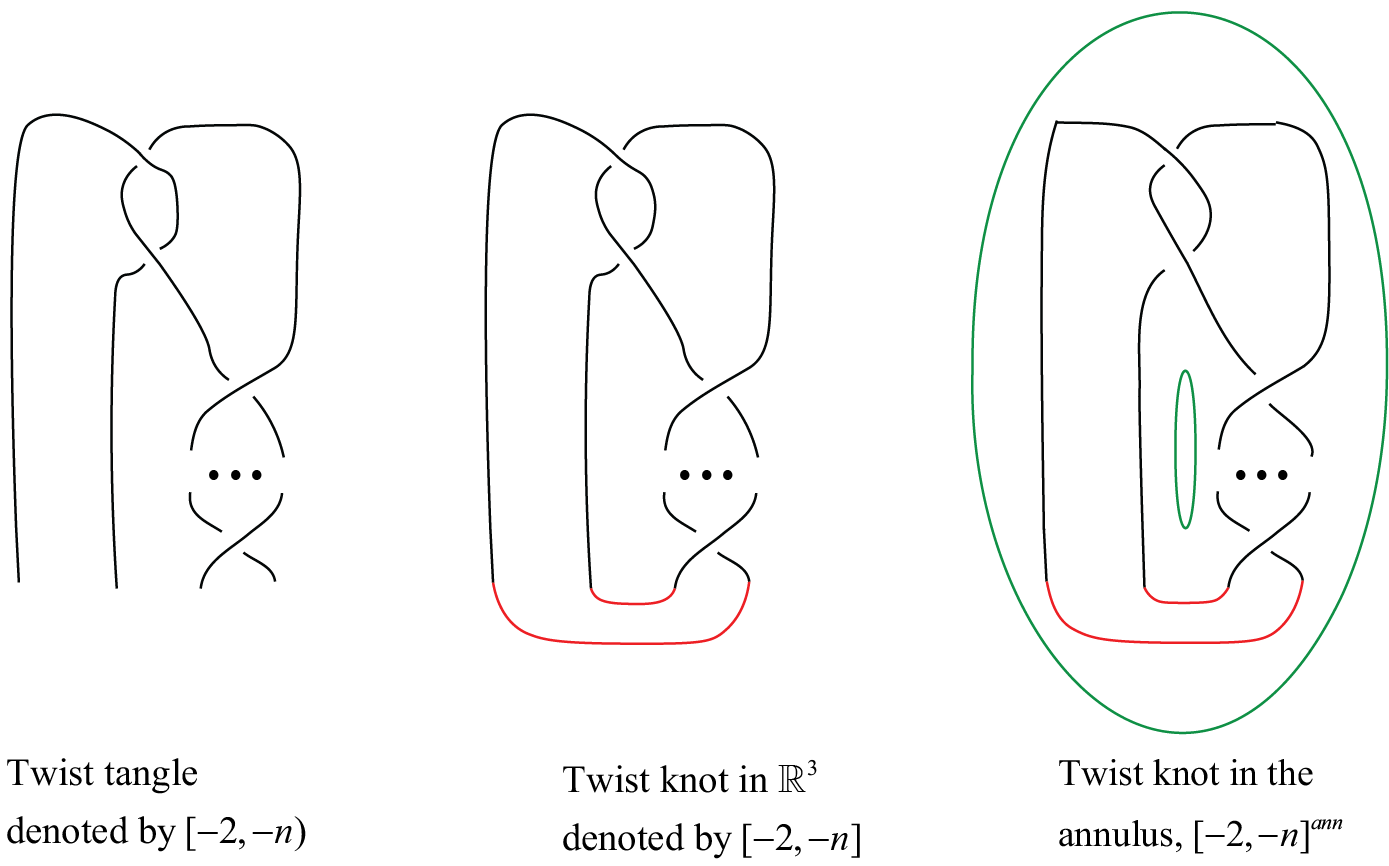}
  
\caption{Twist tangle and twist knots in $\mathbb R^3$ and $Ann \times [0,1].$}
\label{TwistKnot1}
\end{figure}

\begin{figure}[ht]
\centering
\includegraphics[width=0.75\linewidth]{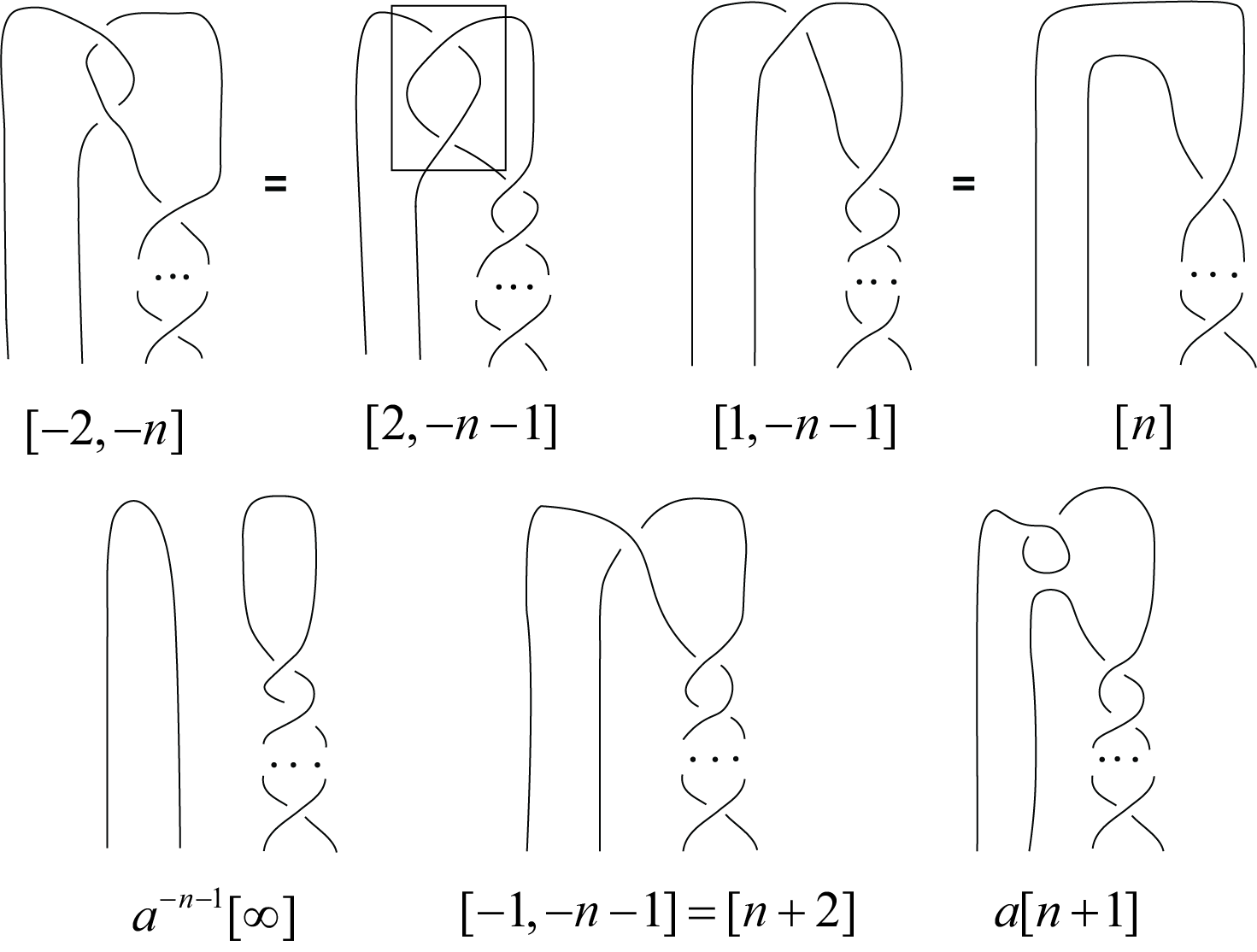}
\caption{Twist knot and its skein relation diagrams (the bottom left figure represents $a^{-n-1}[\infty].$}
\label{TwistKnot}
\end{figure}
\end{example}

\section{Mutations and the Cubic Skein Module}\label{Mut}
A mutation on a link, as defined by J.H. Conway, is a local change of the link 
supported in a 3-disk and modifying a 2-tangle of the link by one of the rotations described in Figure \ref{rotation3}. 

\begin{figure}[ht]
\centering
    \includegraphics[width=0.75\linewidth]{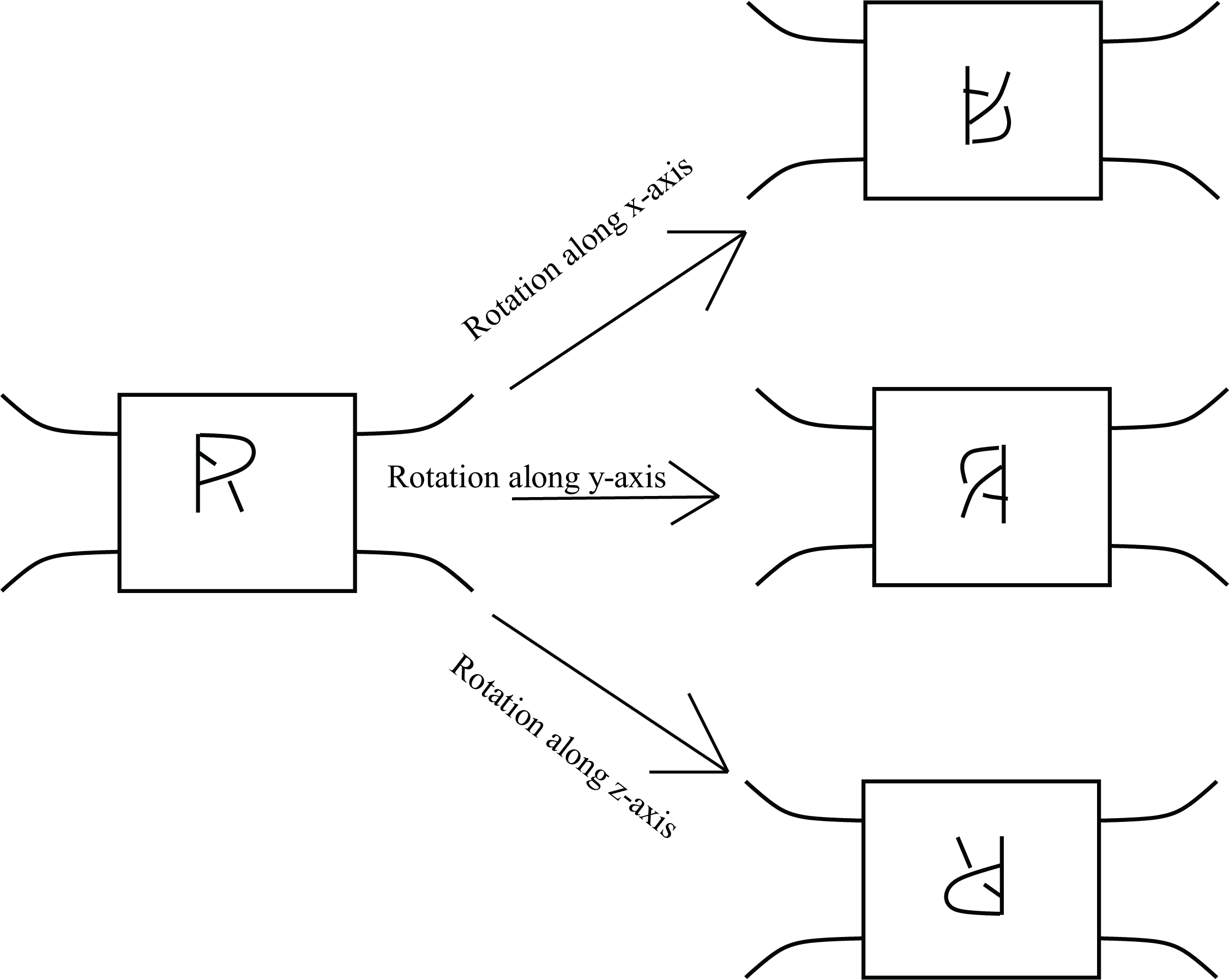}
\caption{Mutations along $x$, $y$ and $z$-axis.}
\label{rotation3}
\end{figure}

\begin{figure}[ht]
\centering
\includegraphics[scale=.3]{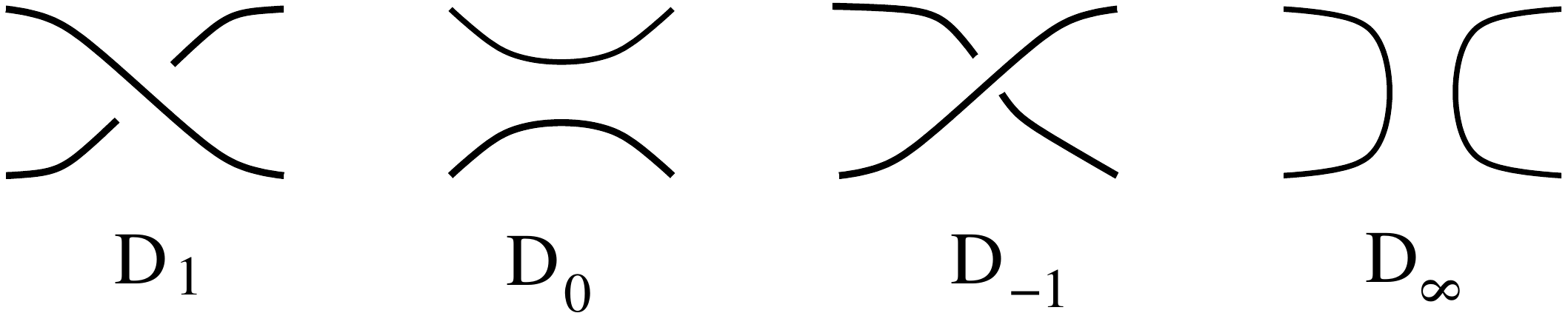}
\caption{Basic $2$-tangles preserved by mutations along $x$, $y$ and $z$-axis.}
\label{D-1D0D1Dinf}
\end{figure}

A reasonable question is to ask whether mutations on a link in a 3-manifold preserve its value in the cubic skein module. A similar question has been partially answered for the fourth skein module by Przytycki and T. Tsukamoto in \cite{PTs}.

\begin{proposition}\label{mutation1}
Assume that a 2-tangle $T$ on which a mutation is performed is, in the cubic skein module, generated by four tangles $D_0$, $D_\infty$, $D_1$ and $D_{-1}$, see Figure \ref{D-1D0D1Dinf}. If a link $L$ is mutated on the tangle $T$ to get link $m(L)$ then $L$ and $m(L)$ are equal in the cubic skein module.
\end{proposition}
\begin{proof}
The proposition follows from the fact that  four generating 2-tangles of Figure \ref{D-1D0D1Dinf} are preserved by mutation.
\end{proof}
We know several cases when assumption on the tangle $T$ of Proposition \ref{mutation1} holds:

\begin{corollary}\ 
The assumption of Proposition \ref{mutation1} holds for
\begin{enumerate}
\item[(1)] $T$ is a rational (that is 2-bridge) tangle. 
\item[(2)] $T$ is a $2$-algebraic tangle including pretzel tangle.
\item[(3)] $T$ is a $3$-algebraic $2$-tangle (that is a 2-tangle obtained from an algebraic $3$-tangle by connecting two boundary points of\ $T$ without intersection, see \cite{PTs}).
\end{enumerate}
\end{corollary}
\begin{proof}  (1) is a special case of (2).\\
(2) follows from the fact that, in the cubic skein module, $2$-algebraic tangles are generated by four generating 2-tangles of Figure \ref{D-1D0D1Dinf}.\\
(3) This follows by careful analysis of 40 basic $3$-algebraic tangles, see Figures \ref{InvertibleTangles} and \ref{Non-InvertibleTangles}, and \cite{PTs}.
\end{proof}

\subsection{Conway and Kinoshita-Terasaka Mutants}

\begin{figure}[ht]
    \centering
    \includegraphics[width=0.75\linewidth]{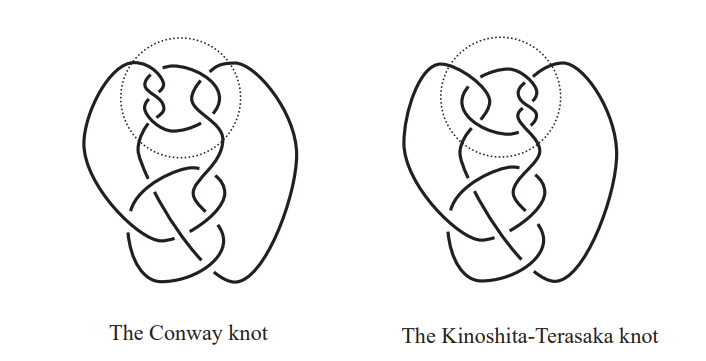}
    \caption{Two popular 11-crossing mutant knots.}
    \label{Pair}
\end{figure}

The Conway and Kinoshita-Terasaka knots is a well-known pair of knots related by mutation, and thus, worth being studied in the cubic skein module. In the calculation, we use the fact that both knots are generated by trivial links to compute an explicit polynomial in the cubic skein module of $S^3$, demonstrating their equality. In fact, equality in the cubic skein module follows from the fact that the following circled tangles in both knots are $2$-algebraic, in fact they are pretzel tangles.\footnote{Recall that $2$-algebraic means ``algebraic" in the sense of Conway.}

We present the first step of computations of the polynomial for the diagram of the Kinoshita-Terasaka knot in the cubic skein module, expressing the diagram on the right in Figure \ref{Pair} as a linear combination of diagrams with strictly fewer crossings. We do the same for the Conway Knot, and then compare their respective equations. By definition, the complement of the circled 2-tangle is the same in both knots. We notice that these partial polynomial computations produce the same equation with isotopic diagrams associated to given $b_i$. This allows us to conclude that these two diagrams produce the same polynomial in $\mathcal{S}_{(4,\infty)}(S^3)$, illustrating the Proposition \ref{mutation1}.

We resolve the circled $2$-algebraic tangle in the Kinoshita-Terasaka knot as follows by resolving crossings $v_{KT}$ and $v_C$. The calculation is as follow:

$$-b_0 \vcenter{\hbox{\includegraphics[scale=1] {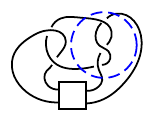}}} = b_1 \vcenter{\hbox{\includegraphics[scale=1] {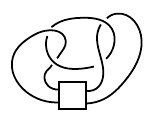}}}+ b_2 \vcenter{\hbox{\includegraphics[scale=1] {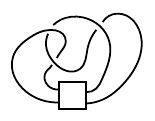}}} +b_3 \vcenter{\hbox{\includegraphics[scale=1] {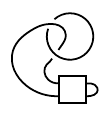}}}+a^3 b_{\infty}\vcenter{\hbox{\includegraphics[scale=1] {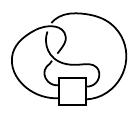}}}.$$

The polynomial for Conway knot in cubic skein module is as follow:

$$-b_0 \vcenter{\hbox{\includegraphics[scale=1] {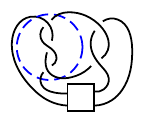}}} = b_1 \vcenter{\hbox{\includegraphics[scale=1] {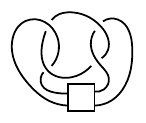}}}+ b_2 \vcenter{\hbox{\includegraphics[scale=1] {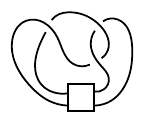}}} +b_3 \vcenter{\hbox{\includegraphics[scale=1] {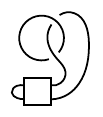}}}+a^3 b_{\infty}\vcenter{\hbox{\includegraphics[scale=1] {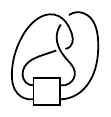}}}.$$

\begin{remark}
Notice that while comparing diagrams associated to $b_i$, the diagrams are identical except for the diagrams of coefficient $b_2$, which are related by Tait Flype.
\end{remark}

\section{Independence of the Order of Resolving Crossings}\label{Inord}
If we perform cubic skein module calculations in two disjoint 3-balls $B_1^3$ and $B_2^3$ then the result does not depend on the order of calculations; that is if we first do calculations in $B_1^3$ and then $B_2^3$, the result is the same as if we had first calculated in $B_2^3$ and then $B_1^3$. \\ This seemingly obvious observation should be stressed and is often used in our calculations, e.g. 2-bridge links algorithm. For most of the readers, this statement is self-explanatory and sufficient to read the rest of the paper, however we think it is useful to dwell more on this property, which we may name {\it no interaction at a distance} (or the locality principle) as it is of independent interest and led to the {\it entropic condition} well studied by researchers in non-associative systems.

\subsection{Historical Perspective}

It was important historically to see that the computation of the Kauffman bracket, HOMFLYPT, and Kauffman/Dubrovnik 
polynomials of links in $\mathbb R^3$ does not depend on the order of crossing resolution. This let the authors of \cite{PrTr}, in the case of the HOMFLYPT polynomial, consider an abstract algebra, which they called the \textit{Conway algebra},  in which the binary operation $*$ satisfies the entropic relation $(a*b)*(c*d)=(a*c)*(b*d)$ (see Subsection \ref{Entropic}). This idea, in the case of Kauffman/Dubrovnik polynomials, was analyzed in \cite{Prz1986} and the case of Kauffman bracket was discussed in \cite{NiPr}.
For a skein relation involving more than 3 terms, the binary relation is replaced by operations with more parameters.

In this paper we work with skein modules but one can consider even more general algebraic structures. For a generalization of a cubic skein module, the adequate structure requires $4$-ary relations e.g. $*:X^4\to X$ satisfying a proper enhancement of the entropic relation. It is above the scope of this paper but here, we want introduce the readers to interpreting skein relations on $n$ diagrams as $(n-1)$-ary operations which satisfy entropic conditions. We encourage further research in this direction; see Subsection \ref{Entropic} for definitions.

\subsection{From Linear Skein Relations to Entropic Condition}\label{FLSRE}

In the case of a general skein module of a 3-manifold, the property of ``no interaction at a distance" should be  precisely formulated.
The setting is as follows. \\ 
Consider two disjoint oriented balls $B_1^3$ and $B_2^3$ in an oriented 3-dimensional manifold $M$ and fix the part of a link outside  $B_1^3$ and $B_2^3$.
We now consider skein relations taking place only inside $B_1^3$  and $B_2^3$ (local condition).
We ask whether a calculation done in $B_1^3$ before calculation in $B_2^3$ gives the same result as calculation done first in $B_2^2$ and then in $B_1^3$. 
One can say that, obviously yes, as we know generally that if $M$ is a module over a  ring $R$ and $M_1$ and $M_2$ are submodules of $M$ then 
$(M/M_1)/M_2 = (M/M_2)/M_1 $ (and it is $M/(M_1+M_2$). 
However it is not exactly what we ask. Here, we have fixed particular skein relations which we apply in differing order.

The general case will be discussed in Subsections \ref{FullGen} and \ref{Entropic} but before, as a preparation, we discuss two special cases, the Kauffman bracket relation (see Subsection \ref{KBE}) and the cubic setting (as it is the main topic of this paper; see Subsection \ref{FLSRE}).

\subsubsection{From Kauffman bracket to classical entropic condition}\label{KBE}

Consider the  familiar case of  the Kauffman bracket skein relation in an oriented 3-manifold $M$. We use the relation $D_1= AD_0 + BD_{\infty}$ (we do not need $B=A^{-1}$ 
at this point). Thus as before consider two disjoint oriented balls $B_1^3$ and $B_2^3$ in $M$ and a link with two crossings $v_1$ in $B_1^3$ and $v_2$ 
in $B_2^3$. Denote our link by $D^{v_1,v_2}_{1,\ 1}$. Now if we resolve first the crossing $v_1$ by the Kauffman bracket skein relation we get 
$  D^{v_1,v_2}_{1,1}= A D^{v_1v_2}_{0,1} +B  D^{v_1,v_2}_{\infty,1}$ then when we resolve crossing $v_2$ we get 
$$D^{v_1,v_2}_{1,1}= A^2D^{v_1v_2}_{0,0} + ABD^{v_1v_2}_{0,\infty} + BAD^{v_1,v_2}_{\infty,0} + B^2D^{v_1,v_2}_{\infty,\infty}.$$
If we change the order of  crossings we resolve so we start from $v_2$ (and then follow with $v_1$) then we get 
$$ D^{v_1,v_2}_{1,\ 1}= AD^{v_1,v_2}_{1,0}+ B D^{v_1,v_2}_{1,\infty,}=  A^2D^{v_1v_2}_{0,0} + AB D^{v_1v_2}_{\infty,0}+BAD^{v_1v_2}_{0,\infty}+B^2D^{v_1,v_2}_{\infty,\infty}.$$
Thus the same as before as long as $A$ and $B$ commute. In fact, it is the main reason that we have to assume that our rings are 
commutative. 

We can however go further and consider not the linear recursive relation but, as we did in \cite{PrTr,NiPr} but to consider a bilinear operation $*:X^2 \to X$ which is allowing to find $D_1$ when $D_0$ and $D_{\infty}$ are known. We can write $D_1= *(D_0,D_\infty) $ or simply 
(as we do usually with binary operations) that $D_1 =D_0*D_\infty$. We check directly that the change of order requires the magma $(X,*)$ to satisfy the condition $$*\bigg(*(D^{v_1,v_2}_{0,\ \ 0},D^{v_1,v_2}_{0,\ \ \infty}),*(D^{v_1,v_2}_{\infty,\ \ 0},D^{v_1,v_2}_{\infty,\ \infty})\bigg)= *\bigg(*(D^{v_1,v_2}_{0,\ \ 0},D^{v_1,v_2}_{\infty,\ \ 0}),*(D^{v_1,v_2}_{0,\ \ \infty},D^{v_1,v_2}_{\infty,\ \infty})\bigg)$$ 
or shortly
$$(D^{v_1,v_2}_{0,\ \ 0}*D^{v_1,v_2}_{0,\ \ \infty})*(D^{v_1,v_2}_{\infty,\ \ 0}*D^{v_1,v_2}_{\infty,\  \infty})=(D^{v_1,v_2}_{0,\ \ 0}*D^{v_1,v_2}_{\infty,\ \ 0})*(D^{v_1,v_2}_{0,\ \ \infty}*D^{v_1,v_2}_{\infty,\  \infty}).$$ 
This can be nicely described by creating $2 \times 2 $ matrix

\[ \{D^{v_1,v_2}_{i,\ \ j}\}_{\{i,j\in \{0,\infty\} \}}=
\left[
\begin{array}{cc}

D^{v_1,v_2}_{0,\ \ 0}   & D^{v_1,v_2}_{0,\ \ \infty}    \\
& \\
 D^{v_1,v_2}_{\infty,\ \ 0}   & D^{v_1,v_2}_{\infty,\ \ \infty}    

\end{array}
\right],
\]
and perform $*$ operation first on rows and then on columns and equal the result if we use first columns and then rows. 
This matrix visualization of the entropic relation allows a natural generalization to more general skein relations. We start with the case of cubic skein relation and its abstraction a $4$-ary relation.

\subsection{Change of Order Calculation: the Case of Cubic Skein Relation and Corresponding Entropic Condition}\label{order-cubic}\
Consider the cubic skein relation $b_3D_3+b_2D_2+b_1D_1+b_0D_0+ b_\infty D_{\infty}=0$ applied in two disjoint disks
$B_1^3$ and $B_2^3$ in an oriented 3-manifold $M$. We use the notation $D^{B_1,B_2}_{D_i,D_j}$ if the link has the 2-tangle $D_i$ in the disk $B_1$ and
the 2-tangle $D_j$ in the disk $B_2$. Thus, the cubic relation taking place in $B_1^3$ has the form
$$b_3D^{B_1,B_2}_{D_3,T} +b_2D^{B_1,B_2}_{D_2,T}+b_1D^{B_1,B_2}_{D_1,T} +b_0D^{B_1,B_2}_{D_0,T} + b_\infty D^{B_1,B_2}_{D_\infty,T} =0,$$
where $T$ is an arbitrary fixed 2-tangle in $B_2$. Similarly the cubic relation in $B_2$ takes the form:
$$b_3D^{B_1,B_2}_{T,D_3} +b_2D^{B_1,B_2}_{T,D_2}+b_1D^{B_1,B_2}_{T,D_1} +b_0D^{B_1,B_2}_{T,D_0} + b_\infty D^{B_1,B_2}_{T,D_\infty} =0,$$
where $T$ is an arbitrary fixed 2-tangle in $B_1$. We now claim that the result of computation is independent of the order of resolution: i.e. using the cubic skein relation to reduce the tangle in $B_1^3$ and then reducing the tangle in $B_2^3$ is the same as reducing first $B_2^3$ and then $B_1^3$ as long as the ring taken for our skein module is commutative. To demonstrate, this we compute $D^{B_1,B_2}_{D_3,D_3}$ using two different orders and for simplicity
we assume that $b_3$ is invertible. Applying the cubic skein relation in $B_1$ first, we have

\begin{eqnarray*}
    D^{B_1,B_2}_{D_3,D_3} &\stackrel{B_1}{=}& (-b_3)^{-1}\sum_{i\in\{2,1,0,\infty\}} b_i D^{B_1,B_2}_{D_i,D_3} \\
    &\stackrel{B_2}{=}& 
b_3^{-1}\sum_{i\in\{2,1,0,\infty\}} b_i \bigg(b_3^{-1}\sum_{j\in\{2,1,0,\infty\}} b_j D^{B_1,B_2}_{D_i,D_j} \bigg) \\
&=& 
b_3^{-1} \sum_{i,j \in\{2,1,0,\infty\}} b_ib_3^{-1}b_j D^{B_1,B_2}_{D_i,D_j}. 
\end{eqnarray*}

If we perform the calculation first in  disk $B_2^3$, we get
$$b_3^{-1} \sum_{i,j \in\{2,1,0,\infty\}} b_jb_3^{-1}b_i D^{B_1,B_2}_{D_i,D_j}, $$
so the result is the same provided $b_jb_i= b_ib_j$ for all $i,j \in \{3,2,1,0,\infty \}$.

\begin{remark}\label{localTBT} The fact that our skein relations are local (we can change the order of calculations) can be leveraged to reduce the integral subtangles of a given rational link to subtangles with $|n_i| \leq 2$.  In particular, we later study (see Proposition \ref{reversecode}) relations of the form $[n_1,n_2,...,n_k] - \epsilon[n_k,...,n_2,n_1]$ where $\epsilon=1$ if $k$ is odd and $\epsilon = -1$ if $k$ is even. These Conway codes represent diagrams of ambient isotopic links. Our locality condition ensures that we can reduce  $n_i$'s in any order.
\end{remark}

\subsubsection{Entropic relation from cubic skein relation}\label{ERCSR}
The direct generalization of the cubic relation into an entropic condition is seen by replacing the computation for $D_i$ (including $i=\infty$) by the 4-ary relation $*_i: X^4\to X$. For example $D_3=*_3(D_2,D_1,D_0,D_{\infty})$. 
Now we work with the cubic relation in two disjoint 3-disks $D_1$ and $D_2$. For example let us compute $D^{B_1,B_2}_{3,\ \ 3}$ by eliminating index $3$ (compare Figure \ref{cross-order}). Our relation is  best visualized by the following matrix:

\[ \{D^{B_1,B_2}_{i,\ \ j}\}_{\{i,j\in \{2,1,0,\infty\} \}}=
\left[
\begin{array}{cccc}
D^{B_1,B_2}_{2,\ \ 2} & D^{B_1,B_2}_{2,\ \ 1} & D^{B_1,B_2}_{2,\ \ 0}   & D^{B_1,B_2}_{2,\ \ \infty}    \\
 &&& \\
D^{B_1,B_2}_{1,\ \ 2} & D^{B_1,B_2}_{1,\ \ 1} & D^{B_1,B_2}_{1,\ \ 0}   & D^{B_1,B_2}_{1,\ \ \infty}    \\
&&& \\
D^{B_1,B_2}_{0,\ \ 2} & D^{B_1,B_2}_{0,\ \ 1} & D^{B_1,B_2}_{0,\ \ 0}   & D^{B_1,B_2}_{0,\ \ \infty}    \\
&&& \\
D^{B_1,B_2}_{\infty,\ \ 2} & D^{B_1,B_2}_{\infty,\ \ 1} & D^{B_1,B_2}_{\infty,\ \ 0}   & D^{B_1,B_2}_{\infty,\ \ \infty}    

\end{array}
\right],
\]
We use the same $4$-ary relation $*_3$ (which we denote by $*:X^4\to X$) in the ball $B_1=B_1^3$ and $B_2=B_2^3$ in opposing orders. The resulting entropic condition is 
$$*\bigg(*(r_1),*(r_2),*(r_3),*(r_4)\bigg)= *\bigg(*(c_1),*(c_2),*(c_3),*(c_4)\bigg),$$
where $r_i$ are rows of the matrix $\{D^{B_1,B_2}_{i,\ \ j}\}_{i,j}$ and $c_j$ are its columns. 

\begin{figure}[ht]
\centering
\scalebox{.30}{\includegraphics{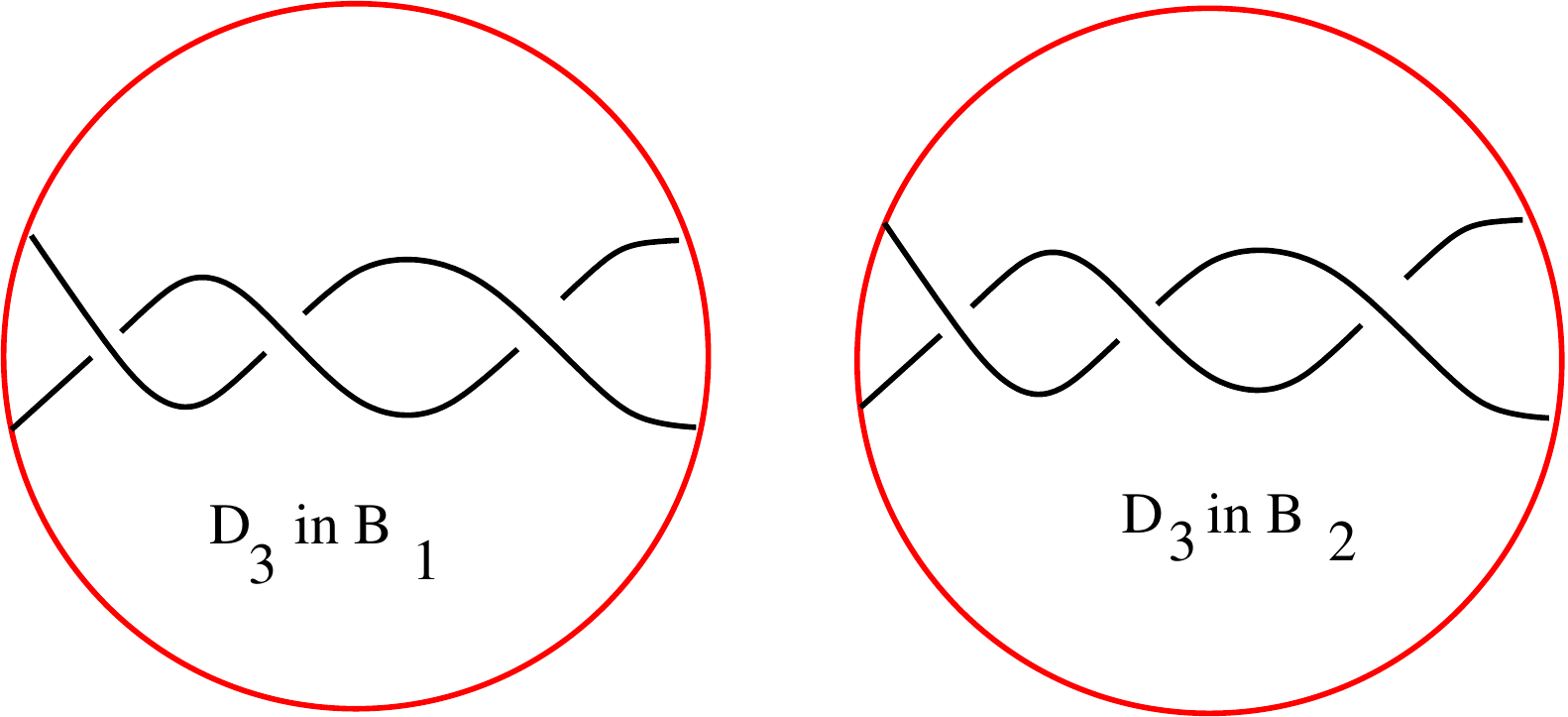}}
\caption{The link $D^{B_1,B_2}_{3,\ 3}$ with two tangles 
$D_3\subset B_1$ and $D_3\subset B_2.$}
\label{cross-order}
\end{figure}

\subsubsection{ Arbitrary local skein relation}\label{FullGen}
 We can interpret arbitrary local skein relations as $n$-ary operations more generally, again stressing the fact that changing the order of disks $B_1^3$ and $B_2^3$ produces the same result if we work with commutative ring $R$. Let us fix notation: $D^{B_1,B_2}_{T_1,T_2}$ denotes a link which is fixed outside $B_1^3$ and $B_2^3$ and in $B_1^3$ and $B_2^3$ contains the tangle $T_1$ and $T_2$ respectively. Assume that in $B_1^3$ we have a skein relation 
$\sum_{i=0}^{n_1} b_i^{(1)}T_i^{(1)}=0$ and in $B_1^3$ we have a skein relation 
$\sum_{i=0}^{n_2} b_i^{(2)}T_i^{(2)}=0$. Assume also, for simplicity that $b_{n_1}^{(1)}$ and $b_{n_2}^{(2)}$ are invertible in the ring $R$. Then we can 
compute $D^{B_1,B_2}_{T_{n_1}^{(1)},T_{n_2}^{(2)}}$ starting from the relation in $B_1^3$ or from relation in $B_2^3$ and for commutative $R$, the result is independent on the order of chosen disks $B_i^3$.
\begin{proof}
Let us perform calculation starting from the relation in the disk $B_1^3$. We have:
\begin{eqnarray*}
    D^{B_1,B_2}_{T_{n_1}^{(1)},T_{n_2}^{(2)}} &=& (-b_{n_1}^{(1)})^{-1}\sum_{i=0}^{n_1-1} b_i^{(1)}D^{B_1,B_2}_{T_{i}^{(1)},T_{n_2}^{(2)}} \\
    &=&
 (-b_{n_1}^{(1)})^{-1}\sum_{i=0}^{n_1-1} b_i^{(1)}\bigg((-b_{n_2}^{(2)})^{-1}\sum_{j=0}^{n_2-1} b_j^{(2)}D^{B_1,B_2}_{T_{i}^{(1)},T_{j}^{(2)}}\bigg) \\
 &=&
(b_{n_1}^{(1)})^{-1}\sum_{0\leq i < n_1, 0\leq j < n_2}b_i^{(1)}(b_{n_2}^{(2)})^{-1}b_j^{(2)}D^{B_1,B_2}_{T_{i}^{(1)},T_{j}^{(2)}}.
\end{eqnarray*}

If we use the skein relation in the disk $B_2^3$ first, we analogously get:
$$D^{B_1,B_2}_{T_{n_1}^{(1)},T_{n_2}^{(2)}}= (b_{n_2}^{(2)})^{-1}\sum_{0\leq i < n_1, 0\leq j < n_2}b_j^{(2)}(b_{n_1}^{(1)})^{-1}b_i^{(1)}D^{B_1,B_2}_{T_{i}^{(1)},T_{j}^{(2)}}.$$
In a commutative ring $R$ we have $b_i^{(1)}b_j^{(2)}=b_j^{(2)}b_i^{(1)}$ for all $i$ and $j$ so both expressions are equal.
\end{proof}
The property of independence of ordering leads now to a general entropic condition.

\subsubsection{From Local Skein Relations to General Entropic Condition}\label{Entropic} 
As part of a generalization we replace skein relations from Subsection \ref{FullGen} by $m$-ary and $n$-ary operations and use the first in the ball $B_1^3$ and the second in $B_2^3$. Here tangles are not necessarily 2-tangles and relations are arbitrary (as long as they are local). For simplicity we use $n$ for $n_1$ and $m$  for $n_2$. The generalized entropic condition guarantees us that the calculation does not depend on the order of balls $B_i^3$ (see Equation \ref{EnGe}).  
Thus the motivation for this section is a generalization of skein modules from first degree equations into ``skein" abstract algebras; that is 
having a family of tangles, say $D_1,D_2,..,D_k,D_{k+1}$ we can compute $D_{n+1}$ 
using an $n$-ary relation $*$: $D_{k+1}=*(D_1,...,D_k)$. For $k=2$ it was originally introduced in \cite{PrTr}. With this motivation we can introduce the concept of the general entropic condition.
We follow here mostly \cite{SmRo,RoSm,Prz6}.\\
Consider a set $X$ with two operations, one $m$-ary, $*_1: X^m \to X$ and the second $n$-ary, $*_2: X^n \to X$.
We say that $(X;*_1,*_2)$ is a generalized entropic algebra (that is, it satisfies the entropic condition) if
for $mn$ elements of $X$, say $D_{i,j}$ where $1\leq i \leq m$ and $1\leq j \leq n$ we have an entropic condition:
\begin{equation}\label{EnGe}
*_2\bigg(*_1(r_1),*_1(r_2),...,*_1(r_m)\bigg) = *_1\bigg(*_2(c_1),*_2(c_2),...,*_2(c_n)\bigg),
\end{equation}
where $r_i=(D_{i,1},D_{i,2},...,D_{i,n})$ and $c_j=(D_{1,j},D_{2,j},...,D_{m,j})$ are the rows and columns, respectively
of the matrix $\{D_{i,j}\}$ described below.

\[ \{D_{i,j}\}=
\left[
\begin{array}{cccccc}
D_{1,1} & D_{1,2} & D_{1,3} &  \cdots & D_{1,n-1} & D_{1,n} \\
D_{2,1} & D_{2,2} & D_{2,3} &  \cdots & D_{2,n-1} & D_{2,n} \\
D_{3,1} & D_{3,2} & D_{3,3} &  \cdots & D_{3,n-1} & D_{3,n} \\
\cdots  & \cdots   &\cdots  &\cdots   &\cdots     & \cdots  \\
D_{m-1,1} & D_{m-1,2} & D_{m-1,3} & \cdots & D_{m-1,n-1} & D_{m-1,n} \\
D_{m,1}   & D_{m,2}   & D_{m,3}   & \cdots & D_{m,n-1}   & D_{m,n}
\end{array}
\right],
\]
We can reformulate our condition by saying that $*_2$ is a homomorphism of the magma $(X,*_2)$ and $*_2$ is a
homomorphism of the magma $(X,*_2)$.
The case of $m=2=n$ is the classical case and the entropic condition can be written as $(a*_1b)*_2(c*_1d)=(a*_1c)*_2(b*_1d)$.
The name ``entropic" means ``inner turning" referring to swapping $b$ and $c$ (see \cite{RoSm,SmRo}.)
Entropic magmas were first considered in \cite{BuMa} in 1929 and the name entropic magma was coined in \cite{Eth} in 1949.

\section{From Quadratic to Cubic Relations}\label{quadtocube}
As mentioned in  Section \ref{secnkaffdubrov}, the quadratic skein relations leading to the HOMFLYPT, Kauffman and Dubrovnik polynomials and their related skein modules are extensively studied.
They are of interest in this paper as they lead to some specific cubic skein relations. This is the topic of this section.

Consider the general quadratic relation for framed unoriented links (we denote coefficients of the equation by $c_2,c_1,c_0,c_\infty$ so not to mix them with general coefficients of a cubic relation, $b_3,b_2,b_1,b_0,b_\infty$ which we already used):

\begin{equation}\label{quadratic1}
c_2D_2+c_1D_1+c_0D_0+ c_{\infty} D_{\infty}=0.
\end{equation}
The classical Kauffman and Dubrovnik polynomials (in framed version) satisfy:
\begin{equation}\label{KaDub1}
D_1+\varepsilon D_{-1}- x(D_0 + \varepsilon D_{\infty})=0.
\end{equation}
It yields the Kauffman polynomial for $\varepsilon =1$ and the Dubrovnik polynomial for $\varepsilon =-1$. We can rewrite \ref{KaDub1} as:
\begin{equation}\label{KaDub2}
D_2 - xD_1 +\varepsilon D_{0} - a^{-1}x \varepsilon D_{\infty}=0.
\end{equation}
Therefore in these cases $c_2=1,c_1=-x, c_0=\varepsilon$, and $c_\infty=-\varepsilon a^{-1} x$.

Furthermore, for $c_\infty$ invertible we get for the trivial component $$t=\frac{a^{-2}{c_2}+a^{-1}c_1+c_0}{-c_{\infty}}\stackrel{KD}{=} \frac{a^{-1}-x +\varepsilon a}{\varepsilon x}= \frac{a+\varepsilon a^{-1}-\varepsilon x}{x}.$$

To create a cubic skein relation we shift Equation \ref{quadratic1} to get\footnote{If $2$ is invertible in the ring, we do not lose any information by considering only $s=1$ and $s=-1$.}
\begin{equation}\label{quadratic2}
c_2D_3+c_1D_2+c_0D_1+ a^{-1}c_{\infty} D_{\infty}=0,
\end{equation}
and consider the linear combination $s$(\ref{quadratic2})- (\ref{quadratic1}) to get
\begin{equation}\label{quadratic3}
sc_2D_3+ (sc_1-c_2)D_2+(sc_0-c_1)D_1- c_0D_0 + (a^{-1}s-1)c_{\infty} D_{\infty}=0,
\end{equation}
Thus by comparing Equation \ref{quadratic3} to the general cubic relation $b_3D_3+b_2D_2+b_1D_1+b_0D_0+ b_{\infty} D_{\infty}=0$  we have $b_3=sc_2$, $b_2=sc_1-c_2$, $b_1=sc_0-c_1$, $b_0=-c_0$, and $b_\infty= (a^{-1}s-1)$. Notice that the coefficient of $D_\infty$ is zero if and only if $s=a$ or $c_\infty =0$. This case was considered in \cite{PTs} so here, unless otherwise stated, 
we consider $b_\infty$ to be invertible. Observe also that $b_3+sb_2+s^2b_1+ s^3b_0=0.$

It is worth writing Formula \ref{quadratic3} using the variables from the Kauffman and Dubrovnik polynomials. We get 

\begin{equation}\label{quadraticKD}
sD_3+ (-sx-1)D_2+(s\varepsilon+x)D_1- \varepsilon D_0 - \varepsilon a^{-1}x(a^{-1}s-1) D_{\infty}=0.
\end{equation}

Let us dwell a little more on the quadratic skein relation:\\
We start from the general quadratic relation
$$c_2D_2+ c_1D_1 + c_0D_0 + c_{\infty} D_{\infty}=0.$$
The denominator of this relation allows us to find $c_\infty t$:
$$a^{-2}c_2 + a^{-1}c_1+ c_0 + c_\infty t =0.$$
The numerator of the equation will involve the positive (right handed) Hopf link diagram:
$$c_2H^+ +(ac_1+ tc_0 + c_{\infty})t=0.$$
and similarly for its mirror image, the negative Hopf link diagram\footnote{The basic quadratic relation can be rewritten as
$$c_2D_0+ c_1D_{-1}+ c_0D_{-2}+a^2 c_{\infty} D_{\infty}=0$$
and from this we can observe that for the mirror limit we have  involution, say $\phi$ on coefficients: $\phi(c_0)=c_2, 
\phi(c_1)=c_1, \phi(a)=a^{-1}$ and 
$\phi(c_{\infty})=a^2c_{\infty}$ (so also $\phi(a^kc_{\infty})= a^{2-k}c_{\infty}$).}:
$$c_0H^{-} +(a^{-1}c_1+tc_2 +a^2c_{\infty})t=0.$$
Using the fact that the Hopf link is amphicheiral and combining the above formulas, we get:
$$0=c_0c_2(H^+-H^-)= \bigg((c_0^2-c_2^2)t + (ac_0c_1-a^{-1}c_2c_1) + (c_0-a^2c_2)c_\infty \bigg)t,$$ 
So the quadratic skein relation leads to the Hopf relation
$$0=(c_0^2-c_2^2)t +(ac_0-a^{-1}c_2)c_1 + (c_0-a^2c_2)c_\infty.$$
Further analysis of the trivial knot relation and the Hopf relation would lead us to the Kauffman and Dubrovnik polynomial and more, but it is beyond the scope of this paper. Here we only see a hint how we may start analyzing cubic skein modules see Section \ref{MRel} for further discussions.

\section{Rational Tangle Algorithm}\label{RaTaAl}

\subsection{Introduction}

The Rational Tangle Algorithm (RTA) describes an efficient way to quickly compute polynomials from diagrams of rational links in $\mathcal{S}_{4,\infty}(S^3)$ by reducing the subtangles of the corresponding rational link using the cubic skein relation. For more on relations in $\mathcal{S}_{4,\infty}(S^3)$, see Section \ref{sec:relations}.

The full statement of the Rational Tangles Algorithm is presented in  Subsection \ref{algorithmstatement}. Generally, the algorithm is similar to computing a polynomial invariant defined by a skein relation; i.e. we use a (cubic) skein relation to rewrite a link diagram as a combination of link diagrams with fewer crossings. Through subtangle reduction and various isotopies, we produce a combination of framed trivial links. Computation concludes by substituting $t$ with a polynomial in $R$, allowing us to express our original link diagram as an element in the ring $R$.

Because computation is diagram-specific, we recall and slightly extend the language of Conway Notation.

\subsection{Rational tangles \& Tangle Operations}
We recall the following basic operations on tangles in service of clarifying notational conventions and defining the Conway code of a Rational Link/Tangle.

\subsubsection{Tangle Operations}\label{tangleoperations} \hfill

\smallskip

Given $n$-tangles $T, S$, we have the following operations.

\begin{enumerate}
    \item \textbf{Addition} $T+S$, is a noncommutative, horizontal attaching of $T$ to $S$, where the NE arc of $T$ is connected to the NW arc of $S$ via a simple curve, and the SE arc of $T$ is similarly connected to the SW arc of $S$.

        In particular, an \textit{integral tangle} denoted $[n]$ for $n \in \mathbb{Z}$ is just the repeated addition of $[-1]$ or $[1]$ tangles, depending on the sign of $n$.

For example, 

\[ [3] = [1] + [1] + [1]. \]

\begin{figure}[H]
            \centering
            \includegraphics[scale=.5]{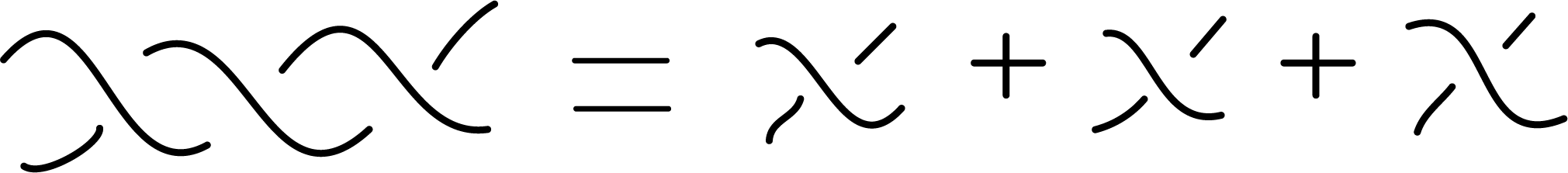}
        \end{figure}
                
    \item \textbf{Multiplication} $T * S$ is a noncommutative vertical attaching of $T$ to $S$ where the SW arc of $T$ is connected to the NW arc of $S$ and the SE arc of $T$ is connected to the NE arc of $S.$
    \item \textbf{Mirror Image} $-T$, is an order two operation wherein the undercrossings of $T$ are swapped with overcrossings and vice versa.
    \item \textbf{Inversion} $T^i = \frac{1}{T} = - T^r$ is an order two operation for rational tangle $T$ where $T$ is rotated 90 $\deg$ counterclockwise and then taken to its mirror image.

    In particular, \textit{vertical tangles}, denoted $\frac{1}{[m]}$ are comprised of $m$ vertical half-twists. For $m \in \mathbb{Z}_{>0}$, $\frac{1}{[m]} = [1] * [1] * \dots * [1]$.

\end{enumerate}

We now recall the definition of $2$-tangles and their numerator/denominator closures.

\begin{definition}[$2$-tangle]
     A $2$-tangle, of which rational tangles are a subset, is a proper embedding of two unoriented arcs in $B^3$ such that the four endpoints lie in the boundary of $B^3$.

We can label the four points of the tangle which lie along the boundary of $B^3$ as NW, NE, SW, and SE. 
\end{definition}

\begin{figure}[H]
    \centering
\includegraphics[width=0.5\linewidth]{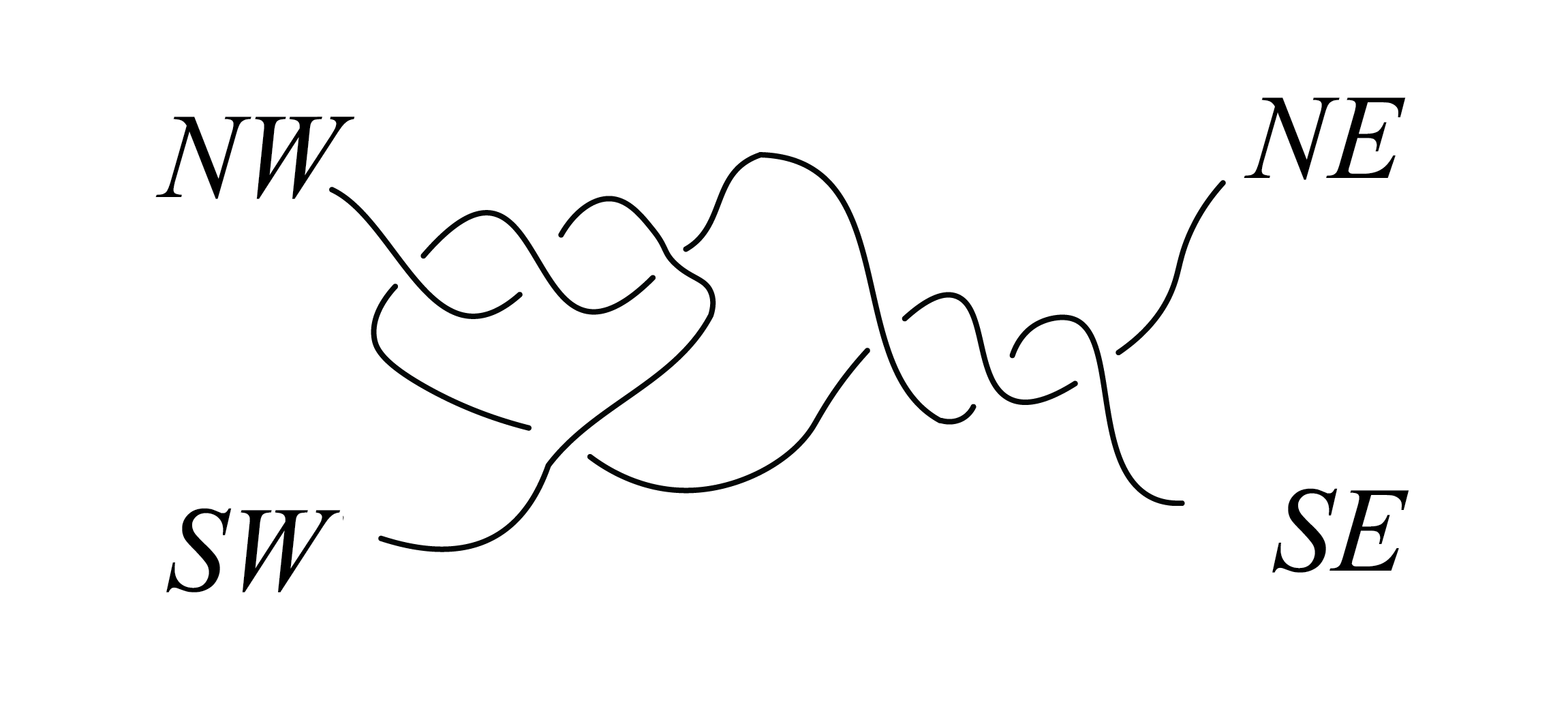}

    \caption{The rational 2-tangle $[3,-1,3]$ with NW, NE, SW, SE points labeled.}
    \label{fig:enter-label}
\end{figure}

\begin{definition}[Denominator closure]
   The denominator closure of a 2-tangle is the link formed by identifying the NW and SW arcs and the NE and SE arcs with two simple curves. See Figure \ref{denclosure}.
\end{definition}

\begin{figure}[H]
    \centering
    \includegraphics[width=0.3\linewidth]{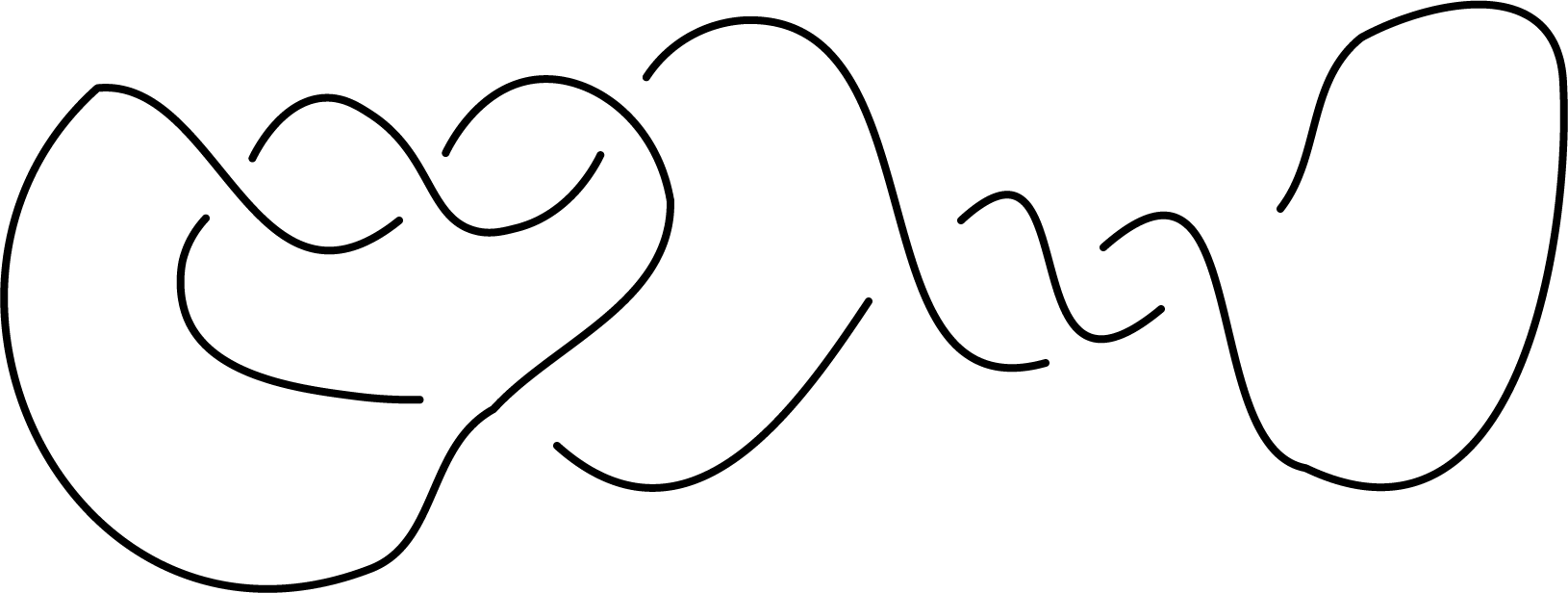}
    \caption{Denominator closure of $[3,-1,3]$, $D([3,-1,3]).$}
    \label{denclosure}
\end{figure}

\begin{definition}[Numerator closure]
    The numerator closure of a 2-tangle is the link formed by identifying the NW and NE arcs and the SW and SE arcs with two simple curves. See Figure \ref{numclosure}
\end{definition}

\begin{figure}[H]
    \centering
    \includegraphics[scale=.3]{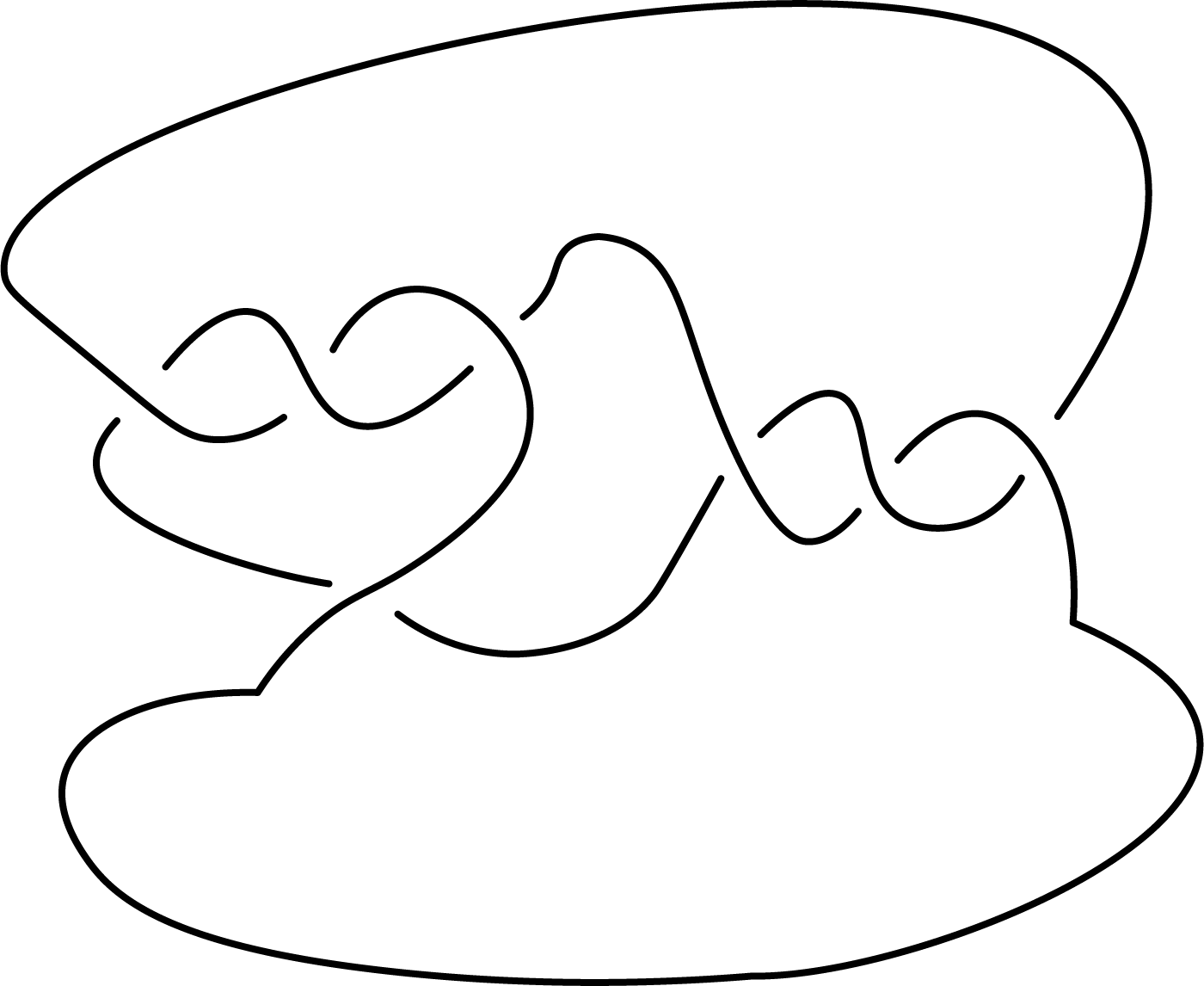}
    \caption{Numerator closure of $[3,-1,3]$, $N([3,-1,3]).$}
    \label{numclosure}
\end{figure}

\begin{remark}\label{num-dom-base}
We identify a set of four elementary rational tangles

\[ B = \left\{
    [0] =  \zero, \ \     
        [\infty] = \infinity, \ \
    [1]=\one, \ \ 
    [-1] = \none \right\} \] which serve as basic elements and to which we reduce all other rational tangles. Computation terminates with the closure of these basic tangles, and  those closures are collected here for easy reference. Again, $t$ denotes the trivial component and $a \in R$ is the element involved in the framing relation.

\begin{align*}
    \operatorname{N}([-1]) & = a^{-1} t =\numnone &     \operatorname{D}([-1]) & = a t = \dennone\\     
    \operatorname{N}([0]) & = t^2 =\numzero & \operatorname{D}([0]) & = t = \denzero \\
    \operatorname{N}([1]) & = a t =\numone &  \operatorname{D}([1]) & = a^{-1} t = \denone \\
    \operatorname{N}([\infty]) & = t =\numinf &  \operatorname{D}([\infty]) & = t^2 =\deninf\\
\end{align*}
\end{remark}

\subsubsection{Rational 2-Tangles}\label{subsection:rational2tangle}

The class of Rational 2-Tangles is a class of links in bijection with the rational numbers as proven by Conway; see for example \cite{KauLam}. Throughout the rest of this chapter, we refer to Rational 2-Tangles simply as rational tangles, but for a more general setting which includes Rational and Algebraic $n$-tangles, see \cite{PTs}, also see Section \ref{GS3AT}.

\bigskip

For convenience of the reader, we give the diagrammatic (inductive) definition of rational tangles and Links. 

\begin{definition}[Rational tangle, Conway; see e.g. \cite{KauLam}]
   rational tangles are the class of 2-tangles generated by the tangle operations of addition and multiplication: that is, those tangles created by consecutive additions and multiplications by the tangles $[\pm 1]$ starting from the tangle $[0]$ or $[\infty]$.
\end{definition}

Recall (see: e.g. \cite{KauLam}) that there are two standard conventions for drawing the diagrams of rational tangles and their closures, rational links; examples of rational tangles drawn in these two conventions are shown in Figures \ref{twobridge} and \ref{rationalstandardform}. We refer to Figure \ref{TwistKnot} and Figure \ref{twobridge}  as the \textit{two-bridge representation}.\footnote{Note that such a representation with two bridges always exists since the class of two-bridge links is the same as the class of rational links.}

The representation in Figure \ref{rationalstandardform}  will be our chosen diagrammatic convention, and will be referred to as the \textit{standard representation}, This diagram is described by an algebraic expression on tangles called the \textit{standard form}; this standard form is concisely notated with Conway Notation---see Definition \ref{standardform}.

\begin{figure}[H]
    \centering
    \includegraphics[scale=.20]{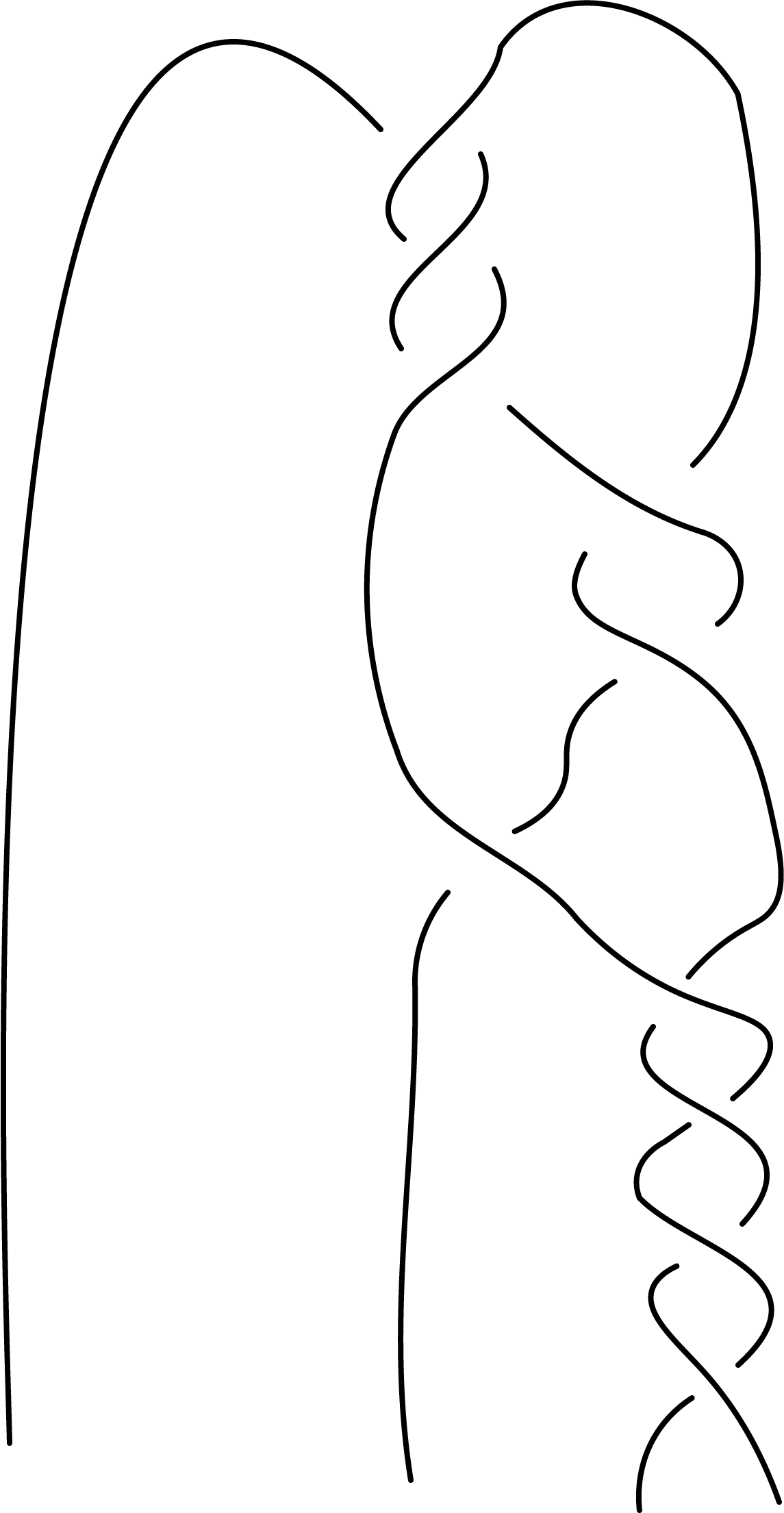}
    \caption{The two-bridge representation of rational tangle $[3,2,-1,4].$}
    
    \label{twobridge}
\end{figure}

\begin{figure}[H]
    \centering
    \includegraphics[scale=.20]{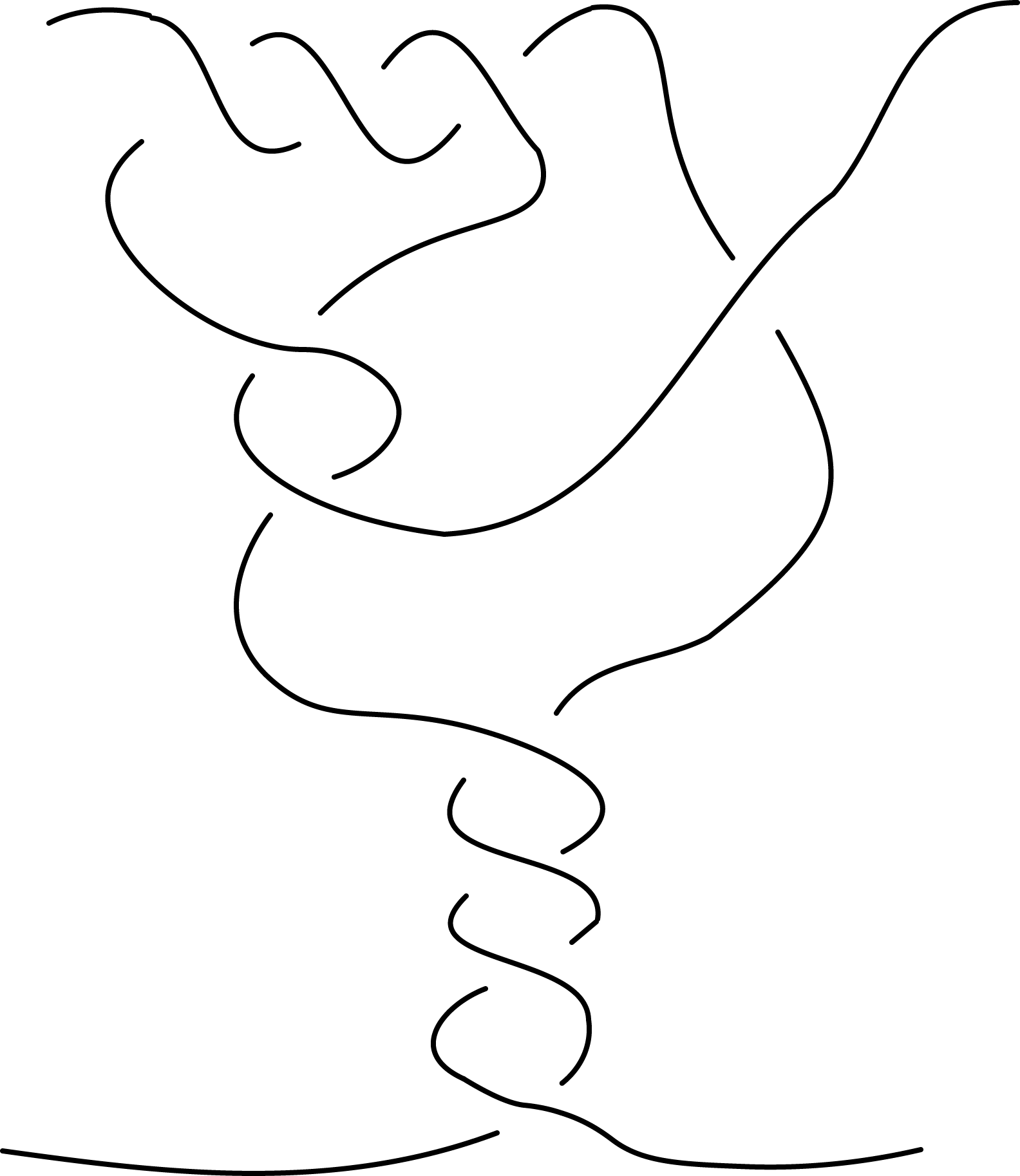}
    \caption{The standard representation diagram of rational tangle $[3,2,-1,4].$}
    \label{rationalstandardform}
\end{figure}

The two figures can be seen to represent the same tangle via a sequence of planar isotopies in Figure \ref{isotopies}.
\begin{figure}[H]
    \centering
    \includegraphics[width=1\linewidth]{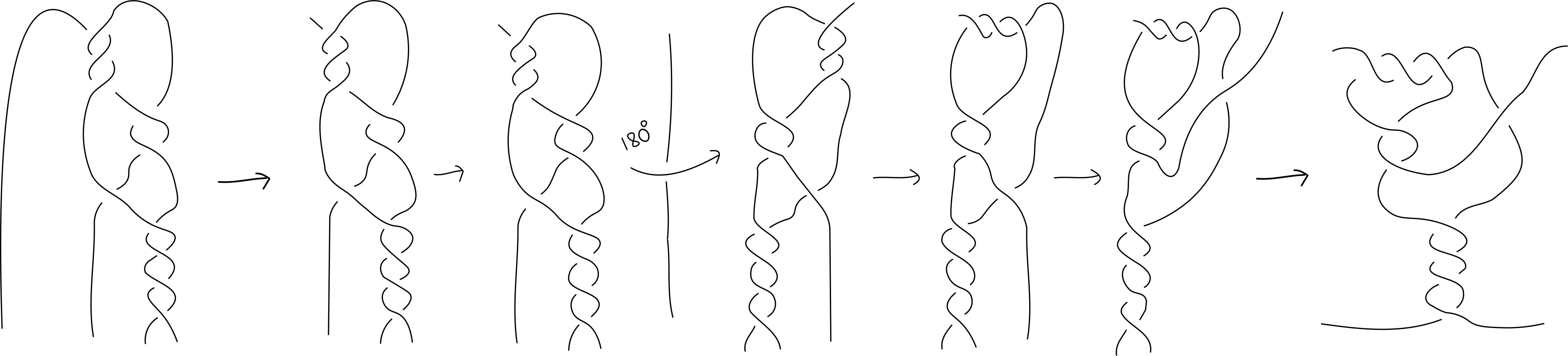}
    \caption{Isotopies taking the two-bridge representation of $[3,2,-1,4]$ to the standard representation.}
    \label{isotopies}
\end{figure}

\bigskip

\subsubsection{Conway Notation}

\begin{definition}[Standard form]\label{standardform}
The standard form of a rational tangle $T$ is an algebraic expression of the form
 
\[ [n_1, n_2, \dots , n_r] :=  (\dots ((([n_1]*\frac{1}{[n_2]})+[n_3])*\frac{1}{[n_4]})+ \dots) \]

with $n_1 \in \mathbb{Z} \cup \{ \infty \}$ and $n_i \in \mathbb{Z} \setminus \{ 0 \}$ for $1 < i \leq r.$ We call the finite length integral sequence $[n_1, n_2, \dots , n_r]$ the \textit{Conway Notation}, or Conway code of $T$. The diagram itself is said to be a \textit{standard representation} of $T$ as in Figure \ref{standardform}.
\end{definition}

Note that a tangle might have multiple equivalent Conway codes. For instance both $[3]$ and $[2,1] = [2]*\frac{1}{[1]}$ are standard form expressions for the right-handed trefoil. In fact, for a given Conway code, Proposition \ref{reversecode} gives another Conway code for an ambient isotopic link. 

Moreover, every rational tangle admits a standard form expression and thus a Conway code. We also note that a Conway code can be used to  denote the rational \textit{link} obtained by taking the ``correct" closure of the rational tangle. 

When we stress that a given Conway code $[n_1, n_2, \dots, n_r]$ signifies a link diagram, we write $\overline{[n_1, n_2, \dots, n_r]}$, though we will drop this overline when the object is clear from context. The following definition provides a description of the correct type of closure operation which yields a rational link  $\overline{[n_1, n_2, \dots, n_r]}$ from a rational tangle with Conway code $[n_1, n_2, \dots, n_r]$.

\begin{definition}[Closure]\label{howtoclose}
Given a rational tangle, $[n_1, n_2, \dots , n_r]$, the corresponding rational link is given by

\[\overline{[n_1, n_2, \dots, n_r]}  = \begin{cases}
    N([n_1, n_2, \dots, n_r]) & \text{ if }r=2k \\
    D([n_1, n_2, \dots, n_r]) & \text{ if }r=2k+1. \\
\end{cases}  \]
\end{definition}

\begin{figure}
    \centering
    \includegraphics[width=0.5\linewidth]{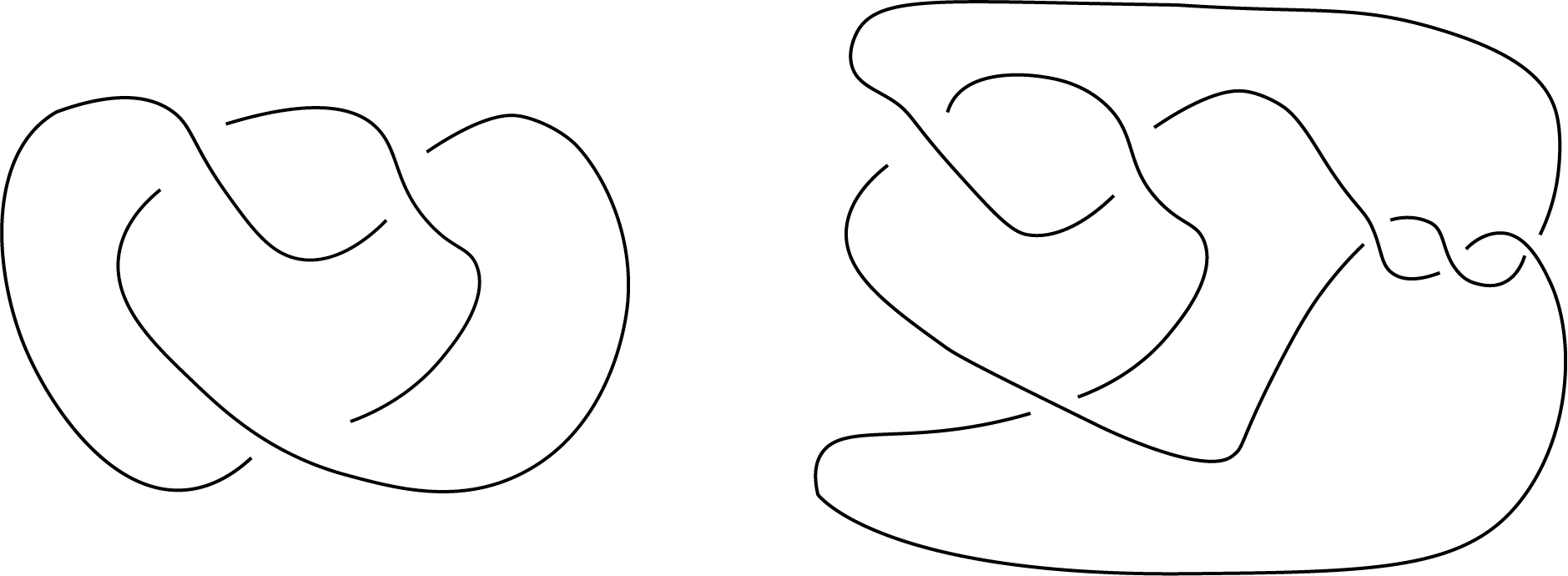}
    \caption{Note the opposing types of closure for $\overline{[2,1]}$ and $\overline{[2,1,3]}.$}
    \label{fig:enter-lab}
\end{figure}

Note that taking the opposite closure of a rational tangle in standard form will produce a different rational link, one that does not share the same Conway code. In fact, we can characterize the rational links produced by the ``wrong" closure of a rational tangle with Conway code $[n_1, n_2, \dots, n_r]$.

\begin{proposition}
  Given odd-length Conway code $[n_1, n_2, \dots, n_{2k+1}]$,
  \[D([n_1, n_2, \dots, n_{2k+1}]) = a^{-n_{2k+1}} \overline{[n_1, n_2, \dots, n_{2k}]}. \] If $k=0,$ $D([n_1]) = a^{-n_1} D([0]) = a^{-n_1} t. $

  \smallskip

  Given an even-length code, we have 
  \[ N([n_1, n_2, \dots, n_{2k}]) = a^{n_{2k}} \overline{[n_1, n_2, \dots, n_{2k-1}]}.\]
\end{proposition}

\begin{proof}
    In the case of an odd-length code with $k=0$, we reduce $n_1$ half twists by Reidemeister 1 moves, noticing that the nugatory crossings from the denominator closure of a positive integral tangle have negative framing and vice versa. Resolving these nugatory crossings, we are left with a trivial link with $a^{-n_1}$ framing. 
    
    For $k>0$, we note $[n_1, \dots, n_{2k+1}]=[n_1, n_2, \dots, n_{2k}] + [n_{2k+1}]$. The denominator closure of this tangle identifies the NE and SE arcs of the integral tangle $[n_{2k+1}]$ with a simple curve, forming $[n_{2k+1}]$ nugatory crossings which we reduce via a sequence of $n_{2k+1}$ Reidemeister 1 moves while accounting for  framing. After reducing, we are left with the denominator closure of the rational tangle $[n_1, n_2, \dots, n_{2k}]$ which is $\overline{[n_1, n_2, \dots, n_{2k}]}$ by definition. See Figure \ref{finalintegraltangletrivialized}.

    The even case is similar, where we view $[n_1, n_2, \dots, n_{2k}]$ as the product $[n_1, n_2, \dots, n_{2k-1}] * \frac{1}{[n_{2k}]}.$ When reducing the nugatory crossings of the vertical tangle $\frac{1}{[n_{2k}]}$, we notice that a positive $n_{2k}$ will produce positively-framed nugatory crossings.

    \begin{figure}[H]
        \centering
\includegraphics[width=0.8\linewidth]{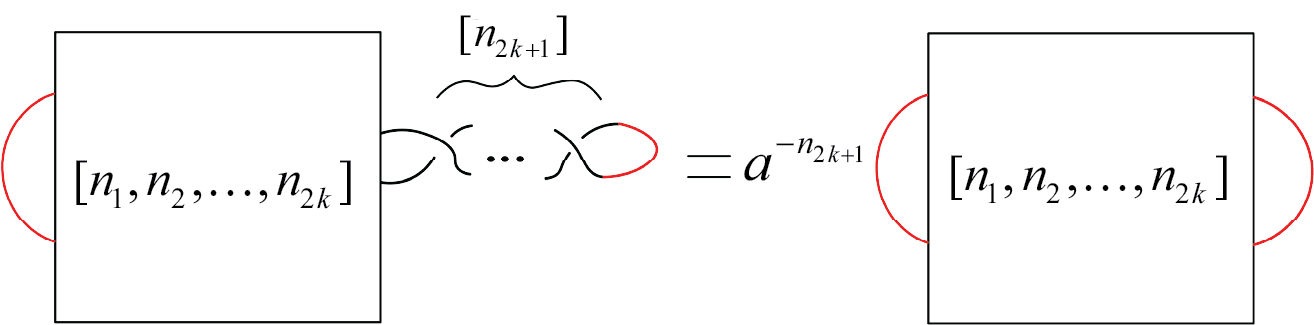}
\caption{Reduction of nugatory crossings in the denominator closure of odd-length Conway code.}
\label{finalintegraltangletrivialized}
    \end{figure}
\end{proof}

As mentioned, Conway codes for a tangle $T$: 1) always exist and 2) are not unique to a tangle/link. The following proposition shows that there are at least two Conway codes for every link. This is useful since, to form relations, we need two isotopic diagrams.

\begin{proposition}\label{reversecode}
For a rational link diagram $D$ given by Conway code $[n_1, n_2, \dots, n_r]$,

\[ D' = \begin{cases}
    [n_r, n_{r-1}, \dots, n_2, n_1] & \text{if $r$ is odd} \\
    [-n_r, -n_{r-1}, \dots, -n_2, -n_1] & \text{if $r$ is even} \\
\end{cases}\]

is an isotopic link.

\end{proposition}

\begin{remark}[Notational convention]
    Finally, we remark that our convention for naming rational tangles (i.e. $\one=[1]$ and not $[-1]$) follows that of Kauffman and Lambropoulou \cite{KauLam}, which yields the mirror image diagrams of those conceived of by Conway in his original paper. We defend this choice by noting the agreement with the notational convention for the cubic skein relation. That is, we can write our cubic skein relation as
    \[ b_0 [0] + b_1 [1] + b_2 [2] + b_3 [3] + b_\infty [ \infty] = 0. \]

    The major drawback is that while Conway code $[3]$ for us denotes the right-handed trefoil, most major references (e.g. Rolfsen) take $[3]$ to be the left-handed trefoil.
\end{remark}

\subsection{Skein Relation Reductions, Closure Operations, \& Mirror Images}

The Rational Tangle Algorithm is presented as a series of identities on Conway codes. Each identity changes either the rational subtangle indicated by $n_1$ or, in the case that $n_1 = \infty$, the subtangle indicated by $n_2$.

These identities on Conway codes encode either 1) reductions via the cubic skein relation or 2) isotopies of tangles. By ``reductions via the cubic skein relation," one should imagine the typical step in the computation of a polynomial invariant, like the Kauffman Bracket Polynomial. That is, given some diagram with a specified local subtangle, such diagram is equal to an algebraic combination of other diagrams with local changes to the identified subtangle. We discuss the second type of identity, ``isotopy of tangles," in Section \ref{isotopy identities}.

\subsubsection{Skein Relation Reductions}

We begin a discussion of skein relation reductions for a given rational tangle $[n_1, n_2, \dots, n_r]$ with $1<n_1<\infty$; i.e. when $n_1$ in the Conway code denotes two or more positive half twists.

Skein relation reductions performed on the tangle identified by $n_1$ for $1<n_1<\infty$ look locally like the cubic skein relation (see Example \ref{FirstStepReduction}) or the cubic skein relation with a shift of a $[-1]$ tangle (see Equation \ref{twotworeduction}). That is, we replace subtangles of two or three positive half-twists by an $R$-algebraic combination of other rational tangles with strictly fewer crossings.

\begin{example}\label{FirstStepReduction}

The first step in the computation of a polynomial for the knot $5_2$, Conway code $[3,2]$, is a skein relation reduction to the rational subtangle $[3]$ identified by $n_1=3$.

This first step, performed whenever $1<n_1<\infty$, yields the following equation:

\[ [3,2] = -\frac{1}{b_3}(b_2 [2,2] + b_1 [1,2] + b_0 [0,2] + b_\infty [\infty, 2]). \]

Pictorially, that is

$$\vcenter{\hbox{\includegraphics[width=0.75\linewidth]{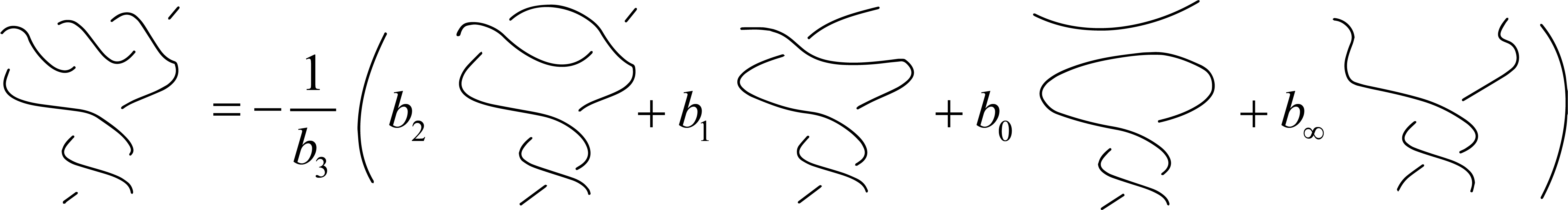}}}$$

\end{example}

\subsubsection{Reducing $n>1$ Half-twists}

In Example \ref{FirstStepReduction}, we simplified $[3,2]$ via a single skein relation reduction to
$[3,2]$ an $R$-algebraic combination of the rational tangles $[2,2]$, $[1,2]$ $[0,2]$ $[\infty, 2]$.

We make two remarks to motivate further discussion.

First, we note that the tangles $[0,2]$ and $[\infty, 2]$ are not in standard form. The RTA will return these to standard form with the various isotopies of tangles, which are employed in the case that $n_1, n_2 \in \{-\infty, -1, 0, 1, \infty \}$. See Subsection \ref{isotopy identities}.

Second, we notice that the reduction of $[3,2]$ also produced Conway code $[2,2]$, which has $1<n_1<\infty$, and therefore can be reduced via skein relation; that is, 

\begin{equation}
 [2,2] = -\frac{1}{b_3}(b_2 [1,2] + b_1 [0,2] + b_0 [-1,2] + a b_\infty [\infty, 2]). \label{twotworeduction} 
 \end{equation}

The following proposition describes an efficient way to reduce any Conway code with $1<n_1<\infty$ to four Conway codes with $n_1 \in \{-\infty, -1, 0, 1, \infty\}$. Essentially, this formula collects the coefficients which appear as the result of multiple cubic skein relation reductions, speeding the efficiency of the algorithm and simplifying its statement.

\begin{proposition}\label{Pnformula}
As in Section \ref{Section 15}, let $P_n^h = b_2' P_{n-1}^h + b_1' P_{n-2}^h + b_0'  P_{n-3}^h=\frac{b_2}{-b_3} P_{n-1}^h + \frac{b_1}{-b_3} P_{n-2}^h + \frac{b_0}{-b_3}  P_{n-3}^h$ with
\begin{itemize}
    \item $P_{-1}^h = 0$
    \item $P_{0}^h = 1$
    \item $P_{1}^h = -\frac{b_2}{b_3}.$
\end{itemize}

Let \[ U_n^h = \sum_{i=0}^{n-2} a^{3-n+i} P_i^h.\]

Then we recall Lemma \ref{U-formula} to see that 
\begin{equation}\label{ntwistreductiontobase}
     [n]  = -\frac{1}{b_3}((-b_3)P_{n-1}^h [1] + (b_1 P_{n-2}^h + b_0 P_{n-3}^h)[0] + b_0 P_{n-2}^h[-1]  + b_\infty U_n^h[\infty] ), 
\end{equation} or, for longer-length Conway codes,

\begin{align*} [n_1, n_2, \dots, n_r]  = -\frac{1}{b_3}((-b_3)P_{{n_1}-1}^h [ 1, n_2 \dots, n_r] + (b_1 P_{{n_1}-2}^h + b_0 P_{{n_1}-3}^h)[ 0, n_2, \dots, n_r] + \\
 \ \ \  \ b_0 P_{{n_1}-2}^h[-1, n_2, \dots, n_r]  + b_\infty U_{n_1}^h[\infty, n_2, \dots, n_r] ) .\end{align*}
\end{proposition}

\begin{proof}
 See Equation \ref{Dn-kiterations}.
\end{proof}

\subsubsection{Closure of Basic Tangles}

The RTA is comprised of identities on \textit{tangles} like in Proposition \ref{Pnformula}. However, its goal is to produce polynomials in $R$ from rational \textit{links}. The transition from $R$-algebraic combinations of tangles to links is as simple as taking the correct closure of each tangle as described in Definition \ref{howtoclose}. Since the  RTA always terminates with a combination of basic tangles whose closures are framed trivial links, we can substitute the polynomial for $t$ to produce a polynomial from a given rational link.

Given a rational tangle $T$ denoted by Conway code $T = [n_1, \dots, n_r]$, the Rational Tangle Algorithm allows us to compute $T$ as an $R$-algebraic combination of elements in 

\begin{equation}\label{4basictangles}
     B = \{ [-1], [0], [1], [\infty] \} = \left\{ \ \none, \zero, \one, \infinity \ \right\} 
\end{equation}
 via cubic skein relation reductions and various isotopies expressed as identities on the Conway code. Taking the (correct; see Prop \ref{howtoclose}) closure of both the original tangle $T$ and the four tangles in $B$, we compute a polynomial in $R$ corresponding to the rational link $\overline{T} = \overline{[n_1, \dots, n_r]}$ after substituting in a polynomial representing the trivial link, namely

\[ t = - \frac{b_0 + b_1 a^{-1} + b_2 a^{-2} + b_3 a^{-3}}{b_\infty}.\]

Because we only have a polynomial for the trivial link a priori, and both the numerator and denominator closures of elements of $B$ yield trivial components (possibly with some framing), reduction to the elements of $B$ is essential for the computation of polynomials for any rational link.

With Prop \ref{Pnformula} and an understanding of the closure operation, we can quickly  compute the cubic skein module polynomial for all positive, length-one Conway codes.

\begin{example}[Computation of Cinquefoil]
For the cinquefoil, the numerator closure of the rational tangle $[5]$, we apply the formula from Prop. \ref{Pnformula}:

\begin{align*}
  [5] & = -\frac{1}{b_3}((-b_3)P_{4}^h [1] + (b_1 P_{3}^h + b_0 P_{2}^h)[0] + b_0 P_{3}^h[-1]  + b_\infty U_5^h[\infty] ) \\
  & = \frac{ \left(a^2 b_2^4-3 a^2 b_1 b_3 b_2^2+2 a^2 b_0 b_3^2 b_2+a^2 b_1^2
   b_3^2\right)}{a^2 b_3^4}[1] +\frac{\left(a^2 b_1 b_2^3-a^2 b_0 b_3 b_2^2-2 a^2 b_1^2 b_3 b_2+2 a^2 b_0 b_1 b_3^2\right)}{a^2 b_3^4}[0] \\
 & + \frac{\left(a^2 b_0 b_2^3-2 a^2 b_0 b_1 b_3 b_2+a^2 b_0^2 b_3^2\right)}{a^2 b_3^4}[-1]  \\
   & + \frac{ \left(a^3 b_2^3 b_{\infty }+a^3 b_0 b_3^2 b_{\infty }-2 a^3 b_1 b_2 b_3 b_{\infty }+a^2 b_1 b_3^2
   b_{\infty }-a^2 b_2^2 b_3 b_{\infty }+a b_2 b_3^2 b_{\infty }-b_3^3 b_{\infty }\right)}{a^2 b_3^4}[\infty].
\end{align*}

Now, we take the numerator closure of the tangle $[5]$ and the resulting tangles in $B$; then we substitute the polynomial for $t$, which yields
{\small 
\begin{align*}
  5_1= N([5]) & = -\frac{1}{b_3}((-b_3)P_{4}^h [1] + (b_1 P_{3}^h + b_0 P_{2}^h)[0] + b_0 P_{3}^h[-1]  + b_\infty U_5^h[\infty] ) \\
  & = \frac{ \left(a^2 b_2^4-3 a^2 b_1 b_3 b_2^2+2 a^2 b_0 b_3^2 b_2+a^2 b_1^2
   b_3^2\right)}{a^2 b_3^4}  at +\frac{\left(a^2 b_1 b_2^3-a^2 b_0 b_3 b_2^2-2 a^2 b_1^2 b_3 b_2+2 a^2 b_0 b_1 b_3^2\right)}{a^2 b_3^4} t^2 \\
 & + \frac{\left(a^2 b_0 b_2^3-2 a^2 b_0 b_1 b_3 b_2+a^2 b_0^2 b_3^2\right)}{a^2 b_3^4} a^{-1} t  \\
   & + \frac{ \left(a^3 b_2^3 b_{\infty }+a^3 b_0 b_3^2 b_{\infty }-2 a^3 b_1 b_2 b_3 b_{\infty }+a^2 b_1 b_3^2
   b_{\infty }-a^2 b_2^2 b_3 b_{\infty }+a b_2 b_3^2 b_{\infty }-b_3^3 b_{\infty }\right)}{a^2 b_3^4}t \\
& = -\frac{1}{a^6 b_3^4 b_\infty^2} \left(a^3 b_0+a^2 b_1+a b_2+b_3\right) \left(a^4 b_2^3 b_{\infty }^2+a^4 b_0 b_3^2 b_{\infty }^2-2 a^4 b_1 b_2 b_3 b_{\infty }^2+a^4
   b_2^4 b_{\infty }+a^4 b_1^2 b_3^2 b_{\infty } \right. \\ & +2 a^4 b_0 b_2 b_3^2 b_{\infty }-3 a^4 b_1 b_2^2 b_3 b_{\infty }+a^3 b_1 b_3^2 b_{\infty }^2-a^3
   b_2^2 b_3 b_{\infty }^2-a^3 b_0 b_1 b_2^3-2 a^3 b_0^2 b_1 b_3^2+a^3 b_0^2 b_2^2 b_3 \\ & +2 a^3 b_0 b_1^2 b_2 b_3+a^2 b_2 b_3^2 b_{\infty }^2+a^2
   b_0 b_2^3 b_{\infty }+a^2 b_0^2 b_3^2 b_{\infty }-2 a^2 b_0 b_1 b_2 b_3 b_{\infty }-a^2 b_1^2 b_2^3-2 a^2 b_0 b_1^2 b_3^2 \\ & +a^2 b_0 b_1 b_2^2
   b_3+2 a^2 b_1^3 b_2 b_3-a b_3^3 b_{\infty }^2-a b_1 b_2^4-2 a b_0 b_1 b_2 b_3^2+a b_0 b_2^3 b_3+2 a b_1^2 b_2^2 b_3-2 b_0 b_1 b_3^3 \\ & +b_0 b_2^2
   b_3^2+2 b_1^2 b_2 b_3^2-b_1 b_2^3 b_3).
\end{align*}}
\end{example}

\subsubsection{Mirror Image Substitution}

\bigskip

How would we reduce negative crossings via the cubic skein relation to elements in our base, $B$? That is, given an arbitrary Conway code  $[n_1, n_2, \dots , n_r]$ with $n_1 < -1$, can we reduce our Conway code to 4 tangles with $n_1 \in \{1,0,-1,\infty\}$?

In particular, we could look to compute a  polynomial  for the mirror image ciquefoil, $[-5]$. Proposition \ref{phi} describes a way to compute the polynomial of a diagram given the polynomial of its mirror image. Thus, having computed $[5],$ we can apply Proposition \ref{phi} to yield the polynomial for $[-5]$.

Note that in general, the negation of every integer in a Conway code corresponds to the mirror image of a given tangle. An example given before was $4_1=[2,2]$ and its (isotopic) mirror image $\overline{4_1}=[-2,-2]$. The involution $\phi$ is first discussed in Subsection \ref{MI} and recalled here.

\begin{proposition}[Mirror Image]\label{phi}
    Let $d', m' \in R$ be polynomials computed from diagrams $D$ and $\overline{D}$ respectively.

    Then $m' = \phi'(d')$
    \begin{align*}
        \phi': R & \to R  \\
        b_0 & \mapsto b_3 \\
        b_1 & \mapsto b_2 \\
        b_2 & \mapsto b_1 \\
        b_3 & \mapsto b_0 \\
        a & \mapsto a^{-1} \text{ in terms without $b_\infty$ }\\
        a^n b_\infty & \mapsto a^{3-n} b_\infty
    \end{align*}
\end{proposition}

\begin{remark}
    It is somewhat easier to present the mirror image substitution, $\phi$, in terms of $b_v$, where $b_\infty  = a^{-3/2} b_v.$  Let $d, m = d',m'|_{b_\infty \to a^{-3/2} b_v}$. Then
    \[ m = \phi(d), \] where

    \begin{align*}
        \phi': R & \to R  \\
        b_0 & \mapsto b_3 \\
        b_1 & \mapsto b_2 \\
        b_2 & \mapsto b_1 \\
        b_3 & \mapsto b_0 \\
        a & \mapsto a^{-1} \\
        b_v & \to b_v 
    \end{align*}

This change of variables is implemented in the Mathematica program.
\end{remark}

Thus, the Rational Tangle Algorithm reduces a number of negative half-twists by reducing the same number of positive half twists in the mirror image diagram (this corresponds to negating the Conway code); we then use the mirror image substitution, $\phi$, on the computed polynomial.

For instance, $[3]$ can be computed using Prop \ref{Pnformula} directly, so we compute (the rational link) $[-3]$ as

\begin{align*}
    [-3] = \phi(-[-3])  =&  \phi([3]) \\
     =& \phi(-\frac{1}{a^6
   b_3^2 b_{\infty }^2}\left(a^3 b_0+a^2 b_1+a b_2+b_3\right) (a^4 b_2 b_{\infty }^2+a^4 b_2^2 b_{\infty }-a^4 b_1 b_3 b_{\infty } \\ 
   & -a^3 b_3 b_{\infty }^2 -a^3
   b_0 b_1 b_2
    +a^3 b_0^2 b_3+a^2 b_0 b_2 b_{\infty }-a^2 b_1^2 b_2+a^2 b_0 b_1 b_3 \\
    &-a b_1 b_2^2+a b_0 b_2 b_3+b_0 b_3^2-b_1 b_2 b_3))  \\
    =& \frac{1}{a^6
   b_0^2 b_{\infty }^2}\left(a^3 b_0+a^2 b_1+a b_2+b_3\right) (a^6 b_0 b_{\infty }^2-a^5 b_1 b_{\infty }^2-a^4 b_1 b_3 b_{\infty } \\
   &+a^3 b_0 b_1 b_2-a^3 b_0^2
   b_3  -a^2 b_1^2 b_{\infty }+a^2 b_0 b_2 b_{\infty }+a^2 b_1^2 b_2-a^2 b_0 b_1 b_3 \\
   &+a b_1 b_2^2-a b_0 b_2 b_3-b_0 b_3^2+b_1 b_2 b_3)).
\end{align*}

Propositions \ref{Pnformula} and \ref{phi} will appear as central components of the algorithm. We now present the remaining identities necessary for the Rational Tangle Algorithm.

\subsubsection{Isotopy Identities}\label{isotopy identities}

The following propositions describe isotopies of rational tangles in terms of their Conway code. Notice that we work only with the subtangles represented by the first and second integers of the given Conway code in the Rational Tangle Algorithm, but many of these identities can be easily extended, and are not presented in their most general form.

\begin{proposition}[$n_1= \pm 1$]\label{beginswithone}
    If $n_1 = \pm 1$, then
    \[ [\pm 1, n_2, n_3,  \dots , n_r] = [\infty, n_2\pm 1, n_3, \dots, n_r].\]
\end{proposition}

\begin{proof}
 We note that, as in Figure \ref{onen2},   $[1]= \frac{1}{[1]}.$ Thus, $[1]*\frac{1}{[n_2]} = \frac{1}{[1]} * \frac{1}{[n_2]} = \frac{1}{[n_2+1]}$.

\begin{figure}[H]
     \centering
\includegraphics[width=0.5\linewidth]{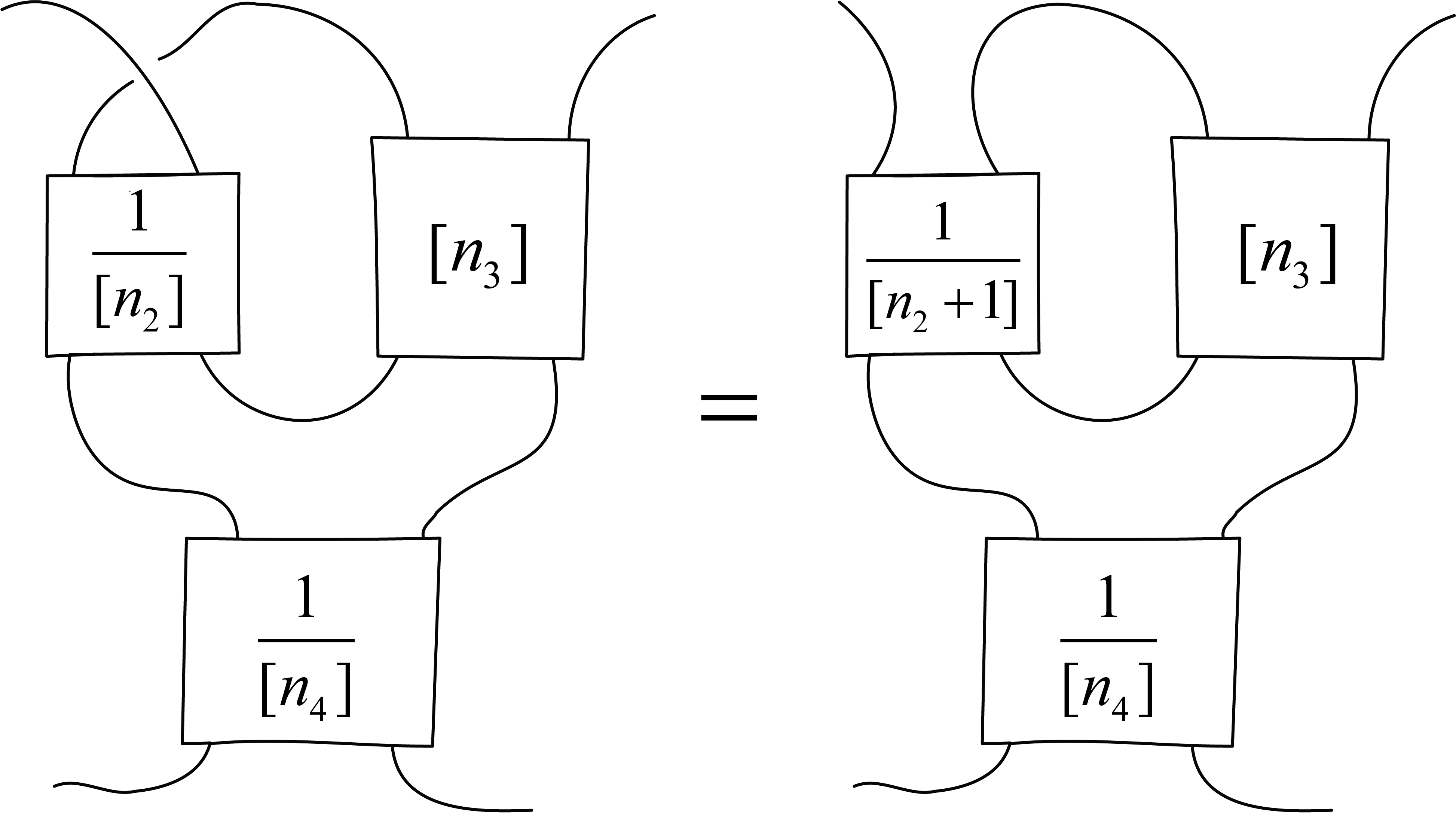}

     \caption{Subtangle $[1,n_2]$ is equivalent to $[\infty, n_2+1].$}
     \label{onen2}
 \end{figure}
  
\end{proof}

\begin{proposition}[$n_1=0$]\label{beginswithzero}
    We have the following cases:
    \begin{enumerate}
        \item $[0,0]=0$
        \item $[0,\infty]=t\cdot [0]$
        \item If $n_2 \neq 0,$ then $[0,n_2] = a^{n_2} [0]$. If $r>2$, then  $[0,n_2, \dots n_r] = a^{n_2} [ n_3, \dots n_r]$ for $r>2.$
    \end{enumerate}
\end{proposition}

\begin{proof}

\begin{itemize}

\item[(1)-(2)]  See Figure \ref{twobasicrelations}.
\item[(3)] See Figure \ref{n1=0 n2 nonzero}.
\end{itemize}

\end{proof}

\begin{figure}[H]
    \centering
\includegraphics[width=0.5\linewidth]{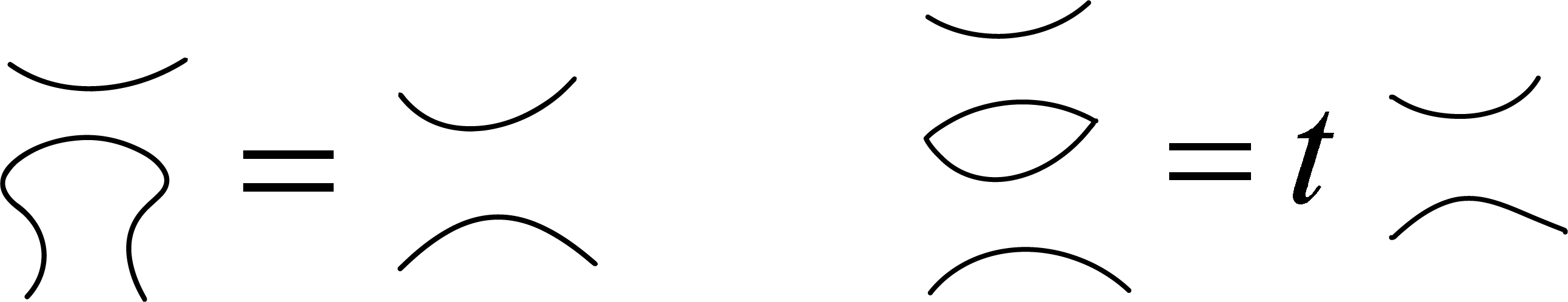}

        \caption{Isotopy taking $[0,0]\to [0]$ and $[0,\infty]\to t [0].$}
    \label{twobasicrelations}
\end{figure}

\begin{figure}[H]
        \centering
        \includegraphics[width=0.25\linewidth]{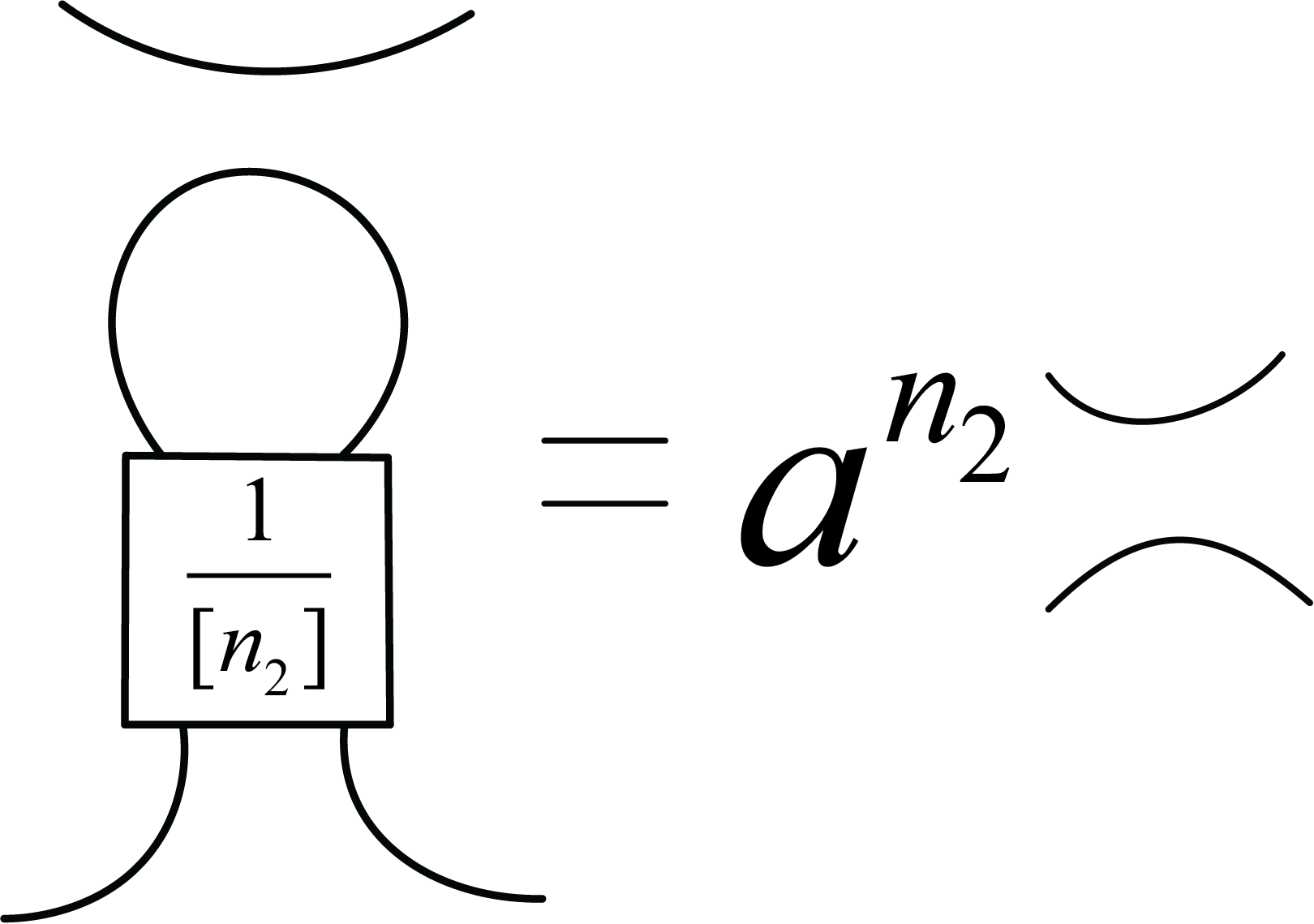}
        \caption{Tangle $[0, n_2]$ untwisting.}
       \label{n1=0 n2 nonzero}
    \end{figure}

\begin{proposition}[$n_1 = \infty$, $r=2$]\label{lengthtwobeginswithinfinity}
    We have the following cases for some specific length two Conway codes
    \begin{enumerate}
        \item $[\infty, \infty] = [0]$
        \item $[\infty, \pm 1] = [\pm 1]$
        \item $[\infty,0] = [\infty] $
    \end{enumerate}
\end{proposition}

\begin{proof} \hspace{1cm}

    See Figure \ref{infinityreductions}.
\end{proof}

\begin{figure}[H]
    \centering
    \includegraphics[width=0.5\linewidth]{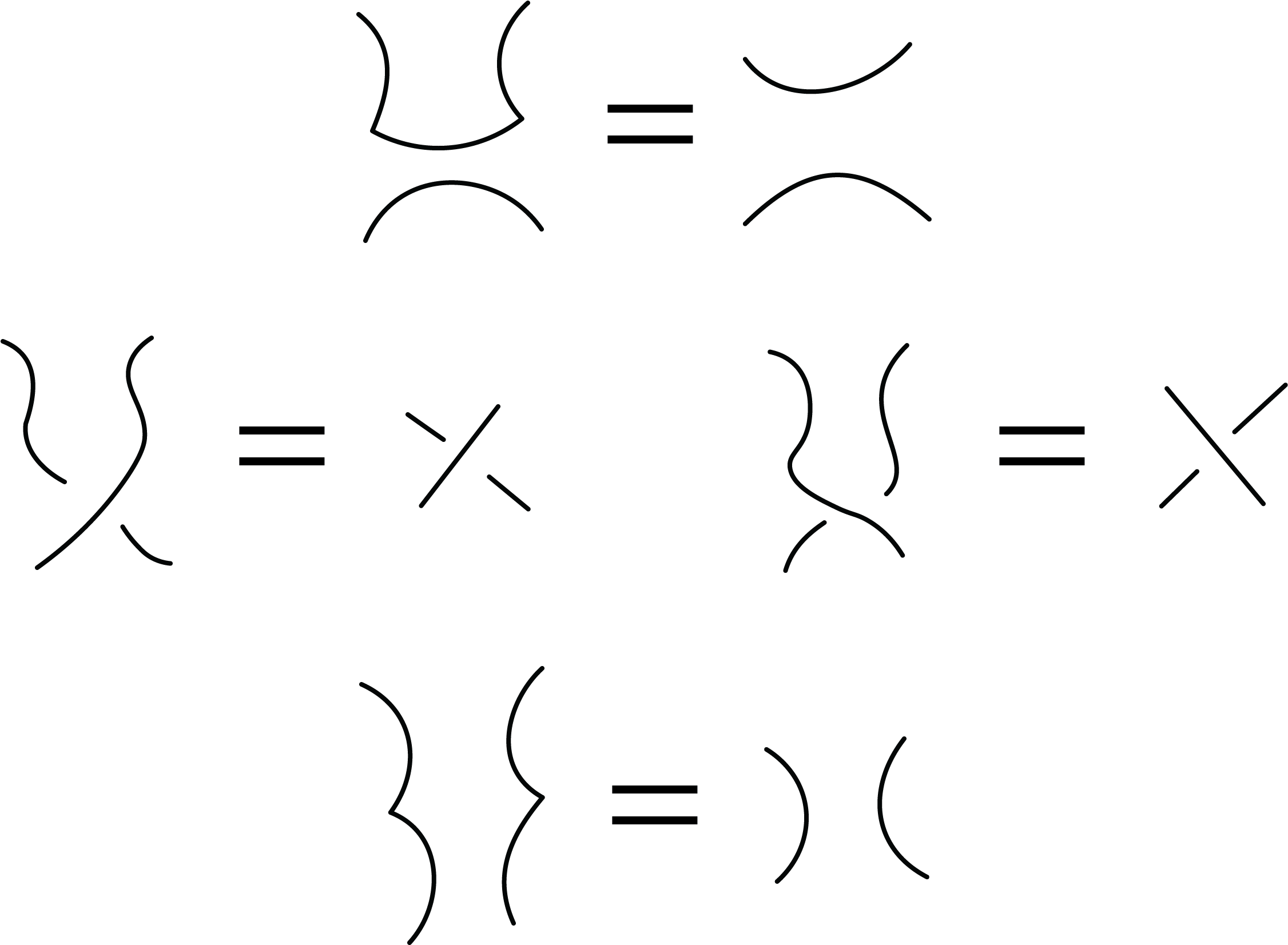}
    \caption{Reductions of tangle multiplication with infinity.}
    \label{infinityreductions}
\end{figure}

Now, we describe how to reduce rational tangles for which $n_1 = \infty$.

\begin{proposition}[$n_1 = \infty, r>2$]\label{longlengthbeginswithinfinity} \hfill

    \begin{enumerate}
        \item If $n_1=\infty$ and $-\infty<n_2 <-1,$ then
        \begin{align*}
            [\infty, n_2, \dots, n_r] & = -\frac{1}{b_3}((-b_3)P_{|n_2|-1}^h [\infty, -1, \dots, n_r] + (b_1 P_{|n_2|-2}^h + b_0 P_{|n_2|-3}^h)[\infty, 0, \dots, n_r]  \\ & \ \ \ \ \ \ + b_0 P_{|n_2|-2}^h[\infty, 1, \dots, n_r]  + b_\infty U_{|n_2|}^h[\infty, \infty, \dots, n_r] ).
        \end{align*}

        \item If $n_1 = \infty$ and $1<n_2<\infty$, then

        \begin{align*}
            [\infty, n_2, n_3, \dots , n_r] & = \phi([\infty, -n_2, -n_3, \dots, -n_r]) \\
            & = \phi(-\frac{1}{b_3}((-b_3)P_{n_2-1}^h [\infty, -1, -n_3, \dots, -n_r] + (b_1 P_{n_2-2}^h + b_0 P_{n_2-3}^h)[\infty, 0, -n_3, \dots, -n_r] +  \\ & \ \ \ \ \ \ \  b_0 P_{n
_2-2}^h[\infty, 1, \dots, n_r]  + b_\infty U_{n_2}^h[\infty, \infty, -n_3, \dots, - n_r] )).
        \end{align*}

     \smallskip

    \item If $n_1= \infty$, $n_2 = 0$, and $r>3$ then
    \[ [\infty, 0, n_3, \dots, n_r] = a^{-n_3} [ \infty , n_4, \dots , n_r]. \]

    If $r=3$, then $[\infty, 0, n_3] = a^{-n_3} [0]$.

    \smallskip

    \item If $n_1 = \infty$ $n_2 = \pm 1$, and $r>2$ then
    \[ [\infty, \pm 1, n_3, \dots, n_r] = [\infty, \infty, n_3 \pm 1, \dots , n_r ]. \]

    \item If $n_1 = \infty$ $n_2 = \infty,$ and $r>2$ then
    \[ [\infty, \infty, n_3, \dots, n_r] = [n_3, \dots, n_r]. \]

    Our definition of Conway code did not initially allow $n_2 = \infty$. We write a Conway code with $n_2 = \infty$ here for convenience, though the code is immediately reduced, as described by this proposition.
    \end{enumerate}
\end{proposition}

\begin{proof}\hfill
\begin{enumerate}
        \item Positive integers in the second coordinate correspond to the tangle $\frac{1}{[n_2]},$ so if $n_2$ is negative, this tangle is precisely $[|n_2|]$ rotated $90$ degrees, which is reduced via cubic skein relation as in Proposition \ref{Pnformula}.

        \item For $n_2 > 1$, $\frac{1}{[n_2]}$ looks like $[-n_2]$ rotated rotated 90 degrees thus we negate the Conway code and apply $\phi$ as in the example calculation of $[-3]$.
         \item  See Figure \ref{infzeron3}.
        \item[(3)-(4)] See Figure \ref{inf1 and infinf}.
\end{enumerate}
                 
\begin{figure}[H]
             \centering
        \includegraphics[width=0.5\linewidth]{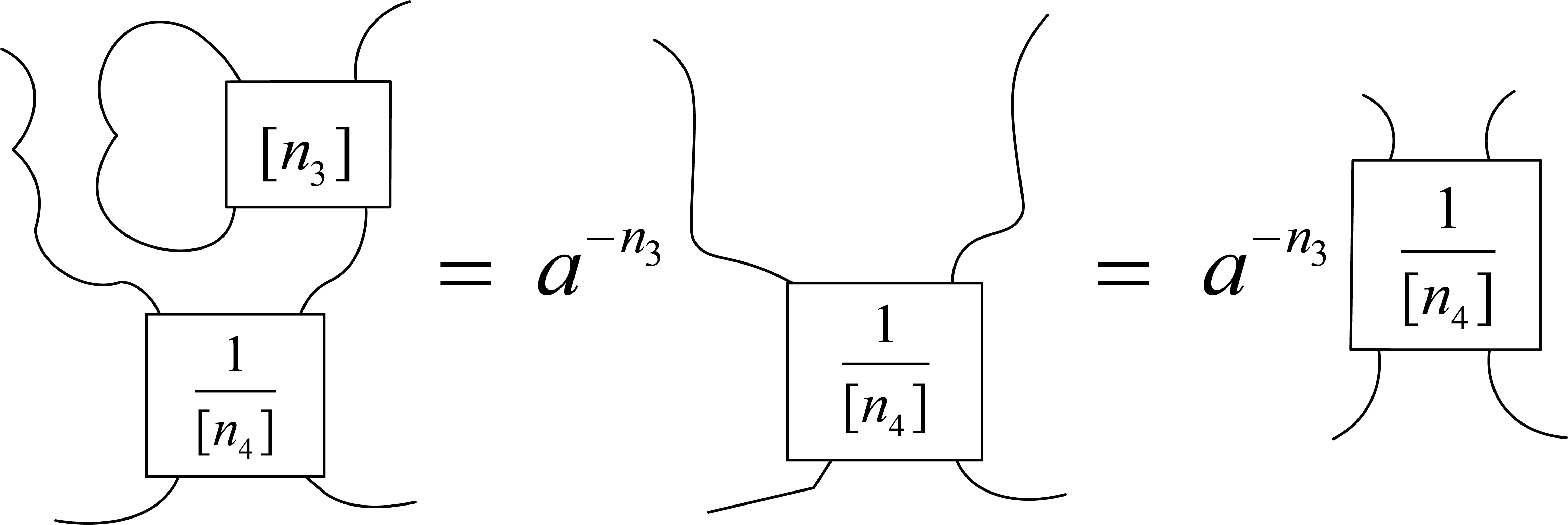}
  \caption{Conway code beginning with $[\infty, 0,n_3,\dots]$ sees the $n_3$ subtangle trivialized.}
             \label{infzeron3}
         \end{figure}

\begin{figure}[H]
    \centering
    \includegraphics[width=0.75\linewidth]{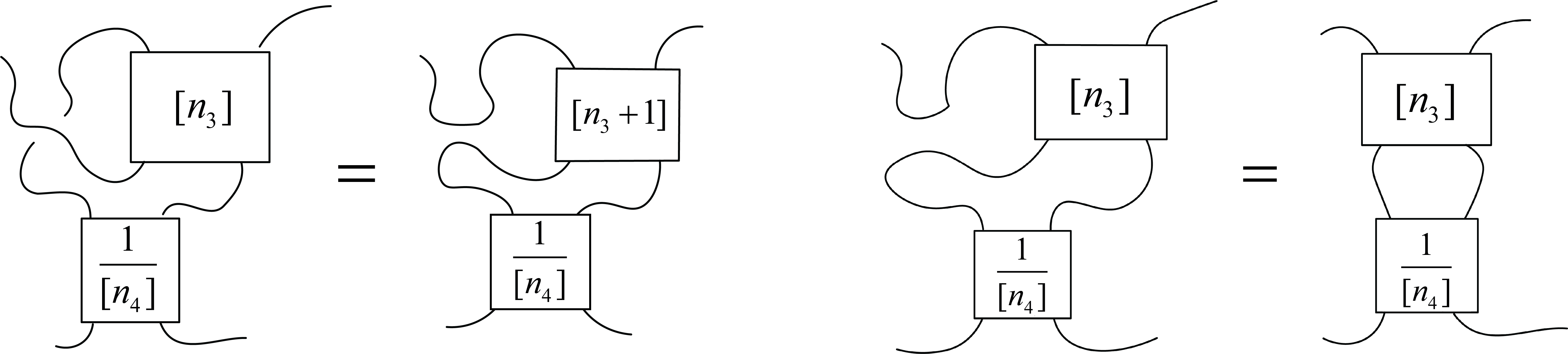}
    \caption{Isotopies for tangles beginning $[\inf, 1,\dots]$ and $[\inf, \inf,\dots].$}
    \label{inf1 and infinf}
\end{figure}

\end{proof}

\subsection{Algorithm}\label{algorithmstatement}

We now present the algorithm in full, which utilizes all of the previously presented propositions.

Begin with Conway code of all nonzero integers: $[n_1, n_2, \dots, n_r]$. All the ``if" statements have disjoint conditions. In each step, the Conway code may produce an $R$-algebraic combination of up to $4$ new Conway codes.

\vspace{1 cm}

\textbf{Rational Tangle Algorithm}

\begin{algorithmic}
\State{ Given Conway code $[n_1, n_2, \dots, n_r],$ }
\If{$[n_1, n_2, \dots, n_r] $ is one of the following,} 
  \State{   \begin{align*}
    [\infty, 0] & = [\infty]   \\
    [0,0] & = [0]  \\
    [\infty, \pm 1] & = [ \pm 1] \\
    [0,\infty] & = t [0] \\
    [0, n_2] & = a^{n_2} [0] \text{ for $n_2 \in \mathbb{Z}\setminus \{0\}. $}
\end{align*} }
\bigskip\Else \bigskip
  \If{$1<n_1<\infty $}
    \State{  \begin{align*}
[n_1, \dots, n_r] &  = -\frac{1}{b_3}((-b_3) P_{n-1} [1, n_2, \dots, n_r] +  (b_1 P_{n-2} +  b_0 P_{n-3}) [0, n_2, \dots, n_r] \\ & \ \ \ \ \ \ \ \ + b_0 P_{n-2} [-1, n_2, \dots, n_r] + a^{-3/2} b_v U_n [\infty, n_2, \dots, n_r]) 
        \end{align*} }

\bigskip \Else \bigskip

\If{$-\infty < n_1 < -1$}
\State { \([n_1, n_2, n_3, \dots, n_r] = \phi([-n_1, -n_2, -n_3, \dots, -n_r]) \)}

\bigskip \Else \bigskip

\If{ $n_1 = 0,$ $r \geq 3,$ and $n_2 \neq \pm \infty$}
\State { $[0, n_2, n_3, \dots, n_r] = a^{n_2} [n_3, n_4, \dots, n_r],$ }

\bigskip \Else \bigskip

\If{ $n_1 = \pm 1$ }

\State{ \( [\pm 1, n_2, n_3, \dots, n_r] = [\infty, n_2 \pm 1, n_3, \dots n_r]. \) }

\bigskip \Else \bigskip

\If{ $n_1 = \infty$ and $-\infty <n_2 < -1$ }

\State{     \begin{align*}
[n_1, \dots, n_r] &  = -\frac{1}{b_3}((-b_3) P_{|n_2|-1} [\infty, -1, \dots, n_r] +  (b_1 P_{|n_2|-2} +  b_0 P_{|n_2|-3}) [\infty, 0, \dots, n_r] \\ & \ \ \ \ \ \ \ \ + b_0 P_{|n_2|-2} [\infty, 1, \dots, n_r] + a^{-3/2} b_v U_{|n_2|} [\infty, \infty, \dots, n_r] \\
        \end{align*} }

\If{$n_1 = \infty$ and $1 <n_2 < \infty$ }
\State{  \( [\infty, n_2, n_3, \dots, n_r] = \phi([\infty, -n_2, -n_3, \dots, -n_r])\) }

\bigskip \Else \bigskip

\If{ $n_1 = \infty,$ $n_2 =0,$ and $r=3,$ }
\State{  \( [\infty, 0, n_3, \dots, n_r] = a^{-n_3} D_\infty \)}

\bigskip \Else \bigskip

 \If{ $n_1 = \infty,$ $n_2 =0,$ and $r>3,$ }
\State{   \( [\infty, 0, n_3, \dots, n_r] = a^{-n_3} [\infty, n_4, \dots, n_r] \) }

\bigskip \Else \bigskip

\If{ $n_1 = \infty,$ $n_2 =\infty,$ and $r>3,$ }
\State{  \( [\infty, \infty, n_3, \dots, n_r] =  [n_3, n_4, \dots, n_r] \) }
\EndIf
\EndIf
\EndIf
\EndIf
\EndIf
\EndIf
\EndIf
\EndIf
\EndIf
\EndIf
\end{algorithmic}

\subsection{Relations in the Rational Tangle Algorithm}\label{relsinrta}

We can form a relation in $\mathcal{S}_{(4,\infty)}(S^3)$ by taking the difference of the computed polynomials for two diagrams of rational links in the same ambient isotopy class. 

Early inquiries into the structure of $\mathcal{S}_{(4,\infty)}(S^3)$ hinted at the possibility of   
\[ \mathcal{S}_{(4,\infty)}(S^3) = \frac{R \mathcal{L}^{\mathit{fr}}}{\langle [-1,-1]-[1,1] \rangle }. \]

In the following section, we provide a counterexample to this; that is, we show that not every relation in $\mathcal{S}_{(4,\infty)}(S^3)$ is contained in the ideal generated by the Hopf Relation. However, the Hopf Relation plays a central role in the structure of the cubic skein module with rational link diagrams computed in the Rational Tangle Algorithm. In fact we have the following closed formulas.

\begin{proposition}

    For $m,n \geq 1$, 

    \begin{enumerate}
        \item    \[ \frac{[m,n]-[-n,-m]}{[-1,-1]-[1,1]} = P_{m-1} \phi(P_{n-1}) - P_{m-2} \phi( P_{n-2}) \] 
        \item \[ \frac{[m,-n]-[n,-m]}{[-1,-1]-[1,1]} = P_{m-1} P_{n-1} - P_{m-2} P_{n-2} \]
        \item \[ \frac{[-m,n]-[-n,m]}{[-1,-1]-[1,1]} = \phi(P_{m-1}) P_{n-1} - \phi(P_{m-2})  P_{n-2} \]
        \item \[ \frac{[-m,-n]-[n,m]}{[-1,-1]-[1,1]} = \phi(P_{m-1}) \phi(P_{n-1}) - \phi(P_{m-2})  \phi(P_{n-2}) \]

    \end{enumerate}

        for polynomials as computed in the Rational Tangle Algorithm. Note that $(3)$ is the mirror image of $(1)$ and $(4)$ is the mirror image of $(2)$.

\end{proposition}

Furthermore, we conjecture the following.

\begin{conjecture}\label{Conj:algorithm}
    Let $[n_1, n_2, \dots, n_r]$ and $\epsilon [n_1, n_2, \dots, n_r]$ be two diagrams of rational links where $\epsilon = 1$ if $r$ is odd, and $\epsilon = -1$. As computed in the Rational Tangles Algorithm, the relation
    \[ [n_1, n_2, \dots, n_r] - \epsilon [n_1, n_2, \dots, n_r] \] is divisible by the Hopf Relation. Therefore, the Rational Tangle Algorithm does not produce any new relations above the Hopf Relation.
\end{conjecture}

\bigskip

If we go outside of the Rational Tangle Algorithm, we can produce more polynomial relations; see Section \ref{MRel}.

\subsection{Pretzel Links Algorithm}\label{sec:pretzels}

\subsubsection{Computing polynomial for pretzel links}
We consider pretzel $2$-tangles of $r$ columns denoted by $P(n_1,n_2,...,n_r)$. The numerator closure of a pretzel tangle, $N(P(n_1, \dots, n_r))$, is a pretzel link.  The diagram of pretzel link $N(P(2,1,-3,1))$ is shown on the right-hand-side of Figure \ref{63knot-pretzel}.\footnote{Following the notational convention for tangle operations from section \ref{tangleoperations}, a pretzel tangle $P(n_1, n_2, \dots , n_r)$ can be written 
   
    \[  \left( \frac{1}{[-n_1]} + \frac{1}{[-n_2]} + \dots + \frac{1}{[-n_r]} \right) \]
    for $n_i \in \mathbb{Z}$.}

Recall that the twist knot of a 2-bridge link $[2,n]$ is also the pretzel knot $N(P(2,\underbrace{1,1,\hdots,1}_n))$, see Figure \ref{TwistKnot1}.

In an effort to produce more relations in $\mathcal{S}_{4,\infty}(\mathbb{R}^3)$, we can compare diagrams which we compute using different algorithms. For instance, Figure \ref{63knot-pretzel} shows $6_3$ as a two-bridge (rational) link with $6$ crossings and a Pretzel Link with $7$. We leave the comparison of their computed polynomials as an exercise to the reader. In particular, is the relation in the ideal generated by the Hopf Relation and the Trefoil Relation \ref{RTrefoilNeg}?

\begin{figure}[ht]
\centering
\includegraphics[scale=0.27]{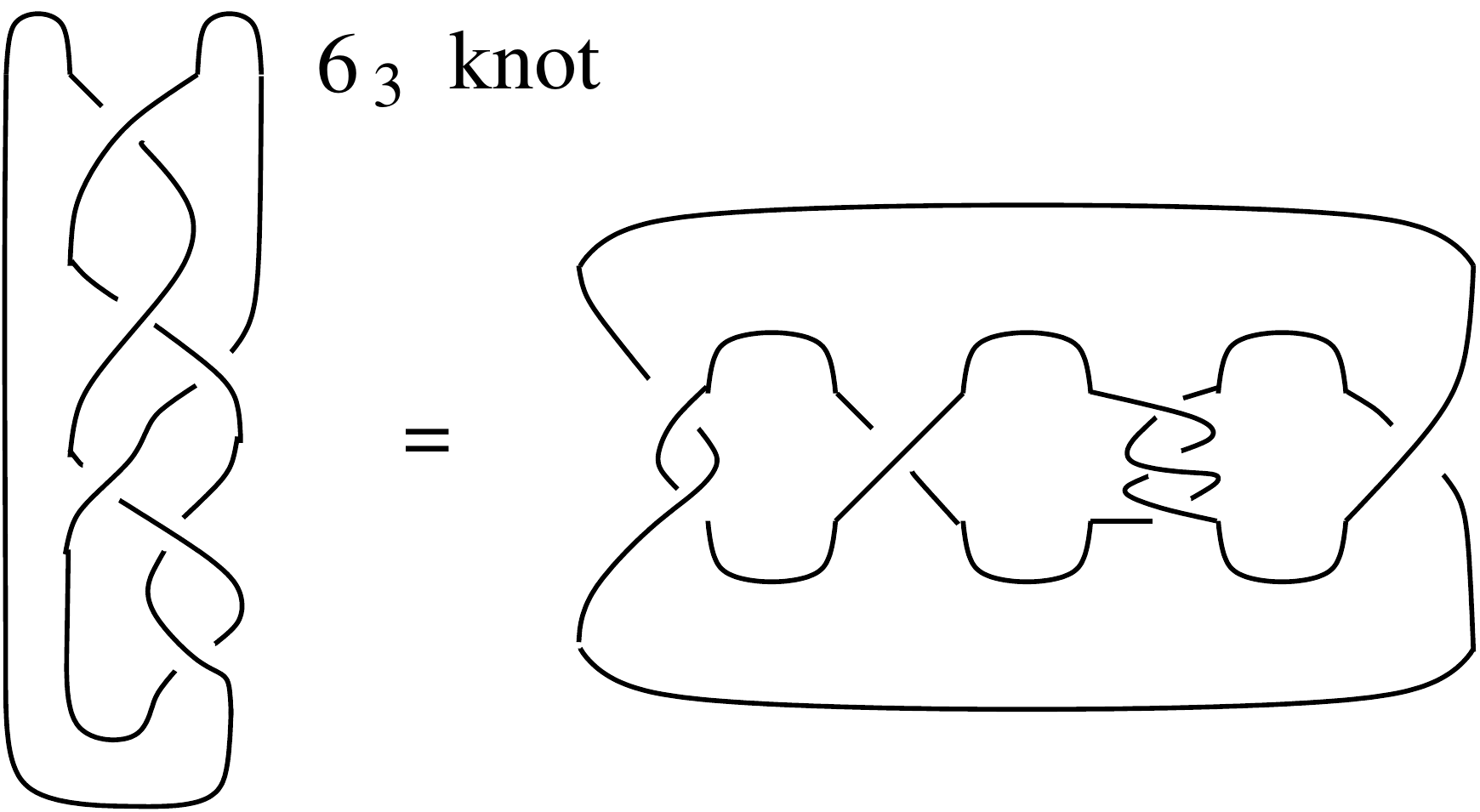} 
\caption{The amphichiral knot $6_3$ is the two-bridge knot $[2,1,1,2]$ and also the pretzel knot $P(2,1,-3,1)$. These diagrams are ambient isotopic but differ by framing.}
\label{63knot-pretzel}
\end{figure}

Like the Rational Tangle Algorithm, the Pretzel Link Algorithm is a means to produce a unique polynomial in $R = \mathbb{Z}[a^{-\pm}, b_0^{\pm}, b_1, b_2, b_3^{\pm}, b_\infty^{\pm} ]$ given a pretzel link diagram. The algorithm proceeds as follows:
\begin{enumerate}
    \item Beginning with $n_1$, if $|n_i| \geq 2$, use the cubic skein relation to reduce the given pretzel link diagram to an algebraic combination of diagrams with strictly fewer crossings. In particular, we use formula \ref{ntwistreductiontobase} to reduce to a $P(n_1, \dots, n_i, \dots n_r)$ to an algebraic combination of $P(n_1, \dots, 1, \dots n_r)$, $P(n_1, \dots, 0, \dots n_r)$, $P(n_1, \dots, -1, \dots n_r)$, and $P(n_1, \dots, \infty, \dots n_r)$ .
    \item If $n_i \in \{1, 0, -1, \infty\}$, continue to $n_{i+1}$.
    \item When each $n_i \in \{1, 0, -1, \infty \}$, various isotopy identities are applied, and some intermediate diagrams are further reduced via the cubic skein relation. This culminates in a linear combination of framed trivial links, for which we substitute Equation \ref{tgenformula}.
\end{enumerate}

The following is a precise statement of the algorithm.

\subsubsection{Pretzel Link Algorithm}

 Let $B= \{-1,0,1, \infty\}$. For a pretzel knot diagram given by $(n_1, n_2, \dots , n_r) := P(n_1, n_2, \dots, n_r)$, 

 \bigskip
 
\begin{algorithmic}
 \If{$r=1$}
 \State{\begin{align*}
        (0) & = t \\
        (1) & = a^{-1} t \\
        (-1) & = a t \\
        (\infty ) = (\pm 1, \mp 1) & = t^2 \\
    \end{align*} }

\bigskip \Else{ beginning with $i=1$,} \bigskip

\If{$n_i \in B$}
\State{Skip to $(i+1)$-st entry, $n_{i+1}$.}

\bigskip \Else \bigskip

\If{$-\infty <n_i < -1$}

\State{ $(n_1, \dots, n_i, \dots, n_r) = \phi((-n_1, \dots, -n_i, \dots, -n_r))$. Proceed at the $i$-th entry.}

\bigskip \Else \bigskip

\If{ $n_i \geq 2$ }

\State{\begin{align*}
           (n_1, \dots, n_i, \dots, n_r) & =  -\frac{1}{b_3}((-b_3) P_{n-1} (n_1, \dots, 1, \dots, n_r) +  (b_1 P_{n-2} +  b_0 P_{n-3}) (n_1, \dots, 0, \dots, n_r) \\ & \ \ \ \ \ \ \ \ + b_0 P_{n-2} (n_1, \dots, -1, \dots n_r) + b_\infty U_n (n_1, \dots, \infty, \dots, n_r)) \\ 
           & \ \ \ \ \ \ \ \ \ \ \  \text{ and proceed to the $(i+1)$-st entry in each of the resultant tuples.}
       \end{align*}         }

   \bigskip \Else \bigskip

\If{ $n_i \in B$ for all $i \leq r$, }

\If{$n_k= 0$ for some $k$}    
\State{\[ (n_1, \dots, n_r) = t^{\text{\# zeros}} \prod_{k=1}^r a^{n_k}. \]}

   \bigskip \Else \bigskip

\If{$n_k \neq 0$ for all $k$ }
    
   \State{  Let $m= - \sum_{k=1}^{r} n_k$ for $n_k \neq \infty$.
    Compute $[m]$ as in formula \ref{ntwistreductiontobase}.}
\EndIf
\EndIf
\EndIf
\EndIf
\EndIf
\EndIf
\EndIf
\end{algorithmic}
   
The results of this procedure will be a polynomial in $R[t]$, in which we substitute Equation \ref{tgenformula} for $t$ to produce a polynomial in $R$.

\section{More Relations towards a new polynomial invariant of knots}\label{MRel}
This section will build off of Section \ref{sec:relations} and Section \ref{quadtocube}. We calculate new relations outside of the Rational Tangle Algorithm (see Section \ref{relsinrta}), though we use much of the same machinery. In particular, we require the invertibility of various coefficients in $R$. We observe that in certain cases a new cubic skein relation arises which yields a quadratic skein relation. In order to rectify this, we observe a restriction on the coefficients with the requirement that the new quadratic skein relation is trivial. We observe that for special substitutions to the quadratic skein relation, we recover the cubic skein relations for the Kauffman 2-variable and Dubrovnik polynomials. 

\subsection{Trivial Knot Relations}\label{Subsec:trivialknot}
Recall that if $b_{\infty}$ is invertible, then we obtain a formula for $t$ given in Equation \ref{eqn:t1}. As discussed in Section \ref{mirror}, sometimes it is advantageous to make the  substitution $b_v = a^{3/2} b_{\infty}$ for the purposes of symmetry.\footnote{If $\bar D$ is a mirror image of $D$ and $D$ is a linear combination of trivial links then $\bar D$ can be obtained from $D$ by the following substitutions $b_0 \leftrightarrow b_3, b_1 \leftrightarrow b_2, a \leftrightarrow a^{-1}$ where $b_v$ is preserved.} The following lemma is a direct result from Proposition \ref{mirror}.

\begin{lemma}
Under the substitution $b_v = a^{3/2} b_{\infty}$, Equation \ref{eqn:t1} is invariant under the substitutions $b_0 \leftrightarrow b_3, b_1 \leftrightarrow b_2, a \leftrightarrow a^{-1}, b_v \leftrightarrow b_v$; this is the mirror image substitution $\phi$ from  Proposition \ref{phi}.
\end{lemma}

Notice that we obtain no new formulas for $t$, from the one given in Equation \ref{eqn:t1}, when computing $t$ after shifting the cubic relation. Suppose we shift Equation \ref{eqn:CubicSkeinRelation} by adding $n$ positive crossings to $D_0$, then

\begin{eqnarray}
     0 &=&b_0 D_0 +b_1D_1 +b_2 D_2+b_3 D_3+ b_{\infty}D_{\infty} \nonumber\\
     &\implies& 0 = b_0 D_n +b_1 D_{n+1}+b_2 D_{n+2} +b_3D_{n+3}+ b_{\infty}a^{-n}D_{\infty} \label{eqn:cubicskeinshiftn}\\
     &\implies& 0 = b_0 a^{-n}t + a^{-n-1} b_1 t+b_2 a^{-n-2}t +b_3a^{-n-3}t+ b_{\infty}a^{-n}t^2 \nonumber\\
     &\implies & t = -\frac{b_0 + b_1a^{-1}+b_2 a^{-2}+b_3a^{-3}}{b_{\infty}}. \nonumber
\end{eqnarray}

However, a new formula for $t$ can be obtained by shifting the cubic relation and adding a negative kink. Notice that this process involves one Reidemeister 3 move and one Reidemeister 1 move; let $t_2$ denote the resulting diagram of the unknot. 

\begin{lemma} If $(b_0 b_1 - b_2 b_3)$ and $b_3$ are invertible, then
    \begin{eqnarray}
        t_2 &=& \frac{-b_0^2 - a^2 b_0 b_2 + b_1 b_3 + a^2 b_3^2 - a^2 b_0 b_{\infty} + a^3 b_3 b_{\infty}}{a (b_0 b_1 - b_2 b_3)}. \label{eqn:t2}
    \end{eqnarray}
\end{lemma}
\begin{proof}
After adding a kink to $D_{-2}$ and using a shifted version of the relation, we obtain a relation using $D_0, D_{\infty}, D_{-1}, D_{-2}$, and a clasp.

\begin{eqnarray}
     \Dntwo &=& a \Dntwotwist \nonumber \\
   0 &=& b_0 \Dnnonetwist +b_1 \Dnzerotwist+b_2 \Dnonetwist +b_3\Dntwotwist+ b_{\infty}\Dninftwist \nonumber \\
   &\implies& 0=b_0 \Dnnonetwist +b_1 a^{-1} \Dzero+b_2 \Dinfty +b_3 a^{-1} \Dntwo+ b_{\infty}a\Dminus \nonumber \\
   &\implies&  0 = b_0\vcenter{\hbox{\includegraphics[scale=.12]{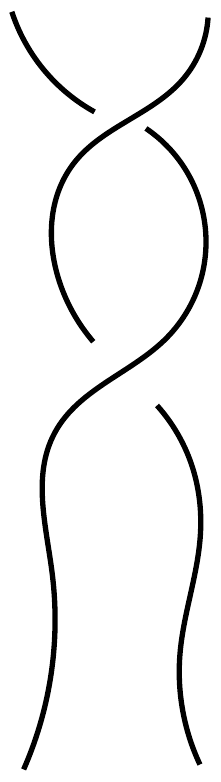}}} +b_1 a^{-1} \Dzero+b_2 \Dinfty +b_3 a^{-1} \Dntwo+ b_{\infty}a\Dminus.\label{eqn:twistkink}
\end{eqnarray}

After placing the clasp in the $D_3$ position of Equation \ref{eqn:CubicSkeinRelation} we have 

\begin{eqnarray}
    0 &=& b_3 \vcenter{\hbox{\includegraphics[scale=.12]{D2vert.pdf}}}+b_2 \Dminus +b_1 \Dinfty + b_0 \Done + b_{\infty} a \Dzero \nonumber \\
    \vcenter{\hbox{\includegraphics[scale=.12]{D2vert.pdf}}} &=& -b_3^{-1} \left(b_2 \Dminus +b_1 \Dinfty + b_0 \Done + b_{\infty} a \Dzero \right). \label{eqn:clasp1}
    \end{eqnarray}
Then by incorporating Equation \ref{eqn:clasp1} into Equation \ref{eqn:twistkink} we have
    \begin{eqnarray}
     0 &=& -b_3^{-1}b_0 \left(b_2 \Dminus +b_1 \Dinfty + b_0 \Done + b_{\infty} a \Dzero \right) \nonumber \\
    &&+b_1 a^{-1} \Dzero+b_2 \Dinfty +b_3 a^{-1} \Dntwo+ b_{\infty}a\Dminus \nonumber \\
    \implies 0 &=& (a^{-1} b_1 - ab_3^{-1} b_{\infty}b_0) \Dzero +(a b_{\infty} - b_3^{-1} b_0 b_2) \Dminus \nonumber \\
    && +(b_2 - b_3^{-1}b_0b_1) \Dinfty -b_3^{-1} b_0^2 \Done + b_3 a^{-1} \Dntwo. \label{eqn:clasp}
\end{eqnarray}
By letting $D_0 = t$ in Equation \ref{eqn:clasp}, we obtain the following formula for $t$;
\begin{eqnarray}
 0 &=& (a^{-1} b_1 - ab_3^{-1} b_{\infty}b_0) t +(a b_{\infty} - b_3^{-1} b_0 b_2) a t  +(b_2 - b_3^{-1}b_0b_1) t^2 -b_3^{-1} b_0^2 a^{-1}t + b_3 a^{-1} a^{-2}t \nonumber \\
t &=& \frac{-b_0^2 - a^2 b_0 b_2 + b_1 b_3 + a^2 b_3^2 - a^2 b_0 b_{\infty} + a^3 b_3 b_{\infty}}{a (b_0 b_1 - b_2 b_3)}.
\end{eqnarray}
\end{proof}

Notice that setting $(b_0 b_1 -b_2b_3)=0$ is similar to setting $b_\infty=0$. Recall that $b_\infty=0$ is related to the Kauffman $2$-variable polynomial; see Section \ref{quadtocube}. In fact, the authors of \cite{PTs} showed that if $b_{\infty}=0$ and $t^i$'s, are linearly independent, then  $(b_0 b_1 -b_2b_3)=0$. By ``linearly independent" we mean no nontrivial linear combination of $t^i$ gives zero.

\begin{question}\cite{PTs}\label{Qu:linearlyindep}
     If $b_{\infty}=0$ and  $(b_0 b_1 -b_2b_3)=0$, does it imply that $t^i$'s are linearly independent?
\end{question}

\begin{definition}[Trivial knot relation]
    Define the trivial knot relation, denoted by $R_t$, as the difference between Equation \ref{eqn:t1} and \ref{eqn:t2}.

    \begin{eqnarray}
&& R_t := (t_2-t_1)\label{eqn:trelation} \\
    0&=& \frac{(b_0 b_1 -b_2b_3)(a^3b_0+a^2b_1 +ab_2+b_3)-a^2b_{\infty}(b_0^2+a^2b_0b_2-b_1b_3-a^2b_3^2+a^2b_0b_{\infty}-a^3b_3b_{\infty})}{a^3 b_{\infty}(b_0 b_1 -b_2b_3)}\nonumber \\
   & = & \frac{1}{a^3(b_0 b_1 - b_2 b_3)} (a^3 b_0^2 b_1 +a^2 b_0 b_1^2 + a b_0 b_1 b_2 + b_0 b_1 b_3 - a^3 b_0 b_2 b_3 - a^2 b_1 b_2 b_3 - a b_2^2 b_3 - b_2 b_3^2 - a^2 b_0^2 b_\infty  \nonumber \\ & &- a^4 b_0 b_2 b_\infty + a^2 b_1 b_3 b_\infty + a^4 b_3^2 b_\infty - a^4 b_0 b_\infty^2 \nonumber + a^5 b_3 b_\infty^2).
\end{eqnarray}
\end{definition}

The Hopf relation, given in Equation \ref{eqn:Hopfrelation} is related to $R_t$. In particular, if we solve for $t$ we obtain Equation \ref{eqn:t2}. Furthermore, we have the following observation.

\begin{lemma} If $b_0$, $b_3$, and $(b_0 b_1 -b_2b_3)$ are invertible, then
    $$ b_0b_3R_{\mathit{Hopf}} = t(b_0 b_1 -b_2b_3)R_t,$$
    where $t$ was substituted into $R_{\mathit{Hopf}}$ by using Equation \ref{eqn:t1}.
\end{lemma}

Notice that, if $b_3$, $b_0$, and $(b_0 b_1 -b_2b_3)$ are invertible, then $R_{t}=0 \iff R_{\mathit{Hopf}}=0$.\\

We find a suitable substitution for $b_0$ and $b_3$ in terms of $b_1, b_2, b_3, b_{\infty}$, and $a$ by setting the numerator equal to zero and solving for $b_3$ from Equation \ref{eqn:solveb3}. 

\begin{eqnarray}
    0 &=& b_0 (a^3 b_0 b_1 + a^2 b_1^2 + a b_1 b_2 + b_1 b_3 - a^2 b_0 b_{\infty} - 
   a^4 b_2 b_{\infty} - a^4 b_{\infty}^2) \label{eqn:solveb3} \\
\implies b_3 &=& \frac{(-a^3 b_0 b_1 - a^2 b_1^2 - a b_1 b_2 + a^2 b_0 b_{\infty} + a^4 b_2 b_{\infty} + 
 a^4 b_{\infty}^2)}{b_1} \nonumber 
\end{eqnarray}

This substitution yields

\begin{eqnarray*}
    R_t &=& (-a^2 b_0 b_1 - a b_1^2 - b_1 b_2 + a b_0 b_{\infty} + a^3 b_2 b_{\infty} + 
    a^3 b_{\infty}^2)  \times \\
    && \frac{(-a^5 b_0 b_1+ b_1^2 - a^4 b_1^2- a^3 b_2b_1 - b_0 b_2 - 
    a^2 b_2^2 + a^4 b_0 b_{\infty} + a^3 b_1 b_{\infty} - a^2 b_2 b_{\infty} + a^6 b_2 b_{\infty} +
     a^6 b_{\infty}^2)}{b_1 (-b_0 b_1^2 - a^3 b_0 b_1 b_2 - a^2 b_1^2 b_2 - a b_1 b_2^2 + 
    a^2 b_0 b_2 b_{\infty} + a^4 b_2^2 b_{\infty} + a^4 b_2 b_{\infty}^2)}.
\end{eqnarray*}

Next consider Equation \ref{eqn:solveb0} and solve for $b_0$.

\begin{eqnarray}
    0 &=& -a^5 b_0 b_1 + b_1^2 - a^4 b_1^2 - b_0 b_2 - a^3 b_1 b_2 - a^2 b_2^2 + 
  a^4 b_0 b_\infty + a^3 b_1 b_\infty \label{eqn:solveb0} \\
  &&- a^2 b_2 b_\infty + a^6 b_2 b_\infty + a^6 b_\infty^2 \nonumber \\
  \implies b_0 &=& \frac{b_1^2 - a^4 b_1^2 - a^3 b_1 b_2 - a^2 b_2^2 + a^3 b_1 b_\infty - a^2 b_2 b_\infty + 
 a^6 b_2 b_\infty + a^6 b_\infty^2}{a^5 b_1 + b_2 - a^4 b_\infty}.
\end{eqnarray}

\begin{lemma}
    Assume $a^5 b_1 + b_2 - 
      a^4 b_\infty$, $(b_0 b_1 -b_2b_3)$, $b_1$, and $b_{\infty}$ are invertible, then $R_t=0$ under the following substitutions
    \begin{eqnarray}
     b_0 &=& \frac{b_1^2 - a^4 b_1^2 - a^3 b_1 b_2 - a^2 b_2^2 + a^3 b_1 b_\infty - a^2 b_2 b_\infty + 
 a^6 b_2 b_\infty + a^6 b_\infty^2}{a^5 b_1 + b_2 - a^4 b_\infty}\label{eqn:b0substitution} \\
        b_3 &=& \frac{-a (a^2 b_1^2 + a b_1 b_2 + b_2^2 - a^4 b_2^2 - a b_1 b_\infty + 
        a^5 b_1 b_\infty - a^4 b_2 b_\infty - a^4 b_\infty^2)}{a^5 b_1 + b_2 - 
      a^4 b_\infty} .\label{eqn:b3substitution}
    \end{eqnarray}
\end{lemma}

\begin{theorem} Suppose $b_0$ is invertible, then the following is a cubic skein relation in the cubic skein module of $S^3$.
 \begin{equation}\label{eqn:cubic2}
        0 =  b_3 D_3 - \frac{b_3(b_1b_3-a^2b_0 b_{\infty})}{a b_0^2}D_2 -  \frac{b_3(a b_3 b_{\infty} -b_0 b_2 )}{b_0^2} D_1 - \frac{b_3^3}{a b_0^2} D_0 - \frac{b_3(b_2 b_3 - b_0 b_1)}{a^2b_0^2} D_{\infty}.
    \end{equation}
\end{theorem}

\begin{proof}
Consider the shifted cubic skein relation given in Equation \ref{eqn:shiftedDn1} and set $D_3$ in the $D_{-1}$ position. 
\begin{eqnarray}
    0 &=& b_0 \Dminus + b_1 \Dzero + b_2 \Done + b_3 \Dtwo + b_{\infty} a \Dinfty \label{eqn:shiftedDn1}\\
    \implies 0&=& b_0 \vcenter{\hbox{\includegraphics[scale=.12]{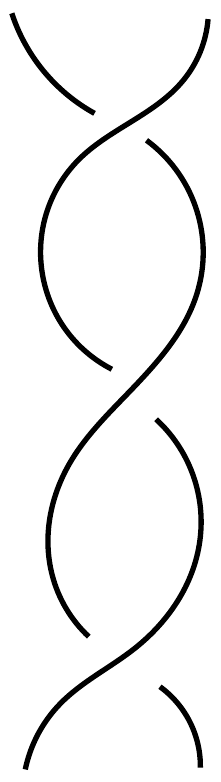}}}+b_1 \vcenter{\hbox{\includegraphics[scale=.12]{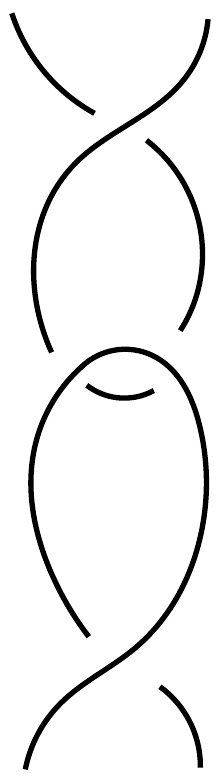}}}+b_2 \vcenter{\hbox{\includegraphics[scale=.12]{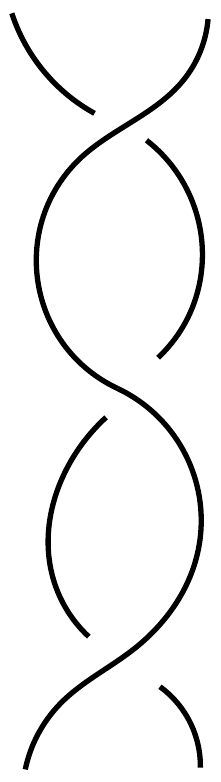}}}+b_3\vcenter{\hbox{\includegraphics[scale=.12]{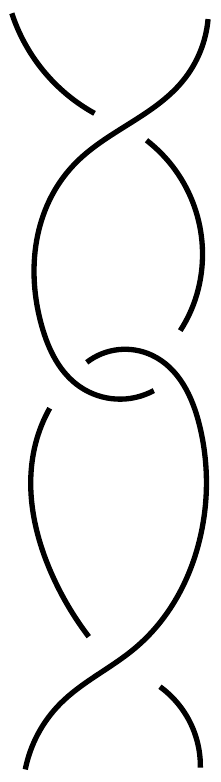}}}+b_{\infty}\vcenter{\hbox{\includegraphics[scale=.12]{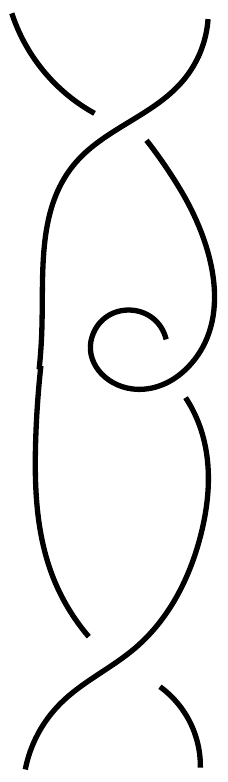}}}. \label{eqn:d3onecrossing}
\end{eqnarray}

Apply Reidemeister 1, 2,  and 3 moves to the diagrams in Equation \ref{eqn:d3onecrossing}.
\begin{eqnarray}
    0&=& b_0 \vcenter{\hbox{\includegraphics[scale=.12]{D3vert.pdf}}}+b_1 \vcenter{\hbox{\includegraphics[scale=.12]{D3vert4.pdf}}}+b_2 \vcenter{\hbox{\includegraphics[scale=.12]{D3vert2.pdf}}}+b_3 a^{-1}\vcenter{\hbox{\includegraphics[scale=.12]{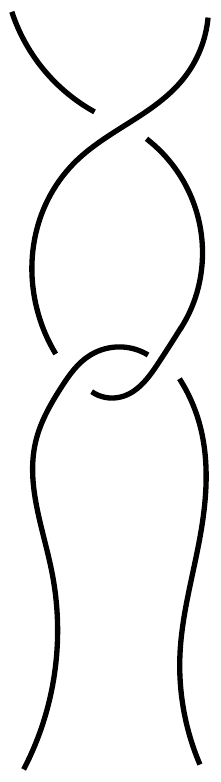}}}+ab_{\infty}\vcenter{\hbox{\includegraphics[scale=.12]{D2vert.pdf}}} \nonumber \\
     0&=& b_0 \vcenter{\hbox{\includegraphics[scale=.12]{D3vert.pdf}}}+a^{-2} b_1 \Dzero+b_2 \Dminus+b_3 a^{-1}\vcenter{\hbox{\includegraphics[scale=.12]{D3vert6.pdf}}}+ab_{\infty}\vcenter{\hbox{\includegraphics[scale=.12]{D2vert.pdf}}}. \label{eqn:twistclasp}
\end{eqnarray}

Now consider a shifted version of the cubic relation given in Equation \ref{eqn:shiftedb2} and set the diagram with coefficient $b_3 a^{-1}$ in Equation \ref{eqn:twistclasp} to the $D_{-2}$ position in Equation \ref{eqn:shiftedb2}.
\begin{eqnarray}
    0 &=& b_0 \Dntwo+ b_1 \Dminus +b_2 \Dzero +b_3 \Done+b_{\infty}a^2 \Dinfty \label{eqn:shiftedb2}\\
    \implies 0 &=& b_0 \vcenter{\hbox{\includegraphics[scale=.12]{D3vert6.pdf}}} + b_1 \vcenter{\hbox{\includegraphics[scale=.12]{D2vert.pdf}}} + b_2 \vcenter{\hbox{\includegraphics[scale=.12]{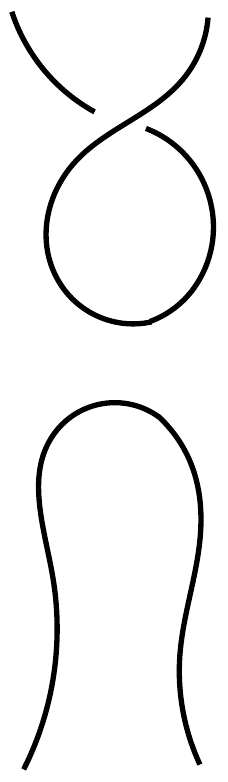}}} + b_3 \vcenter{\hbox{\includegraphics[scale=.12]{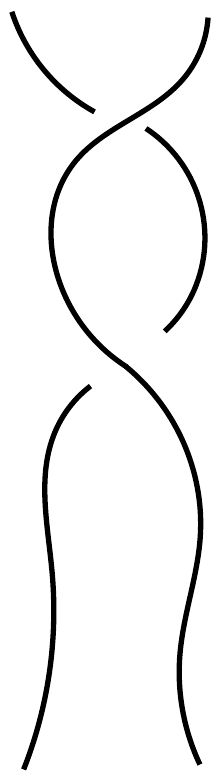}}} +a^2 b_{\infty} \vcenter{\hbox{\includegraphics[scale=.12]{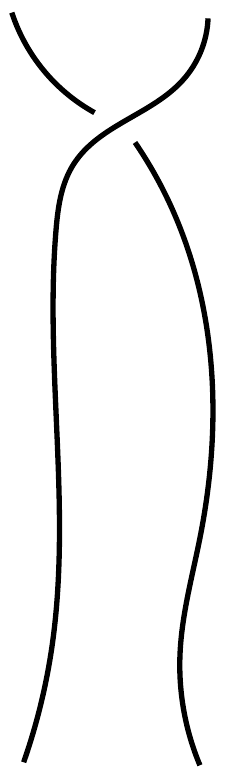}}} \nonumber \\
    \implies  \vcenter{\hbox{\includegraphics[scale=.12]{D3vert6.pdf}}} &=&  -b_0^{-1} \left( b_1 \vcenter{\hbox{\includegraphics[scale=.12]{D2vert.pdf}}} + b_2 a^{-1} \Dzero + b_3 \Dinfty +a^2 b_{\infty} \Dminus \right). \label{eqn:twistclaspformula}
\end{eqnarray}
Finally, incorporate Equation \ref{eqn:twistclaspformula} into Equation \ref{eqn:twistclasp}.

    \begin{eqnarray}
 0 &=& b_0 \vcenter{\hbox{\includegraphics[scale=.12]{D3vert.pdf}}} + b_1 a^{-2} \Dzero +b_2 \Dminus + b_{\infty} a \vcenter{\hbox{\includegraphics[scale=.12]{D2vert.pdf}}}  \nonumber \\
    && -b_3a^{-1}b_0^{-1}\left( b_1 \vcenter{\hbox{\includegraphics[scale=.12]{D2vert.pdf}}} + b_2 a^{-1} \Dzero + b_3 \Dinfty +a^2 b_{\infty} \Dminus \right) \nonumber \\
    0 &=& b_0 \vcenter{\hbox{\includegraphics[scale=.12]{D3vert.pdf}}} - (b_3 a^{-2}b_0^{-1} b_2 - b_1 a^{-2}) \Dzero -(b_3a b_0^{-1} b_{\infty} -  b_2) \Dminus \nonumber \\
    &&- (b_3 a^{-1} b_0^{-1} b_1-  b_{\infty} a) \vcenter{\hbox{\includegraphics[scale=.12]{D2vert.pdf}}} - b_3^2a^{-1} b_0^{-1} \Dinfty. \nonumber
\end{eqnarray}

Therefore we have the following new cubic relation

\begin{eqnarray}
    0 &=& a^2 b_0^2 \vcenter{\hbox{\includegraphics[scale=.12]{D3vert.pdf}}} - (b_2 b_3 - b_0 b_1) \Dzero - a^2(ab_3 b_{\infty}- b_0 b_2) \Dminus- a(b_1 b_3-a^2 b_0 b_{\infty} ) \vcenter{\hbox{\includegraphics[scale=.12]{D2vert.pdf}}} -a b_3^2 \Dinfty. \label{eqn:cubicrelation2}
\end{eqnarray}
    Equation \ref{eqn:cubicrelation2} gives us the desired cubic skein relation.
\end{proof}

Notice that the formula for $t$ obtained from Equation \ref{eqn:cubic2} by placing $t$ in the $D_0$ position is the same as Equation \ref{eqn:t2}. This shows that $R_t$ can be retrieved from the difference between the two cubic skein relations.

\begin{eqnarray}
    0 &=&  b_3 D_3 - \frac{b_3(b_1b_3-a^2b_0 b_{\infty})}{a b_0^2}D_2 -  \frac{b_3(a b_3 b_{\infty} -b_0 b_2 )}{b_0^2} D_1 - \frac{b_3^3}{a b_0^2} D_0 - \frac{b_3(b_2 b_3 - b_0 b_1)}{a^2b_0^2} D_{\infty}  \nonumber \\
    0 &=&  b_3 a^{-3} t - \frac{b_3(b_1b_3-a^2b_0 b_{\infty})}{a b_0^2}a^{-2}t -  \frac{b_3(a b_3 b_{\infty} -b_0 b_2 )}{b_0^2} a^{-1}t - \frac{b_3^3}{a b_0^2} t - \frac{b_3(b_2 b_3 - b_0 b_1)}{a^2b_0^2} t^2 \nonumber \\ \ \ \ 
  \implies t &=& \frac{-b_0^2 - a^2 b_0 b_2 + b_1 b_3 + a^2 b_3^2 - a^2 b_0 b_{\infty} + a^3 b_3 b_{\infty}}{a (b_0 b_1 - b_2 b_3)}. \label{othertrivial} 
\end{eqnarray}

\subsection{The Trefoil Skein Relation}\label{sec:trefoilrelation}
We will now present a new relation that is different from the Hopf relation, called the trefoil relation. In particular we will show that this new relation is obtained from the difference of the two cubic skein relations \ref{eqn:CubicSkeinRelation} and \ref{eqn:cubicrelation2}. Note that this implies $R_t$ (which was shown to be related to the Hopf relation when $b_0, b_3$, and $(b_0b_1-b_2b_3)$  are invertible) and the new trefoil relation come from the difference between the same two cubic skein relations.  \\

Recall that the right trefoil knot, $\overline{3_1}$, is obtained from Equation \ref{eqn:CubicSkeinRelation} by placing the diagram in the $D_3$ position as discusses in Section \ref{subsec:trefoil}. We observe that a new formula for this diagram of the right trefoil is obtained from Equation \ref{eqn:cubic2} by placing the diagram in the $D_3$ position 

 \begin{equation}
   (\overline{3_1})_2 := \trefoilpos  = (a b_0)^{-2}((b_2 b_3 - b_0 b_1) t + a^3(ab_3 b_{\infty}- b_0 b_2) t + a(b_1 b_3-a^2 b_0 b_{\infty} )~H_+ +a b_3^2 t^2).\label{eqn:trefoilpos2}
 \end{equation}

\begin{definition}[Trefoil relation]
    Define the trefoil relation, denoted by $R_{\mathit{tr}_+}$, to be the difference between Equations \ref{eqn:trefoilpos1} and \ref{eqn:trefoilpos2}.

{ \small   \begin{eqnarray}
R_{\mathit{tr}_+} &:=& \overline{3_1} - (\overline{3_1})_2 \nonumber\\ 
     &=& \frac{ (-a^2 b_0^2b_2 - a b_1 b_3^2 + a^3 b_0 b_3 b_{\infty}) ~H_+ + ((b_3-a^3b_0)(b_0b_1-b_2b_3) - a^2 b_{\infty} (b_0^2 + 
  a^2 b_3^2) - a(a b_0^3 +  b_3^3)t) t}{a^2 b_0^2 b_3} \nonumber \\
  &\stackrel{Eqn.~\ref{eqn:hopfpos}}{=}&  t\frac{-a b_0^3 b_2 - a^3 b_0^2 b_2^2 + a^3 b_0^2 b_1 b_3 - 2 b_0 b_1 b_3^2 - 
 a^3 b_0 b_2 b_3^2 - a^2 b_1 b_2 b_3^2 + b_2 b_3^3}{a^2 b_0^2 b_3^2  }  \nonumber\\
 &&+t\frac{a (-a b_0^2 b_1 b_2 + a b_0^3 b_3 - b_1^2 b_3^2 + b_3^4 + a^2 b_0 b_1 b_3 b_{\infty}) t}{a^2 b_0^2 b_3^2 } \nonumber\\
 &&+t\frac{a^2 b_{\infty} (-a b_0^2 b_2 + 2 b_0^2 b_3 + a^2 b_0 b_2 b_3 - b_1 b_3^2 + a^2 b_3^3 +
    a^2 b_0 b_3 b_{\infty})}{a^2 b_0^2 b_3^2 }. \nonumber\\
  \label{eqn:trefoilrel}
\end{eqnarray}}
\end{definition}

\begin{proposition}
    The Hopf relation is not a factor of the trefoil relation $R_{\mathit{tr}_+}$ or vice versa.
\end{proposition}

\begin{proof}
    $R_{\mathit{tr}_+}$ factorizes into the product of $t, a^{-4}, b_0^{-2}, b_3^{-2}, b_{\infty}^{-1},$ and
    \begin{eqnarray*}
        P &=& a^4 b_0^3 b_1 b_2 + a^3 b_0^2 b_1^2 b_2 + a^2 b_0^2 b_1 b_2^2 - a^4 b_0^4 b_3 - 
 a^3 b_0^3 b_1 b_3 - a^2 b_0^3 b_2 b_3 + a b_0^2 b_1 b_2 b_3 - a b_0^3 b_3^2  \\
 &&+ 
 a^3 b_0 b_1^2 b_3^2 + a^2 b_1^3 b_3^2 
 + a b_1^2 b_2 b_3^2 + b_1^2 b_3^3 
 - 
 a^3 b_0 b_3^4 - a^2 b_1 b_3^4 - a b_2 b_3^4 - b_3^5 - a^3 b_0^3 b_2 b_{\infty} - 
 a^5 b_0^2 b_2^2 b_{\infty} \\
 &&- a^4 b_0 b_1^2 b_3 b_{\infty} - a^3 b_0 b_1 b_2 b_3 b_{\infty} - 
 3 a^2 b_0 b_1 b_3^2 b_{\infty} - a^5 b_0 b_2 b_3^2 b_{\infty}
 - a^4 b_1 b_2 b_3^2 b_{\infty} + 
 a^2 b_2 b_3^3 b_{\infty} 
 \\ &&- a^5 b_0^2 b_2 b_{\infty}^2 + 2 a^4 b_0^2 b_3 b_{\infty}^2 + 
 a^6 b_0 b_2 b_3 b_{\infty}^2 - a^4 b_1 b_3^2 b_{\infty}^2 + a^6 b_3^3 b_{\infty}^2 + 
 a^6 b_0 b_3 b_{\infty}^3,
    \end{eqnarray*}
  where $P$ is irreducible.

  The Hopf relation $R_{
  \mathit{Hopf}  }$ factorizes into a product of $t, a^{-3}, b_0^{-1}, b_3^{-1}, b_{\infty}^{-1}$, and 

 \begin{eqnarray*}
      Q &=& a^3 b_0^2 b_1 + a^2 b_0 b_1^2 + a b_0 b_1 b_2 + b_0 b_1 b_3 - a^3 b_0 b_2 b_3 - 
 a^2 b_1 b_2 b_3 - a b_2^2 b_3 - b_2 b_3^2 - a^2 b_0^2 b_{\infty} 
  \\ && - a^4 b_0 b_2 b_{\infty}
+
  a^2 b_1 b_3 b_{\infty} + a^4 b_3^2 b_{\infty} - a^4 b_0 b_{\infty}^2 + a^5 b_3 b_{\infty}^2,\end{eqnarray*}
  where $Q$ is irreducible. Since $P$ and $Q$ are not equal, then $R_{\mathit{tr}_+}$ and $R_{ \mathit{Hopf}  }$ are not even related by a Laurent monomial in the variables $a, b_0, b_1, b_2, b_3, b_{\infty},$ and $t$.
\end{proof}

Since $R_{\mathit{tr}_+}$ and $R_{t}$ are obtained from the difference between the two cubic skein relations, then both relations will become zero if the two cubic skein relations are equal. In order to achieve this, notice that we must make the ratio of the coefficients of $D_3$ and $D_0$ in Equation \ref{eqn:CubicSkeinRelation} be equal the ratio of the coefficients of $D_3$ and $D_0$ in Equation \ref{eqn:cubic2}. This implies that we need $b_3^3 = -ab_0^3$.  Notice that, if we consider the substitution $b_3 = -a^{1/3} b_0$ and $b_1 = -a^{1/3}b_2-a^{5/3}b_{\infty}$, then Equation \ref{eqn:cubic2} coincides with Equation \ref{eqn:CubicSkeinRelation}.

\begin{lemma}\label{lem:substitution2} Suppose $(b_0 b_1 - b_2 b_3)$, $b_0$, $b_\infty$, and $b_3$ are invertible.
    Let $b_3 = -a^{1/3} b_0$ and $b_1 = -a^{1/3} b_2-a^{5/3} b_{\infty}$, then $R_{\mathit{Hopf}} =R_t = R_{\mathit{tr}_+}= 0$.
\end{lemma}

We may also produce a left handed trefoil relation as follows. 

A formula for the left-handed trefoil can be obtained from a shifted version of the cubic skein relation given in Equation \ref{eqn:CubicSkeinRelation} as discussed in Section \ref{subsec:trefoil}. If we instead consider the following regular isotopy on the left-handed trefoil,

 $$ \trefoilneg \sim \vcenter{\hbox{\includegraphics[scale=.6]{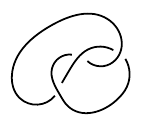}}} \sim a \vcenter{\hbox{\includegraphics[scale=.6]{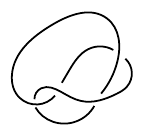}}} \sim a \vcenter{\hbox{\includegraphics[scale=.6]{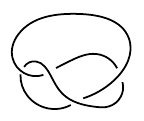}}}$$

 then apply a shifted version of the cubic skein relation given in Equation \ref{eqn:CubicSkeinRelation} to the resulting diagram we obtain a new formula for $3_1$.

 \begin{eqnarray}
     0 &=& b_3 \Dtwo + b_2 \Done + b_1 \Dzero + b_0 \Dminus + b_{\infty} a \Dinfty \nonumber \\
     0 &=& b_3 \vcenter{\hbox{\includegraphics[scale=.6]{trefoilnegside3.pdf}}} + b_2 \vcenter{\hbox{\includegraphics[scale=.6]{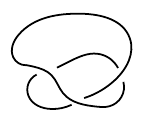}}} + b_1 \vcenter{\hbox{\includegraphics[scale=.6]{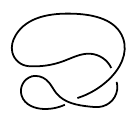}}} + b_0 \vcenter{\hbox{\includegraphics[scale=.6]{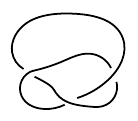}}} + b_{\infty} a \vcenter{\hbox{\includegraphics[scale=.6]{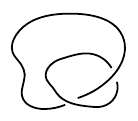}}} \nonumber \\
     \implies (3_1)_2 &=& -a b_3^{-1} \left( b_2 a t +b_1 a^{-2} t + b_0 ~\overline{3_1} +b_{\infty} ~H_+ \right) \label{eqn:lefttrefoil2}
 \end{eqnarray}

\begin{definition}[Left-handed trefoil relation]\label{RTrefoilNeg}
    Define the left-handed trefoil relation, denoted by $R_{\mathit{tr}_-}$, to be the difference between Equations \ref{eqn:lefttrefoil1} and \ref{eqn:lefttrefoil2}.
    Where Equation \ref{eqn:lefthandedtrefoilrelation} is given by substituting Equation \ref{eqn:hopfpos}, \ref{eqn:trefoilpos1}, and \ref{eqn:t1}. 
\begin{eqnarray}
    && R_{\mathit{tr}_-}  := \nonumber\\
     &=& t\frac{-a b_0^3 b_2 - a^3 b_0^2 b_2^2 + a^3 b_0^2 b_1 b_3 - 2 b_0 b_1 b_3^2 - 
 a^3 b_0 b_2 b_3^2 - a^2 b_1 b_2 b_3^2 + b_2 b_3^3}{a b_0 b_3^3  }  \nonumber\\
 &&+t\frac{a (-a b_0^2 b_1 b_2 + a b_0^3 b_3 - b_1^2 b_3^2 + b_3^4 + a^2 b_0 b_1 b_3 b_{\infty}) t}{a b_0 b_3^3  } \nonumber\\
 &&+t\frac{a^2 b_{\infty} (-a b_0^2 b_2 + 2 b_0^2 b_3 + a^2 b_0 b_2 b_3 - b_1 b_3^2 + a^2 b_3^3 +
    a^2 b_0 b_3 b_{\infty})}{a b_0 b_3^3  }.\label{eqn:lefthandedtrefoilrelation}
\end{eqnarray}
    
\end{definition}

Notice that $b_3 R_{\mathit{tr}_-} = a b_0 R_{\mathit{tr}_+}$. Since $a, b_0$, and $b_3$ are assumed to be invertible, then $R_{\mathit{tr}_-}=0$ under the substitutions provided in Lemma \ref{lem:substitution2}.

\begin{remark} $a[2,-2] - [-3]$ example.

The left-handed trefoil knot has the following two descriptions in Conway code---$a[2,-2]$ and $[-3]$---which differ by a Tait Flype and framing (see \cite{KauLam}).

The Rational Tangle Algorithm computes the corresponding diagrams of these left-hand trefoils differently. Taking their difference
\[ a[2,-2] - [-3], \]
produces a relation which is not divisible by the Hopf Relation. 

However, the difference 
\[ (a[2,-2] - [-3]) - (R_{\mathit{tr}_-}) \] 
is divisible by the Hopf Relation, $[-1,-1]-[1,1]$. That is, these relations are equivalent modulo the Hopf Relation.

\end{remark}

\subsection{The $\boldsymbol{6_3}$, Whitehead Link, and Conway Knot Relations}
In this subsection we will show the $6_3$ relation comes from the Whitehead relation and the Whitehead relation comes from calculating the Whitehead link in two different ways. While it is unclear, by looking at the relation, if the Whitehead relation is a combination of the Hopf relation and the trefoil relation, we will show that the Whitehead relation comes from a new quadratic skein relation from studying a family of Conway knot relations. In particular we will show that the new quadratic skein relation can also be obtained from Equations \ref{eqn:CubicSkeinRelation} and \ref{eqn:cubic2} if we assume that the two cubic skein relations are not equal. \\

We start by using a shifted version of the cubic skein relation to calculate the $(6_3)_+$ knot.

\begin{eqnarray*}
    0 &=& b_3 \vcenter{\hbox{\includegraphics[scale=.7]{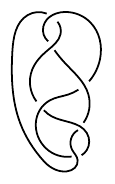}}}+ b_2 \vcenter{\hbox{\includegraphics[scale=.7]{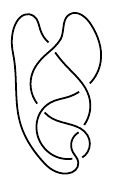}}}+ b_1 \vcenter{\hbox{\includegraphics[scale=.7]{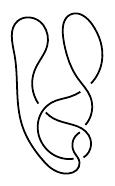}}}+ b_0 \vcenter{\hbox{\includegraphics[scale=.7]{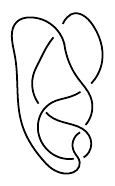}}}+ b_{\infty}\vcenter{\hbox{\includegraphics[scale=.7]{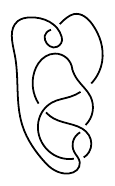}}} \\
    &=& b_3 (6_3)_+ + b_2 \vcenter{\hbox{\includegraphics[scale=.7]{63plus1.pdf}}} + b_1 a ~3_1 +b_0 a^{-1} ~H_- + b_{\infty} a ~(4_1)_+ 
\end{eqnarray*}
\begin{equation}
    (6_3)_+ = - b_3^{-1}(b_2 \vcenter{\hbox{\includegraphics[scale=.7]{63plus1.pdf}}} + b_1 a ~3_1 +b_0 a^{-1} ~H_- + b_{\infty} a ~(4_1)_+) \label{eqn:63one}
\end{equation}

Next we use the same relation to calculate its mirror image. 

\begin{eqnarray}
    0 &=& b_3 \vcenter{\hbox{\includegraphics[scale=.7]{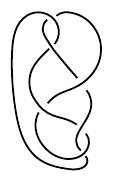}}}+ b_2 \vcenter{\hbox{\includegraphics[scale=.7]{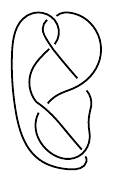}}}+ b_1 \vcenter{\hbox{\includegraphics[scale=.7]{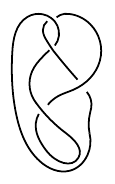}}}+ b_0 \vcenter{\hbox{\includegraphics[scale=.7]{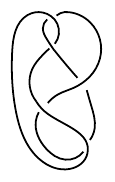}}}+ b_{\infty}\vcenter{\hbox{\includegraphics[scale=.7]{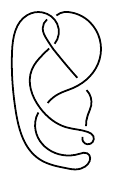}}} \nonumber\\
    &=& b_3 (6_3)_- + b_2 \vcenter{\hbox{\includegraphics[scale=.7]{63neg1.pdf}}} + b_1 a 
~3_1+b_0 a^{-1} ~H_- + b_{\infty} a ~(4_1)_+  \nonumber\\
    (6_3)_- &=&  -b_3^{-1}(b_2 \vcenter{\hbox{\includegraphics[scale=.7]{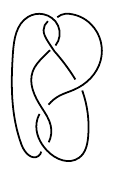}}} + b_1 a ~3_1 +b_0 a^{-1} ~H_- + b_{\infty} a ~(4_1)_+) \label{eqn:63two}
\end{eqnarray}

\begin{definition}[$(6_3)$-relation]
    Define the $(6_3)$-relation, denoted by $R_{6_3}$, as the difference between Equation \ref{eqn:63one} and \ref{eqn:63two}.
\end{definition}

Notice that if we assume the Hopf relation, then $R_{6_3}$ is equal to the product of $b_3^{-1}b_2$ and the difference of two Whitehead links. Furthermore, one can show that no new relations arise from calculating both Whitehead link diagrams without introducing any kinks. However, if we calculate the Whitehead link by introducing a Reidemeister 1 and Reidemeister 3 move, then we obtain a new relation.\\

First we will use the shifted cubic skein relation to obtain a formula for the Whitehead link;

\begin{eqnarray}\label{eqn:whiteheadlink1}
\vcenter{\hbox{\includegraphics[scale=.7]{Whitehead1.pdf}}} &=& -b_0^{-1}  \left( b_1 ~(4_1) + b_2 a^{-1} ~H_- + b_3 a^2 t + b_{\infty} a^2 ~3_1 \right).
\end{eqnarray}

For the second formula we will first introduce a kink (using one Reidemeister 1 and one Reidemeister 3 move).

\begin{eqnarray}
    \vcenter{\hbox{\includegraphics[scale=.7]{63plus1.pdf}}} &=& a \vcenter{\hbox{\includegraphics[scale=.7]{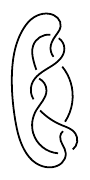}}} 
    = -b_3^{-1} a (b_2 ~3_1 + a^{-2} b_1 ~H_-  +b_0 \vcenter{\hbox{\includegraphics[scale=.7]{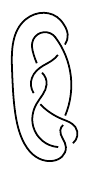}}}+ a b_{\infty} ~(4_1)_-). \label{eqn:whitehead52}
\end{eqnarray}

Notice that for this case we must calculate the twist knot $5_2$.

\begin{eqnarray}
    5_2 &=& -b_3^{-1} (b_2 
    ~(4_1)_- + b_1 ~3_1 + b_0 a^2 t + b_{\infty} ~H_-). \label{eqn:52}
\end{eqnarray}

After incorporating Equation \ref{eqn:52} into Equation \ref{eqn:whitehead52} we have 

\begin{eqnarray}
\vcenter{\hbox{\includegraphics[scale=.7]{63plus1.pdf}}} &=&  
 b_3^{-2} \bigg( a( b_0 b_1 - b_2 b_3) ~3_1 +(a b_0 b_{\infty} - a^{-1} b_1 b_3) ~H_-+ a\left( b_0 b_2-a b_{\infty}b_3 \right) ~(4_1)_-\bigg). \label{eqn:whiteheadlink2}
\end{eqnarray}

\begin{definition}[Whitehead relation]
    Define the Whitehead relation, denoted by $R_{\mathit{Wh}}$, by the difference of Equation \ref{eqn:whiteheadlink1} and \ref{eqn:whiteheadlink2}.

\begin{eqnarray*}
    R_{\mathit{Wh}} &:=& \vcenter{\hbox{\includegraphics[scale=.7]{63plus1.pdf}}}- a \vcenter{\hbox{\includegraphics[scale=.7]{Whitehead2.pdf}}} \\
    &=& \frac{a^2 b_0 b_3 b_{\infty}-a b_0^2 b_2-b_1b_3^2}{b_0 b_3^2} ~4_1 + \frac{a b_0 b_2 b_3 - a b_0^2b_1 -a^2b_3^2 b_{\infty}}{b_0 b_3^2} ~3_1 \\
    &&+ \frac{b_0b_1b_3-b_2 b_3^2-a^2 b_0^2 b_{\infty}}{a b_0 b_3^2} ~H_- +\frac{-a^3b_0^3-a^2b_3^3}{ b_0 b_3^2}t.
\end{eqnarray*}
\end{definition}

One can show that after substituting formulas for $4_1, 3_1, H_-$, and $t$, the relation $R_{\mathit{Wh}}$ is not a factor of the Hopf relation nor the trefoil relation. Since it is not obvious to tell if $R_{\mathit{Wh}}$ is an algebraic combination of the Hopf and trefoil relation we must rely on using a Gr\"oebner basis\footnote{We attempted to calculate the Gr\"oebner basis, however the programs used, such as Mathematica, failed to execute.} or understand the underlying skein relation used to obtain $R_{\mathit{Wh}}$.  The underlying skein relation originally came from studying Conway knots; in particular $[2, -n, m, \dots]$. Notice that the Conway knot $[-2, -n+1, m \dots]$ is obtained from $[2, -n, m, \dots]$ by one Reidemeister 1 and one Reidemeister 3 move where the sequence $\{ m \dots \}$ in $[2, -n, m, \dots]$ and $[-2, -n+1, m \dots]$ are is assumed to be the same. \\

Let $C_{2, -2}$ denote the tangle obtained from the top portion of the Conway knot $[2, -n, m, \dots]$ with 4 crossings given in Figure \ref{fig:Conwayknotrelation}. Similarly, let $C_{-2, -1}$ denote the tangle obtained from the top portion of the Conway knot $[-2, -n+1, m, \dots]$ with three crossings. See Figure \ref{TwistKnot} for a detailed illustration and more details about twist tangles and twist knots.

 \begin{figure}[H]
     \centering
     $$C_{2, -2}:= \vcenter{\hbox{\includegraphics[scale=.7]{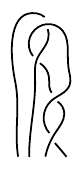}}}= a^{-1}\vcenter{\hbox{\includegraphics[scale=.7]{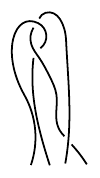}}} $$
     \caption{Tangle $C_{2, -2}$ is related to tangle $C_{-2, -1}$ by one Reidemeister 1 and one Reidemeister 3 move.}
     \label{fig:Conwayknotrelation}
 \end{figure}

 Consider the cubic skein relation in Equation \ref{eqn:CubicSkeinRelation} and place $C_{2, -2}$ in the $D_3$ position;

\begin{eqnarray}
  0 &=& b_3  \vcenter{\hbox{\includegraphics[scale=.8]{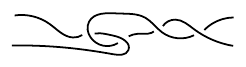}}} + b_2 \vcenter{\hbox{\includegraphics[scale=.8]{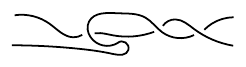}}} + b_1 \vcenter{\hbox{\includegraphics[scale=.8]{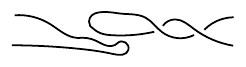}}} \nonumber \\
  &&+ b_0 \vcenter{\hbox{\includegraphics[scale=.8]{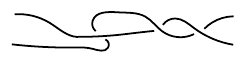}}}+ b_{\infty} \vcenter{\hbox{\includegraphics[scale=.8]{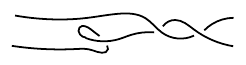}}}.
\end{eqnarray}

If we assume that $b_3$ is invertible, then $C_{2,-2}$ can be written in terms of basic tangles.

\begin{eqnarray}    \vcenter{\hbox{\includegraphics[scale=.8]{RelationConwayT1.pdf}}} &=& -b_3^{-1} ( b_2 D_1 + b_1 a^{-2} D_{\infty} + b_0 D_3 + b_{\infty} a D_2).\label{eqn:C2n2}
\end{eqnarray}

By solving for $D_3$ in Equation \ref{eqn:CubicSkeinRelation} and incorporate it into Equation \ref{eqn:C2n2}, we obtain

\begin{eqnarray}
\vcenter{\hbox{\includegraphics[scale=.7]{RelationConwayT1.pdf}}} &=& -b_3^{-2}(b_0^2 D_0 +  (b_0 b_1-b_2b_3) D_1 + (b_0 b_2 - a b_3 b_{\infty}) D_2 + (b_0 b_{\infty} - a^{-2} b_1 b_3) D_{\infty}. \label{eqn:C2n2one}
\end{eqnarray}

Now consider the cubic skein relation in Equation \ref{eqn:CubicSkeinRelation} and place $C_{-2, -1}$ in the $D_0$ position. 

\begin{eqnarray}
    0 &=& b_0 \vcenter{\hbox{\includegraphics[scale=.7]{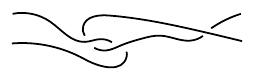}}}+ b_1 \vcenter{\hbox{\includegraphics[scale=.7]{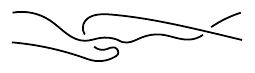}}} + b_2 \vcenter{\hbox{\includegraphics[scale=.7]{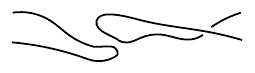}}} \nonumber\\
    && + b_3 \vcenter{\hbox{\includegraphics[scale=.7]{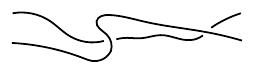}}} + b_{\infty} \vcenter{\hbox{\includegraphics[scale=.7]{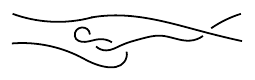}}}.
\end{eqnarray} 

If we assume that $b_0$ is invertible, then $C_{-2, -1}$ can be written in terms of basic tangles.

\begin{equation}\label{eqn:cn2n1}  \vcenter{\hbox{\includegraphics[scale=.7]{RelationConway2T1.pdf}}} = b_0^{-1} (b_1 D_2 + b_2 a^{-1} D_{\infty}+ b_3 D_0 + b_{\infty} a^{2}D_1).
\end{equation}

\begin{definition}[Conway relation]
    Define the Conway relation, denoted by $R_{\text{C}}$, to be a family of relations obtained from taking the difference between Equation \ref{eqn:C2n2one} and $a^{-1}$ times Equation \ref{eqn:cn2n1} where the links formed from $D_0, D_1, D_2$, and $D_{\infty}$ are evaluated using the same algorithm, respectively, in Equations \ref{eqn:C2n2one} and \ref{eqn:cn2n1}.

    \begin{eqnarray}
R_{\text{C}} &=&\vcenter{\hbox{\includegraphics[scale=.8]{RelationConwayT1.pdf}}}-a^{-1} \vcenter{\hbox{\includegraphics[scale=.7]{RelationConway2T1.pdf}}} \nonumber  \\
      &=& a^{-2} b_0^{-1} b_3^{-2}  ((a^2 b_0^3 + a b_3^3) D_0 + (a^2 b_0^2 b_1 - a^2 b_0 b_2 b_3 + a^3 b_3^2 b_{\infty}) D_1  \nonumber \\
      && (a^2 b_0^2 b_2 + a b_1 b_3^2 - a^3 b_0 b_3 b_{\infty}) D_2 +  (-b_0 b_1 b_3 + b_2 b_3^2 + a^2 b_0^2 b_{\infty}) D_{\infty}.
    \end{eqnarray}
\end{definition}
Notice that the Conway relation can be viewed as a quadratic skein relation. In particular this relation is non-trivial when the two cubic skein relations are different. 

\begin{lemma}\label{lemma:quadraticskein}
Suppose $a$ and $b_0$ are invertible, then the following quadratic skein relation is a relation in the cubic skein module. In particular, it is obtained from taking the difference between Equation \ref{eqn:CubicSkeinRelation} and \ref{eqn:cubic2}.

\begin{eqnarray}
    0 &=& -b_0^{-2}a^{-2}[(a^2 b_0^2 b_2 + a b_1 b_3^2 - a^3 b_0 b_3 b_{\infty})D_2 +(a^2 b_0^2 b_1 - a^2 b_0 b_2 b_3 + a^3 b_3^2 b_{\infty})D_1 \nonumber \\
    && +
    (a^2 b_0^3 + a b_3^3) D_0 
    + (-b_0 b_1 b_3 + b_2 b_3^2 + a^2 b_0^2 b_{\infty}) D_{\infty}]. \label{eqn:cubicskein1and2}
\end{eqnarray}
\end{lemma}

Notice that the quadratic skein relation is trivial (the coefficients of Equation \ref{eqn:cubicskein1and2} are all zero) if the coefficients of the two cubic skein relations are equal. This implies that $R_{\text{C}}$ is zero if the coefficients of the two cubic skein relations are equal.
Notice that the substitutions provided in the next two corollaries are the same substitution given in Lemma \ref{lem:substitution2} that yields $R_{\mathit{Hopf}}=R_t=R_{\mathit{tr}_+}=0$.
\begin{corollary}\label{Coro:substituion}
 Let $b_3 = -a^{1/3} b_0$ and $b_1 = -a^{1/3} b_2-a^{5/3} b_{\infty}$, then $R_{\text{C}}=0$.   
\end{corollary}
\begin{proof}
    The coefficients of $D_0$, $D_1$, $D_2$, and $D_{\infty}$, in $b_0^{-1} b_3^2 R_{\text{C}}$ are equal to the respective coefficients of $D_0$, $D_1$, $D_2$, and $D_{\infty}$ in the quadratic skein relation.  
\end{proof}
\begin{corollary}
    \label{coro:whitehead}
    Let $b_3 = -a^{1/3} b_0$ and $b_1 = -a^{1/3} b_2-a^{5/3} b_{\infty}$, then $R_{\text{Wh}}=0$. 
\end{corollary}

\begin{proof}
   $-b_3^2b_0^{-1}a^{-1} R_{\mathit{Wh}}$ is equal to the quadratic skein relation by placing $(H_-)$ in the $D_{\infty}$ position.
\end{proof}

\begin{lemma}\label{lemma:subDubrovnikandKauffman}
Let $b_0 = \epsilon, b_1 = -x, b_2 =1, b_3=0$, and  $b_{\infty} = a^{-1} \epsilon$, then Equations \ref{eqn:CubicSkeinRelation} and \ref{eqn:cubicskein1and2} are equal to
    the shifted Kauffman skein relation for $\epsilon=1$ and the shifted Dubrovnik skein relation for $\epsilon=-1$ given in Equation \ref{KaDub2}. Furthermore, Equation \ref{eqn:cubic2} is trivial.
\end{lemma}

In summary, we have explored new relations obtained by studying different diagrams of small crossing knots. In the process we observed that by allowing invertibility of $a, b_0$, and $b_{\infty}$ we obtain a new cubic and quadratic skein relation. In Corollary \ref{Coro:substituion} we provided a substitution that trivialises the new quadratic skein relation and in Lemma \ref{lemma:subDubrovnikandKauffman} we observed that this new quadratic skein relation recovers the Kauffman 2-variable and Dubrovnik quadratic skein relations. We now conjecture that the appropriate substitutions made to trivialise the new quadratic skein relation leads to a new invariant on a special collection of links. 
\begin{conjecture}\label{Conj:invariant}
Let $\mathcal{C}(D)$ be defined on diagrams of links with the following two axioms, framing relations, and cubic skein relation;

\begin{enumerate}
    \item $\mathcal{C}(\bigcirc) = 1$,
    \item $\mathcal{C}(D  \sqcup \bigcirc) = t \mathcal{C}(D)$,
    \item $\mathcal{C}(\vcenter{\hbox{\includegraphics[scale=1]{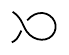}}}) = A^{-3} \mathcal{C}(\vcenter{\hbox{\includegraphics[scale=1]{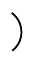}}}) \quad$ and $\quad \mathcal{C}(\vcenter{\hbox{\includegraphics[scale=1]{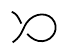}}}) = A^{3} \mathcal{C}(\vcenter{\hbox{\includegraphics[scale=1]{framingrelation1.pdf}}})$,
    \item $-A X \mathcal{C}( D_3) + Y \mathcal{C}(D_2) + (-A Y-A^{5} Z)\mathcal{C}(D_1)+ X \mathcal{C}(D_0)+ Z \mathcal{C}(D_{\infty})=0$,
\end{enumerate}
where $t= (A^{-8}-1) \frac{X}{Z} -A^{-2}(A^{-4}-1)\frac{Y}{Z}+A^2.$ \\

Let $\mathcal{L}$ be the set of links such that for each $L \in \mathcal{L}$, there exists a diagram $D$ of $L$ with $\mathcal{C}(D) \in \mathbb{Z}[A^{\pm 1}, X^{\pm 1}, Y^{\pm 1}, Z^{\pm 1}]$ (such as $3$-algebraic links). Then $\mathcal{C}$ is an invariant of regular isotopy of $\mathcal{L}$.
\end{conjecture}

We end this section with a question on a future direction in the study of cubic skein modules that may unveil a new polynomial invariant of knots from a cubic skein relation.

\begin{question}\label{question:structurecubic}
       Consider the cubic skein module, $\mathcal S_{4,\infty}(S^3)$ with skein relation
\begin{equation}\label{eqn:newCubicSkeinRelation}
-A X  D_3 + Y D_2 + (-A Y-A^{5} Z) D_1+ X D_0+ Z D_{\infty}=0,
\end{equation}
trivial knot relations $\bigcirc=1$ and $D  \sqcup \bigcirc = [(A^{-8}-1) \frac{X}{Z} -A^{-2}(A^{-4}-1)\frac{Y}{Z}+A^2] D$,
and framing relation $D^{(1)}=A^3 D$. What is the structure of this skein module?
\end{question}

\section{Generating Set of 3-algebraic Tangles and Knots}\label{GS3AT}
Algebraic tangles were first introduced by Conway as a natural extension of rational tangles, providing a systematic way to construct complex tangles from simpler building blocks; in Subsection \ref{subsection:rational2tangle} we discussed rational tangles defined from $2$-algebraic tangles. Significant progress on the study of classical $2$-algebraic tangles was later achieved by Bonahon and Siebenmann \cite{BS}, who developed a detailed structural framework for understanding their composition and properties. The broader concept of algebraic tangles of arbitrary type was subsequently formulated and analyzed in \cite{PTs}, extending these ideas beyond the classical case. In this section, we recall the definition of $n$-algebraic tangles and examine in more depth the cases of $2$- and $3$-algebraic tangles, which serve as key examples for the general theory.

We now recall some fundamental properties of the set $\mathcal{T}(n)$ of $n$-tangles, which will be essential for the construction of the $n$-algebraic tangles presented below. The set $\mathcal{T}(n)$ forms a monoid under the operation of horizontal concatenation. This product is read from left to right: if $P_1$ and $P_2$ are two $n$-tangles, then $P_1P_2$ denotes the composition obtained by placing $P_1$ to the left of $P_2$. Moreover, the monoid $\mathcal{T}(n)$ allows a natural action by $D_{2n} \oplus \mathbb{Z}_2$, where $D_{2n}$ is the action of the dihedral group preserving orientation of $4n$ elements generated by the rotation $r$ along the $z$-axis by the angle $2\pi/2n$, and the rotation $\pi$ along the $y$-axis denoted by $m_y$, see Section \ref{Mut}. These symmetries, given by rotation and reflection, act non-trivially on the elements of the monoid \cite{PTs}. We now proceed to define the notion of $n$-algebraic tangles.

\begin{definition}[$n$-algebraic tangles]\label{n-algebraicDefi} \hfill
\begin{enumerate}
    \item $n$-algebraic tangles are defined recursively as follows:
    \begin{itemize}
        \item [(i)] an $n$-tangle with no more than one crossing is $n$-algebraic.
        \item [(ii)] Inductive step: if $P_1$ and $P_2$ are $n$-algebraic tangles, then their twisted sum $r^{i}(P_1)r^{j}(P_2)$ is also an $n$-algebraic tangle, where $r$ denotes the rotation along the $z$-axis by the angle $2\pi/2n$.  
    \end{itemize}
    \item In a more restricted way, if in $(ii)$ above $P_2$ has no more than $k$ crossings, then we call the twisted sum an $(n,k)$-algebraic tangle.
    \item An $n$-algebraic link (respectively an $(n,k)$-algebraic link) is a link with a diagram obtained by closing an $n$-algebraic (respectively an $(n,k)$-algebraic) tangle by pairwise disjoint arcs.
\end{enumerate}   
   
\end{definition}

In this section, we undertake a study of the set of $n$-algebraic tangles, with particular emphasis on the case $n=3$. Our analysis is conducted within the framework of the skein algebra $\mathcal{S}_{4, \infty}(\mathcal{C}_n \cup \mathcal{B}_n;R)$, where $\mathcal{C}_n$ denotes the set of non-invertible basic $n$-tangles, $\mathcal{B}_n$ denotes the set of invertible basic $n$-tangles, and $R$ is a commutative ring with identity. The principal result established in this section is that every $3$-algebraic tangle admits a representation as a linear combination of the $40$ basic $3$-tangles in Figures \ref{InvertibleTangles} and \ref{Non-InvertibleTangles}. These $40$ basic $3$-tangles serve as a basis for the module of $3$-tangles over $R$ under the cubic skein relations defining $\mathcal{S}_{4,\infty}(\mathcal{C}_n \cup \mathcal{B}_n;R)$.

\subsection{2-algebraic Knots and Tangles}
Every $2$-algebraic tangle can be reduced to a linear combination of four basic $2$-tangles illustrated in, for instance, Equation \ref{4basictangles}.

\begin{table}[ht]
    \centering
    \begin{tabular}{c ||c | c|c|c|}
      & $\Dzero$ & $\Done$ & $\Dminus$ & $\Dinfty$ \\[1ex] \hline \hline
     \rule{0pt}{3ex}$\Dzero$    & $D_0$ & $D_{1}$ & $D_{-1}$ & $D_{\infty}$\\[1ex]
     $\Done$    &  $D_1$ & $\frac{b_0 D_{-1} + b_1 D_0 + b_2 D_1 + a b_{\infty} D_{\infty}}{b_3}$ & $D_0$ & $a^{-1}D_{\infty}$\\[1ex]
     $\Dminus$ & $D_{-1}$ & $D_0$& $\frac{b_1 D_{-1}+b_2 D_0 + b_3D_1 + a^2 b_{\infty} D_{\infty}}{b_0}$ & $a D_{\infty} $\\[1ex]
     $\Dinfty$ & $D_{\infty}$ & $a^{-1} D_{\infty}$ & $a D_{\infty}$ & $t D_{\infty}$
    \end{tabular}
    \caption{Multiplication table for basic $2$-tangles.}
    \label{tab:mult2tangle}
\end{table}

\subsection{3-algebraic Knots and Tangles}

We will now focus on the cubic skein algebra $\mathcal{S}_{4, \infty}(\mathcal{C}_n \cup \mathcal{B}_n;R)$.

 In particular, in this section we use the product defined for $\mathcal{S}_{4, \infty}(\mathcal{C}_n \cup \mathcal{B}_n;R)$ to show that any $3$-tangle can be written in terms of the following $40$ basic $3$-tangles\footnote{ But why 40 tangles? Shortly, we note that $40=(1+3)(1+3^2)$, which is a special case of $\Pi_{i=1}^k (1+p^i)$, counting the number of Lagrangians in the symplectic space $Z_p^{2k}$.
 For a full explanation we refer to \cite{DJP,Prz4,Prz5}.} illustrated in Figure \ref{InvertibleTangles} and Figure \ref{Non-InvertibleTangles}. It follows that all 3-algebraic links are generated by trivial links; that is, they belong the submodule $ \mathcal S^{(0)}_{4,\infty}(S^3)$ (see Subsection \ref{ninftysm}).
 Throughout this section we let $\mathcal{B}_3$ denote the set of invertible basic $3$-tangles and $\mathcal{C}_3$ denote the set of non-invertible basic $3$-tangles.
 \begin{figure}[H]
	\centering
    $$\vcenter{\hbox{\begin{overpic}[scale = .9]{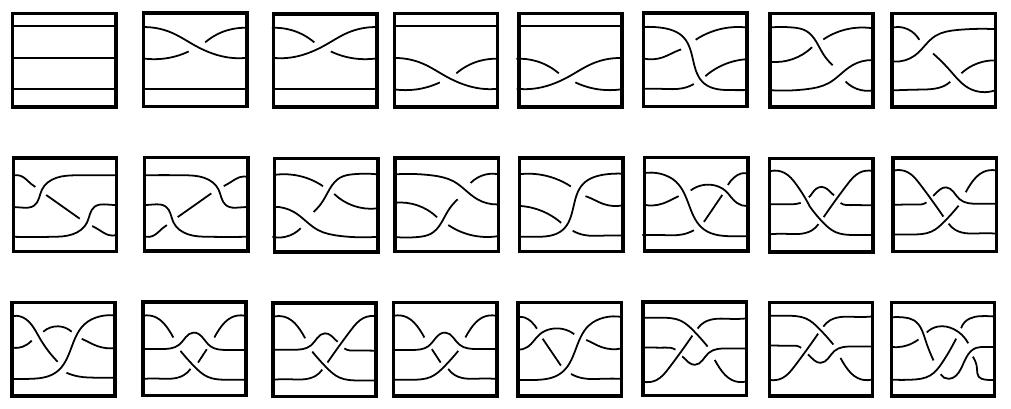}
    \put(80, 119){$S_1$}
     \put(135, 119){$S_1^{-1}$}
      \put(190, 119){$S_2$}
       \put(240, 119){$S_2^{-1}$}
        \put(294, 119){$S_1S_2$}
         \put(345, 119){$S_1S_2^{-1}$}
          \put(399, 119){$S_1^{-1}S_2$}

	\put(8, -8){$S_1S_2^{-1}S_1^{-1}$}
     \put(68, -8){$S_1^{-1}S_2S_1$}
      \put(121, -8){$S_1^{-1}S_2S_1^{-1}$}
      \put(174, -8){$S_1^{-1}S_2^{-1}S_1$}
       \put(225, -8){$S_1^{-1}S_2^{-1}S_1^{-1}$}
        \put(284, -8){$S_2S_1^{-1}S_2$}
         \put(337, -8){$S_2^{-1}S_1S_2^{-1}$}
          \put(391, -8){$(S_1S_2^{-1})^2$}
   
    \put(12, 56){$S_1^{-1}S_2^{-1}$}
     \put(75, 56){$S_2S_1$}
       \put(130, 56){$S_2S_1^{-1}$}
      \put(182, 56){$S_2^{-1}S_1$}
       \put(232, 56){$S_2^{-1}S_1^{-1}$}
        \put(289, 56){$S_1S_2S_1$}
         \put(339, 56){$S_1S_2S_1^{-1}$}
          \put(394, 56){$S_1S_2^{-1}S_1$}
\end{overpic} }}$$
	\caption{Elements of $\mathcal{B}_3,$ invertible braid-type basic $3$-tangles. See \cite{PTs}.}	\label{InvertibleTangles}
\end{figure}

 \begin{figure}[H]
	\centering
    $$\vcenter{\hbox{\begin{overpic}[scale = .9]{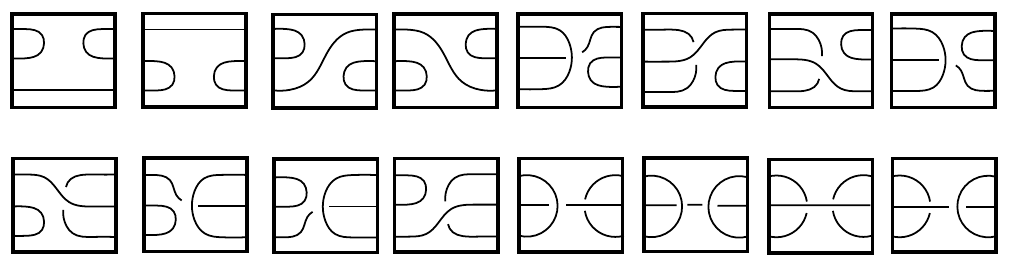}
	\put(18, -8){$U_2S_1$}
     \put(74, -8){$U_2S_1^{-1}$}
      \put(130, -8){$U_1S_2$}
      \put(182, -8){$U_1S_2^{-1}$}
       \put(232, -8){$S_1U_2S_1$}
        \put(284, -8){$S_1U_2S_1^{-1}$}
         \put(338, -8){$S_1^{-1}U_2S_1$}
          \put(390, -8){$S_1^{-1}U_2S_1^{-1}$}
    
    \put(21, 56){$U_1$}
     \put(80, 56){$U_2$}
       \put(132, 56){$U_1U_2$}
      \put(185, 56){$U_2U_1$}
       \put(237, 56){$S_1U_2$}
        \put(292, 56){$S_1^{-1}U_2$}
         \put(345, 56){$S_2U_1$}
          \put(398, 56){$S_2^{-1}U_1$}
\end{overpic} }}$$
	\caption{Elements of $\mathcal{C}_3$, non-invertible basic $3$-tangles. see \cite{PTs}.}	\label{Non-InvertibleTangles}
\end{figure}
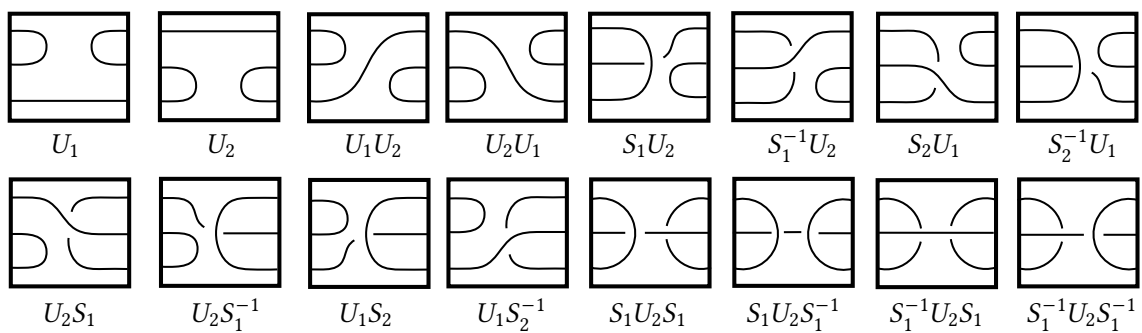

We start with a few immediate observations on properties of basic $3$-tangles that only require Reidemeister 2 moves, the framing relations $D^{(1)} = aD$ and $D^{(-1)} = a^{-1}D$, the relation $D \sqcup \bigcirc = t D$, and planar isotopy. 

For $i, j = 1,2$ where $i\neq j$,

\begin{eqnarray}
    U_iU_jU_i = U_i  &\quad&
    U_i^2 = t U_i \label{eqn:observe1} \\
    U_iS_jU_i = a U_i &\quad&
    U_iS_j^{-1}U_i = a^{-1} U_i \label{eqn:observe2}\\
    U_iS_i = a U_i = S_i U_i &\quad&
    U_iS_i^{-1} = a^{-1} U_i = S_i^{-1} U_i \label{eqn:observe3} \\
(S_i S_j)^{\pm 1}U_i = U_j U_i &\quad &  U_i (S_j S_i)^{\pm 1} = U_i U_j  \label{eqn:observe6}\\
S_i^{\pm 1} U_j = S_j^{\mp 1} U_i U_j && U_i S_j^{\pm 1} = U_i U_j S_i^{\mp 1} \label{eqn:observe8}\\ S_i^{\epsilon_1}U_jS_i^{\epsilon_2} = S_j^{-\epsilon_1}U_iS_j^{-\epsilon_2}, && \text{for } \epsilon_1, \epsilon_2 = \pm 1 \label{eqn:observe9} \\
 &&  S_i S_i^{-1} = S_i^{-1} S_i = e. \label{eqn:observe7}
\end{eqnarray}

We also have braid relations from the Reidemeister 3 move;

\begin{eqnarray}
      S_i S_j S_i &=& S_j S_i S_j \text{ for } i, j =1,2 \text{ where } i\neq j. \label{eqn:braidrelation}\\
      (S_i S_j)^{\pm} S_i^{\mp} &=& S_j^{\mp} (S_i S_j)^{\pm} \text{ for } i, j =1,2 \text{ where } i\neq j.\label{eqn:braidrelation1}
\end{eqnarray}

\begin{lemma}\label{lem:uisisj}
    For $i, j = 1, 2$, where $i\neq j$, 
    
    \begin{eqnarray}
        U_iS_jS_i^{-1} &=& b_0^{-1}(b_1 U_i S_j + b_2 U_i U_j + b_3 U_i S_j^{-1} + a^2 b_{\infty} U_i) \label{eqn:lemma2twist}\\
         S_i^{-1} S_j U_i &=& b_0^{-1}(b_1  S_jU_i + b_2 U_j U_i + b_3  S_j^{-1} U_i+ a^2 b_{\infty} U_i)\label{eqn:lemma2twist2}\\
         S_j^{-2} &=& b_0^{-1}(b_1 S_j^{-1} + b_2 e +b_3 S_j +a^2 b_{\infty}U_j) \label{eqn:lemma2twist4}\\
        U_iS_j^{-1}S_i &=& b_3^{-1} (b_0 U_i S_j^{-1} + b_1 U_i U_j + b_2 U_i S_j + a b_{\infty} U_i )\label{eqn:lemma2negtwist}\\
        S_i S_j^{-1} U_i &=& b_3^{-1} (b_2 S_j U_i+ b_1 U_j U_i +b_0 S_j^{-1} U_i +a b_{\infty} U_i )\label{eqn:lemma2negtwist2}\\
          S_j^2 &=& b_{3}^{-1} (b_0  S_j^{-1} +b_1 e+b_2 S_j + a b_{\infty}  U_j) \label{eqn:lemma2negtwist4}
    \end{eqnarray}
    
\end{lemma}

\begin{proof}
First observe that the minimal diagrams of $U_i S_jS_i^{-1}, S_i^{-1}S_jU_i, U_iS_j^{-2}$ and $S_i^{-2}U_j$ contain a region including $D_{-2}$. Equations \ref{eqn:lemma2twist} through \ref{eqn:lemma2twist4} are obtained by using Table \ref{tab:mult2tangle}, $D_{-1} D_{-1}$. \\
Similarly, the minimal diagrams of $U_iS_j^{-1}S_i, S_iS_j^{-1}U_i$, $U_iS^2$, and $S^2_iU_j$ contain a region including $D_{2}$. Equations \ref{eqn:lemma2negtwist} through \ref{eqn:lemma2negtwist4} are obtained by using Table \ref{tab:mult2tangle}, $D_{1} D_{1}$.
\end{proof}
\begin{lemma}\label{lemma:c3}
    A product of any pair of elements of $\mathcal{C}_3$ in $\mathcal{S}_{4, \infty}(\mathcal{C}_3;R)$ can be generated as a linear
combination of elements in $\mathcal{C}_3$. 
\end{lemma}

\begin{proof}
Let $c_1, c_2 \in \mathcal{C}_3$.\\
First suppose $c_1 c_2$ does not contain as a subword $S_i^{\pm 2}$, $S_iS_j^{-1}$, or $S_i^{-1}S_j$ for $i, j=1,2$ and $i\neq j$. Then Equations \ref{eqn:observe1} through \ref{eqn:observe7} can be applied to show that $c_1 c_2$ is a basic non-invertible $3$-tangle, up to a factor of a monomial in $\mathbb{Z}[a^{\pm1}, t^{\pm 1}]$. This was proven in \cite{PTs} since only isotopy is used to obtain our desired result. \\
Now suppose $c_1 c_2 = U_i (S_j S_j)^{\pm 1} U_i$, then Table \ref{tab:mult2tangle} can be applied to show that $c_1 c_2$ is a linear combination of $U_iS_j^{\pm 1} U_i$ and $ U_i$. It follows by the first case that $c_1 c_2$ is a linear combination of elements of $\mathcal{C}_3$.\\
The same argument applies for $c_1 c_2 = U_2 (S_1 S_1)^{\pm 1} U_2 S_1, U_2 (S_1 S_1)^{\pm 1} U_2 S_1^{-1},$ and $c_1 c_2 =  S_1^{\pm 1} U_2 (S_1 S_1)^{\pm 1} U_2 S_1^{\pm 1}$.\\
Now suppose $c_1c_2 = U_iS_j^{\mp 1}S_i^{\pm 1}U_j$, then by Lemma \ref{lem:uisisj}, $c_1c_2$ is a linear combination of $U_i S_j^{\pm}U_j, U_iU_jU_j,$ and $U_iU_j$. By Equations \ref{eqn:observe1} through \ref{eqn:observe3}, $c_1c_2$ is $U_iU_j$ up to a factor of a Laurent polynomial in $\mathbb{Z}[a^{\pm1}, t^{\pm 1}]$. It follows that $c_1 c_2$ is a linear combination of elements of $\mathcal{C}_3$.\\
The same argument applies for $c_1 c_2= U_1S_2^{\pm 1}S_1^{\mp 1}U_2S_1$, $U_1S_2^{\pm 1}S_1^{\mp 1}U_2S_1^{- 1}, S_1U_2S_1^{\pm 1}S_2^{\mp 1}U_1$, and $S_1^{- 1}U_2S_1^{\pm 1}S_2^{\mp 1}U_1$.
\end{proof}

\begin{lemma}\label{lemma:forb34argument} 
$S_1^{-1}S_2S_1^{-1}S_2, S_2^{-1}S_1S_2^{-1}S_1,$ and $ S_2 S_1^{-1} S_2 S_1^{-1}$ can be reduced to a linear combination of basic $3$-tangles. In particular,

  \begin{eqnarray*}
S_1^{-1}S_2S_1^{-1}S_2&=& \frac{U_1 \left(\frac{a^3 b_1^2 b_{\infty}}{b_0 b_3}+\frac{a^3 b_2 b_{\infty}^2}{b_0 b_3}+\frac{a^3 b_3 b_{\infty}}{b_0}\right)}{b_0}+\frac{U_2 \left(\frac{a^3 b_2 b_{\infty}^2}{b_0 b_3}+\frac{a b_1^2 b_{\infty}}{b_0 b_3}+\frac{a b_3 b_{\infty}}{b_0}\right)}{b_0}+\frac{S_1U_2 \left(\frac{a^3 b_{\infty}^2}{b_0}+\frac{2 a b_1 b_2 b_{\infty}}{b_0 b_3}\right)}{b_0} \label{eqn:lemma4word}\\
      &&+\frac{U_1S_2^{-1} \left(\frac{a^3 b_1 b_{\infty}}{b_3}+\frac{a b_2 b_{\infty}}{b_0}\right)}{b_0}+\frac{U_1S_2 \left(\frac{a^5 b_{\infty}^2}{b_0}+\frac{a^3 b_1 b_2 b_{\infty}}{b_0 b_3}+\frac{a b_1 b_2 b_{\infty}}{b_0 b_3}\right)}{b_0}+\frac{S_1^{-1}U_2 \left(\frac{a b_2 b_{\infty}}{b_0}+\frac{a b_1 b_{\infty}}{b_3}\right)}{b_0}\\
      &&+\frac{U_1U_2 \left(\frac{a^4 b_1 b_{\infty}^2}{b_0 b_3}+\frac{a^4 b_2 b_{\infty}}{b_0}+\frac{a^2 b_1 b_{\infty}^2}{b_0 b_3}+\frac{a^2 b_2 b_{\infty}}{b_0}+\frac{2 a b_2^2 b_{\infty}}{b_0 b_3}\right)}{b_0}+\frac{b_1 b_2 S_2^{-1}S_1}{b_0 b_3} +\frac{b_1 b_2 S_2S_1^{-1}}{b_0 b_3} \\
      &&+\frac{b_1 S_1S_2^{-1}S_1}{b_0}+\frac{b_1 S_2S_1^{-1}S_2}{b_0}+\frac{2 b_2^2 S_1S_2S_1^{-1}}{b_0 b_3}+\frac{b_2 S_1^{-1}S_2S_1^{-1}}{b_3}+\frac{b_2 S_2^{-1}S_1S_2^{-1}}{b_3}+S_1S_2^{-1}S_1S_2^{-1},
      \end{eqnarray*}
      
\begin{eqnarray*}
    S_2^{-1}S_1S_2^{-1}S_1 &=& \frac{U_1 \left(\frac{a^3 b_1^2 b_{\infty}}{b_0 b_3}+\frac{2 a^3 b_2 b_{\infty}^2}{b_0 b_3}+\frac{a^3 b_3 b_{\infty}}{b_0}+\frac{a b_1^2 b_{\infty}}{b_0 b_3}+\frac{a b_3 b_{\infty}}{b_0}\right)}{b_0} +\frac{U_1S_2^{-1} \left(\frac{a^3 b_1 b_{\infty}}{b_3}+\frac{a b_2 b_{\infty}}{b_0}\right)}{b_0} +\frac{b_2 S_1S_2}{b_0}\\
          &&+\frac{U_2 \left(\frac{a^3 b_1^2 b_{\infty}}{b_0 b_3}+\frac{2 a^3 b_2 b_{\infty}^2}{b_0 b_3}+\frac{a^3 b_3 b_{\infty}}{b_0}+\frac{a b_1^2 b_{\infty}}{b_0 b_3}+\frac{a b_3 b_{\infty}}{b_0}\right)}{b_0} +\frac{S_1U_2 \left(\frac{a^3 b_{\infty}^2}{b_0}+\frac{2 a b_1 b_2 b_{\infty}}{b_0 b_3}\right)}{b_0}\\        &&+\frac{S_1^{-1}U_2 \left(\frac{a^3 b_0 b_{\infty}^2}{b_3^2}+\frac{a^3 b_1 b_{\infty}}{b_3}+\frac{a b_2 b_{\infty}}{b_0}+\frac{2 a b_1 b_{\infty}}{b_3}\right)}{b_0} \\
          &&
          +\frac{U_1S_2 \left(\frac{a^5 b_{\infty}^2}{b_0}+\frac{a^3 b_1 b_2 b_{\infty}}{b_0 b_3}+\frac{a^3 b_{\infty}^2}{b_0}+\frac{a b_1 b_2 b_{\infty}}{b_0 b_3}\right)}{b_0}\\
          &&+\frac{U_2S_1^{-1} \left(\frac{a^3 b_1 b_{\infty}}{b_3}+\frac{2 a b_2 b_{\infty}}{b_0}
          +\frac{a b_1 b_{\infty}}{b_3}\right)}{b_0}
          +\frac{a^2 b_2 b_{\infty} S_1^{-1}U_2S_1}{b_3^2} \\
          &&
          +\frac{U_2S_1 \left(\frac{a^5 b_{\infty}^2}{b_0}+\frac{a^3 b_1 b_2 b_{\infty}}{b_0 b_3}+\frac{a b_1 b_2 b_{\infty}}{b_0 b_3}\right)}{b_0}
        +\frac{S_2^{-1}U_1 \left(\frac{a b_2 b_{\infty}}{b_0}+\frac{a b_1 b_{\infty}}{b_3}\right)}{b_0}\\
          &&+\frac{U_1U_2 \left(\frac{a^4 b_1 b_{\infty}^2}{b_0 b_3}+\frac{a^4 b_2 b_{\infty}}{b_0}+\frac{a^2 b_1 b_{\infty}^2}{b_0 b_3}+\frac{a^2 b_2 b_{\infty}}{b_0}+\frac{2 a b_2^2 b_{\infty}}{b_0 b_3}\right)}{b_0}
          \\
          &&+\frac{U_2U_1 \left(\frac{a^4 b_1 b_{\infty}^2}{b_0 b_3}+\frac{a^4 b_2 b_{\infty}}{b_0}+\frac{a^2 b_1 b_{\infty}^2}{b_0 b_3}+\frac{a^2 b_2 b_{\infty}}{b_0}+\frac{2 a b_2^2 b_{\infty}}{b_0 b_3}\right)}{b_0}+\frac{S_1^{-1}U_2S_1^{-1} \left(\frac{a^3 b_{\infty}^2}{b_3}+\frac{a^2 b_0^2 b_{\infty}}{b_3^2}+\frac{a b_1 b_{\infty}}{b_0}\right)}{b_0}\\
          &&+\frac{S_2U_1 \left(\frac{a^3 b_{\infty}^2}{b_0}+\frac{a^2 b_0 b_1 b_{\infty}}{b_3^2}+\frac{a^2 b_2 b_{\infty}}{b_3}+\frac{2 a b_1 b_2 b_{\infty}}{b_0 b_3}\right)}{b_0}+\frac{a b_3 b_{\infty} S_1U_2S_1^{-1}}{b_0^2}+\frac{2 b_1^2 S_1^{-1}}{b_0 b_3}+\frac{2 b_1 S_2S_1^{-1}S_2}{b_0}\\
          &&+\frac{2 b_1 b_2 S_1^{-1}S_2}{b_0 b_3}+\frac{b_1 b_2 S_1S_2^{-1}}{b_0 b_3}+\frac{b_1 b_2 S_2^{-1}S_1}{b_0 b_3}+\frac{2 b_1 b_2 S_2S_1^{-1}}{b_0 b_3}+\frac{2 b_1 S_1S_2^{-1}S_1}{b_0}+\frac{b_2 S_2S_1}{b_0}+\frac{b_1 S_1^{-1}S_2^{-1}}{b_3}\\
          &&+\frac{2 b_2^2 S_1^{-1}S_2S_1}{b_0 b_3}+\frac{2 b_2^2 S_1S_2S_1^{-1}}{b_0 b_3}+\frac{b_1 S_2^{-1}S_1^{-1}}{b_3}
          +\frac{2 b_2 S_1^{-1}S_2S_1^{-1}}{b_3}+\frac{2 b_2 S_2^{-1}S_1S_2^{-1}}{b_3}+S_1S_2^{-1}S_1S_2^{-1},
          \end{eqnarray*}

and 

          \begin{eqnarray*}
        S_2 S_1^{-1} S_2 S_1^{-1} &=& \frac{U_1 \left(\frac{a^3 b_1^2 b_{\infty}}{b_0 b_3}+\frac{a^3 b_2 b_{\infty}^2}{b_0 b_3}+\frac{a^3 b_3 b_{\infty}}{b_0}\right)}{b_0}+\frac{U_2 \left(\frac{a^3 b_2 b_{\infty}^2}{b_0 b_3}+\frac{a b_1^2 b_{\infty}}{b_0 b_3}+\frac{a b_3 b_{\infty}}{b_0}\right)}{b_0}+\frac{S_1U_2 \left(\frac{a^3 b_{\infty}^2}{b_0}+\frac{2 a b_1 b_2 b_{\infty}}{b_0 b_3}\right)}{b_0} \\
        &&+\frac{S_1^{-1}U_2 \left(\frac{a^3 b_0 b_{\infty}^2}{b_3^2}+\frac{a^3 b_1 b_{\infty}}{b_3}+\frac{a b_2 b_{\infty}}{b_0}+\frac{2 a b_1 b_{\infty}}{b_3}\right)}{b_0}+\frac{U_1S_2^{-1} \left(\frac{a^3 b_1 b_{\infty}}{b_3}+\frac{a b_2 b_{\infty}}{b_0}\right)}{b_0}+\frac{b_2 S_1S_2}{b_0}\\
        &&+\frac{S_2U_1 \left(\frac{a^2 b_0 b_1 b_{\infty}}{b_3^2}+\frac{a^2 b_2 b_{\infty}}{b_3}\right)}{b_0}+\frac{a^2 b_2 b_{\infty} S_1^{-1}U_2S_1}{b_3^2} +\frac{a b_3 b_{\infty} S_1U_2S_1^{-1}}{b_0^2}+\frac{2 b_1^2 S_1^{-1}}{b_0 b_3}+\frac{b_1 b_2 S_1^{-1}S_2}{b_0 b_3}\\
        && +\frac{b_1 b_2 S_2^{-1}S_1}{b_0 b_3}+\frac{2 b_1 b_2 S_2S_1^{-1}}{b_0 b_3} +\frac{U_1U_2 \left(\frac{a^4 b_1 b_{\infty}^2}{b_0 b_3}+\frac{a^4 b_2 b_{\infty}}{b_0}+\frac{a^2 b_1 b_{\infty}^2}{b_0 b_3}+\frac{a^2 b_2 b_{\infty}}{b_0}+\frac{2 a b_2^2 b_{\infty}}{b_0 b_3}\right)}{b_0} \\
        &&+\frac{U_1S_2 \left(\frac{a^5 b_{\infty}^2}{b_0}+\frac{a^3 b_1 b_2 b_{\infty}}{b_0 b_3}+\frac{a^3 b_{\infty}^2}{b_0}+\frac{a b_1 b_2 b_{\infty}}{b_0 b_3}\right)}{b_0}
        +\frac{S_1^{-1}U_2S_1^{-1} \left(\frac{a^3 b_{\infty}^2}{b_3}+\frac{a^2 b_0^2 b_{\infty}}{b_3^2}+\frac{a b_1 b_{\infty}}{b_0}\right)}{b_0}\\
        &&
        +\frac{U_2S_1^{-1} \left(\frac{a b_2 b_{\infty}}{b_0}+\frac{a b_1 b_{\infty}}{b_3}\right)}{b_0}
        +\frac{b_1 S_1S_2^{-1}S_1}{b_0}+\frac{b_1 S_2S_1^{-1}S_2}{b_0}
        +\frac{2 b_2^2 S_1S_2S_1^{-1}}{b_0 b_3}+\frac{b_2 S_2S_1}{b_0}+\frac{b_1 S_1^{-1}S_2^{-1}}{b_3}\\
        &&+\frac{b_1 S_2^{-1}S_1^{-1}}{b_3}+\frac{b_2 S_1^{-1}S_2S_1^{-1}}{b_3}+\frac{b_2 S_2^{-1}S_1S_2^{-1}}{b_3}+S_1S_2^{-1}S_1S_2^{-1}.\label{eqn:lemma4word2}
    \end{eqnarray*}
\end{lemma}

\begin{proof} First note by isotopy we may introduce two 2-twist regions to both $ S_i^{-1}S_jS_i^{-1}S_j$ and $ S_2 S_1^{-1} S_2 S_1^{-1}$.
        \begin{eqnarray}
         S_i^{-1} S_j S_i^{-1} S_j &\stackrel{Eqn.~\ref{eqn:observe7}}{=}& S_i^{-1}S_i^{-1}S_i  S_j S_i^{-1} S_j  \stackrel{Eqn.~\ref{eqn:braidrelation1}}{=} S_i^{-2} S_j^{-1} S_i S_j^2, \label{eqn:lemiso1}\\
             S_2 S_1^{-1} S_2 S_1^{-1} &\stackrel{Eqn.~\ref{eqn:observe7}}{=}&  S_2 S_1^{-1} S_1^{-1}S_1S_2 S_1^{-1} \stackrel{Eqn.~\ref{eqn:braidrelation1}}{=} S_2 S_1^{-2} S_2^{-1}S_1S_2. \label{eqn:lemiso2}
        \end{eqnarray}
    Consider Equation \ref{eqn:lemiso1},
    \begin{eqnarray*}
        S_i^{-1} S_j S_i^{-1} S_j &=&   S_i^{-2} S_j^{-1} S_i S_j^2  \\
&\stackrel{Eqn.~\ref{eqn:lemma2twist4}}{=}& b_0^{-1}(b_1 (S_i^{-1}S_j^{-1}) S_i S_j^2  + b_2 S_j^{-1} S_i S_j^2  +b_3 S_iS_j^{-1} S_i S_j^2  +a^2 b_{\infty}U_iS_j^{-1} (S_i S_j) S_j )\\
&\stackrel{Eqn.~\ref{eqn:braidrelation1}~\&~\ref{eqn:observe7}}{=}& b_0^{-1}(b_1 S_jS_i^{-1}S_j  + b_2 S_j^{-1} S_i S_j^2  +b_3 S_iS_j^{-1} S_i S_j^2  +a^2 b_{\infty}U_i S_i S_j S_i^{-1}S_j ) \\
&\stackrel{Eqn.~\ref{eqn:observe3}}{=}& b_0^{-1}(b_1 S_jS_i^{-1}S_j  + b_2 S_j^{-1} S_i S_j^2  +b_3 S_iS_j^{-1} S_i S_j^2  +a^3 b_{\infty}U_i S_j S_i^{-1}S_j ).
    \end{eqnarray*}
Similarly, 
$$S_2 S_1^{-1} S_2 S_1^{-1} = b_0^{-1}(b_3 S_1^{-1}S_2S_1^2S_2 +b_2 S_1S_2 + b_1 S_2^2S_1^{-1} +a^2 b_{\infty} S_1^{-1}U_2S_1^{2}S_2).$$
For $S_j^{-1} S_i S_j^2$, we have
\begin{eqnarray}
    S_j^{-1} S_i S_j^2 &\stackrel{Eqn.~\ref{eqn:lemma2negtwist4}}{=} & b_{3}^{-1} (b_0   S_j^{-1} S_i S_j^{-1} +b_1  S_j^{-1} S_i +b_2  S_j^{-1} S_i S_j + a b_{\infty}   S_j^{-1} S_i U_j) \nonumber \\
&\stackrel{Eqn.~\ref{eqn:lemma2twist2}}{=} &  b_{3}^{-1} (b_0   S_j^{-1} S_i S_j^{-1} +b_1  S_j^{-1} S_i +b_2  S_j^{-1} S_i S_j) \nonumber\\
&&+ ab_{3}^{-1} b_{\infty} b_0^{-1}(b_1  S_iU_j + b_2 U_i U_j + b_3  S_i^{-1} U_j+ a^2 b_{\infty} U_j).\label{eqn:lemmap1}
\end{eqnarray}

For $S_iS_j^{-1} S_i S_j^2$,
\begin{eqnarray*}
    S_iS_j^{-1} S_i S_j^2 &\stackrel{Eqn.~\ref{eqn:lemmap1}}{=} &  b_{3}^{-1} (b_0   S_iS_j^{-1} S_i S_j^{-1} +b_1 S_i S_j^{-1} S_i +b_2  S_iS_j^{-1} (S_i S_j)) \nonumber\\
&&+ ab_{3}^{-1} b_{\infty} b_0^{-1}(b_1  S_i^2U_j + b_2 (S_iU_i) U_j + b_3  S_iS_i^{-1} U_j+ a^2 b_{\infty} S_iU_j) \\
&\stackrel{Eqn.~\ref{eqn:observe3} ~\& ~\ref{eqn:braidrelation1} }{=} &  b_{3}^{-1} (b_0   S_iS_j^{-1} S_i S_j^{-1} +b_1 S_i S_j^{-1} S_i +b_2  S_i^2 S_jS_i^{-1}) \nonumber\\
&&+ ab_{3}^{-1} b_{\infty} b_0^{-1}(b_1  S_i^2U_j + a b_2 U_i U_j + b_3  (S_iS_i^{-1}) U_j+ a^2 b_{\infty} S_iU_j) \\
&\stackrel{Eqn.~\ref{eqn:observe7} }{=} & b_{3}^{-1} (b_0   S_iS_j^{-1} S_i S_j^{-1} +b_1 S_i S_j^{-1} S_i)+ ab_{3}^{-1} b_{\infty} b_0^{-1}(a b_2 U_i U_j + b_3  U_j+ a^2 b_{\infty} S_iU_j) \nonumber\\
&&+b_{3}^{-1}b_2  S_i^2 S_jS_i^{-1}+ ab_{3}^{-1} b_{\infty} b_0^{-1}b_1  S_i^2U_j \\
&\stackrel{Eqn.~\ref{eqn:lemma2negtwist4}}{=} &   b_{3}^{-1} (b_0   S_iS_j^{-1} S_i S_j^{-1} +b_1 S_i S_j^{-1} S_i)+ ab_{3}^{-1} b_{\infty} b_0^{-1}(a b_2 U_i U_j + b_3  U_j+ a^2 b_{\infty} S_iU_j) \nonumber\\
&&+b_{3}^{-2}b_2   (b_0  S_i^{-1}S_jS_i^{-1} +b_1 S_jS_i^{-1}+b_2 S_iS_jS_i^{-1}+ a b_{\infty}  U_iS_jS_i^{-1}) \\
&& + ab_{3}^{-2} b_{\infty} b_0^{-1}b_1   (b_0  S_i^{-1}U_j +b_1 U_j+b_2 S_iU_j+ a b_{\infty}  U_iU_j).
\end{eqnarray*}

Therefore, by Equation~\ref{eqn:lemma2twist},
\begin{eqnarray}
  S_iS_j^{-1} S_i S_j^2   
&=& b_{3}^{-1} (b_0   S_iS_j^{-1} S_i S_j^{-1} +b_1 S_i S_j^{-1} S_i)+ ab_{3}^{-1} b_{\infty} b_0^{-1}(a b_2 U_i U_j + b_3  U_j+ a^2 b_{\infty} S_iU_j) \nonumber\\
&& + ab_{3}^{-2} b_{\infty} b_0^{-1}b_1   (b_0  S_i^{-1}U_j +b_1 U_j+b_2 S_iU_j+ a b_{\infty}  U_iU_j)\nonumber\\
&&+b_{3}^{-2}b_2   (b_0  S_i^{-1}S_jS_i^{-1} +b_1 S_jS_i^{-1}+b_2 S_iS_jS_i^{-1}) \nonumber\\
&& +ab_{3}^{-2}b_2b_{\infty} b_0^{-1}(b_1 U_i S_j + b_2 U_i U_j + b_3 U_i S_j^{-1} + a^2 b_{\infty} U_i).
\end{eqnarray}
For $S_j^{-1} S_i S_j^2S_i$

\begin{eqnarray*}
    S_j^{-1} S_i S_j^2S_i &\stackrel{Eqn.~\ref{eqn:lemmap1}}{=} &  b_{3}^{-1} (b_0   S_j^{-1} S_i S_j^{-1}S_i +b_1  S_j^{-1} S_i^2 +b_2  S_j^{-1} S_i S_jS_i) \\
    &\stackrel{Eqn.~\ref{eqn:braidrelation}~\&~ \ref{eqn:observe7}}{=} & b_{3}^{-1} (b_0   S_j^{-1} S_i S_j^{-1}S_i +b_1  S_j^{-1} S_i^2 +b_2 S_iS_j).
    \end{eqnarray*}

For $U_i S_j S_i^{-1}S_j$,
\begin{eqnarray*}
    U_i S_j S_i^{-1}S_j &\stackrel{Eqn.~\ref{eqn:lemma2twist}}{=} & b_0^{-1}(b_1 U_i S_j^2+ b_2 U_i U_j S_j+ b_3 U_i S_j^{-1} S_j+ a^2 b_{\infty} U_iS_j)\\
    &\stackrel{Eqn.~\ref{eqn:observe3} ~\& ~\ref{eqn:observe7} }{=} &  b_0^{-1}(b_1 U_i S_j^2+ a b_2 U_i U_j+ b_3 U_i + a^2 b_{\infty} U_iS_j)\\
&\stackrel{Eqn.~\ref{eqn:lemma2negtwist4}}{=} & b_0^{-1}( a b_2 U_i U_j+ b_3 U_i + a^2 b_{\infty} U_iS_j)\\
&&+b_0^{-1}b_1 b_{3}^{-1} (b_0  U_iS_j^{-1} +b_1 U_i+b_2 U_iS_j + a b_{\infty}  U_iU_j).
\end{eqnarray*}

For $S_1^{-1}U_2S_1^2S_2$,

\begin{eqnarray}
    S_1^{-1}U_2S_1^2S_2 &\stackrel{Eqn.~\ref{eqn:lemma2negtwist4}}{=} & b_3^{-1}(b_0 S_1^{-1}U_2S_1^{-1}S_2 +b_1 S_1^{-1}U_2S_2 +b_2 S_1^{-1}U_2S_1S_2 +a b_{\infty} S_1^{-1}U_2U_1S_2) 
    \nonumber\\
&\stackrel{Eqn.~\ref{eqn:observe8}, ~\ref{eqn:observe6}~\&~\ref{eqn:observe3}}{=} &
 b_3^{-1}(b_0 S_1^{-1}U_2S_1^{-1}S_2 +a b_1 S_1^{-1}U_2 +b_2 S_1^{-1}U_2U_1 +a b_{\infty} S_1^{-1}U_2S_1^{-1}) \nonumber\\
&\stackrel{Eqn.~\ref{eqn:lemma2negtwist}~\&~\ref{eqn:observe8}}{=} &
 b_3^{-1}(a b_1 S_1^{-1}U_2 +b_2 S_2U_1 +a b_{\infty} S_1^{-1}U_2S_1^{-1}) \nonumber\\
 &&+b_3^{-2}b_0(b_0 S_1^{-1}U_2S_1^{-1}+b_1 S_1^{-1}U_2U_1 + b_0 S_1^{-1}U_2S_1 +a b_{\infty}S_1^{-1}U_2) \nonumber\\
  &\stackrel{Eqn.~\ref{eqn:observe8}}{=} &
 b_3^{-1}(a b_1 S_1^{-1}U_2 +b_2 S_2U_1 +a b_{\infty} S_1^{-1}U_2S_1^{-1}) \nonumber\\
 &&+b_3^{-2}b_0(b_0 S_1^{-1}U_2S_1^{-1}+b_1 S_2U_1 + b_0 S_1^{-1}U_2S_1 +a b_{\infty}S_1^{-1}U_2).
\end{eqnarray}
\end{proof}

\begin{theorem}
Every $3$-algebraic tangle can be reduced to a linear combination of basic $3$-tangles.
\end{theorem}

\begin{proof}
    By Lemma \ref{lemma:c3}, we only need to check that for any  $b_1, b_2 \in \mathcal{B}_3$ and any $c \in \mathcal{C}_3$, $b_1 b_2$ and $b_1 c$ can be reduced to a linear combination of basic $3$-tangles in $\mathcal{S}_{4, \infty}(\mathcal{C}_3 \cup \mathcal{B}_3)$. To do so, we will consider partitions of $\mathcal{B}_3$ and $\mathcal{C}_3$. \\

    Let $\mathcal{B}_{3,1} = \{ S_i^{\epsilon} \}_{i=1,2, \epsilon=\pm 1}$,  $\mathcal{B}_{3,2} = \{ S_i^{\epsilon_1} S_j^{\epsilon_2} \}_{i,j=1,2, i\neq j, \epsilon_1, \epsilon_2 = \pm 1}$,  $\mathcal{B}_{3,3} = \{ S_1^{\epsilon_1} S_2^{\epsilon_2} S_1^{\epsilon_3}, S_2 S_1^{-1}S_2, S_2^{-1} S_1 S_2^{-1}\}_{\epsilon_1, \epsilon_2, \epsilon_3 = \pm 1}$, and  $\mathcal{B}_{3,4} = \{ (S_1 S_2^{-1})^2 \}$. Notice that the sets are pairwise disjoint and $\mathcal{B}_3 = \{\text{e}\} \cup (\cup_{i=1}^4 \mathcal{B}_{3,i}$).
    Also, let $\mathcal{C}_{3,1} = \{ U_1, U_2 \}$, $\mathcal{C}_{3,2} = \{ U_1 U_2, U_2 U_1 \}$, $\mathcal{C}_{3,3} = \{ S_i^{\epsilon} U_j, U_i S_j^{\epsilon} \}_{i, j = 1,2, i\neq j, \epsilon=\pm 1}$, $\mathcal{C}_{3,4} = \{ S_1^{\epsilon_1} U_2 S_1^{\epsilon_2}\}_{\epsilon_1, \epsilon_2 = \pm 1}$. Similarly, $\mathcal{C}_3 = \cup_{i=1}^4 \mathcal{C}_{3,i}$. 

We have 27 cases to consider. 

\begin{enumerate}
    \item[Case 1] $x \text{e}= \text{e} x =x$ for all $x \in \mathcal{C}_3 \cup \mathcal{B}_3$.
    \item[Case 2] Let $b_1, b_2 \in \mathcal{B}_{3,1}$. Then either $b_1b_2 \in \mathcal{B}_{3,2}$ or it is of the form $S^{\pm 2}_{i}$ for $i=1,2$. For the later case, Equations \ref{eqn:lemma2twist4} and \ref{eqn:lemma2negtwist4} imply that $b_1 b_2$ are linear combinations of elements in $\mathcal{B}_{3,1}$, $\{\text{e}\}$ and $\mathcal{C}_{3,1}$.
    \item[Case 3] Let $b_1 \in \mathcal{B}_{3,1}$ and $b_2 \in \mathcal{B}_{3,2}$, then either:
    \begin{enumerate}
        \item  $b_1b_2 \in \mathcal{B}_{3,3}$ by applying Equations \ref{eqn:braidrelation} and \ref{eqn:braidrelation1} when necessary.
        \item $b_2b_1\in \mathcal{B}_{3,3}$ by applying Equations \ref{eqn:braidrelation} and \ref{eqn:braidrelation1} when necessary.
        \item $b_1b_2 =S_i^{\epsilon_1} S_i^{-\epsilon_1} S_j^{\epsilon_2}= S_j^{\epsilon_2} \in \mathcal{B}_{3,1}$, by Equation \ref{eqn:observe7}, where $\epsilon_1, \epsilon_2 = \pm 1$ and $i,j=1,2, i\neq j$. 
        \item $b_2b_1 = S_j^{\epsilon_2} S_i^{\epsilon_1} S_i^{-\epsilon_1} = S_j^{\epsilon_2}\in \mathcal{B}_{3,1}$, by Equation \ref{eqn:observe7}, where $\epsilon_1, \epsilon_2 = \pm 1$ and $i,j=1,2, i\neq j$. 
        \item $b_1 b_2 = S_i^{\epsilon_1} S_i^{\epsilon_1} S_j^{\epsilon_2}$, where $\epsilon_1, \epsilon_2 = \pm 1$ and $i,j=1,2, i\neq j$. By Equations \ref{eqn:lemma2twist4} and \ref{eqn:lemma2negtwist4}, $b_1 b_2$ are linear combinations of elements in $\mathcal{B}_{3,1}, \mathcal{B}_{3,2},$ and $\mathcal{C}_{3,3}$. 
        \item $b_2 b_1 = S_j^{\epsilon_2} S_i^{\epsilon_1} S_i^{\epsilon_1}$, where $\epsilon_1, \epsilon_2 = \pm 1$ and $i,j=1,2, i\neq j$. By Equations \ref{eqn:lemma2twist4} and \ref{eqn:lemma2negtwist4}, $b_2 b_1$ are linear combinations of elements in $\mathcal{B}_{3,1}, \mathcal{B}_{3,2},$ and $\mathcal{C}_{3,3}$. 
    \end{enumerate}
   \item[Case 4]  Let $b \in \mathcal{B}_{3,1}$ and $c \in \mathcal{C}_{3,1}$, then either
   \begin{enumerate}
       \item $cb \in \mathcal{C}_{3,3}$.
       \item $bc \in \mathcal{C}_{3,3}$.
       \item $cb = U_iS_i^{\pm 1}=a^{\pm 1} U_i$ by Equation \ref{eqn:observe3}.
       \item $bc = S_i^{\pm 1}U_i= a^{\pm 1} U_i$ by Equation \ref{eqn:observe3}.
   \end{enumerate}
   \item[Case 5] Let $b \in \mathcal{B}_{3,1}$ and $c \in \mathcal{C}_{3,2}$, then either
   \begin{enumerate}
       \item $bc = S^{\pm 1}_iU_iU_j= a^{\pm} U_j$ by Equation \ref{eqn:observe3}.
        \item $cb = U_iU_jS^{\pm 1}_j = a^{\pm} U_i$ by Equation \ref{eqn:observe3}.
       \item $bc = S^{\pm 1}_i U_jU_i= S^{\mp 1}_jU_i \in \mathcal{C}_{3,3}$ by Equation \ref{eqn:observe8}.
       \item $cb =  U_jU_iS^{\pm 1}_j= U_j S^{\mp 1}_i \in \mathcal{C}_{3,3}$ by Equation \ref{eqn:observe8}.
   \end{enumerate}
    \item[Case 6]  Let $b \in \mathcal{B}_{3,2}$ and $c \in \mathcal{C}_{3,1}$, then $bc = b_1 b_2 c$ and $cb = c b_1 b_2$ where $b_1, b_2 \in \mathcal{B}_{3,1}$. By Case 4 on $b_2c$ and $cb_1$, then either
   \begin{enumerate}
        \item $bc= a^{\pm 1}b_1c$. By Case 4 and since $b_1b_2\in \mathcal{B}_{3,2}$, $b_1c \in  \mathcal{C}_{3,3}$.
        \item $cb = a^{\pm1}cb_2$. By Case 4 and since $b_1b_2\in \mathcal{B}_{3,2}$, $cb_2 \in \mathcal{C}_{3,3}$.
       \item $bc = S_i^{\epsilon} S_j^{\epsilon} U_i=U_jU_i \in \mathcal{C}_{3,2}$ by Equation \ref{eqn:observe6}.
         \item $bc = S_i^{\epsilon} S_j^{-\epsilon} U_i$.  By Equations \ref{eqn:lemma2negtwist} and \ref{eqn:lemma2twist}, $bc$ is a linear combination of elements in $\mathcal{C}_{3,1}\cup \mathcal{C}_{3,2}\cup \mathcal{C}_{3,3}$.
         \end{enumerate}
   \item[Case 7] Let $b \in \mathcal{B}_{3,1}$ and $c \in \mathcal{C}_{3,3}$, then either:
   \begin{enumerate}
            \item $bc=b_1 b_2 c_1$ where $b_1, b_2 \in \mathcal{B}_{3,1}$ and $c_1 \in \mathcal{C}_{3,1}$. By Case 2, 
            \begin{enumerate}
            \item $bc=b_3c_1$ where $b_3\in \mathcal{B}_{3,2}$. $bc \in \mathcal{S}_{4,\infty}(\mathcal{C}_3)$ by Case 6.
            \item $bc$ is a linear combination of elements of the form $b_1 c$, $c'c$, $c$ where $b_1 \in \mathcal{B}_{3,1}$ and $c' \in \mathcal{C}_{3,1}$. By Lemma \ref{lemma:c3} and Case 4, $bc \in \mathcal{S}_{4,\infty}(\mathcal{C}_3)$.
            \end{enumerate}
            \item $cb= c_1 b_1 b_2 $ where $b_1, b_2 \in \mathcal{B}_{3,1}$ and $c_1 \in \mathcal{C}_{3,1}$. By Case 2,
            \begin{enumerate}
            \item $cb=c_1b_3$ where $b_3\in \mathcal{B}_{3,2}$. $cb \in \mathcal{S}_{4,\infty}(\mathcal{C}_3)$ by Case 6.
            \item $cb$ is a linear combination of elements of the form $cb_1 $, $cc'$, $c$ where $b_1 \in \mathcal{B}_{3,1}$ and $c' \in \mathcal{C}_{3,1}$. By Lemma \ref{lemma:c3} and Case 4, $cb \in \mathcal{S}_{4,\infty}(\mathcal{C}_3)$.
            \end{enumerate}
        \item $bc$ or $cb = b_1 c_1 b_2$, where $b_1, b_2 \in \mathcal{B}_{3,1}$ and $c_1 \in \mathcal{C}_{3,1}$. By Case 4 on $c_1 b_2$ for $cb$ and $b_1c_1$ for $bc$, 
        \begin{enumerate}
            \item $bc=a^{\pm 1} c_1b_2 = a^{\pm 1} c$.
            \item $cb= a^{\pm 1} b_1c_1= a^{\pm 1} c$.
             \item $bc$ or $cb=S_i^{\epsilon_1}U_jS_i^{\epsilon_2}$. By Equation \ref{eqn:observe9}, $bc$ or $cb \in \mathcal{C}_{3,4}$.
        \end{enumerate}
   \end{enumerate}
    \item[Case 8]  Let $b \in \mathcal{B}_{3,3}$ and $c \in \mathcal{C}_{3,1}$, then $bc = b_1b_2c$ where $b_1 \in \mathcal{B}_{3,1}$ and $b_2 \in \mathcal{B}_{3,2}$. By Case 6, $b_2c \in \mathcal{S}_{4, \infty}(\mathcal{C}_{3,1} \cup \mathcal{C}_{3,2} \cup \mathcal{C}_{3,3})$. Therefore, by Cases 4, 5, 7, we have $bc \in \mathcal{S}_{4, \infty}(\mathcal{C}_{3})$.
    \item[Case 9] Let $b \in \mathcal{B}_{3,2}$ and $c \in \mathcal{C}_{3,3}$, then 
    \begin{enumerate}
        \item $bc = b b_1 c_1$, where $b_1 \in \mathcal{B}_{3,1}$ and $c_1 \in \mathcal{C}_{3,1}$. By Case 3,
        \begin{enumerate}
            \item $bc = b_2 c_1$, where $b_2 \in \mathcal{B}_{3,3}$. By Case 8, $bc \in \mathcal{S}_{4,\infty}(\mathcal{C}_3 ).$
            \item $bc = b_3 c_1$, where $b_3 \in \mathcal{B}_{3,1}$. By Case 4, $bc \in  \mathcal{S}_{4,\infty}(\mathcal{C}_3)$.
            \item $bc$ is a linear combination of elements of the form $b' c_1$, $b'' c_1$, and $c'c_1$ where $b' \in \mathcal{B}_{3,1}, b'' \in \mathcal{B}_{3,2}$, and $c'   \in \mathcal {C}_{3,3}$. Case 4 and 6 and Lemma \ref{lemma:c3}, $bc \in  \mathcal{S}_{4,\infty}(\mathcal{C}_3 )$.
        \end{enumerate}
        \item $bc = b c_1 b_1 $, where $b_1 \in \mathcal{B}_{3,1}$ and $c_1 \in \mathcal{C}_{3,1}$. By Case 6, 
        \begin{enumerate}
            \item $bc = a^{\pm 1} b' c $ where $b'\in \mathcal{B}_{3,1}$. By Case 7, $bc \in \mathcal{S}_{4,\infty}(\mathcal{C}_3 )$.
            \item $bc = c' b_1$ where $c' \in \mathcal{C}_{3,2}$. $bc \in  \mathcal{C}_{3,1}$ by Case 5.
            \item $bc$ is a linear combination of elements of the form $c' b_1$, $c'' b_1$, $c'''b_1$, where $c' \in \mathcal{C}_{3,1}, c'' \in \mathcal{C}_{3,2}$, and $c''' \in \mathcal{C}_{3,3}$. 
            By Case 4, 5, and 7, $bc \in \mathcal{S}_{4,\infty}(\mathcal{C}_3 )$
         \end{enumerate}
    \end{enumerate}
    Similarly, $cb \in  \mathcal{S}_{4,\infty}(\mathcal{C}_3)$.
   \item[Case 10] Let $b \in \mathcal{B}_{3,1}$ and $c \in \mathcal{C}_{3,4}$, then
   \begin{enumerate}
   \item $bc = b_1b_2c_1$, where $b_1, b_2 \in \mathcal{B}_{3,1}$ and $c_1 \in \mathcal{C}_{3,3}$. By Case 2,
   \begin{enumerate}
        \item $bc$ is a linear combination of elements of the form $b'c$, $c$, and $c_1c$ where $b' \in \mathcal{B}_{3,1}$ and $c_1 \in \mathcal{C}_{3,1}$. $bc \in \mathcal{S}_{4,\infty}(\mathcal{C}_3)$ by Lemma \ref{lemma:c3} and Case 7.
        \item $bc=c_1\in \mathcal{C}_{3,3}$
        \item $bc = b_3 c_1$ where $b_3 \in \mathcal{B}_{3,2}$. By Case 8, $bc \in  \mathcal{S}_{4,\infty}(\mathcal{C}_3)$.
   \end{enumerate}
    \item $cb = c_1b_1b_2$, where $b_1, b_2 \in \mathcal{B}_{3,1}$ and $c_1 \in \mathcal{C}_{3,3}$. Similarly, by Case 2 and 7 and Lemma \ref{lemma:c3}, $cb \in \mathcal{S}_{4,\infty}(\mathcal{C}_3)$.
    \end{enumerate}
   \item[Case 11] Let $b \in \mathcal{B}_{3,2}$ and $c \in \mathcal{C}_{3,2}$, then $bc = b_1 b_2 c$ where $b_1, b_2 \in \mathcal{B}_{3,1}$. By Case 5, $b_2 c \in \mathcal{S}_{4, \infty}(\mathcal{C}_3)$. By cases 4, 5, 7, and 10, we have $bc \in \mathcal{S}_{4, \infty}(\mathcal{C}_3)$. By Cases 4, 5, 7, and 10, we have $bc \in \mathcal{S}_{4,\infty}(\mathcal{C}_3 )$.

   \item[Case 12] Let $b \in \mathcal{B}_{3,2}$ and $c \in \mathcal{C}_{3,4}$, then $bc = b_1 b_2 c$ where $b_1, b_2 \in \mathcal{B}_{3,1}$. By Case 10, $b_2 c \in \mathcal{S}_{4, \infty}(\mathcal{C}_3)$. By cases 4, 5, 7, and 10, we have $bc \in \mathcal{S}_{4, \infty}(\mathcal{C}_3)$. By cases 4, 5, 7, and 10, we have $bc \in \mathcal{S}_{4, \infty}(\mathcal{C}_3)$. By Cases 4, 5, 7, and 10, we have $bc \in \mathcal{S}_{4,\infty}(\mathcal{C}_3 )$.
   
    \item[Case 13] Let $b \in \mathcal{B}_{3,3}$ and $c \in \mathcal{C}_{3,2}$, then $bc = b_1 b_2 c$ where $b_1 \in \mathcal{B}_{3,1}$ and $b_2 \in \mathcal{B}_{3,2}$. By Case 11, $b_2c \in \mathcal{S}_{4, \infty}(\mathcal{C}_3)$. By cases 4, 5, 7, and 10, we have $bc \in \mathcal{S}_{4, \infty}(\mathcal{C}_3)$. By Cases 4, 5, 7, and 10, we have $bc \in \mathcal{S}_{4,\infty}(\mathcal{C}_3 )$.
    
   \item[Case 14] Let $b \in \mathcal{B}_{3,3}$ and $c \in \mathcal{C}_{3,3}$, then $bc = b_1 b_2 c$ where $b_1 \in \mathcal{B}_{3,1}$ and $b_2 \in \mathcal{B}_{3,2}$. By Case 9, $b_2c \in \mathcal{S}_{4, \infty}(\mathcal{C}_3)$. By cases 4, 5, 7, and 10, we have $bc \in \mathcal{S}_{4, \infty}(\mathcal{C}_3)$. By Cases 4, 5, 7, and 10, we have $bc \in \mathcal{S}_{4,\infty}(\mathcal{C}_3 )$.

    \item[Case 15] Let $b \in \mathcal{B}_{3,3}$ and $c \in \mathcal{C}_{3,4}$, then $bc = b_1 b_2 c$ where $b_1 \in \mathcal{B}_{3,1}$ and $b_2 \in \mathcal{B}_{3,2}$. By Case 12, $b_2c \in \mathcal{S}_{4, \infty}(\mathcal{C}_3)$.  By cases 4, 5, 7, and 10, we have $bc \in \mathcal{S}_{4, \infty}(\mathcal{C}_3)$. By Cases 4, 5, 7, and 10, we have $bc \in \mathcal{S}_{4,\infty}(\mathcal{C}_3 )$.

   \item[Case 16]  Let $b \in \mathcal{B}_{3,4}$ and $c \in \mathcal{C}_{3,1}$, then $bc = b_1 b_2 c$ where $b_1, b_2 \in \mathcal{B}_{3,2}$. By Case 6, $b_2c \in \mathcal{S}_{4, \infty}(\mathcal{C}_3)$. By Cases 6, 9, 11, and 12, we have $bc \in \mathcal{S}_{4,\infty}(\mathcal{C}_3 )$.
   
   \item[Case 17] Let $b \in \mathcal{B}_{3,4}$ and $c \in \mathcal{C}_{3,2}$, then $bc = b_1 b_2 c$ where $b_1, b_2 \in \mathcal{B}_{3,2}$. By Case 11, $b_2c \in \mathcal{S}_{4, \infty}(\mathcal{C}_3)$.  By Cases 6, 9, 11, and 12, we have $bc \in \mathcal{S}_{4,\infty}(\mathcal{C}_3 )$.
   
 \item[Case 18] Let $b \in \mathcal{B}_{3,4}$ and $c \in \mathcal{C}_{3,3}$, then $bc = b_1 b_2 c$ where $b_1, b_2 \in \mathcal{B}_{3,2}$. By Case 9, $b_2c \in \mathcal{S}_{4, \infty}(\mathcal{C}_3)$. By Cases 6, 9, 11, and 12, we have $bc \in \mathcal{S}_{4,\infty}(\mathcal{C}_3 )$.
 
   \item[Case 19] Let $b \in \mathcal{B}_{3,4}$ and $c \in \mathcal{C}_{3,4}$, then $bc = b_1 b_2 c$ where $b_1, b_2 \in \mathcal{B}_{3,2}$. By Case 12, $b_2c \in \mathcal{S}_{4, \infty}(\mathcal{C}_3)$. By Cases 6, 9, 11, and 12, we have $bc \in \mathcal{S}_{4,\infty}(\mathcal{C}_3 )$.

    \item[Case 20] Let  $b_1 \in \mathcal{B}_{3,1}$ and $b_2 \in \mathcal{B}_{3,3}$, then either:
    \begin{enumerate}
        \item $b_1b_2 = S_2^{-\epsilon} S_2^{\epsilon} S_1^{-\epsilon}S_2^{\epsilon}= S_1^{-\epsilon}S_2^{\epsilon} \in \mathcal{B}_{3,2}$, by Equation \ref{eqn:observe7}, where $\epsilon = \pm 1$. 
        \item $b_1b_2 = S_1^{-\epsilon_1} S_1^{\epsilon_1} S_2^{\epsilon_2}S_1^{\epsilon_3}= S_2^{\epsilon_2}S_1^{\epsilon_3} \in \mathcal{B}_{3,2}$,  by Equation \ref{eqn:observe7}, where $\epsilon_1, \epsilon_2, \epsilon_3 = \pm 1$.
         \item $b_1b_2 = S_2^{\epsilon} S_2^{\epsilon} S_1^{-\epsilon}S_2^{\epsilon}$. By Equations \ref{eqn:lemma2twist4} and \ref{eqn:lemma2negtwist4}, $b_1b_2$ is a linear combination of elements of the form $S_2^{\pm}S_1^{-\epsilon}S_2^{\epsilon},  S_1^{-\epsilon}S_2^{\epsilon}$, and $ U_2 S_1^{-\epsilon}S_2^{\epsilon}$. By Equations \ref{eqn:braidrelation}, \ref{eqn:braidrelation1}, \ref{eqn:lemma2twist} and \ref{eqn:lemma2negtwist}, $b_1b_2 \in \mathcal{S}_{4,\infty}(\mathcal{C}_3 \cup \mathcal{B}_3 )$.
        \item $b_1b_2 = S_1^{\epsilon_1} S_1^{\epsilon_1} S_2^{\epsilon_2}S_1^{\epsilon_3}$. By Equations \ref{eqn:lemma2twist4} and \ref{eqn:lemma2negtwist4}, $b_1b_2$ is a linear combination of elements of the form $S_1^{\pm 1}S_2^{\epsilon_2}S_1^{\epsilon_3},  S_2^{\epsilon_2}S_1^{\epsilon_3}$, and $ U_1S_2^{\epsilon_2}S_1^{\epsilon_3}$. By Equations \ref{eqn:lemma2twist} and \ref{eqn:lemma2negtwist}, $b_1b_2 \in \mathcal{S}_{4,\infty}(\mathcal{C}_3 \cup \mathcal{B}_3 )$. 
        
         \item $b_1b_2 = S_2^{\epsilon_1} S_1^{\epsilon_1} S_2^{\epsilon_2}S_1^{\epsilon_3} =  S_1^{\epsilon_2} S_2^{\epsilon_1} S_1^{\epsilon_1} S_1^{\epsilon_3}$, by Equations \ref{eqn:braidrelation} and \ref{eqn:braidrelation1} where $\epsilon_1, \epsilon_2, \epsilon_3 = \pm 1$. By Equations \ref{eqn:observe7}, \ref{eqn:lemma2twist4} and \ref{eqn:lemma2negtwist4}, $b_1b_2$ is a linear combination of elements of the form $S_1^{\epsilon_2} S_2^{\epsilon_1} S_1^{\pm 1}, S_1^{\epsilon_2} S_2^{\epsilon_1}$, and $S_1^{\epsilon_2} S_2^{\epsilon_1} U_1$. By Equations \ref{eqn:observe6} and  \ref{eqn:lemma2negtwist2}, $b_1b_2 \in \mathcal{S}_{4,\infty}(\mathcal{C}_3 \cup \mathcal{B}_3 )$.

         \item $b_1b_2 = S_2 S_1^{-1} S_2S_1= S_2 S_2S_1S_2^{-1}$ by Equation \ref{eqn:braidrelation1}. By Equation \ref{eqn:lemma2negtwist4} $b_1b_2$ is a linear combination of elements of the form $S_2^{\pm 1}S_1S_2^{-1}, S_1S_2^{-1}, U_2S_1S_2^{-1}$. By Equations \ref{eqn:braidrelation1} and \ref{eqn:lemma2twist}, $b_1b_2 \in \mathcal{S}_{4,\infty}(\mathcal{C}_3 \cup \mathcal{B}_3 )$. 
         
           \item $b_1b_2 = S_2 S_1^{-1} S_2^{-1}S_1^{-1}=S_2  S_2^{-1}S_1^{-1}S_2^{-1}=S_1^{-1}S_2^{-1} \in \mathcal{B}_{3,2}$ by Equations \ref{eqn:braidrelation1} and \ref{eqn:observe7}.
           
          \item $b_1b_2 = S_2 S_1^{-1} S_2^{-1}S_1=S_2 S_2S_1^{-1} S_2^{-1}$ by Equation \ref{eqn:braidrelation1}. By Equation \ref{eqn:lemma2negtwist4} $b_1b_2$ is a linear combination of elements of the form $S_2^{\pm 1}S_1^{-1} S_2^{-1}, S_1^{-1} S_2^{-1}$, and $U_2S_1^{-1} S_2^{-1}$. By Equations \ref{eqn:observe6} and \ref{eqn:braidrelation1}, $b_1b_2 \in \mathcal{S}_{4,\infty}(\mathcal{C}_3 \cup \mathcal{B}_3 )$. 
          
          \item $b_1b_2 = S_2^{-1} S_1 S_2S_1=S_2^{-1}S_2 S_1 S_2=S_1 S_2 \in \mathcal{B}_{3,2}$, by Equations \ref{eqn:braidrelation} and \ref{eqn:observe7}.
          
            \item $b_1b_2 = S_2^{-1} S_1 S_2^{-1}S_1^{-1}=S_2^{-1} S_2^{-1}S_1^{-1}S_2$. By Equation \ref{eqn:lemma2twist4} $b_1b_2$ is a linear combination of elements of the form $S_2^{\pm 1}S_1^{-1}S_2, S_1^{-1}S_2, U_2S_1^{-1}S_2$.  By Equations \ref{eqn:braidrelation1} and \ref{eqn:lemma2negtwist}, $b_1b_2 \in \mathcal{S}_{4,\infty}(\mathcal{C}_3 \cup \mathcal{B}_3 )$.

              \item $b_1b_2 = S_2^{-1} S_1 S_2S_1^{-1}= S_2^{-1} S_2^{-1}S_1 S_2$. By Equation \ref{eqn:lemma2twist4} $b_1b_2$ is a linear combination of elements of the form $S_1 S_2, S_2^{\pm 1}S_1 S_2$, and $U_2S_1 S_2$. By Equations \ref{eqn:braidrelation1} and \ref{eqn:observe6}, $b_1b_2 \in \mathcal{S}_{4,\infty}(\mathcal{C}_3 \cup \mathcal{B}_3 )$.

                \item $b_1b_2 = S_2 S_1^{-1} S_2S_1^{-1}$. By Lemma \ref{lemma:forb34argument}, $b_1b_2 \in \mathcal{S}_{4, \infty}(\mathcal{C}_3 \cup \mathcal{B}_3)$.
              \item $b_1b_2 = S_2^{-1} S_1 S_2^{-1}S_1$. By Lemma \ref{lemma:forb34argument}, $b_1b_2 \in \mathcal{S}_{4, \infty}(\mathcal{C}_3 \cup \mathcal{B}_3)$.
               \item $b_1b_2 = S_1^{-1} S_2 S_1^{-1}S_2$. By Lemma \ref{lemma:forb34argument}, $b_1b_2 \in \mathcal{S}_{4, \infty}(\mathcal{C}_3 \cup \mathcal{B}_3)$.

               \item $b_1b_2 = S_1^{-1} S_2^{-1} S_1S_2^{-1}= S_2S_1^{-1} S_2^{-1} S_2^{-1}$. By Equation \ref{eqn:lemma2twist4} $b_1b_2$ is a linear combination of elements of the form $S_2S_1^{-1} S_2^{\pm 1}, S_2S_1^{-1},$ and $S_2S_1^{-1} U_2$. By Equations \ref{eqn:braidrelation1} and \ref{eqn:lemma2negtwist2}, $b_1b_2 \in \mathcal{S}_{4, \infty}(\mathcal{C}_3 \cup \mathcal{B}_3)$.

               \item $b_1b_2 = S_1 S_2 S_1^{-1}S_2= S_2^{-1}S_1 S_2 S_2$.  By Equation \ref{eqn:lemma2negtwist4} $b_1b_2$ is a linear combination of elements of the form $S_2^{-1}S_1 S_2^{\pm 1}, S_2^{-1}S_1$, and $S_2^{-1}S_1 U_2$. By Equations \ref{eqn:braidrelation1} and \ref{eqn:lemma2twist2}, $b_1b_2 \in \mathcal{S}_{4, \infty}(\mathcal{C}_3 \cup \mathcal{B}_3)$.

               \item $b_1b_2 = S_1 S_2^{-1} S_1S_2^{-1} \in \mathcal{B}_{3,4}$.

    \end{enumerate}
   
    \item[Case 21] Let $b_1, b_2 \in \mathcal{B}_{3,2}$. Then $b_1b_2 = b_3b_4b_2$ where $b_3, b_4 \in \mathcal{B}_{3,1}$. By Case 3 we have three cases to consider
    \begin{enumerate}
        \item $b_4b_2 \in \mathcal{B}_{3,1}$. By Case 2 we have $b_1b_2 \in \mathcal{S}_{4,\infty}(\mathcal{B}_3 \cup \mathcal{C}_3)$.
        \item $b_4b_2$ is a linear combination of elements in $\mathcal{B}_{3,1}, \mathcal{B}_{3,2},$ and $\mathcal{C}_{3,3}$. By Cases 2, 3, and 7 we have $b_1b_2 \in \mathcal{S}_{4,\infty}(\mathcal{B}_3 \cup \mathcal{C}_3)$.
        \item $b_2b_2 \in \mathcal{B}_{3,3}$. By Case 20  we have $b_1b_2 \in \mathcal{S}_{4,\infty}(\mathcal{B}_3 \cup \mathcal{C}_3)$.
    \end{enumerate}

    \item[Case 22] Let $b_1 \in \mathcal{B}_{3,1}$, $b_2 \in \mathcal{B}_{3,4}$, then either:

    \begin{enumerate}
        \item $b_1b_2 = S_1^{-1} S_1 S_2^{-1}S_1 S_2^{-1} = S_2^{-1}S_1 S_2^{-1} \in \mathcal{B}_{3,3}$.
        \item $b_1b_2 = S_2 S_1 S_2^{-1}S_1 S_2^{-1}= S_1^{-1}S_2 S_1 S_1 S_2^{-1}$. By Equation \ref{eqn:lemma2negtwist4}, $b_1b_2$ is a linear combination of elements of the form
        $S_1^{-1}S_2 S_2^{-1}, S_1^{-1}S_2 S_1 S_2^{-1}, S_1^{-1}S_2 S_1^{-1} S_2^{-1}$ and $S_1^{-1}S_2 U_1 S_2^{-1}$. By Cases 3, 9, and 20, $b_1b_2 \in \mathcal{S}_{4,\infty}(\mathcal{B}_3 \cup \mathcal{C}_3)$.
        \item $b_1b_2 = S_1 S_1 S_2^{-1}S_1 S_2^{-1}$. By Equation \ref{eqn:lemma2negtwist4}, $b_1b_2$ is a linear combination of elements of the form...
        \item $b_1b_2 = S_2^{-1} S_1 S_2^{-1}S_1 S_2^{-1}$. By Lemma \ref{lemma:forb34argument}, $b_1 b_2$ is a linear combination of elements of the form $b S_2^{-1}$ where $b \in \mathcal{B}_3 \cup \mathcal{C}_3$. Note that by Equation \ref{eqn:lemma2twist4}, $(S_1 S_2^{-1}S_1 S_2^{-1})S_2^{-1}$ is a linear combination of elements of the form $S_1 S_2^{-1}S_1 S_2^{\pm 1}, S_1 S_2^{-1}S_1,$ and $S_1 S_2^{-1}S_1 U_2$ so by Cases 8 and 20 we have that if $b\in \mathcal{B}_{3,4}$, then $b S_2^{-1} \in  \mathcal{S}_{4,\infty}(\mathcal{B}_3 \cup \mathcal{C}_3)$. Therefore, by applying cases 2, 3, 4, 5, 7, 10, and 20 to the remaining sub-cases, we have $b_1b_2 \in \mathcal{S}_{4,\infty}(\mathcal{B}_3 \cup \mathcal{C}_3)$.
    \end{enumerate}

    \item[Case 23] Let  $b_1 \in \mathcal{B}_{3,2}$ and $b_2 \in \mathcal{B}_{3,3}$, then $b_1 b_2 = b_3b_4b_2$, where $b_3, b_4 \in \mathcal{B}_{3,1}$. By Case 20, $b_4b_2 \in \mathcal{S}_{4,\infty}(\mathcal{B}_3 \cup \mathcal{C}_3)$. Therefore, by cases 2, 3, 4, 5, 7, 10, 20, and 22, $b_1b_2 \in \mathcal{S}_{4,\infty}(\mathcal{B}_3 \cup \mathcal{C}_3)$.

 \item[Case 24] Let $b_1, b_2 \in \mathcal{B}_{3,3}$, then $b_1 b_2 = b_3b_4b_2$, where $b_3 \in \mathcal{B}_{3,1}$ and $b_4 \in \mathcal{B}_{3,2}$. By Case 23, $b_4b_2 \in \mathcal{S}_{4,\infty}(\mathcal{B}_3 \cup \mathcal{C}_3)$. Therefore, by cases 2, 3, 4, 5, 7, 10, 20, and 22, $b_1b_2 \in \mathcal{S}_{4,\infty}(\mathcal{B}_3 \cup \mathcal{C}_3)$.

   \item[Case 25] Let  $b_1 \in \mathcal{B}_{3,2}$ and $b_2 \in \mathcal{B}_{3,4}$, then $b_1 b_2 = b_3b_4b_2$, where $b_3, b_4 \in \mathcal{B}_{3,1}$. By Case 22, $b_4b_2 \in \mathcal{S}_{4,\infty}(\mathcal{B}_3 \cup \mathcal{C}_3)$. Therefore, by cases 2, 3, 4, 5, 7, 10, 20, and 22, $b_1b_2 \in \mathcal{S}_{4,\infty}(\mathcal{B}_3 \cup \mathcal{C}_3)$.

   \item[Case 26] Let  $b_1 \in \mathcal{B}_{3,3}$ and $b_2 \in \mathcal{B}_{3,4}$, then $b_1 b_2 = b_3b_4b_2$, where $b_3 \in \mathcal{B}_{3,1}$ and $b_4 \in \mathcal{B}_{3,2}$. By Case 25, $b_4b_2 \in \mathcal{S}_{4,\infty}(\mathcal{B}_3 \cup \mathcal{C}_3)$. Therefore, by cases 2, 3, 4, 5, 7, 10, 20, and 22, $b_1b_2 \in \mathcal{S}_{4,\infty}(\mathcal{B}_3 \cup \mathcal{C}_3)$.

  \item[Case 27] Let  $b_1, b_2 \in \mathcal{B}_{3,4}$, then $b_1 b_2 = b_3b_4b_2$, where $b_3 \in \mathcal{B}_{3,1}$ and $b_4 \in \mathcal{B}_{3,3}$. By Case 26, $b_4b_2 \in \mathcal{S}_{4,\infty}(\mathcal{B}_3 \cup \mathcal{C}_3)$. Therefore, by cases 2, 3, 4, 5, 7, 10, 20, and 22, $b_1b_2 \in \mathcal{S}_{4,\infty}(\mathcal{B}_3 \cup \mathcal{C}_3)$.
\end{enumerate}

\end{proof}

\section{Algebraic Perspective on Longer Skein Relations}\label{Sec:longerrelations}

Our paper is the first comprehensive study of cubic skein modules. Thus, we found it important
to discuss in this section all known approaches to cubic skein relations, including the
Yang-Baxter operators with cubic minimal polynomials and cubic Hecke algebras (and cubic Temperley-Lieb algebras).
These approaches are very promising but not yet completely developed as a tool to analyze cubic skein modules.
The first technique, developed in 1986, was the Akutsu-Wadati invariants, which appear alongside the study of cubic skein relations in \cite{AW}.\footnote{V. F. R. Jones predicted that the Akutsu-Wadati
invariants are obtained using classical knot invariants (e.g. HOMFLYPT polynomial) of cables (or their linear combinations) of links. The paper by Kanenobu seems to be a step in this direction \cite{Kan}. See Figure \ref{fig:8-ended tangles}.

} Even before that, Jones \cite{Jones}, Turaev \cite{Tur1} showed how invariants of links can be obtained from solutions to Yang-Baxter operators; the minimal polynomial of the Yang-Baxter operator leads to a skein relation (see also very influential paper by Reshetikhin and Turaev \cite{RT}). This work intertwined with  the use of quantum groups and with Witten's work \cite{Witten} which we will not get into in this paper. 

\subsection{Cubic Hecke Algebras}

Special cases of the general cubic skein relation have been studied in the classical case when  $M=S^3$. Two approaches  have been proposed: (a) using solutions of the Yang-Baxter equation to construct polynomial invariants (see Subsection \ref{SS11.3}), and (b) using cubic Hecke algebras.  In this subsection, we discuss the second approach, that is, defining invariants by constructing Markov trace functions on cubic Hecke algebras\footnote{Here we recall that a  Markov trace on the group algebra of the braid group is a functional 
$t$ satisfying 
\begin{enumerate}
\item $t(xy) = t(yx),  x, y \in B_{n}, $
\item $t(xb_n) = zt(x), t(xb^{-1}_{n}) = \bar{z}t(x)  x \in B_{n}$, 
\end{enumerate}
with  $z,\bar{z}\in \mathbb{C^{*}}$, called parameters.
}. 
In 1957, Coxeter \cite{Cox} was studying the quotient of the braid group $B_n$ modulo the $k$-th power of generators, that is, $B_n/(\sigma_i^{k}),$ and showed that this quotient group is finite if and only if $\frac{1}{n}+\frac{1}{k}+\frac{1}{2}>1$. In particular, for $k=3$, the group is finite if and only if $n\leq 5$. The cubic skein relation can be thought of as a deformation of the relation  $\sigma_i^3=1$.
In 1995, Louis Funar studied the cubic Hecke algebra
$$H(Q,n)=\mathbb{C}[B_{n}]/(Q(b_{j})), j=1,\ldots,n-1,$$
where $Q(x)$ is some cubic polynomial with $Q(0) \neq 0$ \cite{Fun}. Consider the quotient,
$$K_{n}(\gamma)=H(Q_{\gamma},n) / I_{n},$$
where $Q_{\gamma}(x)=x^{3}-\gamma$ and $I_{n}$\footnote{Note that, in $K_n(\gamma)$ one can give an ordering on the representatives in form of the elements in the free monoid generated by $b_{i}$'s and their inverses.} is a two-sided ideal generated by $$b_{i+1}b_{i}^{2}b_{i+1}+b_{i}b_{i+1}^{2}b_{i}+b_{i}^{2}b_{i+1}b_{i}+b_{i}b_{i+1}b_{i}^{2}+b_{i}^{2}b_{i+1}^{2}+b_{i+1}^{2}b_{i}^{2}+\gamma b_{i}+\gamma b_{i+1}$$ for each $n$ and $i=1,\ldots,n-2$. He proved that $K_{n}(\gamma)$ is finite dimensional.

Dabkowski and Przytycki proved that the Montesinos-Nakanishi $3$-move conjecture does not hold \cite{DP1,DP2}. Przytycki and Tsukamoto \cite{PTs} considered a deformation of the relation $b_{i}^{3}=1$ to analyze cubic skein modules for $b_{\infty}=0$. Then,
Bellingeri and Funar wrote another paper \cite{BF} in which they extended the idea of the first paper \cite{Fun} and searched for invariants obtained from cubic Hecke algebras that can be recursively determined by the unknot, and they call such a system of skein relations complete. More precisely, they constructed a deformation of the quotient of the cubic Hecke algebra in \cite{Fun}, and considered the Markov trace on the deformations. They found two invariants (also, see the next paragraph) whose system of skein relations are complete. Furthermore, the invariants distinguish all knots with minimal crossing number at most $10$ with the same HOMFLYPT polynomial.

However, around 2004-05, it was pointed out by Orevkov that there is a gap in a proof of the invariance of the trace under Markov moves in \cite{BF} which originates in \cite{Fun}. According to Orevkov \cite{Ore}, the main idea in \cite{Fun, BF} is correct, i.e. to define a Markov trace on $K_{n}$, it is sufficient to check a finite number of identities though the number of them is much bigger than that indicated in \cite{Fun, BF} and the algorithm of the computation is much more complicated. Orevkov also pointed out explicitly why the invariant in the main results in \cite{Fun, BF} are not well-defined. Basically in \cite{Fun} and \cite{BF}, the authors only considered one type of ambiguity when computing the Markov trace in different order. Orevkov completed the set of relations needed to quotient by in order to ensure well-definedness, as mentioned in Remark 2.11 in \cite{Ore}.

Secondly, in $2008$, Marin noticed that when $k$, the base ring, is a field of characteristic $0$, the tower of the algebra $K_{n}(1)$ collapsed, meaning $K_{n}(1)=0$ for $n\geq 5$. In \cite{CM}, they discussed mainly the cases when  $k=\mathbb{Z}$ and $k=\mathbb{Z}_{2}$, in which $K_{n}(1)$ is no longer trivial.

\begin{figure}[ht]
    \centering
    \includegraphics[scale=.48]{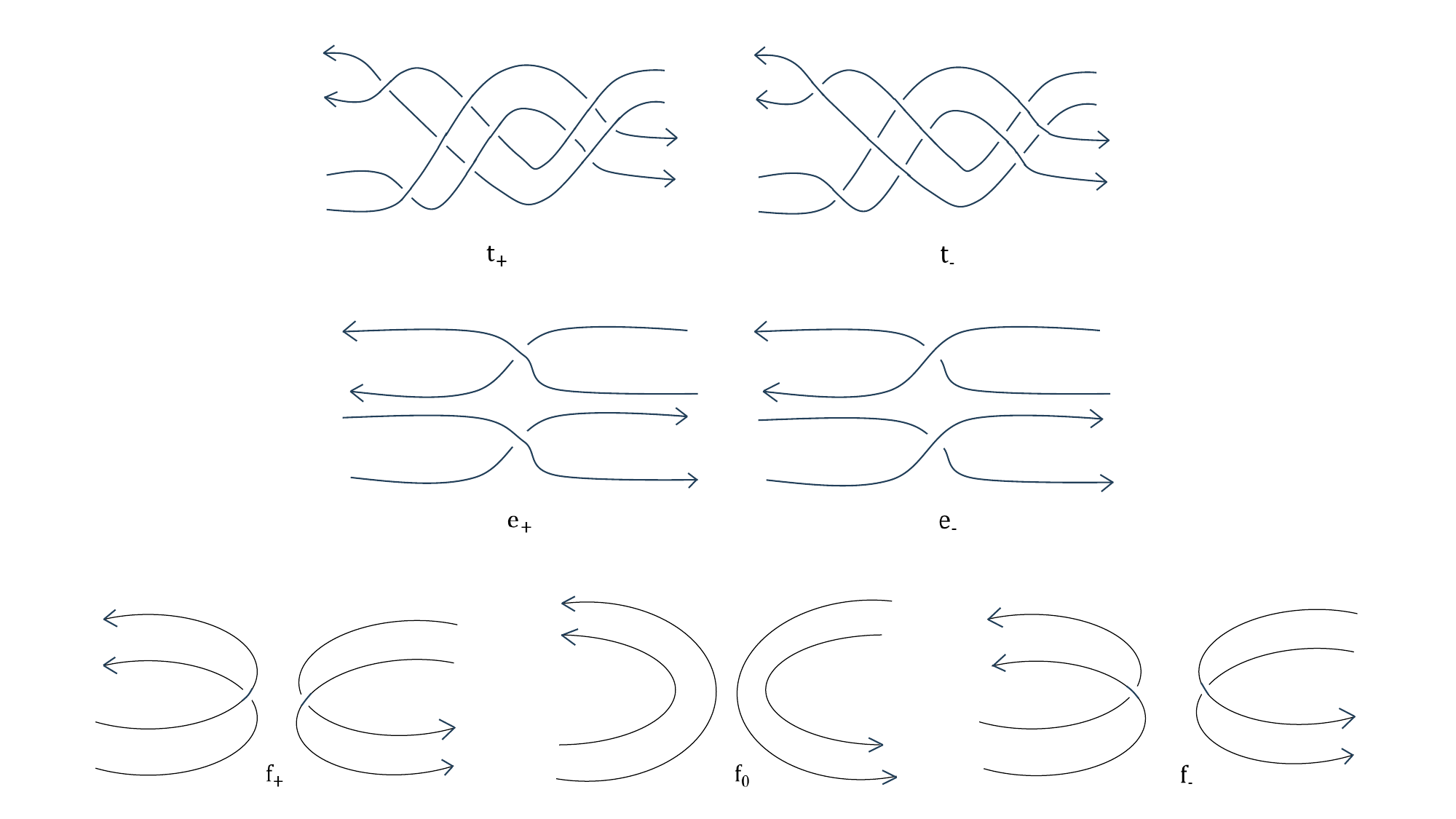}
    \caption{8-ended tangles from \cite{Kan}.}
    \label{fig:8-ended tangles}
\end{figure}

\begin{figure}[ht]
    \centering
    \includegraphics[scale=.28]{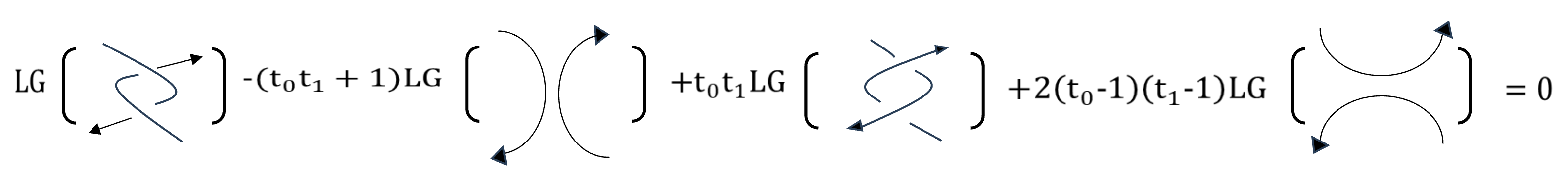}
    \caption{The skein relation of the LG polynomial.}
    \label{fig:LG}
\end{figure}

\subsection{Links-Gould Invariant}

The Links-Gould invariant was developed by David DeWit and Louis H. Kauffman when Kauffman visited Gould in Australia in the summer of 1997. The first paper on the so-called Links-Gould invariant resulted from this visit and David DeWit obtained his PhD thesis from this work. Kauffman and Gould began this project from a conversation that they had at the Fields Institute in Toronto a year prior to Kauffman’s visit to Australia. See these recent papers on the Links-Gould invariant: \cite{GHKKW,GHKKST}. 

The Links-Gould (LG) invariant is a two-variable polynomial invariant of oriented links, derived from a one-parameter family of four-dimensional representations of the quantum superalgebra $U_q(gl(2|1));$ see \cite{LG}. It is known that this invariant satisfies a cubic skein relation. De Wit, Kauffman and Links \cite{DKL} showed that the LG invariant is powerful by evaluating some links. In particular, De Wit \cite{DeWit} showed that the invariant is complete for all prime knots of up to 10 crossings. Moreover, the invariant is complete up to mutation, for all prime knots (including reflections) with less or equal to 11 crossings \cite{DL}. However, it fails to distinguish certain non-mutant pairs of prime knots with 12 crossings. Additionally, the mutants within the class of prime knots with 11 or 12 crossings can be classified using the invariant. In the same paper, the authors also study chiral prime knots with at most 12 crossings. Further, they demonstrate  that every mutation-insensitive link invariant
fails to distinguish the chirality of some prime knots with 14 crossings. We list several properties of the LG invariant below:
\begin{enumerate}
    \item $LG(trivial \ knot)=1.$
    \item It vanishes for split links.
    \item $LG(L \# L')=LG(L) LG(L').$ Note that, here the formula does not depend on the choice of the point where the connected sum operation is made.    
    \item $LG(L^{*};t_{0},t_{1})=LG(L;t^{-1}_{0},t^{-1}_{1}),$ where $ L^{*} $is the reflection of $ L.$
    
    \item $LG(L;t_{0},t_{1})= LG(L;t_{1},t_{0}). $
\end{enumerate}

In \cite{I}, let $\mathcal{L_{A}}$ be the set of all oriented algebraic links in $S^{3}$ up to ambient isotopy and $\mathcal{T_{A}}$ be the set of all oriented algebraic 2-tangles in $B^{3}$ up to relative ambient isotopy, and let $\mathcal{S_{A}^{T}}$($\mathcal{S_{A}^{L}}$) denote the skein module of $S^{3}$ as the free module generated by $\mathcal{T_{A}}$($\mathcal{L_{A}}$) modulo the local skein relations generated by the LG cubic skein relations. Atushi Ishii showed that $\mathcal{S_{A}^{T}}$ is generated by the $-1,0,1,\infty,\frac{1}{2}$ tangles with all possible orientations and $\mathcal{S_{A}^{L}}$ is generated by the trivial knot. In particular, he proved that the LG invariant of all algebraic links can be calculated recursively. In the same paper, Ishii also provided a formula for the LG invariant of 2-bridges links.

 Przytycki \cite{Przcan}, and Lickorish and Lipson \cite{LL}  showed that if two knots are mutant knots, their 2-cable links share the same HOMFLYPT and Kauffman polynomials. Stoimenow \cite{Sto} found the first examples of four pairs of non-mutant 12-crossing knots whose 2-cable links share the same HOMFLYPT polynomials. They are
$$\{12_{341},12_{627}\},\{12_{1305},12_{1872}!\},\{12_{1378},12_{1704}\},\{12_{1423},12_{1704}\},$$
where $12_{1872}!$ is the mirror image of $12_{1872}$, and $$\{12_{1378},12_{1423}\}$$ is a mutant pair. As mentioned above, it is known that mutant links share the same LG polynomial \cite{DKL}, while De Wit and Links \cite{DL} searched for non-mutant pair of prime knots sharing the same LG polynomial. Surprisingly, the above four pairs are examples of non-mutant knots whose LG polynomials are the same for each pair. This fact motivated Taizo Kanenobu to consider the relation between the HOMFLYPT polynomial of a 2-cable link and the LG polynomial. In \cite{Kan}, Kanenobu discovered a skein relation for the HOMFLYPT polynomial of 2-cable links, which is similar to the skein relation for LG polynomial.

The following theorem motivates the comparison of the skein relation of the HOMFLYPT polynomial applied to 2-cable links (see Figure \ref{fig:8-ended tangles}) and the skein relation of the LG polynomial (see Figure \ref{fig:LG}).

\begin{theorem}[\cite{Kan}]
    $v^{-5}P(t_{+})+v^{5}P(t_{-})=v^{-3}P(e_{+})+v^{3}P(e_{-})+(v^{-3}P(f_{+})+(v^{-1}+v)P(f_{0})+v^{3}P(f_{-}))z^{2}.$
\end{theorem}

\subsection{The Yang-Baxter Equation} \label{SS11.3}

\

During 1986-88, Miki Wadati and his coauthors published a series of papers (e.g. \cite{AW}) about the connection between exact solvable models in physics and knot theory. They focused on exact solvable models in quantum mechanics, statistical mechanics and inverse scattering problems, and pointed out the key ingredient is the Yang-Baxter equation, shared among these models. They used the solutions of the Yang-Baxter equation from those models to obtain representations of the braid group and constructed a family of link invariants. The representation they use is related to a half-integer spin, say $s$, with $N=2s+1$. When $s=\frac{1}{2}$, it gives the Jones polynomial. Furthermore, they developed a method to generalize the 1-variable invariant to a 2-variable invariant by considering a higher dimensional representation of the braid group and constructed a trace function. We provide the main idea about their construction below.

\

Let $S:V\otimes V \rightarrow V\otimes V$ be a solution of the Yang-Baxter equation, obtained from the exact solvable models after taking the limit of the spectral parameter. For each integer $n>1$, this leads to a representation of the braid group $B_{n}$ on $V^{\otimes n}$ by $$b_{i}\mapsto Id_{V}^{\otimes (i-1)}\otimes S \otimes Id_{V}^{\otimes n-i-1}.$$ We denote this map by $\phi_{n}$, and the image of $b_{i}$ by $g_{i}$. From this representation, they construct a trace function which is a link invariant in one variable. For the 2-variable generalization, they use composite strings. More precisely, they consider the braid group $B_{(N-1)n}$ and group together $(N-1)$ strings as one, denoted by $\beta_{i}$. $B_{(N-1)n}$ has a representation $\phi_{(N-1)n}$; they use this representation and provide another representation of $B_{n}$ on $V^{\otimes (N-1)n}$ as follows. Each $\beta_{i}$ is sent to $$G_{i}=P^{(N)}_{(i-1)k+1}P^{(N)}_{ik+1} \overline{G}^{N}_{i}P^{(N)}_{(i-1)k+1}P^{(N)}_{ik+1},$$ where $\overline{G}^{N}_{i}=g_{i}^{(1)}g_{i}^{(2)}\cdot \cdot \cdot g_{i}^{(N-1)}$, with $g_{i}^{(l)}=g_{ik+1-l}g_{ik+2-l}\cdot \cdot \cdot g_{(i+1)k-l}$ and $P_{i}^{N}=P_{i}^{(N-1)}h^{(N)}_{i+N-3}P_{i}^{(N-1)}$ a projector recursively defined, where $P^{(2)}_{i}=1$ and $h_{j}^{(N)}=\frac{\tau_{N-2}}{\tau_{N-1}}(\frac{t^{N-2}}{\tau_{N-2}}+g_{j})$ and $\tau_{m}=1+t+t^{2}+\cdot \cdot \cdot + t^{m-1}$. We can easily visualize $G_{i}$ using Figure \ref{fig:gi}. With the new representation of $B_{n}$, following Ocneanu's idea (which gives the HOMFLYPT polynomial as a 2-variable generalization of the Jones polynomial), they also build a trace function for each $N$, a link invariant of two variables.

\begin{figure}[ht]
    \centering
    \includegraphics[scale=.5]{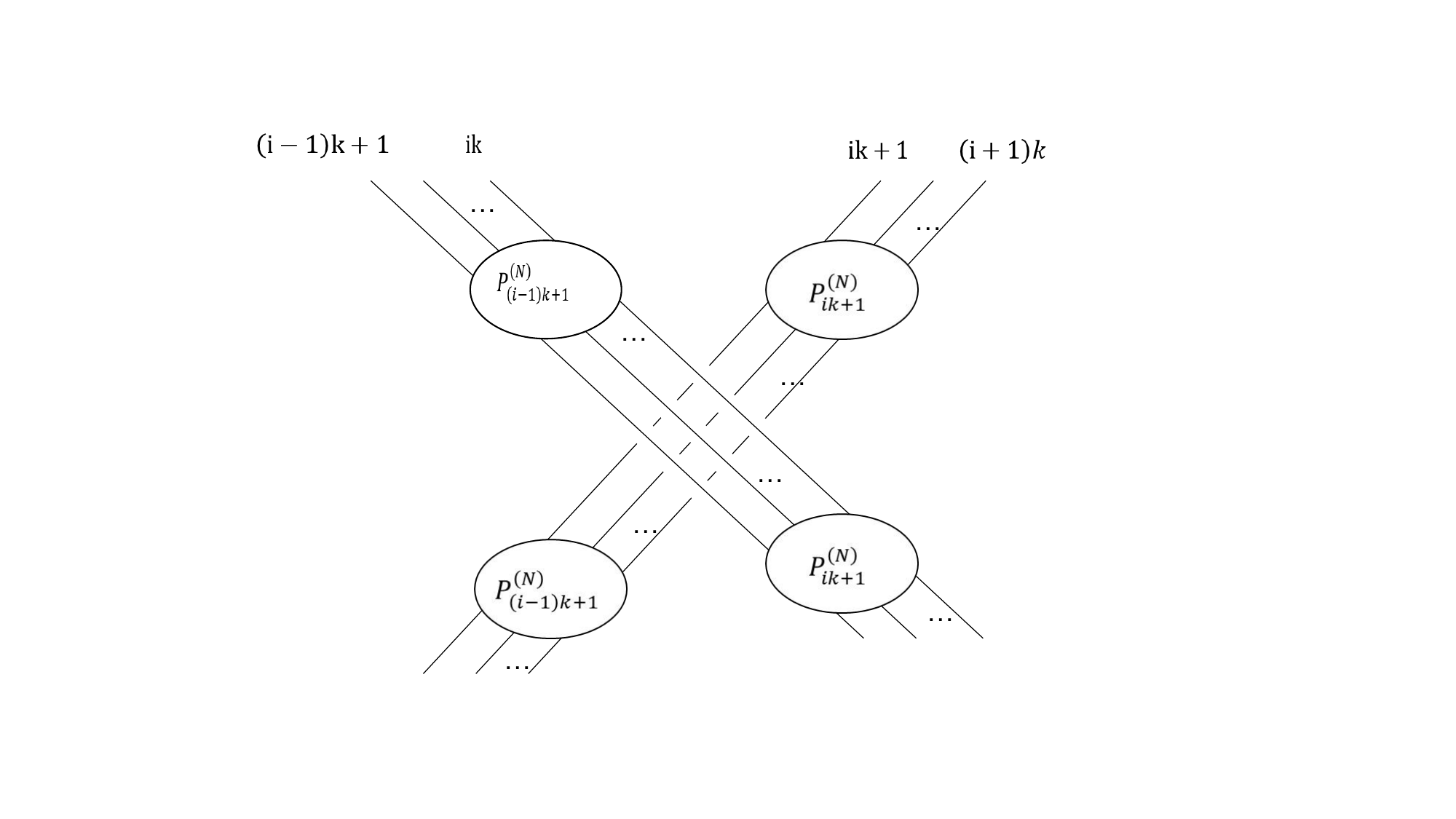}
    \caption{$G_{i}.$}
    \label{fig:gi}
\end{figure}

\section{Future Research}\label{sec:future}

In this section, we list the conjectures and questions provided throughout this paper and present additional ones.

\begin{named}{Conjecture \ref{mono}}
Let $f:B^3 \to M^3$ be an embedding of a disk into manifold then the induced map on skein modules $f_*:\mathcal S_{4,\infty}(B^3,R) \to \mathcal{S}_{4,\infty}(M^3,R)$ is a monomorphism. As before, $R = \mathbb Z[ b_0^{\pm 1}, b_1,b_2,b_3^{\pm 1},b_{\infty}^{\pm 1},a^{\pm 1}]$. 
\end{named}

If the conjecture holds, then it is natural to consider the ring in the cubic skein module to be equal to $\mathcal{S}_{4,\infty}( \mathbb{R}^3)$.

\begin{named}{Conjecture \ref{Conj:algorithm}}
    Let $[n_1, n_2, \dots, n_r]$ and $\epsilon [n_1, n_2, \dots, n_r]$ be two diagrams of rational links where $\epsilon = 1$ if $r$ is odd, and $\epsilon = -1$. As computed in the Rational Tangles Algorithm, the relation
    \[ [n_1, n_2, \dots, n_r] - \epsilon [n_1, n_2, \dots, n_r] \] is divisible by the Hopf Relation. Therefore, the Rational Tangles Algorithm does not produce any new relations above the Hopf Relation.
\end{named}

\begin{named}{Question \ref{Qu:linearlyindep}}\cite{PTs}
     If $b_{\infty}=0$ and  $(b_0 b_1 -b_2b_3)=0$, does it imply that the $t^i$'s are linearly independent in the Kauffman bracket skein module?
\end{named}        

\begin{named}{Conjecture \ref{Conj:invariant}}
    Let $\mathcal{C}(D)$ be defined on diagrams of framed links with the following  axioms, framing relations, and cubic skein relation;

\begin{enumerate}
    \item $\mathcal{C}(\bigcirc) = 1$,
    \item $\mathcal{C}(D  \sqcup \bigcirc) = t \mathcal{C}(D)$,
    \item $\mathcal{C}(\vcenter{\hbox{\includegraphics[scale=1]{framingrelation2.pdf}}}) = A^{-3} \mathcal{C}(\vcenter{\hbox{\includegraphics[scale=1]{framingrelation1.pdf}}}) \quad$ and $\quad \mathcal{C}(\vcenter{\hbox{\includegraphics[scale=1]{framingrelation3.pdf}}}) = A^{3} \mathcal{C}(\vcenter{\hbox{\includegraphics[scale=1]{framingrelation1.pdf}}})$,
    \item $-A X \mathcal{C}( D_3) + Y \mathcal{C}(D_2) + (-A Y-A^{5} Z)\mathcal{C}(D_1)+ X \mathcal{C}(D_0)+ Z \mathcal{C}(D_{\infty})=0$,
\end{enumerate}
where $t= (A^{-8}-1) \frac{X}{Z} -A^{-2}(A^{-4}-1)\frac{Y}{Z}+A^2.$ Let $\mathcal{L}$ be the set of framed links such that for each $L \in \mathcal{L}$, there exists a diagram $D$ of $L$ with $\mathcal{C}(D) \in$
\ \ $\mathbb{Z}[A^{\pm 1}, X^{\pm 1}, Y^{\pm 1}, Z^{\pm 1}]$ (such as $3$-algebraic links). Then $\mathcal{C}$ is an invariant of regular isotopy of $\mathcal{L}$.
\end{named}

\begin{named}{Question \ref{question:structurecubic}}
    Consider the cubic skein module, $\mathcal S_{4,\infty}(S^3)$ with skein relation

$$-A X  D_3 + Y D_2 + (-A Y-A^{5} Z) D_1+ X D_0+ Z D_{\infty}=0,$$

trivial knot relations, $\bigcirc=1$, $D  \sqcup \bigcirc = [(A^{-8}-1) \frac{X}{Z} -A^{-2}(A^{-4}-1)\frac{Y}{Z}+A^2] D$,
and framing relation $D^{(1)}=A^3 D$. What is the structure of this skein module?
\end{named}

\begin{conjecture}
    Consider the solid torus. The wrapping number of a link in the solid torus is determined by its value in the cubic skein module.
\end{conjecture}

\begin{conjecture}\label{7-coloring}
    If two links are equivalent in the cubic skein module, then they have the same number of Fox 7-colorings.
\end{conjecture}

The explanation and generalization of the last conjecture is given in the following subsection.

\subsection{Fox 7-colorings and 3-2 Moves} Conjecture \ref{7-coloring} is a special case of the following interesting question: what is the shortest skein relation for the number of Fox $n$-colorings. In fact, Jaeger and Jones apparently solved the problem (at least for prime $n$) a long time ago.
Fran{\c c}ois Jaeger told Przytycki \cite{Prz5} that he
knew how to form a short  skein relation
(of the type  $(\frac{p+1}{2},\infty)$) involving
 spaces of $p$-colorings. If $col_p(L)=|Col_p(L)|$ denotes the order of the space of Fox $p$-colorings  of the link $L$, then
among $p+1$ links $L_0,L_1,...,L_{p-1}$, and $L_{\infty}$, $p$ of them has the same order  $col_p(L)$ and one has its order $p$ times larger
\cite{Prz2}. This leads to the relation of type $(p,\infty)$.  The relation between the Jones polynomial (or the Kauffman bracket)
 and $col_3(L)$ has the form: $col_3(L)= 3|V(e^{\pi i/3})|^2$ and the formula relating the Kauffman polynomial and $col_5(L)$
has the form: $col_5(L)= 5|F(1,e^{2\pi i/5} + e^{-2\pi i/5})|^2$.  This seems to suggest that the connection discovered by Jaeger involved Gaussian sums.

Similarly as in the case of $5$-moves, not every link can be reduced via $7$-moves
to a trivial link. The $7$-move is however a combination of $(3, 2)$-moves\footnote{To be precise, a 7-move is a combination of a $(-3,-2)$ and a $(2, 3)$-move; compare
Figures \ref{32moveexamplea}.} which
might be sufficient for a reduction. We say that two links are $(3, 2)$-move
equivalent if there is a sequence of $\pm (2, 3)$-moves and their inverses, $\pm(3, 2)$-
moves, which converts one to the other

Recall that a $(3,2)$-move is made by replacing $D_3$ with $D_{-1/2}$ as shown below:

$$ \vcenter{\hbox{\begin{overpic}[scale = .8]{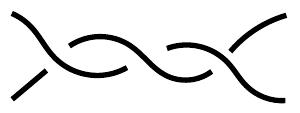}
	\put(50, 0){$D_3$}
\end{overpic} }} \xrightarrow{(3,2)-move} \vcenter{\hbox{\begin{overpic}[scale = .8]{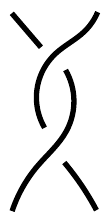}
	\put(10, -4){$D_{-1/2}$}
\end{overpic} }}$$ \ \\

Two useful moves that is a variation of the $(3,2)$-move is given in Figure \ref{32movevariationfigure}.
\begin{figure}[ht]
\centering
\begin{subfigure}{.49\textwidth}
     \centering
       $$ \vcenter{\hbox{\begin{overpic}[scale = 1]{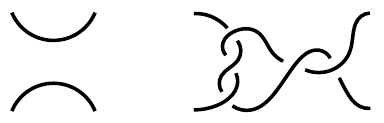}
	\put(62, 29){$\rightarrow$}
\end{overpic} }}$$
    \caption{$(2, 3)$-move on $D_0$}
        \label{32moveexamplea}
\end{subfigure}
   \begin{subfigure}{.49\textwidth}
     \centering
       $$ \vcenter{\hbox{\begin{overpic}[scale = 1]{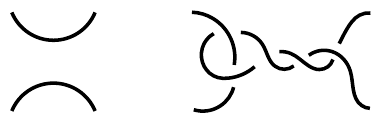}
	\put(62, 29){$\rightarrow$}
\end{overpic} }}$$
    \caption{$(-3, -2)$-move on $D_0$}
        \label{32moveexampleb}
\end{subfigure}
\caption{Variations of $(3,2)$-moves.} \label{32movevariationfigure}
\end{figure}

\begin{lemma}\cite{Prz5,Prz6}\label{D0toD7lemma}
\begin{enumerate}
\item[(1)]

The $7$-move is a combination of $(3,2)$-moves involving mirror image and inverses.
\item[(2)] $col_7 (D) = col_7(D_{(3,2)})$ where $D_{(3,2)}$ is $D$ with a $(3,2)$-move applied.
\end{enumerate}
\end{lemma}

\section{Acknowledgments}
R. P. B. was supported by Dr. Max R\"ossler, the Walter Haefner Foundation, and the ETH Z\"urich Foundation. D. I. acknowledges the support by the Australian Research Council grant DP210103136 and DP240102350. G. M-V. was supported by the Matrix-Simons travel grant and acknowledges the support of the National Science Foundation through Grant DMS-2212736. J. H. P. was partially supported by Simons Collaboration Grant for Mathematicians under grant 3637794. X. W. is partially supported by the National Natural Science Foundation of China (Grant No. 11901229, 12371029, 22341304 and W2412041).

\end{document}